\documentclass[ 10pt, american ]{article}

\usepackage[american]{babel}

\title{ Fractured Poroelastic Media in the Limit of Vanishing Aperture }

\author{
\textsc{Maximilian Hörl}\thanks{Institute of Applied Analysis and Numerical Simulation, University of Stuttgart, Pfaffenwaldring~57, D-70569 Stuttgart, Germany (\href{mailto:maximilian.hoerl@mathematik.uni-stuttgart.de}{maximilian.hoerl@mathematik.uni-stuttgart.de}, \href{mailto:christian.rohde@mathematik.uni-stuttgart.de}{christian.rohde@mathematik.uni-stuttgart.de}).} 
\and \textsc{Kundan Kumar}\thanks{Center for Modeling of Coupled Subsurface Dynamics, Department of Mathematics, University of Bergen, Allegaten 41, N-5007 Bergen, Norway (\href{mailto:kundan.kumar@uib.no}{kundan.kumar@uib.no}).} 
\and \textsc{Christian Rohde}\footnotemark[1]
}

\date{}

\newcommand{\headerleft}{Hörl $\cdot$ Kumar $\cdot$ Rohde}
\newcommand{\headerright}{ Fractured Poroelastic Media in the Limit of Vanishing Aperture }

\usepackage[utf8]{inputenc}
\usepackage[T1]{fontenc}

\usepackage{amsmath,amssymb,amsfonts,amsthm,mathtools}

\usepackage{bm}
\usepackage{stmaryrd}
\SetSymbolFont{stmry}{bold}{U}{stmry}{m}{n}

\usepackage{upgreek}
\usepackage{units}
\usepackage{scalerel}
\usepackage{csquotes}
\usepackage{enumitem}
\usepackage{subcaption}
\usepackage{longtable}
\usepackage{array}
\usepackage{arydshln}
\usepackage{multirow}
\usepackage{hyperref}
\usepackage[noabbrev,nameinlink,capitalize]{cleveref}
\usepackage[backend=biber, style=numeric, maxnames=10, giveninits=true]{biblatex}
\usepackage{appendix}

\usepackage[osf,sc]{mathpazo}
\usepackage[scaled=0.85]{beramono}

\usepackage[a4paper, margin=1.5in]{geometry}
\let\originalleft\left
\let\originalright\right
\renewcommand{\left}{\mathopen{}\mathclose\bgroup\originalleft}
\renewcommand{\right}{\aftergroup\egroup\originalright}

\theoremstyle{plain}
\newtheorem{definition}{Definition}
\newtheorem{assumption}[definition]{Assumption}
\newtheorem{notation}[definition]{Notation}
\newtheorem{proposition}[definition]{Proposition}
\newtheorem{corollary}[definition]{Corollary}
\newtheorem{lemma}[definition]{Lem\-ma}
\newtheorem{theorem}[definition]{Theorem}

\theoremstyle{remark}
\newtheorem{remark}[definition]{Remark}

\Crefname{notation}{Notation}{Notations}
\Crefname{subsection}{Section}{Sections}
\Crefname{assumption}{Assumption}{Assumptions}
\Crefname{paragraph}{Section}{Sections}

\crefname{equation}{}{}
\Crefname{equation}{}{}

\newcommand*{\dimq}[1]{{\tilde{#1}}}
\newcommand*{\refq}[1]{{#1}^\star}
\newcommand{\rmb}{{\mathrm{b}}}
\newcommand{\rmd}{{\operatorname{d}}}
\newcommand{\rmD}{{\mathrm{D}}}
\newcommand{\rmf}{{\mathrm{f}}}
\newcommand{\rmN}{{\mathrm{N}}}
\newcommand{\rmt}{{\mathrm{t}}}

\newcommand{\pmf}{{\pm\rmf}}

\newcommand{\fracfac}[2][]{#1\lceil #2 #1\rfloor}
\newcommand{\norm}[2][]{{#1\Vert #2 #1\Vert}}
\newcommand{\abs}[2][]{{#1\vert #2 #1\vert}}
\newcommand{\fraceps }{{\fracfac{\!\epsilon\! }}}

\newcommand*{\jump}[2][]{#1\llbracket #2 #1\rrbracket}
\newcommand{\eff }{{\mathrm{eff}}}

\newcommand{\normiii}[2][]{{#1\vert\kern-0.25ex #1\vert\kern-0.25ex #1\vert #2 
    #1\vert\kern-0.25ex #1\vert\kern-0.25ex #1\vert}}

\newcommand{\nablapar}{\nabla_{\!\smallpar}\hspace{0.5pt}}
\newcommand{\nablaperp}{\nabla_{\!\vct{N}}}

\renewcommand{\rho}{\varrho}
\renewcommand{\epsilon}{\varepsilon}

\let\temp\phi
\let\phi\varphi
\let\varphi\temp

\newcommand{\upgreekmap}[1]{%
  \ifx#1\alpha\upalpha
  \else\ifx#1\beta\upbeta
  \else\ifx#1\gamma\upgamma
  \else\ifx#1\delta\updelta
  \else\ifx#1\epsilon\upepsilon
  \else\ifx#1\zeta\upzeta
  \else\ifx#1\eta\upeta
  \else\ifx#1\theta\uptheta
  \else\ifx#1\iota\upiota
  \else\ifx#1\kappa\upkappa
  \else\ifx#1\lambda\uplambda
  \else\ifx#1\mu\upmu
  \else\ifx#1\nu\upnu
  \else\ifx#1\xi\upxi
  \else\ifx#1\pi\uppi
  \else\ifx#1\rho\uprho
  \else\ifx#1\sigma\upsigma
  \else\ifx#1\tau\uptau
  \else\ifx#1\phi\upphi
  \else\ifx#1\chi\upchi
  \else\ifx#1\psi\uppsi
  \else\ifx#1\omega\upomega
  \else\ifx#1\Gamma\Upgamma
  \else\ifx#1\Delta\Updelta
  \else\ifx#1\Theta\Uptheta
  \else\ifx#1\Lambda\Uplambda
  \else\ifx#1\Xi\Upxi
  \else\ifx#1\Pi\Uppi
  \else\ifx#1\Sigma\Upsigma
  \else\ifx#1\Upsilon\Upupsilon
  \else\ifx#1\Phi\Upphi
  \else\ifx#1\Psi\Uppsi
  \else\ifx#1\Omega\Upomega
  \else#1 
  \fi\fi\fi\fi\fi\fi\fi\fi\fi\fi\fi\fi\fi\fi\fi\fi\fi\fi\fi\fi\fi\fi\fi\fi\fi\fi\fi\fi\fi\fi\fi\fi\fi
}

\DeclareRobustCommand{\matr}[1]{%
  \ifmmode 
    \ifcat\noexpand#1\relax 
      \boldsymbol{\upgreekmap{#1}}
    \else
      \mathbf{#1}
    \fi 
  \else
    \textbf{#1}
  \fi
}

\newcommand*{\vct}[1]{{\bm{#1}}}
\newcommand*{\tsr}[1]{{\mathbb{#1}}}

\newcommand*{\smallpar}{{\scaleobj{.75}{\parallel}}}
\newcommand*{\smallperp}{{\scaleobj{.75}{\perp}}}
\newcommand*{\smallo}{{\scalebox{.9}{$\scriptstyle\mathcal{O}$}}}

\newcommand*{\nablaeps}{\nabla^{\fracfac{\!\epsilon\!}}}

\makeatletter
\newcommand{\customlabel}[2]{%
  \begingroup
    \def\@currentlabel{#2}%
    \label{#1}%
  \endgroup
}
\makeatother

\newcolumntype{L}[1]{>{\raggedright\arraybackslash}p{#1}}

\usepackage[ headsepline=.5pt ]{scrlayer-scrpage}
\lohead{ \headerleft }
\rohead{ \headerright }%
\allowdisplaybreaks[3]

\numberwithin{equation}{section}
\numberwithin{definition}{section}

\makeatletter
\renewenvironment{abstract}{%
    \if@twocolumn
      \section*{\abstractname}%
    \else
      \begin{center}%
        {\bfseries  \abstractname\par}%
      \end{center}%
    \fi
}{}
\makeatother

\newbibmacro{string+doi}[1]{%
  \iffieldundef{doi}{ \iffieldundef{url}{#1}{\href{\thefield{url}}{#1}}}{\href{http://dx.doi.org/\thefield{doi}}{#1}}}
\DeclareFieldFormat{title}{\usebibmacro{string+doi}{\mkbibemph{#1}}}
\DeclareFieldFormat[article,inbook,incollection,inproceedings,patent,thesis,unpublished]{title}{\usebibmacro{string+doi}{\mkbibquote{#1\isdot}}}

\renewbibmacro*{volume+number+eid}{%
  \printfield{volume}%
  \setunit*{\addnbthinspace}
  \printfield{number}%
  \setunit{\addcomma\space}%
  \printfield{eid}}
\DeclareFieldFormat[article]{number}{\mkbibparens{#1}}

\ifpdf
\hypersetup{
  pdftitle={Fractured Poroelastic Media in the Limit of Vanishing Aperture},
  pdfauthor={M. Hörl, K. Kumar, and C. Rohde}
}
\fi

\bibliography{references}

\begin{document}

\maketitle

\begin{abstract}
\noindent 
We consider a poroelastic medium with a thin heterogeneity, also referred to as a fracture. 
Fluid flow and mechanical deformation inside both bulk and fracture are governed by the quasi-static Biot equations.
The fracture's material coefficients, such as hydraulic conductivity and elasticity, are assumed to scale with powers of the width-to-length ratio~$\epsilon$ of the fracture.
Based on a~priori estimates, we rigorously derive limit models as~$\epsilon \rightarrow 0$ and identify different limit regimes.
We obtain five regimes for the hydraulic conductivity and two for the elasticity.
While many cases yield discrete fracture models, others result in two-scale limit problems dominated by normal flow or deformation. 
\end{abstract}

\newlength{\kwlabelwidth}
\settowidth{\kwlabelwidth}{\textbf{Key Words}} 
\begin{itemize}[labelwidth=\kwlabelwidth,labelsep=1.4em,leftmargin=!]
\item[\textbf{Key Words}] Fractured porous media,  discrete fracture model, vanishing aperture, poroelasticity 
\end{itemize}

\newlength{\msclabelwidth}
\settowidth{\msclabelwidth}{\textbf{MSC Codes}} 
\begin{itemize}[labelwidth=\msclabelwidth,labelsep=1em,leftmargin=!]
\item[\textbf{MSC Codes}] 35B40, 35B45, 74F10, 74R99, 76S05
\end{itemize}

\section{Introduction}
\label{sec:1}
Thin heterogeneities in porous materials are characterized by extreme geometries with a small aperture but wide lateral extent. 
They may be partially sealed or highly porous and, despite their small thickness, can exert a strong influence on the mechanical behavior and fluid flow in the surrounding porous material.
Examples include geological fractures, thin sedimentary layers, membranes, and composite interfaces.
Applications arise not only in environmental and geotechnical contexts,
where we mention subsurface energy storage and exploitation, carbon sequestration, and groundwater flow, but also in engineered porous materials or biological tissues. 

We consider a poroelastic medium at the macroscopic scale, where fluid flow and mechanical deformation are governed by the Biot equations~\cite{biot41}.
Here, accurately modeling thin heterogeneities is challenging as their extreme aperture-to-length ratios cannot be resolved efficiently in corresponding numerical discretizations. 
 A common solution is the use of discrete fracture models, in which fractures are represented as lower-dimensional manifolds embedded within the surrounding poroelastic matrix.
 In this context, the term fracture is used in a generic sense to refer to any thin heterogeneity whose total aperture is well separated from both the pore scale and the size of the overall domain such that a lower-dimensional representation is appropriate. 
 While discrete fracture models substantially reduce the computational cost, careful formulation and coupling are required to preserve the physical accuracy in the interaction between flow and deformation across the interfaces.

A mathematically rigorous way to derive discrete fracture models is to consider the governing equations first in a full-dimensional fracture domain and analyze the convergence of the corresponding solutions towards a limit model as the width-to-length ratio of the fracture vanishes. 
This approach, also referred to as the limit of vanishing aperture, is the approach that we follow here.
Specifically, we consider a poroelastic medium with an isolated thin heterogeneity---here generically referred to as \enquote{fracture}---where the flow and deformation in both bulk and fracture are governed by the quasi-static Biot equations. 
Like the bulk domain, the fracture is therefore treated as poroelastic medium but characterized by distinct material properties, such as hydraulic conductivity and elasticity tensor. 
The material coefficients within the fracture domain are assumed to scale with arbitrary powers of the fracture's width-to-length ratio such that the choice of these exponents leads to different asymptotic regimes in the resulting limit models. 
The key scaling parameters governing the limit behavior are those associated with the hydraulic conductivity and elastic stiffness within the fracture.
Here, we identify five distinct limit regimes for the hydraulic conductivity and two regimes for the elastic response.
While many of these cases lead to discrete fracture models, others do not permit a lower-dimensional reduction of the mechanics or the flow equation within the fracture.
In such cases, flow or deformation in normal direction across the fracture dominates, resulting in a two-scale limit problem that still involves the full-dimensional fracture---however, not as a thin domain but rescaled in normal direction. 
To the authors’ knowledge, this work provides the first rigorous analysis that simultaneously considers the limit of vanishing aperture for coupled fluid flow and mechanical deformation. Accounting for the flow–mechanics coupling and the correct scalings with respect to the fracture's width-to-length ratio---arising from different model coefficients and spatial directions---renders the mathematical analysis particularly challenging.

Rigorous convergence results in the limit of vanishing aperture for flow in non-deformable porous media have been obtained in the literature for Darcy flow in both bulk and fracture domains~\cite{hoerl24,huy74,morales10,morales12,sanchez74}, Darcy bulk flow coupled with a Stokes fracture~\cite{morales17}, and unsaturated Richards flow in both bulk and fracture~\cite{list20}.
In particular, as in the present study, the works~\cite{hoerl24,list20} consider different scalings of the hydraulic conductivity within the fracture with respect to its width-to-length ratio.
Moreover, the literature provides rigorous convergence results for poroelastic plates without a surrounding porous medium. 
Here, we mention~\cite{marciniak15,mikelic19,mikelic16}, where isolated poroelastic plates or shells are analyzed, and~\cite{gahn22}, where rigorous homogenization results for a poroelastic plate with surrounding Stokes flow are presented. 
In contrast to~\cite{gahn22}, where the thickness of the poroelastic layer is of pore-scale size, the fracture aperture in our setting is well separated from the pore scale as in \cite{marciniak15,mikelic19,mikelic16}.
Consequently, homogenization from the pore scale and the limit of vanishing aperture constitute two separate procedures that have to be performed consecutively. 
Another rigorous strategy relies on Fourier analysis applied to infinite fracture geometries and has been applied for Darcy flow in both bulk and fractures~\cite{gander21} and Darcy flow coupled with a Stokes fracture~\cite{gander23}. 
This approach also allows for the derivation of quantitative error estimates.

In addition to rigorous methods, formal derivations of discrete fracture models are widely found in the literature, often based on vertical averaging across the fracture~\cite{ahmed17,brenner18,burbulla22,lesinigo11,martin05,rybak20}.
In particular, \cite{bukac17} considers a poroelastic bulk medium coupled with Brinkman flow inside the fracture, \cite{bergkamp22} investigates a poroelastic bulk with thin-film flow inside the fracture, and~\cite{brezina24} is concerned with a poroelastic medium in both bulk and fracture as in our case. 
Another class of formal approaches employs asymptotic expansions for the derivation of discrete fracture models~\cite{dugstad21,dugstad22,kumar20}.
Further related studies include~\cite{pop17}, where both formal and rigorous methods have been applied for reactive transport with nonlinear transmission conditions in fractured porous media, and~\cite{boon23}, where a hybrid-dimensional poromechanics model is developed based on a tailored mixed-dimensional calculus.
More complex models also account for the propagation of fractures and the evolution of the fracture aperture.
A popular approach to capture such dynamics are phase-field models~\cite{lee18,mikelic15}, which can also be combined with discrete fracture formulations~\cite{burbulla23}.
In addition, we mention models incorporating contact mechanics to account for fracture closure and frictional effects~\cite{berge20,bonaldi24,boon23,dehoop26,franceschini20}, approaches for handling curved~\cite{hoerl24,paranamana21,wang23} and immersed fractures~\cite{angot09,berrone17}, as well as models for networks of intersecting fractures~\cite{amir23,antonietti22,berrone17,formaggia14,schwenck15}.

This paper is structured as follows.
First, in \Cref{sec:2}, we introduce the quasi-static Biot equations in a poroelastic medium with an isolated fracture and define the width-to-length ratio~$\epsilon$ of the fracture as a scaling parameter.  
Then, in \Cref{sec:3}, we state a~priori estimates for the solutions of the full-dimensional problem and formulate weak and strong convergence results as~$\epsilon \rightarrow 0$.
Depending on how the material coefficients scale with~$\epsilon$, we identify different asymptotic regimes and present the corresponding limit models in \Cref{sec:4}. 
The proofs of the a~priori estimates and convergence theorems are given in the \Cref{sec:5,sec:6,sec:7}.
The existence and uniqueness of weak solutions to the quasi-static Biot equations is discussed in~\Cref{sec:A}. 
A list of symbols is provided in~\Cref{sec:B}.

\section{Poroelasticity with Full-Dimensional Fracture}
\label{sec:2}
We consider a poroelastic medium with a thin heterogeneity that we subsequently refer to as \enquote{fracture}. 
First, in \Cref{sec:2.1}, we introduce notations that are used throughout the paper, before defining the geometric setting in \Cref{sec:2.2}.
The quasi-static Biot equations, which model fluid flow and mechanical deformation within the poroelastic medium, are presented in \Cref{sec:2.3}.
In \Cref{sec:2.4}, we non-dimensionalize the Biot system from \Cref{sec:2.3} and introduce the width-to-length ratio~$\epsilon$ of the fracture as a scaling parameter.
While the fracture is treated as a full-dimensional poroelastic subdomain throughout this section, we eventually aim to derive limit models as $\epsilon \rightarrow 0$ (see \Cref{sec:4}).
In \Cref{sec:2.5}, we reformulate the Biot problem in dimensionless form, introduce a corresponding weak formulation, and discuss its wellposedness.
Finally, in \Cref{sec:2.6}, we apply a coordinate transformation to the weak formulation from \Cref{sec:2.5}, which allows us to define function spaces that are independent of the scaling parameter~$\epsilon$.

\subsection{Notation}
\label{sec:2.1}
Let $n \in \mathbb{N}$ with $n \ge 2$ denote the dimension of a poroelastic medium; relevant in applications are the cases $n\in\{ 2,3\}$.
For vectors~$\vct{a},\vct{b} \in \mathbb{R}^n$ and matrices~$\matr{A}, \matr{B} \in \mathbb{R}^{n\times n }$, we write $\vct{a} \cdot \vct{b}$ for the Euclidean and $\matr{A} : \matr{B} $ for the Frobenius product. 
Vectors and vector-valued functions and function spaces are written in bold italics~(e.g., $\vct{u}$), matrices and matrix-valued functions and function spaces in upright bold~(e.g., $\matr{K}$), and fourth-order tensors and suchlike tensor-valued functions in blackboard bold (e.g., $\tsr{C}$). 
In particular, for an open set $D \subset \mathbb{R}^n$, we use the Lebesgue spaces $L^r(D)$, $\vct{L}^r (D) := L^r (D )^n  $, and $\matr{L}^r (D) := L^r (D)^{n\times n} $ for $r\in [1,\infty]$, as well as the Sobolev spaces~$W^{1,r} (D )$ with $H^1(D) := W^{1,2} (D)$ and $\vct{H}^1 ( D) := H^1 (D)^n  $. 
In addition, for $\emptyset \not= \Gamma \subset \partial D$, we write $H^1_{0,\Gamma} ( D) := \{ v \in H^1 (D) \ \vert\ v\vert_\Gamma = 0 \}$ for the space of $H^1$-functions with zero trace on~$\Gamma$ and set $\vct{H}^1_{0,\Gamma } (D) := H^1_{0,\Gamma} (D)^n$. 
If $\Gamma = \partial D$, we simply write~$H^1_0 (D)$ and $\vct{H}^1_0 (D)$.
The scalar product in $L^2 (D)$, $\vct{L}^2 (D)$, and $\matr{L}^2 (D)$ is denoted by $\langle \cdot ,\cdot \rangle_D $, where the specific space is clear from the context.
Besides, given an open bounded interval~$J\subset \mathbb{R}$ and a Banach space~$X$, we write $L^r (J ; X ) $, $W^{1,r} (J ; X )$,  and $H^1 (J  ; X)$ for the corresponding Bochner spaces.

Further, we use the symbols~$\lesssim$ and $\gtrsim$ to express inequalities that hold up to a multiplicative constant independent of the scaling parameter~$\epsilon$ as introduced in \Cref{sec:2.4}. 
In the \Cref{sec:2.2,sec:2.3}, we state the corresponding SI~unit in square brackets when first introducing a dimensional quantity (e.g., pressure~$\dimq{p}$~$[\unit{Pa}]$).
Moreover, we mark dimensional quantities and domains with a tilde to distinguish them from their non-dimensional counterparts that are introduced in \Cref{sec:2.4}. 
Constant reference quantities are labeled by a star.
Besides, quantities related to the transformed system presented in \Cref{sec:2.6} are marked with a hat.
Finally, we refer to the \Cref{not:pmf,not:fracfac,not:odot} below, where a compact notation tailored to the geometric setting introduced in \Cref{sec:2.2} is defined.
We also refer to Appendix~\ref{sec:B} for a comprehensive list of symbols.

\subsection{Geometric Setting} 
\label{sec:2.2}
The geometric setting as defined below is sketched in \Cref{fig:fulldim_dimensional}. 
\begin{figure}
\centering
\includegraphics[scale=1.1]{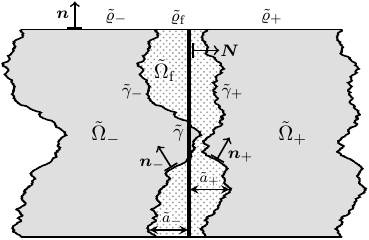}
\caption{Geometry in $\dimq{\Omega}$ from  \cref{eq:omega} for the full-dimensional Biot system~\eqref{eq:fulldim-dimensional}.}
\label{fig:fulldim_dimensional}
\end{figure}
We consider an isolated fracture~$\dimq{\Omega}_\rmf$ inside a poroelastic medium. 
The fracture domain~$\dimq{\Omega}_\rmf \subset \mathbb{R}^n$ is parameterized by an interface~$\overline{\dimq{\gamma }} \subset \mathbb{R}^n $ and by aperture functions~$\dimq{a}_\pm \in \mathcal{C}^{0,1} (\overline{\dimq{\gamma}})$~$[\unit{m}]$. 
Here, the interface $\overline{\dimq{\gamma}}$ is assumed to be a compact and connected submanifold with $\mathcal{C}^{0,1}$-boundary of the hyperplane~$ \{ \dimq{\vct{y}} \in \mathbb{R}^n \ \vert\  \dimq{\vct{y}} \cdot \vct{N} = 0 \} \subset \mathbb{R}^n$ with unit normal~$\vct{N} \in \mathbb{R}^n$, i.e.,  $\overline{\dimq{\gamma}}$ has zero curvature and dimension~$n-1$. 
We write $\dimq{\gamma}$ for the interior of~$\overline{\dimq{\gamma}}$.
The total aperture~$\dimq{a} := \dimq{a}_+ + \,\dimq{a}_- $ is required to be positive on~$\smash{\overline{\dimq{\gamma}}}$.
Specifically, we now define the fracture domain~$\dimq{\Omega}_\rmf$ by
\begin{subequations}
\begin{align}
\dimq{\Omega}_\rmf &:= \bigl\{ \dimq{\vct{y}} + \dimq{s} \vct{N} \, \big\vert \, \dimq{\vct{y}} \in \dimq{\gamma } , \ -\dimq{a}_- ( \dimq{\vct{y}} ) < \dimq{s} < \dimq{a}_+ (\dimq{\vct{y}}) \bigr\} . \label{eq:Omegaf_dim}
\intertext{The fracture domain~$\dimq{\Omega}_\rmf$ is surrounded by two poroelastic bulk domains~$\dimq{\Omega}_+$ and~$\dimq{\Omega}_-$ on its left- and right-hand side given by}
\label{eq:Omega+_dim} \dimq{\Omega}_+ & := \bigl\{  \dimq{\vct{y}} + \dimq{s} \vct{N} \ \big\vert\ \dimq{\vct{y}} \in \dimq{\gamma }, \ 0 < \dimq{s} - \dimq{a}_+ (\dimq{\vct{y}}) < \refq{L} \bigr\}   ,  \\
\label{eq:Omega-_dim} \dimq{\Omega}_- &:=  \bigl\{\dimq{\vct{y}} + \dimq{s} \vct{N} \ \big\vert\ \dimq{\vct{y}} \in \dimq{\gamma}, \ - \refq{L}  < \dimq{s} +  \dimq{a}_- (\dimq{\vct{y}}) < 0 \bigr\} . 
\end{align}%
\label{eq:Omegapmf_dim}%
\end{subequations}%
Here, $\refq{L}$~$[\unit{m}]$ denotes a characteristic value of the longitudinal extension of the fracture, e.g., $\refq{L} := \sup\nolimits_{\dimq{\vct{y}}_1 , \dimq{\vct{y}}_2 \in \dimq{\gamma}} \vert \dimq{\vct{y}}_1 - \dimq{\vct{y}}_2 \vert$.

Subsequently, we use the following notation. 
\begin{notation}
\label{not:pmf}
We use the index~\enquote{$\pmf$} to refer to the corresponding triple of mathematical objects (e.g., functions, constants, domains, spaces) with respect to the subdomains $\dimq{\Omega}_+$, $\dimq{\Omega}_-$, and~$\dimq{\Omega}_\rmf$.
Besides, the index~\enquote{$\pm$} refers to the corresponding pair of mathematical objects with respect to the bulk subdomains~$\dimq{\Omega}_+$ and~$\dimq{\Omega}_-$. 
Multiple indices within a term refer to the same pair or triple.
For example, we write
\begin{align}
\dimq{\Omega}_\pm := \bigl( \dimq{\Omega}_+ , \dimq{\Omega}_- \bigr) \enspace\ \text{or} \enspace\ \dimq{\nabla} \cdot ( \matr{\alpha}_\pmf \dimq{\vct{u}}_\pmf ) := \bigl( \dimq{\nabla} \cdot ( \matr{\alpha}_+ \dimq{\vct{u}}_+ ) , \nabla \cdot ( \matr{\alpha}_- \dimq{\vct{u}}_- ) ,  \dimq{\nabla} \cdot ( \matr{\alpha}_\rmf \dimq{\vct{u}}_\rmf ) \bigr) ,
\end{align}
\end{notation}
Further, we write
\begin{align}
\label{eq:gammapm} \dimq{\gamma}_\pm &:= \partial\dimq{\Omega}_\pm \cap \partial\dimq{\Omega}_\rmf = \bigl\{ \dimq{\vct{y}} \pm \dimq{a}_\pm (\dimq{\vct{y}}) \vct{N} \in \mathbb{R}^n \ \big\vert\ \dimq{\vct{y}} \in \dimq{\gamma} \bigr\}
\end{align}
for the internal boundaries between the bulk domains~$\dimq{\Omega}_\pm$ and the fracture domain~$\dimq{\Omega}_\rmf$, as well as 
\begin{align}
\rho_\pm &:= \partial \dimq{\Omega}_\pm \setminus \overline{\dimq{\gamma}}_\pm , \quad\enspace
\rho_\rmf := \partial \dimq{\Omega}_\rmf \setminus \bigl( \dimq{\gamma}_+ \cup \dimq{\gamma}_- \bigr)
\label{eq:rho}%
\end{align}%
for the external boundaries of the subdomains~$\dimq{\Omega}_\pmf$.
In particular, the external boundaries~$\dimq{\rho}_\pmf$ are disjointly subdivided into  sections with Dirichlet and Neumann boundary conditions, i.e., 
\begin{align}
\dimq{\rho}_\pmf = \dimq{\rho}_{\pmf, \rmD } \cup \dimq{\rho}_{\pmf, \rmN } , \quad\enspace \dimq{\rho}_{\pmf, \rmD } \cap \dimq{\rho}_{\pmf, \rmN } = \emptyset . \label{eq:rhoDN}
\end{align}
We assume that $\abs{\dimq{\rho}_{+ ,\rmD }} > 0$ or $\abs{\dimq{\rho}_{- ,\rmD }} > 0$.
Finally, we denote the overall poroelastic domain~$\dimq{\Omega}$ by
\begin{align}\label{eq:omega}
\dimq{\Omega} := \dimq{\Omega}_+ \cup \dimq{\Omega}_- \cup \dimq{\Omega}_\rmf \cup \dimq{\gamma}_+ \cup \dimq{\gamma}_- .
\end{align}
\begin{remark}[More general geometries]
\begin{enumerate}[label=(\roman*)]
\item More generally, the results in this paper also hold for the case where the parameterizing fracture interface~$\overline{\dimq{\gamma}}\subset \mathbb{R}^n$ is a compact and connected $\mathcal{C}^2$-submanifold of dimension~$(n-1)$ with $\mathcal{C}^{0,1}$-boundary. 
For the treatment of such kind of curved fractures, we refer to~\cite{hoerl24}, where Darcy flow in a non-deformable fractured porous medium is considered.
As in~\cite{hoerl24}, we expect that no additional curvature terms appear.
\item The assumption that the total aperture~$a$ is strictly positive can also be relaxed to nonnegativity. 
In this case, the sets where the aperture vanishes are either located on the external boundary or within the domain such that continuity conditions are imposed directly between the bulk domains.
The analysis then naturally requires the use of aperture-weighted Lebesgue and Sobolev spaces; see~\cite{hoerl24} for the corresponding treatment.
\item The geometry of the bulk domains~$\dimq{\Omega}_\pm$, as defined in \cref{eq:Omegapmf_dim}, depends on the fracture aperture functions~$\dimq{a}_\pm$, even at the outer boundaries that are located far away from the fracture.
This choice ensures that varying the width-to-length ratio of the fracture induces only a translation of the bulk domains in normal direction (cf.~\cref{eq:trafo_pm}).
More generally, our results remain valid when the bulk domains~$\dimq{\Omega}_\pm$ are arbitrary $\mathcal{C}^{0,1}$-domains such that $\partial \dimq{\Omega}_\pm \cap  \partial \dimq{\Omega}_\rmf = \dimq{\gamma}_\pm$, with the subdomains $\dimq{\Omega}_+$, $\dimq{\Omega}_-$, and $\dimq{\Omega}_\rmf$ being pairwise disjoint.
In this case, the bulk domains are decomposed into interior and exterior regions with respect to the fracture, where only the interior region is deformed in normal direction as the width-to-length ratio of the fracture changes, while the exterior region remains unchanged. 
For details, we refer to \cite{hoerl24}.
\end{enumerate}
\end{remark}

\subsection{Full-Dimensional Biot System in Dimensional Form}
\label{sec:2.3}
We consider the flow of a single-phase fluid and the mechanical deformation of the poroelastic domain~$\dimq{\Omega}$. 
Each of the subdomains $\dimq{\Omega}_+$, $\dimq{\Omega}_-$, and $\dimq{\Omega}_\rmf$ is treated as a (possibly inhomogeneous and anisotropic) poroelastic medium where the flow and deformation are governed by the quasi-static Biot equations. 
In particular, we use the word \enquote{fracture} to refer to a thin heterogeneity inside a poroelastic medium (e.g., a crack filled with debris) and treat the fracture domain~$\dimq{\Omega}_\rmf$ equally as a poroelastic medium with (significantly) different material coefficients.
We refer to~\cite{biot41} for the original derivation of Biot's equations and to~\cite{clopeau01,mikelic12} for a derivation based on homogenization theory.

Let $\dimq{I} := ( 0, \dimq{T} ) $ denote a time interval with final time~$\dimq{T} < \infty$~$[\unit{s}]$.
The flow and deformation in~$\dimq{\Omega} \times \dimq{I}$ are modeled by the following problem.

Find the pressure head~$\dimq{p}_\pmf \colon \dimq{\Omega}_\pmf \times \dimq{I} \rightarrow \mathbb{R}$~$[\unit{m}]$ and the displacement vector $\dimq{\vct{u}}_\pmf \colon \dimq{\Omega}_\pmf \times \dimq{I} \rightarrow \mathbb{R}^n$~$[\unit{m}]$ such that
\begin{subequations}
\begin{alignat}{2}
-\dimq{\nabla} \cdot \dimq{\matr{\sigma}}_\pmf ( \dimq{\vct{u}}_\pmf , \dimq{p}_\pmf ) &= \dimq{\vct{f}}_\pmf \quad\enspace &&\text{in } \dimq{\Omega}_\pmf \times \dimq{I} , \\
\label{eq:fulldim-dimensional_b}\partial_\dimq{t} \bigl( \refq{\rho}\! \refq{g} \dimq{\omega}_\pmf \dimq{p}_\pmf + \dimq{\nabla} \cdot ( \matr{\alpha}_\pmf \dimq{\vct{u}}_\pmf ) \bigr) - \dimq{\nabla} \cdot \bigl( \dimq{\matr{K}}_\pmf  \dimq{\nabla} \dimq{p}_\pmf \bigr) &= \dimq{q}_\pmf \quad\enspace &&\text{in } \dimq{\Omega}_\pmf \times \dimq{I}, 
\end{alignat}%
where the total stress $\dimq{\matr{\sigma}}_\pmf ( \dimq{\vct{u}}_\pmf , \dimq{p}_\pmf )$~$\unit{[Pa ]}$ is given by
\begin{align}
\dimq{\matr{\sigma}}_\pmf ( \dimq{\vct{u}}_\pmf , \dimq{p}_\pmf ) &:=  \dimq{\tsr{C}}_\pmf \dimq{\matr{e}} ( \dimq{\vct{u}}_\pmf ) - \refq{\rho} \bigl( \refq{g} \dimq{p}_\pmf + \dimq{G}_\pmf \bigr) \matr{\alpha}_\pmf 
\end{align}
with the strain tensor~$\dimq{\matr{e}}(\dimq{\vct{u}}_\pmf)$~$[\unit{1}]$ given by \cref{eq:edim}.
Across the internal interfaces~$\dimq{\gamma}_\pm$, we require the continuity of pressure head, displacement vector, normal Darcy velocity, and normal stress, i.e.,
\begin{alignat}{2}
    \dimq{p}_\pm &= \dimq{p}_\rmf \quad\enspace &&\text{on } \dimq{\gamma}_\pm \times \dimq{I} , \\
    \dimq{\vct{u}}_\pm &= \dimq{\vct{u}}_\rmf \quad\enspace &&\text{on } \dimq{\gamma}_\pm \times \dimq{I} , \\
    \dimq{\matr{K}}_\pm  \dimq{\nabla} \dimq{p}_\pm  \cdot \vct{n}_\pm &= \dimq{\matr{K}}_\rmf  \dimq{\nabla} \dimq{p}_\rmf  \cdot \vct{n}_\pm \quad\enspace &&\text{on } \dimq{\gamma}_\pm \times \dimq{I} , \\
\dimq{\matr{\sigma}}_\pm ( \dimq{\vct{u}}_\pm , \dimq{p}_\pm ) \vct{n}_\pm &=  
\dimq{\matr{\sigma}}_\rmf ( \dimq{\vct{u}}_\rmf , \dimq{p}_\rmf ) \vct{n}_\pm \quad\enspace    &&\text{on } \dimq{\gamma}_\pm \times \dimq{I} ,
\end{alignat}
where $\vct{n}_\pm$ denotes the unit normals on~$\dimq{\gamma}_\pm$ that are external normals with respect to the fracture domain~$\dimq{\Omega}_\rmf$.
Besides, given the initial pressure head~$\dimq{p}_{0, \pmf} \colon \dimq{\Omega}_\pmf \rightarrow \mathbb{R}$~$[\unit{m}]$, we impose the boundary and initial conditions 
\begin{alignat}{2}
\dimq{\vct{u}}_\pmf &= \vct{0} \quad\enspace &&\text{on } \dimq{\rho}_\pmf \times \dimq{I} , \\
\dimq{p}_\pmf &= 0 \quad\enspace &&\text{on } \dimq{\rho}_{\pmf , \rmD} \times \dimq{I} , \\
\dimq{\matr{K}}_\pmf \dimq{\nabla} \dimq{p}_\pmf \cdot \vct{n} &= 0 \quad\enspace &&\text{on } \dimq{\rho}_{\pmf , \rmN} \times \dimq{I} , \\
\dimq{p}_\pmf ( \cdot , 0 ) &= \dimq{p}_{0, \pmf} \quad\enspace &&\text{on } \dimq{\Omega}_\pmf ,
\end{alignat}
where $\vct{n}$ denotes the outer unit normal on~$\partial \dimq{\Omega}$. 
Given~$\dimq{p}_{0,\pmf} \colon \dimq{\Omega}_\pmf \rightarrow \mathbb{R}$~$[\unit{m}]$, the initial displacement vector~$\dimq{\vct{u}}_\pmf (\cdot , 0) =: \dimq{\vct{u}}_{0, \pmf} \colon \dimq{\Omega}_\pmf \rightarrow \mathbb{R}^n$~$[\unit{m}]$ satisfies%
\begin{align}
- \nabla \cdot \dimq{\matr{\sigma}}_\pmf ( \dimq{\vct{u}}_{0,\pmf} , \dimq{p}_{0,\pmf} )  &= \dimq{\vct{f}}_\pmf (\cdot, 0) \quad\enspace &&\text{in } \dimq{\Omega}_\pmf  , \\
\dimq{\vct{u}}_{0,\pm } &= \dimq{\vct{u}}_{0,\rmf} \quad\enspace &&\text{on } \dimq{\gamma}_\pm , \\
      \dimq{\matr{\sigma}}_\pm ( \dimq{\vct{u}}_{0,\pm} , \dimq{p}_{0,\pm} ) \vct{n}_\pm   
       &= \dimq{\matr{\sigma}}_\rmf ( \dimq{\vct{u}}_{0,\rmf} , \dimq{p}_{0,\rmf} ) \vct{n}_\pm   
    &&\text{on } \dimq{\gamma}_\pm  , \\
\dimq{\vct{u}}_{0, \pmf} &= \vct{0} \quad\enspace &&\text{on } \dimq{\rho}_\pmf .
\end{align}%
\label{eq:fulldim-dimensional}%
\end{subequations}%
In \cref{eq:fulldim-dimensional}, $\dimq{\tsr{C}}_\pmf \colon \dimq{\Omega}_\pmf \rightarrow \mathbb{R}^{n \times n \times n \times n}$~$[\unit{Pa}]$ is the elasticity tensor and $\dimq{\matr{e}} ( \dimq{\vct{u}}_\pmf )$~$[\unit{1}]$ denotes the strain tensor given by
\begin{align}
\dimq{\matr{e}} ( \dimq{\vct{u}}_\pmf ) := \frac{1}{2} \bigl( \dimq{\nabla} \dimq{\vct{u}}_\pmf + ( \dimq{\nabla}  \dimq{\vct{u}}_\pmf )^\rmt \bigr) . \label{eq:edim}
\end{align}
Further, $\matr{\alpha}_\pmf \colon \dimq{\Omega}_\pmf \rightarrow \mathbb{R}^{n\times n }$~$[\unit{1}]$ denotes the Biot tensor---considered here as a matrix following~\cite{mikelic12}---and $\dimq{\omega}_\pmf \colon \dimq{\Omega}_\rmf \rightarrow \mathbb{R}$~$[\unit{Pa^{-1}}]$ denotes the inverse Biot modulus. 
Besides, $\dimq{\matr{K}}_\pmf \colon \dimq{\Omega}_\pmf \rightarrow \mathbb{R}^{n\times n }$~$[\unit{m/s}]$ is the hydraulic conductivity matrix and~$\refq{\rho} \in \mathbb{R}^+$~$[\unit{kg/m^3}]$ denotes the density of the fluid.
The function~$\dimq{G}_\pmf\colon \dimq{\Omega}_\pmf \rightarrow \mathbb{R}$~$[\unit{m^2 /s^2}]$ is given by
\begin{align}
\dimq{G}_\pmf (\dimq{\vct{x}}) := \dimq{\vct{g}} \cdot \dimq{\vct{x}}, \label{eq:Gdim}
\end{align}
where $\dimq{\vct{g}} \in \mathbb{R}^n$~$[\unit{m/s^2}]$ is the gravitational acceleration with $\refq{g} := \abs{\dimq{\vct{g}}}$~$[\unit{m/s^2}]$. 
Moreover, $\dimq{\vct{f}}_\pmf  \colon \dimq{\Omega}_\pmf \times \dimq{I} \rightarrow \mathbb{R}^n$~$[\unit{Pa/m}]$ is a body force term and $\dimq{q}_\pmf  \colon \dimq{\Omega}_\pmf \times \dimq{I} \rightarrow \mathbb{R}$~$[\unit{s^{-1}}]$ a fluid source term.

We remark that the system~\eqref{eq:fulldim-dimensional} is formulated in terms of the pressure head $\dimq{p}_\pmf$~$[\unit{m}]$, which is related to the pressure~$\dimq{\pi}_\pmf$~$[\unit{Pa}]$ via
\begin{align}
\dimq{\pi}_\pmf = \refq{\rho} \bigl( \refq{g} \dimq{p}_\pmf + \dimq{G}_\pmf \bigr). \label{eq:pi}
\end{align} 

\subsection{Non-Dimensionalization}
\label{sec:2.4}
Let $\refq{a}$~$[\unit{m}]$ be a characteristic value for the aperture~$\dimq{a}$ of the fracture.
Then, we define the characteristic width-to-length ratio of the fracture as 
\begin{align}
\epsilon := \frac{\refq{a}}{\refq{L}} > 0 . \label{eq:epsilon}
\end{align}
Subsequently, we assume that the material coefficients inside the fracture domain~$\dimq{\Omega}_\rmf$ depend on~$\epsilon$.
Moreover, we treat the width-to-length ratio~$\epsilon$ as a scaling parameter, i.e., we  analyze how the Biot system~\eqref{eq:fulldim-dimensional} behaves in the limit~$\epsilon \rightarrow 0$. 
As a first step, we non-dimensionalize the Biot system~\eqref{eq:fulldim-dimensional} in order to emphasize the $\epsilon$-dependencies of the system.
While \Cref{sec:3} will deal with the sequence of Biot problems~\eqref{eq:fulldim-dimensional} for $\epsilon \in ( 0 ,\epsilon^\ast ]$ with $\epsilon^\ast > 0$ sufficiently small and the convergence as $\epsilon \rightarrow 0$, this section is first concerned with the Biot problem~\eqref{eq:fulldim-dimensional} for a fixed value of the width-to-length ratio~$\epsilon > 0$ of the fracture. 

Let $\refq{K}_\rmb$~$[\unit{m/s}]$ be a characteristic value of the hydraulic conductivity in the bulk domains~$\dimq{\Omega}_\pm$. 
Then, with the reference time~$\refq{t} := \refq{L} / \refq{K}_\rmb$~$[\unit{s}]$, we define the non-di\-men\-sio\-nal position vector and time as 
\begin{align}
\vct{x} = \frac{\dimq{x}}{\refq{L}} , \quad t = \frac{\dimq{t}}{\refq{t}} . \label{eq:xt}
\end{align}
This also necessitates a rescaling of derivative operators 
\begin{align}
\nabla = \refq{L} \dimq{\nabla} , \quad \partial_t = \refq{t} \partial_\dimq{t}  . \label{eq:nabla}
\end{align}
In addition, we introduce the non-dimensional domains and interfaces
\begin{equation} 
\begin{alignedat}{4}
\Omega_\pmf^\epsilon &:= \frac{\dimq{\Omega}_\pmf}{\refq{L}} , \quad &\Omega^\epsilon &:= \frac{\dimq{\Omega}}{\refq{L}} \quad &\gamma_\pm^\epsilon &:= \frac{\dimq{\gamma}_\pm}{\refq{L}} ,  \quad &\gamma &:= \frac{\dimq{\gamma}}{\refq{L}},  \\
\rho_\pmf^\epsilon &:= \frac{\dimq{\rho}_\pmf}{\refq{L}} , \quad &\rho_{\pmf, \rmD}^\epsilon &:= \frac{\dimq{\rho}_{\pmf, \rmD}}{\refq{L}} , \quad &\rho_{\pmf, \rmN}^\epsilon &:= \frac{\dimq{\rho}_{\pmf, \rmN}}{\refq{L}}, \quad &I &:= \frac{\dimq{I}}{\refq{t}} ,
\end{alignedat}%
\label{eq:nondim1}%
\end{equation}%
as well as the non-dimensional quantities
\begin{equation} 
\begin{alignedat}{5}
p_\pmf^\epsilon &:= \frac{\dimq{p}_\pmf}{\refq{L}}, \quad &\vct{u}_\pmf^\epsilon &:= \frac{\dimq{u}_\pmf}{\refq{L}}, \quad &a_\pm &:= \frac{\dimq{a}_\pm}{\refq{a}} , \quad &a &:= \frac{\dimq{a}}{\refq{a}}, \quad &\tsr{C}_\pm^\epsilon &:= \frac{\dimq{\tsr{C}}_\pm}{\refq{C}_\rmb},  \\
\tsr{C}_\rmf^\epsilon &:= \frac{\dimq{\tsr{C}}_\rmf}{\refq{C}_\rmf} , \quad &\vct{f}_\pm^\epsilon &:= \frac{\dimq{\vct{f}}_\pm}{\refq{f}_\rmb} , \quad &\vct{f}_\rmf^\epsilon &:= \frac{\dimq{\vct{f}}_\rmf}{\refq{f}_\rmf} , \quad &\omega_\pm^\epsilon &:= \frac{\dimq{\omega}_\pm}{\refq{\omega}_\rmb} , \quad &\omega_\rmf^\epsilon &:= \frac{\dimq{\omega}_\rmf}{\refq{\omega}_\rmf},  \\
\matr{K}_\pm^\epsilon &:= \frac{\dimq{\matr{K}}_\pm}{\refq{K}_\rmb} , \quad &\matr{K}_\rmf^\epsilon &:= \frac{\dimq{\matr{K}}_\rmf}{\refq{K}_\rmf} , \quad &q_\pm^\epsilon &:= \frac{\dimq{q}_\pm}{\refq{q}_\rmb} , \quad &q_\rmf^\epsilon &:= \frac{\dimq{q}_\rmf}{\refq{q}_\rmf}, \quad &G_\pmf^\epsilon &:= \frac{\dimq{G}_\pmf}{\refq{g}\!\refq{L}} , \\
\vct{u}_{0,\pmf}^\epsilon &:= \frac{\dimq{\vct{u}}_{0,\pmf}}{\refq{L}} , \quad &p_{0,\pmf}^\epsilon &:= \frac{\dimq{p}_{0,\pmf}}{\refq{L}}, \quad &T &:= \frac{\dimq{T}}{\refq{t}}, \quad &\vct{g} &:= \frac{\dimq{g}}{\refq{g}},
\end{alignedat}%
\label{eq:nondim2}%
\end{equation}%
where $\refq{C}_\rmb := \refq{\rho}\!\refq{g}\!\refq{L}$~$[\unit{Pa}]$, $\refq{f}_\rmb := \refq{\rho}\!\refq{g}$~$[\unit{Pa / m}]$, $\refq{\omega}_\rmb := 1 / \refq{C}_\rmb$~$[\unit{Pa^{-1}}]$, $\refq{q}_\rmb := \refq{K}_\rmb / \refq{L}$~$[\unit{s^{-1}}]$.
In particular, the dimensionless bulk and fracture domain~$\Omega_\pmf^\epsilon$ are given by
\begin{subequations}
\begin{align}
\Omega_+^\epsilon  &  = \bigl\{ \vct{y} + s\vct{N} \ \big\vert\ \vct{y} \in \gamma , 0 < s - \epsilon a_+ ( \vct{y} ) < 1 \bigr\} , \\
\Omega_-^\epsilon & = \bigl\{ \vct{y} + s\vct{N} \ \big\vert\ \vct{y} \in \gamma , -1 < s + \epsilon a_-(\vct{y} ) < 0 \bigr\} , \\
\Omega_\rmf^\epsilon &= \bigl\{ \vct{y} + s \vct{N} \ \big\vert\ \vct{y} \in \gamma , \ - \epsilon a_-(\vct{y}) < s <  \epsilon a_+ (\vct{y}) \bigr\} . \label{eq:Omegafeps}
\end{align}%
\label{eq:Omegapmf_eps}%
\end{subequations}%
Thus, as the width-to-length ratio~$\epsilon$ of the fracture decreases, the fracture domain~$\Omega_\rmf^\epsilon$ is compressed in normal direction, while the bulk domains~$\Omega_\pm^\epsilon$ are translated in normal direction.
The overall domain~$\Omega^\epsilon$ is given by
\begin{align}
\Omega^\epsilon = \Omega_+^\epsilon \cup \Omega_-^\epsilon \cup \Omega_\rmf^\epsilon \cup \gamma_+^\epsilon \cup \gamma_-^\epsilon . \label{eq:Omegaeps}
\end{align}
A sketch of the non-dimensional geometry and the corresponding limit geometry as $\epsilon\rightarrow 0$ is shown in \Cref{fig:nondim}.

We assume that the material coefficients and right-hand sides inside the fracture domain scale with certain powers of~$\epsilon$ as $\epsilon \rightarrow 0$. 
This $\epsilon$-dependency is retrieved in the non-di\-men\-sio\-na\-li\-za\-tion process through the characteristic reference quantities inside the fracture.
Specifically, we assume that there exist parameters~$\nu_\tsr{C}, \nu_{\!\vct{f}} , \nu_\omega , \nu_\matr{K} , \nu_q \in \mathbb{R} $ such that the reference quantities inside the fracture scale like 
\begin{align}
\refq{C}_\rmf = \epsilon^{\nu_\tsr{C}} \refq{C}_\rmb , \quad \refq{f}_\rmf = \epsilon^{\nu_{\!\vct{f}}} \refq{f}_\rmb ,\quad \refq{\omega}_\rmf = \epsilon^{\nu_\omega } \refq{\omega}_\rmb , \quad \refq{K}_\rmf = \epsilon^{\nu_\matr{K}} \refq{K}_\rmb , \quad \refq{q}_\rmf = \epsilon^{\nu_q} \refq{q}_\rmb . \label{eq:eps_scaling1}
\end{align} 
Further, we also consider an~$\epsilon$-dependency of the (already dimensionless) Biot tensor~$\matr{\alpha}_\pmf$. 
\begin{figure}
\centering
\includegraphics[scale=1.03]{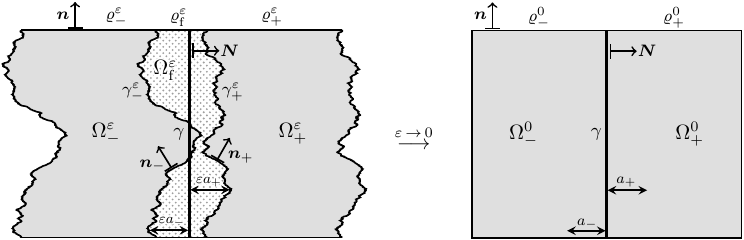}
\caption{Sketch of the geometry for the dimensionless full-dimensional Biot system~\eqref{eq:fulldim-nondim} (left) and the corresponding limit geometry as $\epsilon \rightarrow 0$ (right).}
\label{fig:nondim}
\end{figure}
Here, we write 
\begin{align}
{\matr{\alpha}}_\pm^\epsilon (\vct{x}  ) := \matr{\alpha}_\pm ( \dimq{\vct{x}} ) , \quad  \epsilon^{\!\matr{\nu}_{\!\matr{\alpha}}} \matr{\alpha}_\rmf^\epsilon (\vct{x}  )  :=  \matr{\alpha}_\rmf (\dimq{\vct{x}} ),  \label{eq:eps_scaling2}
\end{align}
and assume that the Biot tensor~$\matr{\alpha}_\rmf^\epsilon$ inside the fracture is block-diagonal with respect to the fracture, i.e., $\matr{\alpha}_\rmf^\epsilon$ commutes with the orthogonal projection in normal direction: $\matr{\alpha}_\rmf^\epsilon \vct{N} \vct{N}^\rmt = \vct{N} \vct{N}^\rmt  \matr{\alpha}_\rmf^\epsilon $.
In addition, we consider different scalings in normal and tangential direction, i.e., $\matr{\nu}_{\!\matr{\alpha}}$ in \cref{eq:eps_scaling2} denotes the matrix
\begin{align}
{\matr{\nu}}_{\!\matr{\alpha}} = \bigl( \matr{I}_n - \vct{N} \vct{N}^\rmt \bigr)  \nu_{\!\matr{\alpha}}^\smallpar +  \vct{N} \vct{N}^\rmt  \, \nu_{\!\matr{\alpha}}^\smallperp \in \mathbb{R}^{n\times n }  \label{eq:nualpha}
\end{align} 
for parameters~$\nu_{\matr{\alpha}}^\smallpar , \nu_{\matr{\alpha}}^\smallperp \in \mathbb{R}$, where $\matr{I}_n \in \mathbb{R}^{n\times n}$ denotes the identity matrix. The expression~$\epsilon^{\!\matr{\nu}_{\!\matr{\alpha}}}  \in \mathbb{R}^{n\times n}$ in \cref{eq:eps_scaling2} is to be understood as a matrix exponential, i.e.,
\begin{align}
\epsilon^{\!\matr{\nu}_{\!\matr{\alpha}}} := \exp ( \ln(\epsilon ) \matr{\nu}_{\!\matr{\alpha}}) = \epsilon^{\nu_{\!\matr{\alpha}}^\smallpar} \bigl( \matr{I}_n - \vct{N} \vct{N}^\rmt \bigr) + \epsilon^{\nu_{\!\matr{\alpha}}^\smallperp} \vct{N}\vct{N}^\rmt .
\end{align}%
\vspace{-1\baselineskip}
\begin{remark}
The non-dimensionalization above aims to extricate the dependencies of the Biot system~\eqref{eq:fulldim-dimensional} on the width-to-length ratio~$\epsilon$ of the fracture. 
We refrain from introducing other dimensionless numbers that are not relevant for the limit~$\epsilon \rightarrow 0$ and, for instance, disregard characteristic length scale differences between the domain~$\dimq{\Omega}$, the displacement vector~$\dimq{\vct{u}}_\pmf$, and the pressure head~$\dimq{p}_\pmf$.
\end{remark}

\subsection{Full-Dimensional Biot System in Dimensionless Form}
\label{sec:2.5}
We can now formulate the Biot system~\eqref{eq:fulldim-dimensional} in dimensionless form. 
For this, it is convenient to introduce the following notation.
\begin{notation} 
\label{not:fracfac}
We continue to use \Cref{not:pmf} for the non-dimensional domains~$\Omega_\pmf^\epsilon$.
For a triple~$\zeta_\pmf$ of functions on~$\Omega_\pmf^\epsilon$ as introduced in \Cref{not:pmf} and a factor~$\lambda$ that is only accounted for in the fracture domain~$\Omega_\rmf^\epsilon$, we write
\begin{align}
\fracfac{\lambda } \zeta_\pmf := \bigl( \zeta_+ , \zeta_- , \lambda \zeta_\rmf \bigr) .
\end{align}
\end{notation}
The dimensionless form of the Biot problem~\eqref{eq:fulldim-dimensional} now reads as follows.

Find the pressure head~$p_\pmf^\epsilon \colon \Omega_\pmf^\epsilon \times I \rightarrow \mathbb{R}$ and the displacement vector~$\vct{u}_\pmf^\epsilon \colon \Omega_\pmf^\epsilon \times I \rightarrow \mathbb{R}^n$ such that 
\begin{subequations}
\begin{alignat}{2}
\hspace{8em} -\nabla \cdot {\matr{\sigma}}_\pmf^\epsilon ( \vct{u}_\pmf^\epsilon , p_\pmf^\epsilon ) &= \fracfac{\epsilon^{\nu_{\!\vct{f}}\! } } \vct{f}_\pmf^\epsilon \quad\enspace &&\text{in } \Omega_\pmf^\epsilon \times I , \\[8pt]
\notag \partial_t \bigl( \fracfac{\epsilon^{\nu_\omega \! }} \omega_\pmf^\epsilon p_\pmf^\epsilon + \nabla \cdot \bigl(  \fracfac{\epsilon^{\!\matr{\nu}_{\!\matr{\alpha}}\! }} \matr{\alpha}_\pmf^\epsilon \vct{u}_\pmf^\epsilon  \bigr) \bigr) \hspace{2em} & \\
-\: \nabla \cdot \bigl( \fracfac{\epsilon^{\nu_\matr{K}}\! }\matr{K}_\pmf^\epsilon  \nabla p_\pmf^\epsilon \bigr)
&= \fracfac{\epsilon^{\nu_q}\! } q_\pmf^\epsilon \quad\enspace &&\text{in } \Omega_\pmf^\epsilon \times I
\end{alignat}
and the coupling conditions
\begin{alignat}{2}
p_\pm^\epsilon &= p_\rmf^\epsilon \quad\enspace &&\text{on } \gamma_\pm^\epsilon \times I , \\
\vct{u}_\pm^\epsilon &= \vct{u}_\rmf^\epsilon \quad\enspace &&\text{on } \gamma_\pm^\epsilon \times I , \\
\matr{K}_\pm^\epsilon \nabla p_\pm^\epsilon \cdot \vct{n}_\pm &= \epsilon^{\nu_\matr{K}} \matr{K}_\rmf^\epsilon \nabla p_\rmf^\epsilon  \cdot \vct{n}_\pm \quad\enspace &&\text{on } \gamma_\pm^\epsilon \times I , \\
{\matr{\sigma}}_\pm^\epsilon ( \vct{u}_\pm^\epsilon , p_\pm^\epsilon ) \vct{n}_\pm &= {\matr{\sigma}}_\rmf^\epsilon ( \vct{u}_\rmf^\epsilon , p_\rmf^\epsilon ) \vct{n}_\pm \quad\enspace &&\text{on } \gamma_\pm^\epsilon \times I ,
\intertext{as well as the boundary and initial conditions}
\vct{u}_\pmf^\epsilon &= \vct{0} \quad\enspace &&\text{on } \rho_{\pmf}^\epsilon \times I , \\
p_\pmf^\epsilon &= 0 \quad\enspace &&\text{on } \rho_{\pmf, \rmD}^\epsilon \times I , \\
\fracfac{\epsilon^{\nu_\matr{K} \! }} \matr{K}_\pmf^\epsilon \nabla p_\pmf^\epsilon \cdot \vct{n} &= 0 \quad\enspace &&\text{on } \rho_{\pmf, \rmN}^\epsilon \times I , \\
\label{eq:fulldim-nondim-initial-p} p_\pmf^\epsilon ( \cdot , 0 ) &= p_{0, \pmf}^\epsilon \quad\enspace &&\text{on } \Omega_\pmf^\epsilon 
\end{alignat}%
are satisfied.
Here, the total stress~${\matr{\sigma}}_\pmf^\epsilon ( \vct{u}_\pmf^\epsilon , p_\pmf^\epsilon )$ is given by 
\begin{align} 
{\matr{\sigma}}_\pmf^\epsilon ( \vct{u}_\pmf^\epsilon , p_\pmf^\epsilon ) &:= \fracfac{\epsilon^{\nu_{\tsr{C}} \! } } \tsr{C}_\pmf^\epsilon \matr{e} ( \vct{u}_\pmf^\epsilon ) -  \bigl( p_\pmf^\epsilon + G_\pmf^\epsilon \bigr)\fracfac{\epsilon^{\!\matr{\nu}_{\!\matr{\alpha}}\! }}  \matr{\alpha}_\pmf^\epsilon ,
\end{align} 
where $2 \matr{e}( \vct{u}_\pmf^\epsilon ) := \nabla \vct{u}_\pmf^\epsilon + (\nabla \vct{u}_\pmf^\epsilon )^\rmt $.
Given the initial pressure head $p_{0, \pmf}^\epsilon\colon \Omega_\pmf^\epsilon \rightarrow \mathbb{R}$, the initial displacement vector~$\vct{u}_{0, \pmf}^\epsilon := \vct{u}_\pmf^\epsilon ( \cdot , 0 ) \colon \Omega_\pmf^\epsilon \rightarrow \mathbb{R}^n$ satisfies
\begin{align}
-\nabla \cdot {\matr{\sigma}}_\pmf^\epsilon ( \vct{u}_{0,\pmf}^\epsilon , p_{0,\pmf}^\epsilon ) &= \fracfac{\epsilon^{\nu_{\!\vct{f}} \! } } \vct{f}_\pmf^\epsilon (\cdot, 0) \enspace &&\text{in } \Omega_\pmf^\epsilon , \\
\vct{u}_{\pm , 0}^\epsilon &= \vct{u}_{\rmf, 0}^\epsilon \quad\enspace &&\text{on } \gamma_\pm^\epsilon , \\
{\matr{\sigma}}_\pm^\epsilon ( \vct{u}_{0,\pm}^\epsilon , p_{0,\pm}^\epsilon ) \vct{n}_\pm &= {\matr{\sigma}}_\rmf^\epsilon ( \vct{u}_{0,\rmf}^\epsilon , p_{0,\rmf}^\epsilon ) \vct{n}_\pm  &&\text{on } \gamma_\pm^\epsilon , \\
\vct{u}_{0,\pmf }^\epsilon &= \vct{0} \quad\enspace &&\text{on } \rho_\pmf^\epsilon .
\end{align}%
\label{eq:fulldim-nondim}%
\end{subequations}%
\indent Next, we define the function spaces
\begin{subequations}
\begin{align}
\label{eq:spaces_eps_a} \vct{V}^\epsilon &:= \Bigl\{ \vct{v}_\pmf^\epsilon \in  \vct{H}^1_{0, \rho_\pmf^\epsilon }\! \bigl( \Omega_\pmf^\epsilon \bigr)  \ \Big\vert \ \vct{v}_\pm^\epsilon = \vct{v}_\rmf^\epsilon \enspace \text{on } \gamma_\pm^\epsilon \Bigr\} \cong \vct{H}^1_0 (\Omega^\epsilon ), \\
\label{eq:spaces_eps_b} \Phi^\epsilon &:= \Bigl\{ \phi_\pmf^\epsilon \in  H^1_{0, \rho_{\pmf , \rmD}^\epsilon }\! \bigl( \Omega_\pmf^\epsilon \bigr)  \ \Big\vert \ \phi_\pm^\epsilon = \phi_\rmf^\epsilon \enspace \text{on } \gamma_\pm^\epsilon \Bigr\} \cong H^1_{0, \rho^\epsilon_\rmD }  (\Omega^\epsilon ) , \\[1pt]
\label{eq:spaces_eps_c} \vct{\Xi}^\epsilon &:= \bigl\{ \vct{\xi}_\pmf^\epsilon \in \vct{H}_\mathrm{div} ( \Omega_\pmf^\epsilon ) \ \big\vert\ \vct{\xi}_\pm \cdot \vct{n}_\pm = \vct{\xi}_\rmf \cdot \vct{n}_\pm \enspace\text{on } \gamma_\pm^\epsilon \bigr\} \cong \vct{H}_\mathrm{div} (\Omega^\epsilon )  , \\[1pt]
\label{eq:spaces_eps_d} \Lambda^{\!\epsilon} &:=  L^2 \bigl( \Omega_\pmf^\epsilon \bigr) \cong L^2 (\Omega^\epsilon ),
\end{align}%
\label{eq:spaces_eps}%
\end{subequations}%
as well as $\vct{\Lambda}^{\!\epsilon} := (\Lambda^{\!\epsilon} )^n$ and $\matr{\Lambda}^{\!\epsilon} := ( \Lambda^{\!\epsilon} )^{n\times n}$, where $\rho^\epsilon_\rmD := \bigcup_{i \in \{ + , - ,\rmf \} } \!  \rho^\epsilon_{i,\rmD}$ in \cref{eq:spaces_eps_b} and $\vct{H}_{\mathrm{div}} (D) := \{ \vct{v} \in \vct{L}^2 (D) \ \vert \ \nabla \cdot \vct{v} \in L^2 (D) \}$ in \cref{eq:spaces_eps_c}.
The spaces in \cref{eq:spaces_eps} are tailored to our geometric setting with two bulk and a fracture subdomain and are isomorphic to the corresponding spaces on the overall domain~$\Omega^\epsilon$ as indicated.
While the spaces~$\vct{V}^\epsilon $ and~$\Phi^\epsilon$ will respectively serve as solution spaces for the pressure head~$p_\pmf^\epsilon$ and the displacement~$\vct{u}_\pmf^\epsilon$, the space~$\vct{\Xi}^\epsilon$ is only introduced to impose a regularity assumption on the initial velocity~$-\matr{K}_\pmf^\epsilon \nabla p_{0, \pmf}^\epsilon$.
We write $\langle \cdot , \cdot \rangle_{\epsilon }$ for the scalar product on~$\Lambda^{\!\epsilon}$ given by 
\begin{align}
\bigl\langle \phi_\pmf, \psi_\pmf \bigr\rangle_{\!\epsilon } = \sum\nolimits_{i\in\{+,-,\rmf\} } \bigl\langle \phi_i , \psi_i\bigr\rangle_{\Omega_i^\epsilon} \label{eq:inner_Lambda}
\end{align}
and use the same notation for the scalar products on~$\vct{\Lambda}^{\!\epsilon}$ and~$\matr{\Lambda}^{\!\epsilon}$, with the specific space being clear from the context.
Besides, we make the following assumptions on the model coefficients, right-hand sides, and initial data.
\begin{assumption} \label{asm:biot} We assume that
\begin{enumerate}[label=(\roman*),itemsep=0pt,labelsep=0em,topsep=0pt,leftmargin=*,align=left,widest={$(\mathbb{A}_{\vct{f}\!,q})$}]
\item[$(\mathbb{A}_{\matr{\alpha}})$] $\matr{\alpha}_\pmf^\epsilon \colon \Omega_\pmf^\epsilon \rightarrow \mathbb{R}^{n\times n}$ is symmetric and piecewise constant and $\matr{\alpha}_\rmf^\epsilon$ is block-diagonal with respect to the fracture, i.e., $\matr{\alpha}_\rmf^\epsilon \vct{N} \vct{N}^\rmt =  \vct{N}\vct{N}^\rmt \matr{\alpha}_\rmf^\epsilon$,
\item[$(\mathbb{A}_{\vct{f}\!,q})$]  $\vct{f}_\pmf^\epsilon \in H^2 (I ; \vct{\Lambda}^{\!\epsilon })$ and $q_\pmf^\epsilon \in H^1 ( I ; \Lambda^{\!\epsilon } )$,
\item[$(\mathbb{A}_{p_0})$] \customlabel{asm:p0}{$(\mathbb{A}_{p_0})$} $p_{0,\pmf}^\epsilon \in \Phi^\epsilon$ and $\matr{K}_\pmf^\epsilon \nabla p_{0 ,\pmf}^\epsilon \in \vct{\Xi}^\epsilon $,
\item[$(\mathbb{A}_{\omega})$]  $\omega_\pmf^\epsilon \in L^\infty (\Omega_\pmf^\epsilon)$ is almost everywhere uniformly positive, i.e., $\omega_\pmf^\epsilon \gtrsim 1 $,
\item[$(\mathbb{A}_{\tsr{C}})$]  $\tsr{C}_\pmf^\epsilon \in L^\infty ( \Omega_\pmf^\epsilon ; \mathbb{R}^{n \times n \times n \times n} )$ is almost everywhere uniformly elliptic, i.e., ${\tsr{C}}^\epsilon_\pmf \matr{M} : \matr{M} \gtrsim \matr{M} : \matr{M}$ for all symmetric~$\matr{M} \in \mathbb{R}^{n\times n }$, and satisfies the symmetry properties
\begin{align*}
( \tsr{C}_\pmf^\epsilon )_{ijkl} = ( \tsr{C}_\pmf^\epsilon )_{klij} , \quad ( \tsr{C}_\pmf^\epsilon )_{ijkl} = ( \tsr{C}_\pmf^\epsilon )_{jikl} , \quad ( \tsr{C}_\pmf^\epsilon )_{ijkl} = ( \tsr{C}_\pmf^\epsilon )_{ijlk},
\end{align*}
\item[$(\mathbb{A}_{\matr{K}})$]  $\matr{K}_\pmf^\epsilon \in \matr{L}^\infty (\Omega_\pmf^\epsilon ) $ is almost everywhere symmetric and uniformly elliptic, i.e., \linebreak[4] ${\matr{K}}_\pmf^\epsilon \vct{w} \cdot \vct{w}  \gtrsim \vert \vct{w} \vert^2$ for all $\vct{w} \in \mathbb{R}^n$.
\end{enumerate}
\end{assumption}
Then, a weak formulation of the Biot system~\eqref{eq:fulldim-nondim} is given by the following problem. 

Find the pressure head~$p_\pmf^\epsilon \in H^1 ( I ; \Lambda^{\!\epsilon} ) \cap L^2 (I ; \Phi^\epsilon )$ and the displacement vector~$\vct{u}_\pmf^\epsilon \in H^1 (I ; \vct{V}^\epsilon )$ such that $p_\pmf^\epsilon ( 0 ) = p_{0, \pmf}^\epsilon$ and%
\begin{subequations}
\begin{align}
\begin{split}
&\bigl\langle \fracfac{\epsilon^{\nu_\tsr{C}}\! } \tsr{C}_\pmf^\epsilon \matr{e} ( \vct{u}_\pmf^\epsilon ) , \matr{e} ( \vct{v}_\pmf^\epsilon )  \bigr\rangle_{\!\epsilon }  - \bigl\langle  \fracfac{\epsilon^{\!\matr{\nu}_{\!\matr{\alpha}}\! }}   p_\pmf^\epsilon \matr{\alpha}_\pmf^\epsilon   , \nabla \vct{v}_\pmf^\epsilon \bigr\rangle_{\!\epsilon } \\
&\hspace{4.5cm} =  \bigl\langle \fracfac{\epsilon^{\nu_{\!\vct{f}}\! }} \vct{f}_\pmf^\epsilon , \vct{v}_\pmf^\epsilon \bigr\rangle_{ \! \epsilon } + \bigl\langle  \fracfac{\epsilon^{\!\matr{\nu}_{\!\matr{\alpha}}\! }}  G_\pmf^\epsilon \matr{\alpha}_\pmf^\epsilon  , \nabla \vct{v}_\pmf^\epsilon \bigr\rangle_{\!\epsilon } , 
\end{split} \\
\begin{split}
&\bigl\langle \fracfac{\epsilon^{\nu_\matr{K}}\! }\matr{K}_\pmf^\epsilon  \nabla p_\pmf^\epsilon  , \nabla \phi_\pmf^\epsilon \bigr\rangle_{\! \epsilon } + \bigl\langle \fracfac{\epsilon^{\nu_\omega }\! }  \omega_\pmf^\epsilon \partial_t p_\pmf^\epsilon  , \phi_\pmf^\epsilon  \bigr\rangle_{ \! \epsilon }  \\
&\hspace{4.25cm} +  \bigl\langle  \fracfac{\epsilon^{\!\matr{\nu}_{\!\matr{\alpha}}\! }}   \phi_\pmf^\epsilon \matr{\alpha}_\pmf^\epsilon   , \nabla \partial_t \vct{u}_\pmf^\epsilon \bigr\rangle_{\!\epsilon }  
= \bigl\langle \fracfac{\epsilon^{\nu_q}\! } q_\pmf^\epsilon , \phi_\pmf^\epsilon \bigr\rangle_{ \!  \epsilon } 
\end{split}
\intertext{hold for all $\vct{v}_\pmf^\epsilon \in \vct{V}^\epsilon$ and $\phi_\pmf^\epsilon \in \Phi^\epsilon $ and for almost all $t \in I$. Besides, for all $\vct{v}_\pmf^\epsilon \in \vct{V}^\epsilon$, the initial displacement vector~$\vct{u}_\pmf^\epsilon ( 0) =: \vct{u}_{0, \pmf}^\epsilon $ satisfies the compatibility condition }%
 \begin{split}
 &\bigl\langle \fracfac{\epsilon^{\nu_\tsr{C}}\! } \tsr{C}_\pmf^\epsilon \matr{e} ( \vct{u}_{0,\pmf}^\epsilon ) , \matr{e} ( \vct{v}_\pmf^\epsilon )  \bigr\rangle_{\!\epsilon }  - \bigl\langle  \fracfac{\epsilon^{\!\matr{\nu}_{\!\matr{\alpha}}\! }}   p_{0,\pmf}^\epsilon \matr{\alpha}_\pmf^\epsilon   , \nabla \vct{v}_\pmf^\epsilon \bigr\rangle_{\!\epsilon } \\
&\hspace{3.75cm} =  \bigl\langle \fracfac{\epsilon^{\nu_{\!\vct{f}}\! }} \vct{f}_\pmf^\epsilon ( 0) , \vct{v}_\pmf^\epsilon \bigr\rangle_{ \! \epsilon } + \bigl\langle  \fracfac{\epsilon^{\!\matr{\nu}_{\!\matr{\alpha}}\! }}  G_\pmf^\epsilon \matr{\alpha}_\pmf^\epsilon  , \nabla \vct{v}_\pmf^\epsilon \bigr\rangle_{\!\epsilon } .
 \end{split}
\end{align}%
\label{eq:fulldim-weak}%
\end{subequations}%
\indent The wellposedness of the Biot problem~\eqref{eq:fulldim-weak} with full-dimensional fracture is guaranteed by the following proposition. 
\begin{proposition}
\label{prop:wellposed-fulldim}
Given \Cref{asm:biot}, the Biot problem~\eqref{eq:fulldim-weak} admits unique solutions $p_\pmf^\epsilon \in H^1 ( I ; \Phi^\epsilon ) \cap L^\infty (I ; \Phi^\epsilon ) \cap W^{1,\infty } (I ; \Lambda^{\!\epsilon } )$ and $\vct{u}_\pmf^\epsilon \in W^{1,\infty} ( I ; \vct{V}^\epsilon )$ for all~$\epsilon > 0$. 
\end{proposition}
\begin{proof}
By defining the functions in~\cref{eq:fulldim-weak} as functions on the overall domain~$\Omega^\epsilon$, e.g., 
\begin{align*}
 \matr{K}^\epsilon (\vct{x} ) = \begin{cases}
 		\matr{K}_+^\epsilon (\vct{x} ) &\text{if } \vct{x} \in \Omega_+^\epsilon , \\
 		\matr{K}_-^\epsilon (\vct{x} ) &\text{if } \vct{x} \in \Omega_-^\epsilon , \\
 		\epsilon^{\nu_\matr{K}} \matr{K}_\rmf^\epsilon (\vct{x} ) &\text{if } \vct{x} \in \Omega_\rmf^\epsilon ,
 \end{cases} 
\end{align*}
it is easy to see that the domain-decomposed Biot problem~\eqref{eq:fulldim-weak} with fracture is equivalent to a single-domain Biot problem as formulated in \cref{eq:biot_sd_weak} in Appendix~\ref{sec:A}.
The existence and uniqueness of solutions to the single-domain Biot problem is established in Appendix~\ref{sec:A} (see \Cref{thm:wellposed_biot_sd}) and directly implies the stated result.
\end{proof}

\subsection{Coordinate Transformation}
\label{sec:2.6}
The Biot problem~\eqref{eq:fulldim-weak} is formulated in terms of $\epsilon$-dependent domains and hence $\epsilon$-dependent function spaces. 
This makes it difficult to derive a~priori estimates for the solution with respect to an $\epsilon$-independent norm.
Thus, in order to obtain $\epsilon$-independent function spaces, we apply the coordinate transformation 
\begin{subequations}
\begin{alignat}{3}
\label{eq:trafo_pm} \vct{T}_\pm^\epsilon &\colon &\Omega_\pm^\epsilon &\rightarrow \Omega_\pm^0 , \quad &\vct{x} &\mapsto \vct{x} \mp \epsilon a_\pm \bigl( \vct{x} - ( \vct{x} \cdot \vct{N}) \vct{N} \bigr) \vct{N} , \\
\vct{T}_\rmf^\epsilon &\colon &\Omega_\rmf^\epsilon &\rightarrow \Omega_\rmf^1 , \quad &\vct{x} &\mapsto \vct{x} + \bigl( \epsilon^{-1} - 1 \bigr) ( \vct{x} \cdot \vct{N} ) \vct{N}
\end{alignat}%
\label{eq:trafo}%
\end{subequations}%
to the Biot problem~\eqref{eq:fulldim-weak}.
This means that the transformation~$\vct{T}_\pm^\epsilon$ translates the bulk domains~$\Omega_\pm^\epsilon$ in normal direction so that they directly abut on the interface~$\gamma$.
In contrast, the transformation~$\vct{T}_\rmf^\epsilon$ rescales the normal coordinate inside the fracture domain~$\Omega_\rmf^\epsilon$ by~$\epsilon$, which also affects normal derivatives. 
In particular, the corresponding Jacobian determinants are 
\begin{align}
\det \bigl( \nabla \vct{T}_\pm^\epsilon ( \vct{x} ) \bigr) = 1 \quad\ \text{and} \quad\ \det \bigl( \nabla \vct{T}_\rmf^\epsilon (\vct{x} ) \bigr) = \epsilon^{-1} .
\end{align}
The inverse transformations~$(\vct{T}_\pm^\epsilon )^{-1}$ and $( \vct{T}_\rmf^\epsilon )^{-1}$ are given by 
\begin{subequations}
\begin{alignat}{2}
( \vct{T}_\pm^\epsilon )^{ -1} \! &\colon \Omega_\pm^0 \rightarrow \Omega_\pm^\epsilon , \quad &\vct{x} &\mapsto \vct{x} \pm \epsilon a_\pm \bigl( \vct{x} - ( \vct{x} \cdot \vct{N}) \vct{N} \bigr) \vct{N}  , \\
( \vct{T}_\rmf^\epsilon )^{ -1} \! &\colon \Omega_\rmf^1 \rightarrow \Omega_\rmf^\epsilon , \quad &\vct{x} &\mapsto \vct{x} + ( \epsilon - 1 ) (\vct{x} \cdot \vct{N} ) \vct{N} .
\end{alignat}
\end{subequations}

For $\phi_\rmf \in H^1 (\Omega_\rmf^1 ) $ and $\vct{v}_\rmf \in \vct{H}^1 (\Omega_\rmf^1 ) $, the gradient~$\nabla \phi_\rmf$ and Jacobian~$\nabla \vct{v}_\rmf$ under the transformation~$\vct{T}_\rmf^\epsilon$ satisfy
\begin{align}
\nabla \bigl( \phi_\rmf \circ \vct{T}_\rmf^\epsilon \bigr) ( \vct{x} ) &= \nabla^\epsilon \phi_\rmf \bigl( \vct{T}_\rmf^\epsilon ( \vct{x} ) \bigr) \quad\ \text{and}\quad\ \nabla \bigl( \vct{v}_\rmf \circ \vct{T}_\rmf^\epsilon \bigr) (\vct{x} ) = \nabla^\epsilon \vct{v}_\rmf \bigl( \vct{T}_\rmf^\epsilon ( \vct{x} ) \bigr).
\end{align}
Here, we define the $\epsilon$-scaled gradient~$\nabla^\epsilon \phi_{\rmf}$ and Jacobian~$\nabla^\epsilon \vct{v}_{\rmf}$ by 
\begin{subequations}
\begin{alignat}{2}
\nabla^\epsilon \phi_{\rmf} &:= \nabla_{\!\smallpar} \phi_{\rmf} + \epsilon^{-1} \nabla_{\!\vct{N}} \phi_{\rmf} , \quad\enspace &\nabla^\epsilon \vct{v}_{\rmf} &:= \nabla_{\!\smallpar} \vct{v}_{\rmf} + \epsilon^{-1} \nabla_{\!\vct{N}} \vct{v}_{\rmf} ,\\
\intertext{with the normal and tangential components given by}
\nabla_{\!\vct{N}} \phi_{\rmf} &:= ( \partial_{\!\vct{N}} \phi_{\rmf} ) \vct{N} , \quad\enspace &\nabla_{\!\vct{N}} \vct{v}_{\rmf} &:= ( \partial_{\!\vct{N}} \vct{v}_\rmf )  \vct{N}^\rmt , \\
\nabla_{\!\smallpar} \phi_{\rmf} &:= \nabla \phi_{\rmf} - \nabla_{\!\vct{N}} \phi_{\rmf} , \quad\enspace   &\nabla_{\!\smallpar} \vct{v}_{\rmf} &:= \nabla \vct{v}_{\rmf} - \nabla_{\!\vct{N}} \vct{v}_{\rmf},
\end{alignat}%
\label{eq:nablaeps}%
\end{subequations}%
where $\partial_{\!\vct{N}} \phi_{\rmf} := \nabla \phi_{\rmf} \cdot \vct{N}$ and $\partial_{\!\vct{N}} \vct{v}_\rmf := (\nabla \vct{v}_\rmf ) \vct{N}$.
Besides, we write 
\begin{align} \label{eq:straineps}
\matr{e}^\epsilon (\vct{v}_{\rmf} ) &:= \tfrac{1}{2} \bigl( \nabla^\epsilon \vct{v}_{\rmf} + ( \nabla^\epsilon \vct{v}_{\rmf} )^\rmt \bigr) 
\end{align}
for the $\epsilon$-scaled strain tensor in~$\Omega_\rmf^1$.  

Further, for $\phi_\pm \in H^1 (\Omega_\pm^0 )$ and $\vct{v}_\pm \in \vct{H}^1 (\Omega_\pm^0) $, we obtain
\begin{subequations}
\begin{align}
\nabla \bigl( \phi_\pm \circ \vct{T}_\pm^\epsilon \bigr) (\vct{x} ) &= \matr{M}_\pm^\epsilon \bigl( \vct{T}_\pm^\epsilon (\vct{x} ) \bigr)  \nabla \phi_\pm \bigl( \vct{T}_\pm^\epsilon ( \vct{x} ) \bigr), \\
\nabla \bigl( \vct{v}_\pm \circ \vct{T}_\pm^\epsilon \bigr) (\vct{x} ) &=  \nabla \vct{v}_\pm \bigl( \vct{T}_\pm^\epsilon (\vct{x} ) \bigr) \bigl[\matr{M}_\pm^\epsilon \bigl( \vct{T}_\pm^\epsilon ( \vct{x}) \bigr) \bigr]^\rmt .
\end{align}%
\label{eq:gradeps}%
\end{subequations}%
Here, we have defined
\begin{align}
\matr{M}_\pm^\epsilon \colon \Omega_\pm^0 \rightarrow \mathbb{R}^{n\times n }, \quad \vct{x} \mapsto \matr{I}_n \mp \epsilon \nablapar a_\pm \bigl( \vct{x} - (\vct{x} \cdot \vct{N} ) \vct{N}\bigr) \vct{N}^\rmt . \label{eq:Meps}
\end{align}
In addition, we set $\matr{M}_\rmf^\epsilon (\vct{x} ) := \matr{I}_n$ in order to use the compact notation~$\matr{M}_\pmf^\epsilon$.
Besides, we denote the transformed strain tensor in~$\Omega_\pm^0$ by
\begin{align}
\matr{E}_\pm^\epsilon \bigl( \vct{v}_\pm \bigr) (\vct{x} ) := \tfrac{1}{2} \Bigl( \nabla \vct{v}_\pm (\vct{x} ) \bigl[ \matr{M}_\pm^\epsilon (\vct{x} )  \bigr]^\rmt  + \matr{M}_\pm^\epsilon (\vct{x} ) \bigl[ \nabla \vct{v}_\pm (\vct{x} ) \bigr]^\rmt \Bigr) \label{eq:Eeps}
\end{align}
and set $\matr{E}_\rmf^\epsilon ( \vct{v}_\rmf ) := \matr{e}^\epsilon (\vct{v}_\rmf ) $.%

Next, we define the transformed model coefficients, right-hand sides, and initial data by
\begin{equation} 
\begin{alignedat}{3}
\hat{\tsr{C}}_\pmf^\epsilon &:= \tsr{C}_\pmf^\epsilon \circ (\vct{T}_\pmf^\epsilon )^{-1} , \quad\enspace &\hat{\matr{\alpha}}_\pmf^\epsilon &:= \matr{\alpha}_\pmf^\epsilon \circ (\vct{T}_\pmf^\epsilon )^{-1} , \quad\enspace &\hat{\omega}_\pmf^\epsilon &:= \omega_\pmf^\epsilon \circ (\vct{T}_\pmf^\epsilon)^{-1} , \\
\hat{\matr{K}}_\pmf^\epsilon &:= \matr{K}_\pmf^\epsilon \circ (\vct{T}_\pmf^\epsilon )^{-1} , \quad\enspace &\hat{\vct{f}}_\pmf^\epsilon &:= \vct{f}_\pmf^\epsilon \circ (\vct{T}_\pmf^\epsilon )^{-1} , \quad\enspace &\hat{q}_\pmf^\epsilon &:= q_\pmf^\epsilon \circ (\vct{T}_\pmf^\epsilon )^{-1} , \\
\hat{G}_\pmf^\epsilon &:= G_\pmf^\epsilon \circ (\vct{T}_\pmf^\epsilon )^{-1} , \quad\enspace &\hat{p}_{0,\pmf}^\epsilon &:= p_{0, \pmf}^\epsilon \circ (\vct{T}_\pmf^\epsilon )^{-1} .
\end{alignedat}%
\label{eq:paramtrafo}%
\end{equation}%
In addition, we introduce the following notation.
\begin{notation} \label{not:odot}
We continue to use the \Cref{not:pmf,not:fracfac} for the transformed domains~$\Omega_+^0$, $\Omega_-^0$, and~$\Omega_\rmf^1$. 
Besides, we write the index~\enquote{$\odot$\!} in a bulk-fracture triple of mathematical objects (with index~\enquote{$\pmf$}) when referring to the index~\enquote{$0$} in the bulk domains~$\Omega_\pm^0$ and to the index~\enquote{$1$} in the fracture domain~$\Omega_\rmf^1$. 
In addition, we use the \enquote{fracture bracket}~$\fracfac{\cdot }$ not only to indicate that a prefactor in a bulk-fracture triple is only present for the fracture domain~$\Omega_\rmf^1$ as introduced in \Cref{not:fracfac}, but also as an index~\enquote{$\fracfac{\epsilon } \!$} to indicate a dependency on the scaling parameter~$\epsilon$ that is only present inside the fracture domain~$\Omega_\rmf^1$.  
For example, we write
\begin{subequations}
\begin{align}
\Omega_\pmf^\odot &:= \bigl( \Omega_+^0 , \Omega_-^0 , \Omega_\rmf^1 \bigr) , \\ \fracfac{\epsilon^{\nu_\matr{K} + 1 }} \hat{\matr{K}}_\pmf^\epsilon  \nabla^{\fracfac{\!\epsilon\! }} \hat{p}_\pmf^\epsilon &:= \bigl( \hat{\matr{K}}_+^\epsilon \nabla \hat{p}_+^\epsilon ,\ \hat{\matr{K}}_-^\epsilon \nabla \hat{p}_-^\epsilon ,\ \epsilon^{\nu_\matr{K} + 1} \hat{\matr{K}}_\rmf^\epsilon \nabla^\epsilon \hat{p}^\epsilon_\rmf \bigr) .
\end{align}
\end{subequations}
\end{notation}

Further, we define the transformed function spaces 
\begin{subequations}
\begin{align}
\label{eq:V} \vct{V} &:= \Bigl\{ \vct{v}_\pmf \in  \vct{H}^1_{0, \rho_\pmf^\odot } \bigl( \Omega_\pmf^\odot \bigr)  \ \Big\vert \ \vct{v}_\pm = \Pi_\pm \vct{v}_\rmf \enspace \text{on } \gamma \Bigr\} , \\
\label{eq:Phi} \Phi &:= \Bigl\{ \phi_\pmf \in H^1_{0, \rho_{\pmf, \rmD}^\odot } \! \bigl( \Omega_\pmf^\odot \bigr)  \ \Big\vert \ \phi_\pm = \Pi_\pm \phi_\rmf \enspace \text{on } \gamma \Bigr\}  , \\[2pt]
\label{eq:Lambda} \Lambda &:=  L^2 \bigl( \Omega_\pmf^\odot \bigr) .
\end{align}
\end{subequations}
Here, $\Pi_\pm$ denotes the projections from the fracture boundaries~$\gamma_\pm^1$ onto the bulk boundary~$\gamma$ given by
\begin{align}
\Pi_\pm \colon L^2 (\gamma_\pm^1 ) \rightarrow L^2 (\gamma ) , \quad \bigl( \Pi_\pm \phi_\pm \bigr) (\vct{x} ) := \phi_\pm \bigl( \vct{T}_\pm^1 (\vct{x} ) \bigr) .
\label{eq:proj}
\end{align}
Moreover, let $\vct{\Lambda} := \Lambda^n$ and $\matr{\Lambda} := \Lambda^{n \times n }$.
We write~$\langle \cdot , \cdot \rangle$ for the scalar product in~$\Lambda$, $\vct{\Lambda}$, and~$\matr{\Lambda}$ with analogous definition as in \cref{eq:inner_Lambda} and the specific space being clear from the context.
Then, the application of the coordinate transformation~\eqref{eq:trafo} to the Biot problem~\eqref{eq:fulldim-weak} yields the following transformed problem.

Find the pressure head~$\hat{p}_\pmf^\epsilon \in H^1 (I ; \Lambda ) \cap L^2 ( I ; \Phi )$ and the displacement vector~$\hat{\vct{u}}_\pmf^\epsilon \in H^1 (I ; \vct{V} )$ such that 
\begin{subequations}
\begin{align}
&\hat{p}_\pmf^\epsilon ( 0 ) = \hat{p}_{0 , \pmf}^\epsilon
\intertext{and the equations}
\label{eq:fulldim-weak-trafo-b}  
\begin{split}
& \bigl\langle \fracfac{\epsilon^{\nu_\tsr{C} + 1}} \hat{\tsr{C}}_\pmf^\epsilon \matr{E}_\pmf^\epsilon ( \hat{\vct{u}}_\pmf^\epsilon ) , \matr{E}_\pmf^\epsilon ( \vct{v}_\pmf )  \bigr\rangle  - \bigl\langle \fracfac{\epsilon^{\!\matr{\nu}_{\!\matr{\alpha}} + \matr{I}_n\!  } } \hat{p}_\pmf^\epsilon \hat{\matr{\alpha}}_\pmf^\epsilon  , \nabla^{\fracfac{\!\epsilon\! }} \vct{v}_\pmf \bigl[ \matr{M}_\pmf^\epsilon \bigr]^\rmt \bigr\rangle  \\
&\hspace{2.3cm} =  \bigl\langle \fracfac{\epsilon^{\nu_{\!\vct{f}} + 1} }\hat{\vct{f}}_\pmf^\epsilon , \vct{v}_\pmf \bigr\rangle + \bigl\langle  \fracfac{\epsilon^{\!\matr{\nu}_{\!\matr{\alpha}} + \matr{I}_n\!  } } \hat{G}_\pmf^\epsilon \hat{\matr{\alpha}}_\pmf^\epsilon , \nabla^\fraceps \vct{v}_\pmf \bigl[ \matr{M}_\pmf^\epsilon \bigr]^\rmt \bigr\rangle  ,  
\end{split} \\
\label{eq:fulldim-weak-trafo-c}   
\begin{split} &\bigl\langle \fracfac{\epsilon^{\nu_\matr{K} + 1 }} \hat{\matr{K}}_\pmf^\epsilon \matr{M}_\pmf^\epsilon \nablaeps \hat{p}_\pmf^\epsilon  , \matr{M}_\pmf^\epsilon \nablaeps \phi_\pmf \bigr\rangle + \bigl\langle \fracfac{\epsilon^{\nu_\omega +1 }}  \hat{\omega}_\pmf^\epsilon \partial_t \hat{p}_\pmf^\epsilon  , \phi_\pmf \bigr\rangle  \\
&\hspace{2.1cm}  +  \bigl\langle \fracfac{\epsilon^{\!\matr{\nu}_{\!\matr{\alpha}} + \matr{I}_n\!  } } \phi_\pmf \hat{\matr{\alpha}}_\pmf^\epsilon  , \nablaeps \partial_t \hat{\vct{u}}_\pmf^\epsilon \bigl[ \matr{M}_\pmf^\epsilon \bigr]^\rmt \bigr\rangle = \bigl\langle \fracfac{\epsilon^{\nu_q + 1}} \hat{q}_\pmf^\epsilon , \phi_\pmf \bigr\rangle
\end{split}
\intertext{hold for all $\vct{v}_\pmf \in \vct{V}$ and $\phi_\pmf \in \Phi$ and almost all~$t \in I$. Besides, for all  $\vct{v}_\pmf \in \vct{V}$,  the initial displacement vector~$\hat{\vct{u}}_{0,\pmf}^\epsilon \in \vct{V}$ satisfies}
\label{eq:fulldim-weak-trafo-d}  
\begin{split}
& \bigl\langle \fracfac{\epsilon^{\nu_\tsr{C} + 1}} \hat{\tsr{C}}_\pmf^\epsilon \matr{E}_\pmf^\epsilon ( \hat{\vct{u}}_{0,\pmf}^\epsilon ) , \matr{E}_\pmf^\epsilon ( \vct{v}_\pmf )  \bigr\rangle  - \bigl\langle \fracfac{\epsilon^{\!\matr{\nu}_{\!\matr{\alpha}} + \matr{I}_n\!  } } \hat{p}_{0,\pmf}^\epsilon \hat{\matr{\alpha}}_\pmf^\epsilon  , \nablaeps \vct{v}_\pmf  \bigl[ \matr{M}_\pmf^\epsilon \bigr]^\rmt \bigr\rangle  \\
&\hspace{2cm} =  \bigl\langle \fracfac{\epsilon^{\nu_{\!\vct{f}} + 1} }\hat{\vct{f}}_\pmf^\epsilon ( 0) , \vct{v}_\pmf \bigr\rangle + \bigl\langle  \fracfac{\epsilon^{\!\matr{\nu}_{\!\matr{\alpha}} + \matr{I}_n\!  } } \hat{G}_\pmf^\epsilon \hat{\matr{\alpha}}_\pmf^\epsilon , \nablaeps \vct{v}_\pmf \bigl[ \matr{M}_\pmf^\epsilon \bigr]^\rmt \bigr\rangle . 
\end{split}
\end{align}%
\label{eq:fulldim-weak-trafo}%
\end{subequations}%

The transformed Biot problem~\eqref{eq:fulldim-weak-trafo} is wellposed as a consequence of \Cref{prop:wellposed-fulldim}.
\begin{corollary}
Given \Cref{asm:biot}, the transformed Biot problem~\eqref{eq:fulldim-weak-trafo} has unique solutions $\hat{p}_\pmf^\epsilon \in H^1 (I ; \Phi ) \cap L^\infty ( I; \Phi ) \cap W^{1,\infty } ( I ; \Lambda )$ and~$\hat{\vct{u}}_\pmf^\epsilon \in W^{1,\infty } (I ; \vct{V} )$ for all~$\epsilon > 0$. 
\end{corollary}

\section{A~Priori Estimates and Convergence}
\label{sec:3}
In the following, we consider the family of solutions~$\{ \hat{p}_\pmf^\epsilon \}_{\epsilon \in ( 0, \epsilon^\ast] } \subset \Phi$ and $\{ \hat{\vct{u}}_\pmf^\epsilon \}_{\epsilon \in ( 0, \epsilon^\ast] }  \linebreak[2] \subset \vct{V}$ of the transformed Biot problem~\eqref{eq:fulldim-weak-trafo} for $\epsilon^\ast > 0$ sufficiently small. 
We first formulate \Cref{asm:eps} on the model coefficients, right-hand sides, and initial data, as well as \Cref{asm:nu} on the scaling parameters from \cref{eq:eps_scaling1,eq:eps_scaling2}.
Then, in \Cref{thm:apriori}, we present a~priori estimates and, as a consequence, find a subsequence~$\{ \epsilon_k \}_{k \in \mathbb{N}} \subset ( 0, \epsilon^\ast]$ with $\epsilon_k \searrow 0$ as $k \rightarrow \infty$ along which these solutions families converge in a weak sense; see \Cref{cor:conv}. 
The proof of \Cref{thm:apriori} is deferred to \Cref{sec:5}.
The weak convergence results in \Cref{cor:conv} enable us to identify the corresponding limit models for~$\epsilon_k \searrow 0$ as presented in \Cref{sec:4}, with the proofs given in~\Cref{sec:6}; 
uniqueness of the limit solutions then implies convergence of the entire sequence.
In addition, once the limit models have been identified, we establish strong convergence results, stated in \Cref{thm:strongconv} and proven in \Cref{sec:7}.
We continue to use the notation introduced in \Cref{sec:2}, and in particular recall the bulk-fracture triple notation introduced in the \Cref{not:pmf,not:fracfac,not:odot}, as well as the definition of the $\epsilon$-scaled gradient~$\nabla^\epsilon$ and strain tensor~$\matr{e}^\epsilon $ in \cref{eq:nablaeps,eq:straineps}.

We assume that the $\epsilon$-dependent model coefficients, right-hand sides, and initial data of the transformed Biot problem~\eqref{eq:fulldim-weak-trafo} satisfy the following properties.
Some of these assumptions have already been stated in \Cref{asm:biot}, where they were required to prove the wellposedness of the Biot problem~\eqref{eq:fulldim-weak} with fixed~$\epsilon$ in \Cref{prop:wellposed-fulldim}, and are restated here in terms of the transformed functions defined in \cref{eq:paramtrafo}.
In addition, since we now consider a sequence of problems parameterized by~$\epsilon$, we introduce new assumptions concerning the strong convergence and uniform boundedness of the model coefficients, right-hands sides, and initial data as~$\epsilon \to 0$. 
\begin{assumption} 
\label{asm:eps}
For $\epsilon > 0 $, we assume that 
\begin{enumerate}[label=(\roman*),itemsep=3pt,topsep=0pt,leftmargin=*,align=left,widest={$(\mathbb{A}_{p_0})$\hspace*{-0.8em}}]
\item[$(\mathbb{A}_{\matr{\alpha}})$] \customlabel{asm:alpha_eps}{$(\mathbb{A}_{\matr{\alpha}})$}
\begin{itemize}
\item $\hat{\matr{\alpha}}_\pmf^\epsilon \colon \Omega_\pmf^\odot \rightarrow \mathbb{R}^{n\times n}$ is symmetric and piecewise constant,  
\item $\hat{\matr{\alpha}}_\rmf^\epsilon$ is block-diagonal with respect to the fracture, i.e., $\hat{\matr{\alpha}}_\rmf^\epsilon \vct{N}\! \vct{N}^\rmt = \vct{N}\! \vct{N}^\rmt  \hat{\matr{\alpha}}_\rmf^\epsilon $,
\item $\hat{\matr{\alpha}}_\pmf^\epsilon \rightarrow \hat{\matr{\alpha}}_\pmf$ strongly in~$\matr{L}^\infty (\Omega_\pmf^\odot )$ as $\epsilon \rightarrow 0$,
\end{itemize}
\item[$(\mathbb{A}_{\vct{f}})$] \customlabel{asm:f_eps}{$(\mathbb{A}_{\vct{f}})$}
\begin{itemize}
\item $\norm[\big ]{\hat{\vct{f}}_\pmf^\epsilon }_{H^2 ( I ; \vct{\Lambda} )}  \lesssim 1$, 
\item $\hat{\vct{f}}_\pmf^\epsilon \rightarrow \hat{\vct{f}}_\pmf$ strongly in $H^1 (I ; \vct{\Lambda} )$ as $\epsilon \rightarrow 0$, 
\end{itemize}
\item[$(\mathbb{A}_{q})$] \customlabel{asm:q_eps}{$(\mathbb{A}_{q})$}
\begin{itemize}
\item $\norm[\big ]{\hat{q}_\pmf^\epsilon }_{H^1 ( I ; \Lambda )}  \lesssim 1$, 
\item $\hat{q}_\pmf^\epsilon \rightarrow \hat{q}_\pmf$ strongly in~$L^2 ( I ; \Lambda )$ as $\epsilon \rightarrow 0$, 
\end{itemize}
\item[$(\mathbb{A}_{p_0})$] \customlabel{asm:p0_eps}{$(\mathbb{A}_{p_0})$}
\begin{itemize}
\item $p_{0,\pmf}^\epsilon = \hat{p}_{0,\pmf}^\epsilon \circ \vct{T}_\pmf^\epsilon$ satisfies \ref{asm:p0} in \Cref{asm:biot},
\item $\norm{\fracfac{\epsilon^{( \nu_\matr{K} +1 )/2 } } \nabla^\fraceps \hat{p}_{0,\pmf}^\epsilon }^2_{\vct{\Lambda}} \lesssim 1 $, \vspace{2pt}
\item $\hat{p}_{0,\pmf}^\epsilon \rightarrow \hat{p}_{0,\pmf}$ strongly in~$\Phi$ as $\epsilon \rightarrow 0$,
\end{itemize}
\item[$(\mathbb{A}_{\omega})$] \customlabel{asm:omega_eps}{$(\mathbb{A}_{\omega})$}
\begin{itemize}
\item $\hat{\omega}_\pmf^\epsilon \in L^\infty (\Omega_\pmf^\odot)$ is almost everywhere uniformly positive, i.e., $\omega_\pmf^\epsilon \gtrsim 1 $, 
\item $\hat{\omega}_\pmf^\epsilon \rightarrow \hat{\omega}_\pmf$ strongly in $L^\infty (\Omega_\pmf^\odot )$ as $\epsilon \rightarrow 0$,
\end{itemize}
\item[$(\mathbb{A}_{\tsr{C}})$] \customlabel{asm:C_eps}{$(\mathbb{A}_{\tsr{C}})$}
\begin{itemize}
\item $\hat{\tsr{C}}_\pmf^\epsilon \in L^\infty ( \Omega_\pmf^\odot ; \mathbb{R}^{n \times n \times n \times n} )$ is almost everywhere uniformly elliptic, i.e., \linebreak[4] $\hat{\tsr{C}}^\epsilon_\pmf \matr{M} : \matr{M}  \gtrsim \matr{M} : \matr{M}$ for all symmetric $\matr{M} \in \mathbb{R}^{n\times n }$, 
\item  $\hat{\tsr{C}}_\pmf^\epsilon$ satisfies the symmetry properties
\begin{align*}
( \hat{\tsr{C}}_\pmf^\epsilon )_{ijkl} = ( \hat{\tsr{C}}_\pmf^\epsilon )_{klij} , \quad ( \hat{\tsr{C}}_\pmf^\epsilon )_{ijkl} = ( \hat{\tsr{C}}_\pmf^\epsilon )_{jikl} , \quad ( \hat{\tsr{C}}_\pmf^\epsilon )_{ijkl} = ( \hat{\tsr{C}}_\pmf^\epsilon )_{ijlk},
\end{align*}
\item $\hat{\tsr{C}}_\pmf^\epsilon \rightarrow \hat{\tsr{C}}_\pmf$ strongly in~$L^\infty ( \Omega_\pmf^\odot ; \mathbb{R}^{n\times n \times n \times n } )$ as $\epsilon \rightarrow 0$,
\end{itemize}
\item[$(\mathbb{A}_{\matr{K}})$] \customlabel{asm:K_eps}{$(\mathbb{A}_{\matr{K}})$}
\begin{itemize}
\item $\hat{\matr{K}}_\pmf^\epsilon \in \matr{L}^\infty (\Omega_\pmf^\odot ) $ is almost everywhere symmetric and uniformly elliptic, i.e., $\hat{\matr{K}}_\pmf^\epsilon \vct{w} \cdot \vct{w} \gtrsim \vert \vct{w} \vert^2$ for all $\vct{w} \in \mathbb{R}^n$,
\item $\hat{\matr{K}}_\pmf^\epsilon \rightarrow \hat{\matr{K}}_\pmf$ strongly in~$\matr{L}^\infty (\Omega_\pmf^\odot )$ as $\epsilon \rightarrow 0$.
\end{itemize}
\end{enumerate}
\end{assumption}
Besides, the convergence of the gravitational term~$\hat{G}_\pmf^\epsilon$  can be characterized as follows as $\epsilon \rightarrow 0$.
\begin{lemma} \label{lem:G}
As $\epsilon \rightarrow 0$, we have 
\begin{align}
\hat{G}_\pmf^\epsilon \rightarrow \hat{G}_\pmf \quad\enspace \text{strongly in } L^\infty (\Omega_\pmf^\odot ) , 
\end{align}
where $\hat{G}_\pm ( \vct{x} ) := \vct{x} \cdot \vct{g}$ and $\hat{G}_\rmf ( \vct{x} ) := ( \vct{x} - (\vct{x} \cdot \vct{N} ) \vct{N} ) \cdot \vct{g}$.
\end{lemma}
\begin{proof}
The functions $\hat{G}_\pmf^\epsilon \colon \Omega_\pmf^\odot \rightarrow \mathbb{R}$ as defined in \cref{eq:paramtrafo} are given by 
\begin{align*}
\hat{G}_\pm^\epsilon ( \vct{x} ) = \bigl( \vct{x} \pm \epsilon a_\pm \bigl( \vct{x} - (\vct{x} \cdot \vct{N} ) \vct{N} \bigr)  \vct{N} \bigr) \cdot \vct{g} , \quad\enspace \hat{G}_\rmf^\epsilon ( \vct{x} ) = \bigl( \vct{x} + (\epsilon - 1) (\vct{x} \cdot \vct{N}) \vct{N} \bigr) \cdot \vct{g}  
\end{align*}
so that the result follows immediately.
\end{proof}
Further, we make the following assumptions on the scaling parameters from \cref{eq:eps_scaling1,eq:eps_scaling2}.
\begin{assumption} 
\label{asm:nu}
We assume that 
\begin{enumerate}[label=(\roman*),labelsep=0.2em,leftmargin=*,align=left,widest={$(\mathbb{A}_{\nu_\vct{f}})$}]
\item[$(\mathbb{A}_{\nu_\omega})$]  \customlabel{asm:nuomega}{$(\mathbb{A}_{\nu_\omega})$}  $\nu_\omega \ge -1$,
\item[$(\mathbb{A}_{\nu_\tsr{C}})$] \customlabel{asm:nuC}{$(\mathbb{A}_{\nu_\tsr{C}})$} $\nu_\tsr{C} \in \mathbb{R}$,
\item[$(\mathbb{A}_{\nu_\matr{K}})$] \customlabel{asm:nuK}{$(\mathbb{A}_{\nu_\matr{K}})$} $\nu_{\matr{K}}\in \mathbb{R}$, 
\item[$(\mathbb{A}_{\nu_q})$] \customlabel{asm:nuq}{$(\mathbb{A}_{\nu_q})$} $2 \nu_q \ge \nu_\omega - 1 $ or $(2 \nu_q \ge \nu_\matr{K} - 3$ and $\nu_q \ge -1)$, 
\item[$(\mathbb{A}_{\nu_\vct{f}})$] \customlabel{asm:nuf}{$(\mathbb{A}_{\nu_\vct{f}})$} $2 \nu_{\!\vct{f}} \ge \nu_{\tsr{C}} - 3$ and $\nu_{\!\vct{f}} \ge -1$,
\item[$(\mathbb{A}_{\nu_{\!\matr{\alpha}}})$] \customlabel{asm:nualpha}{$(\mathbb{A}_{\nu_{\!\matr{\alpha}}})$} $2 \nu_{\!\matr{\alpha}}^\smallpar \ge \max\{ \nu_{\tsr{C}} , 0\}  - 1$ and $2\nu_{\!\matr{\alpha}}^\smallperp \ge \max \{ \nu_{\tsr{C}} , 0 \} - 1$,
\item[$(\mathbb{A}_{\rho_\rmD} )$] \customlabel{asm:dirichlet}{$(\mathbb{A}_{\rho_\rmD} )$} $\abs{\rho^0_{+, \rmD}} > 0$ and $\abs{\rho^0_{-, \rmD}} > 0$ if $\nu_\matr{K} > 0$, otherwise  $\abs{\rho^0_{+, \rmD}} > 0$ or $\abs{\rho^0_{-, \rmD}} > 0$.
\end{enumerate}
\end{assumption}
While the value of the scaling parameter~$\nu_\tsr{C}$ is not restricted in \Cref{asm:eps}~\ref{asm:nuC} and the results in \Cref{thm:apriori,cor:conv} below generally hold for $\nu_\tsr{C} \in \mathbb{R}$, we only obtain closed limit models for the case~$\nu_\tsr{C} \ge 1$ (see \Cref{sec:4}).

We now obtain the following a~priori estimates for the solutions of the transformed Biot problem~\eqref{eq:fulldim-weak-trafo}.
\begin{theorem}
\label{thm:apriori} 
Let $\theta(\nu ) := \max\{ 0 , \tfrac{1}{2} ( \nu - 1)\}$ and $\iota ( \nu )  : = \tfrac{1}{2} ( \nu + 1 ) $.
Given the \Cref{asm:eps,asm:nu}, there exists $\epsilon^\ast > 0$ such that, for all $\epsilon \in ( 0, \epsilon^\ast ]$, the solutions~$\hat{p}_\pmf^\epsilon \in \Phi$ and $\hat{\vct{u}}_\pmf^\epsilon \in\vct{V}$ of the transformed Biot system~\eqref{eq:fulldim-weak-trafo} satisfy the a~priori estimates 
\begin{subequations}
\begin{align}
\label{eq:apriori_a} \norm[\big ]{\fracfac[\big ]{\epsilon^{\theta ( \nu_{\tsr{C}} )  } } \hat{\vct{u}}_\pmf^\epsilon }_{L^\infty ( I ; \vct{\Lambda})} + \norm[\big ]{\fracfac[\big ]{\epsilon^{\theta ( \nu_{\tsr{C}} ) } } \hat{\vct{u}}_\pmf^\epsilon }_{H^1 ( I ; \vct{\Lambda})}  &\lesssim 1 , \\
\label{eq:apriori_b} \norm[\big ]{\fracfac[\big ]{\epsilon^{\max\{ \frac{1}{2}\! ,\, \iota (\nu_\tsr{C} ) \}\! }}  \nabla^\fraceps \hat{\vct{u}}_\pmf^\epsilon }_{L^\infty ( I ; \matr{\Lambda }) } +  \norm[\big ]{\fracfac[\big ]{\epsilon^{\max\{ \frac{1}{2}\! ,\, \iota (\nu_\tsr{C} ) \}\!}}  \nabla^\fraceps \hat{\vct{u}}_\pmf^\epsilon }_{H^1 ( I ; \matr{\Lambda }) }  &\lesssim 1 , \\
\label{eq:apriori_c}  \norm[\big ]{\fracfac[\big ]{\epsilon^{\iota (\nu_\tsr{C} )}} \matr{e}^\fraceps (  \hat{\vct{u}}_\pmf^\epsilon)  }_{L^\infty ( I ; \matr{\Lambda }) } +  \norm[\big ]{\fracfac[\big ]{\epsilon^{\iota (\nu_\tsr{C} )}}  \matr{e}^\fraceps ( \hat{\vct{u}}_\pmf^\epsilon ) }_{H^1 ( I ; \matr{\Lambda }) }  &\lesssim 1 , \\
\label{eq:apriori_d} \norm[\big ]{\fracfac[\big ]{\epsilon^{ \min\{ \iota ( \nu_\omega ) ,\,  \theta ( \nu_\matr{K} )  \} \! }} \hat{p}_\pmf^\epsilon }_{L^\infty ( I ; \Lambda ) } + \norm[\big ]{\fracfac[\big ]{\epsilon^{\iota (\nu_\omega ) }} \hat{p}_\pmf^\epsilon }_{H^1 ( I ; \Lambda ) }  &\lesssim 1 , \\
\label{eq:apriori_e} \norm[\big ]{\fracfac[\big ]{\epsilon^{\iota (\nu_\matr{K} )}} \nabla^\fraceps \hat{p}_\pmf^\epsilon }_{L^\infty (I ; \vct{\Lambda})}  &\lesssim 1  .
\end{align}
\end{subequations}
\vspace{-1.25\baselineskip}
\end{theorem}%
\begin{proof}
See \Cref{sec:5}.
\end{proof}%
The a~priori estimates in \Cref{thm:apriori} imply the existence of a subsequence along which the solutions of the transformed Biot problem~\eqref{eq:fulldim-weak-trafo} converge weakly in the following sense as $\epsilon \rightarrow 0 $. 
\begin{corollary} 
\label{cor:conv}
We define the space
\begin{subequations}
\begin{align}
H^1_{\!\vct{N}} (\Omega_\rmf^1 ) &:= \bigl\{ \phi_\rmf \in L^2 ( \Omega_\rmf^1 ) \ \big\vert\ \nablaperp \phi_\rmf \in \vct{L}^2 ( \Omega_\rmf^1 )  \bigr\} ,  
\intertext{equipped with the norm}
\norm{\phi_\rmf}_{H^1_{\!\vct{N}} (\Omega_\rmf^1 ) }^2 &:= \norm{\phi_\rmf}_{L^2 (\Omega_\rmf^1 ) }^2 + \norm{\nablaperp \phi_\rmf }_{L^2 (\Omega_\rmf^1 ) }^2  ,
\end{align}%
\label{eq:H1N}%
\end{subequations}%
and set $\vct{H}^1_{\!\vct{N}} (\Omega_\rmf^1 ) :=  H^1_{\!\vct{N}} (\Omega_\rmf^1 )^n$.
Besides, we write $\Phi^\#$ for the closure of~$\Phi$ in $H^1 (\Omega_+^0 ) \times H^1 (\Omega_-^0 ) \times H^1_{\! \vct{N}} (\Omega_\rmf^1 ) $ and $\vct{V}^\#$ for the closure of~$\vct{V}$ in $\vct{H}^1 (\Omega_+^0 ) \times \vct{H}^1 (\Omega_-^0 ) \times \vct{H}^1_{\! \vct{N}} (\Omega_\rmf^1 ) $.
Then, given the \Cref{asm:eps,asm:nu}, there exist limit functions
\begin{align*}
\hat{p}_\pm^\# &\in H^1 \bigl( I ; L^2 ( \Omega_\pm^0 ) \bigr) \cap L^\infty \bigl( I ; H^1 ( \Omega_\pm^0 ) \bigr) , \\
\hat{p}_\rmf^\# &\in \begin{cases} 
L^\infty \bigl( I ; H^1 ( \Omega_\rmf^1 ) \bigr) &\text{if } \nu_\matr{K} \le -1 , \\
L^\infty \bigl( I ; H^1_{\!\vct{N}} ( \Omega_\rmf^1 ) \bigr)  &\text{if } \nu_\matr{K} \le 1 , \\
H^1 \bigl( I ; L^2 ( \Omega_\rmf^1 ) \bigr) \cap L^\infty \bigl( I ; L^2 ( \Omega_\rmf^1 ) \bigr) &\text{if } \nu_\omega = -1
\end{cases}
\end{align*}
and a subsequence $\{ \epsilon_k \}_{k \in \mathbb{N}} \subset ( 0, \epsilon^\ast ]$ with $\epsilon_k \searrow 0$ as $k \rightarrow 0$ such that 
\begin{subequations}
\begin{alignat}{3}
\label{eq:conv_p_a} \hat{p}_\pm^{\epsilon_k} &\rightharpoonup \hat{p}_\pm^\# \quad\enspace &&\text{weakly in } H^1 \bigl( I ; L^2 ( \Omega_\pm^0 ) \bigr) ,   \\
\label{eq:conv_p_b}  \hat{p}_\pm^{\epsilon_k} &\rightharpoonup \hat{p}_\pm^\# \quad\enspace &&\text{weakly in } L^2 \bigl( I ; H^1 ( \Omega_\pm^0 ) \bigr) ,   \\
\label{eq:conv_p_c} \hat{p}_\rmf^{\epsilon_k} &\rightharpoonup \hat{p}_\rmf^\# \quad\enspace &&\text{weakly in } L^2 \bigl( I ; H^1 ( \Omega_\rmf^1 ) \bigr) \quad\enspace &&\text{if } \nu_\matr{K} \le -1 , \\ 
\label{eq:conv_p_d} \hat{p}_\rmf^{\epsilon_k} &\rightharpoonup \hat{p}_\rmf^\# \quad\enspace &&\text{weakly in } L^2 \bigl( I ; H^1_{\!\vct{N}} ( \Omega_\rmf^1 ) \bigr) \quad\enspace &&\text{if } \nu_\matr{K} \le 1 ,\\
\label{eq:conv_p_e} \hat{p}_\rmf^{\epsilon_k} &\rightharpoonup \hat{p}_\rmf^\# \quad\enspace &&\text{weakly in } H^1 \bigl( I ; L^2 ( \Omega_\rmf^1 ) \bigr) \quad &&\text{if } \nu_\omega = -1  .
\end{alignat}
\end{subequations}
In particular, we have $\hat{p}_\pmf^\# \in L^\infty ( I ; \Phi )$ if $\nu_\matr{K} \le -1$ and $\hat{p}_\pmf^\# \in L^\infty ( I ; \Phi^\# ) $ if $\nu_\matr{K} \le 1$.
Moreover, with $\theta(\nu ) := \max\{ 0 , \tfrac{1}{2} ( \nu - 1)\}$, there exist limit functions
\begin{align*}
\hat{\vct{u}}_\pm^\# &\in H^1 \bigl( I ; \vct{H}^1 ( \Omega_\pm^0 ) \bigr) \cap L^\infty \bigl( I ; \vct{H}^1 ( \Omega_\pm^0 ) \bigr) , \\
\hat{\vct{u}}_\rmf^\# &\in H^1 \bigl( I ; \vct{H}^1_{\!\vct{N}} ( \Omega_\rmf^1 ) \bigr) \cap L^\infty \bigl( I ; \vct{H}^1_{\!\vct{N}} ( \Omega_\rmf^1 ) \bigr) 
\end{align*}
such that
\begin{subequations}
\begin{alignat}{3}
\label{eq:conv_u_a} \hat{\vct{u}}_\pm^{\epsilon_k} &\rightharpoonup \hat{\vct{u}}_\pm^\# \quad\enspace &&\text{weakly in } H^1 \bigl( I ; \vct{H}^1 ( \Omega_\pm^0 ) \bigr) ,   \\
\label{eq:conv_u_b} \epsilon_k^{\theta ( \nu_\tsr{C})}\hat{\vct{u}}_\rmf^{\epsilon_k} &\rightharpoonup \hat{\vct{u}}_\rmf^\# \quad\enspace &&\text{weakly in } H^1 \bigl( I ; \vct{H}^1_{\!\vct{N}} ( \Omega_\rmf^1 ) \bigr) .
\end{alignat}
\end{subequations}
We have $\hat{\vct{u}}_\pmf^\# \in H^1 ( I ; \vct{V}^\# ) \cap L^\infty ( I ; \vct{V}^\# ) $ if $\nu_\tsr{C} \le 1$.
\end{corollary}
\begin{proof}
The result follows directly from \Cref{thm:apriori}. 
\end{proof}
Using the weak convergences in \Cref{cor:conv}, we can identify the limit models presented in \Cref{sec:4} and eventually obtain the following strong convergence result.
We emphasize that the strong convergences in \Cref{thm:strongconv} are not required for the derivation of the limit models and obtained after identifying the limit problems.
\begin{theorem} \label{thm:strongconv}
Let $\nu_\tsr{C} \ge 1$, $\iota ( \nu ) := \tfrac{1}{2} ( \nu + 1)$, and $2 \matr{e}_\vct{N} (\vct{v} ) := \nablaperp \vct{v} + (\nablaperp \vct{v} )^\rmt $.
Let the \Cref{asm:eps,asm:nu} hold and additionally assume that \ref{asm:dirichlet} in \Cref{asm:nu} holds for $\nu_\matr{K} > -1$.
Besides, if $\nu_\matr{K} > 1$ and $\nu_\omega > - 1$, let $2 \nu_{\!\matr{\alpha}}^\smallpar >  \nu_\tsr{C} + \min\{ \nu_\omega - 2, \nu_\matr{K} - 4 \}$ and $2 \nu_{\!\matr{\alpha}}^\smallperp > \nu_\tsr{C} + \min\{ \nu_\omega , \nu_\matr{K} -2  \}$.
\begin{enumerate}[label=(\roman*),wide]
\item If $\nu_\matr{K} > 1$ and $\nu_\omega > - 1$, let additionally $2\nu_q > \min \{ \nu_\omega - 1 , \nu_\matr{K} +1 \}$.
Then, as $k\rightarrow \infty$, the following strong convergences hold. 
\begin{subequations}
\begin{alignat}{3}
\hat{\vct{u}}_\pm^{\epsilon_k} &\rightarrow \hat{\vct{u}}_\pm^\# \quad\ &&\text{strongly in } L^2 \bigl( I ; \vct{H}^1 ( \Omega_\pm^0 ) \bigr) , \quad\ &&  \\
\epsilon_k^{\iota ( \nu_\tsr{C} ) } \matr{e}^{\epsilon_k} (\hat{\vct{u}}_\rmf^{\epsilon_k} ) &\rightarrow \matr{e}_\vct{N} \bigl( \hat{\vct{u}}_\rmf^\# \bigr)  \quad\ &&\text{strongly in } L^2 \bigl( I ; \matr{L}^2 (\Omega_\rmf^1 ) \bigr) , \quad\ && \\
\hat{p}_\pm^{\epsilon_k} &\rightarrow \hat{p}_\pm^\# \quad\ &&\text{strongly in } L^2  \bigl( I ; H^1 ( \Omega_\pm^0 ) \bigr) , \quad\ && \\
\hat{p}_\rmf^{\epsilon_k} &\rightarrow \hat{p}_\rmf^\# \quad\ &&\text{strongly in } L^2  \bigl( I ; H^1 ( \Omega_\rmf^1 ) \bigr) \quad\ &&\text{if } \nu_\matr{K} \le -1 , \\
\hat{p}_\rmf^{\epsilon_k} &\rightarrow \hat{p}_\rmf^\# \quad\ &&\text{strongly in } L^2  \bigl( I ; H^1_{\!\vct{N}} ( \Omega_\rmf^1 ) \bigr) \quad\ &&\text{if } \nu_\matr{K} \le 1  , \\
\hat{p}_\rmf^{\epsilon_k} &\rightarrow \hat{p}_\rmf^\# \quad\ &&\text{strongly in } L^2  \bigl( I ; L^2 ( \Omega_\rmf^1 ) \bigr) , \quad\ && \\
\epsilon_k^{-1} \nablaperp \hat{p}_\rmf^{\epsilon_k} &\rightarrow \hat{\xi}^\# \vct{N} \quad\ &&\text{strongly in } L^2 (I ; \vct{L}^2 (\Omega_\rmf^1 ) ) \quad\ &&\text{if } \nu_\matr{K} = -1 , \\
\epsilon_k \nablapar \hat{p}_\rmf^{\epsilon_k} &\rightarrow \vct{0} \quad\ &&\text{strongly in } L^2 ( I ; \vct{L}^2 ( \Omega_\rmf^1 ) ) \quad\ &&\text{if } \nu_\matr{K} = 1 ,
\end{alignat}%
\label{eq:strongconv_a}%
\end{subequations}%
where $\hat{\xi}^\# := - ( \hat{\matr{K}}_\rmf \vct{N} \cdot \vct{N} )^{-1} \hat{\matr{K}}_\rmf \nablapar \hat{p}_\rmf^\# \cdot \vct{N}$.
\item Moreover, let additionally $2 \nu_q > \nu_\omega -1$ in \ref{asm:nuq} in \Cref{asm:nu} and
\begin{subequations}
\begin{alignat}{3}
\epsilon^{\iota ( \nu_\matr{K} ) } \nabla^\epsilon \hat{p}_{0,\rmf}^\epsilon &\rightarrow \vct{0} \quad\ &&\text{strongly in } L^2 ( I ; \vct{L}^2 ( \Omega_\rmf^1 ) ) \quad\ &&\text{if } \nu_\matr{K} < -1 , \\
\epsilon^{-1} \nablaperp \hat{p}_{0,\rmf}^\epsilon &\rightarrow \hat{\xi}^\#_0 \vct{N} \quad\ &&\text{strongly in } L^2 ( I ; \vct{L}^2 ( \Omega_\rmf^1 ) ) \quad\ &&\text{if } \nu_\matr{K} = -1 , \\
\epsilon^{\frac{\nu_\matr{K} -1}{2}} \nablaperp \hat{p}_{0,\rmf}^\epsilon &\rightarrow \vct{0} \quad\ && \text{strongly in } L^2 ( I ; \vct{L}^2 ( \Omega_\rmf^1 ) ) \quad\ &&\text{if } \nu_\matr{K} \in (-1, 1) 
\end{alignat}%
\end{subequations}%
 as $\epsilon \rightarrow 0$ in \ref{asm:p0_eps} in \Cref{asm:eps}, where $\hat{\xi}^\#_0 :=  -( \hat{\matr{K}}_\rmf \vct{N} \cdot \vct{N} )^{-1} \hat{\matr{K}}_\rmf \nablapar \hat{p}_{0, \rmf } \cdot \vct{N}$. 
Then, also the following strong convergences hold true as $k\rightarrow \infty$. 
\end{enumerate}
\begin{subequations}
\begin{alignat}{3}
\hat{\vct{u}}_\pm^{\epsilon_k} &\rightarrow \hat{\vct{u}}_\pm^\# \quad\ &&\text{strongly in } H^1 \bigl( I ; \vct{H}^1 ( \Omega_\pm^0 ) \bigr) , \\
\epsilon_k^{ \iota ( \nu_\tsr{C} ) } \matr{e}^\epsilon (\hat{\vct{u}}_\rmf^{\epsilon_k} ) &\rightarrow \matr{e}_\vct{N} \bigl( \hat{\vct{u}}_\rmf^\# \bigr)  \quad\ &&\text{strongly in } H^1 \bigl( I ; \matr{L}^2 (\Omega_\rmf^1 ) \bigr) , \\
\hat{p}_\pm^{\epsilon_k} &\rightarrow \hat{p}_\pm^\# \quad\ &&\text{strongly in } H^1 \bigl( I ; L^2 ( \Omega_\pm^0 ) \bigr) , \\
\hat{p}_\rmf^{\epsilon_k} &\rightarrow \hat{p}_\rmf^\# \quad\ &&\text{strongly in } H^1 \bigl( I ; L^2 ( \Omega_\rmf^1 ) \bigr) \quad \text{ if } \nu_\omega = -1.
\end{alignat}%
\label{eq:strongconv_b}%
\end{subequations}%
\vspace*{-15pt}
\end{theorem}
\begin{proof}
See \Cref{sec:7}.
\end{proof}
\begin{remark}
The weak convergences in \Cref{cor:conv} and the strong convergences in \Cref{thm:strongconv} hold for the entire sequences $\{ \hat{\vct{u}}_\pmf^\epsilon \}_{\epsilon \in ( 0,\epsilon^\ast ]}$ and $\{ \hat{p}_\pmf^\epsilon \}_{\epsilon \in ( 0,\epsilon^\ast ]}$, provided that the uniqueness of solutions to the corresponding limit problems in \Cref{sec:4} is established. 
For $\nu_\matr{K} \not= -1$, uniqueness holds without additional assumptions; see \Cref{thm:Klm1,thm:Km1,thm:Km1p1,thm:Kp1,thm:Kgp1}.
\end{remark}

\section{Limit Models} 
\label{sec:4}
Having established the existence of a weakly convergent subsequence as the width-to-length ratio of the fracture vanishes (see \Cref{cor:conv}), we subsequently present the resulting limit models and corresponding convergence and uniqueness theorems.
In many cases, the resulting limit models can be reduced to a discrete fracture model, i.e., a model on the bulk domains~$\Omega_\pm^0$ and the fracture interface~$\gamma$ only.
However, in certain cases, when flow or deformation in normal direction across the fracture dominates, we obtain two-scale limit problems that still depend on the rescaled full-dimensional fracture~$\Omega_\rmf^1$, either for the pressure head, the displacement vector, or both.

The specific form of the limit models as $\epsilon \rightarrow 0$ depends on the choice of the scaling coefficients in \Cref{asm:nu}.
Two key parameters determine the respective solution spaces: the scaling coefficient~$\nu_\tsr{C}$ of the elasticity tensor for the mechanics equation~\eqref{eq:fulldim-weak-trafo-b} and the scaling coefficient~$\nu_\matr{K}$ of the hydraulic conductivity matrix for the flow equation~\eqref{eq:fulldim-weak-trafo-c}.
Other scaling exponents mainly determine whether specific terms vanish or persist in the limit model. 
These cases can be treated uniformly by introducing effective model coefficients that are set to zero whenever the corresponding terms disappear. 
The various scaling parameters can be combined independently so that we obtain a whole family of possible limit models.
In the following, we briefly discuss the role of the different scaling coefficients. 
A quick reference to the different limit regimes as~$\epsilon \rightarrow 0$ can be found in \Cref{tab:regimes}.
\pagebreak[2]
\begin{table}
\setlength\dashlinedash{0.5pt}
\setlength\dashlinegap{1.25pt}
\centering
\caption{Quick reference to the limit model regimes as $\epsilon \rightarrow 0$.}
\label{tab:regimes}
 \begin{subtable}[t]{\textwidth}
 \centering    
 \caption{Role of the key scaling parameters~$\nu_\tsr{C}$ and~$\nu_\matr{K}$: interpretation of the fracture and resulting model type for pressure ($\nu_\matr{K}$) and displacement ($\nu_\tsr{C}$), distinguishing between discrete-fracture (DF) and two-scale (2S) models, with references to the corresponding strong (SF) and weak (WF) formulations.}
 \renewcommand*{\arraystretch}{1.25}
 \small
 \begin{tabular}{r|l|l|l|l}
   \multirow{2}{*}{$\nu_\tsr{C} = 1$} & \multirow{2}{*}{soft} & 2S & SF \cref{eq:mechlimit_nuC1} & WF \cref{eq:weak_mechlimit_C1}  \\
   \cdashline{3-5}
   & & DF & SF \cref{eq:mechlimit_nuC1_df} & WF \cref{eq:weak_mechlimit_discrete_C1}  \\
   \hdashline
   $\nu_\tsr{C} > 1$ & very soft & 2S & SF \cref{eq:mechlimit_nuCg1} & WF \cref{eq:weak_mechlimit_nuCg1} \\
   \hline 
   $\nu_\matr{K} < -1$ & ideal conduit & DF & SF \cref{eq:Klm1} & WF \cref{eq:weak_C1_Klm1} \\
   \hdashline
   $\nu_\matr{K} = -1$ & conduit & DF & SF \cref{eq:C1_Km1} & WF \cref{eq:weak_C1_Km1} \\
   \hdashline
   $\nu_\matr{K} \in (-1, 1)$ & neutral & DF & SF \cref{eq:Km11} & WF \cref{eq:weak_C1_Km11} \\
   \hdashline
   \multirow{2}{*}{$\nu_\matr{K} = 1$} & \multirow{2}{*}{barrier} & 2S & SF \cref{eq:K1} & WF \cref{eq:weak_C1_K1}  \\
   \cdashline{3-5} 
   & & DF & SF \cref{eq:K1_df} & WF \cref{eq:weak_C1_K1_df} \\
   \hdashline
   $\nu_\matr{K} > 1$ & wall & DF & SF \cref{eq:strong_C1Kg1} & WF \cref{eq:weak_C1_Kg1}
  \end{tabular}
 \end{subtable}%
\hfill
\vspace{12pt}
  \begin{subtable}[t]{\textwidth}
  \centering
  \caption{Role of the secondary scaling parameters~$\nu_{{\matr{\alpha}}}^\perp$, $\nu_\omega$, $\nu_q$, and $\nu_\vct{f}$: affected terms that persist or vanish in the fracture limit models.}
  \small
  \renewcommand*{\arraystretch}{1.25}
  \begin{tabular}{r|l}
   2$\nu_{\matr{\alpha}}^\smallperp = \nu_\tsr{C} - 1$ & Biot coupling  \\
   \hdashline 
   2$\nu_{\matr{\alpha}}^\smallperp > \nu_\tsr{C} - 1$ & no Biot coupling  \\
   \hline
   $\nu_\omega = -1$ & storage term   \\
   \hdashline 
   $\nu_\omega > -1$ & no storage term \\
  \end{tabular}
  \hspace{0.75cm}
  \begin{tabular}{r|l}
   $\nu_q = -1$ & fluid source \\
   \hdashline 
    $\nu_q > -1$ & no fluid source  \\
   \hline
    $2\nu_\vct{f} = \nu_\tsr{C} -3$ & body force  \\
   \hdashline 
    $2\nu_\vct{f} > \nu_\tsr{C} -3$ & no body force
  \end{tabular}
 \end{subtable}
\end{table}
\begin{itemize}
\item The coefficient $\nu_\tsr{C}$ regulates the elasticity of the fracture in the limit~$\epsilon\rightarrow 0$.
\begin{itemize}
\item For $\nu_\tsr{C} = 1$, we obtain a soft fracture with a jump of the normal bulk displacement across the interface~$\gamma$. 
This yields both a full-dimensional limit model and a discrete fracture model.
The former includes an ODE for the normal derivative of the displacement inside the fracture, while the latter solves only for the bulk displacement such that the normal displacement jump across~$\gamma$ scales with an effective normal elasticity.
Details can be found in \Cref{sec:4.2.1}. 
\item For $\nu_\tsr{C} > 1$, we obtain a very soft fracture, as detailed in \Cref{sec:4.2.2}. The original mechanics equation~\eqref{eq:fulldim-weak-trafo-b} vanishes. Instead, we obtain an ODE inside the full-dimensional fracture domain for the normal derivative of~$\hat{\vct{u}}^\#_\rmf$, i.e., the weak limit of the scaled fracture displacement vector~$\epsilon_k^{(\nu_\tsr{C} - 1)/2} \hat{\vct{u}}_\rmf^{\epsilon_k}$. 
This equation is not directly coupled to the bulk displacement;  interaction occurs only indirectly through the flow equation.
\item We do not consider limit models with $\nu_\tsr{C} < 1$ in this paper as we do not obtain closed limit models with a Biot-type coupling between flow and mechanics inside the fracture for this case. 
A key problem is that the domain-decomposed Korn inequality in \Cref{lem:korn_decomp} is not applicable here. We conjecture that this stiff case leads to an elastic plate-type model in the limit.
\end{itemize}
\item The coefficient $\nu_\matr{K}$ controls the hydraulic conductivity of the fracture in the limit model. 
We distinguish between five cases in analogy to~\cite{hoerl24}, where Darcy flow without deformation is considered.
\begin{itemize}
\item For $\nu_\matr{K} < -1$, as detailed in in \cref{sec:4.3.1}, the fracture becomes an ideal conduit and we obtain a discrete fracture model where the pressure head on the fracture interface~$\gamma$ is completely constant in space.
\item \Cref{sec:4.3.2} is concerned with the case $\nu_\matr{K} =-1$. 
Here, we obtain a discrete fracture model with a conductive fracture such that the fracture pressure head satisfies a PDE on the interface~$\gamma$ and the bulk pressure head is continuous across~$\gamma$. 
\item \Cref{sec:4.3.3} treats the case $\nu_\matr{K} \in (-1, 1)$ of a neutral fracture, where the pressure head is continuous across the interface~$\gamma$ and the hydraulic conductivity within the fracture does not affect the solution of the limit model.
\item For $\nu_\matr{K} = 1$, the fracture becomes a permeable barrier as detailed in \Cref{sec:4.3.4}. 
The limit model includes an ODE in normal direction inside the full-dimensional fracture.
A reduction to a discrete fracture model is only possible for special cases.
\item For $\nu_\matr{K} > 1$, the fracture turns into an impermeable barrier with no flow from either side across the interface~$\gamma$. For details, see \Cref{sec:4.3.5}.
\end{itemize}
\item The coefficient $\nu_\omega$ determines the presence of the storage term inside the fracture in the limit model. The storage term is present for $\nu_\omega = -1$ and vanishes for $\nu_\omega > -1$.
\item The coefficient $\nu_{\!\matr{\alpha}}^\smallpar$ for the scaling of the Biot tensor~$\hat{\matr{\alpha}}_\rmf$ in tangential direction does not have an impact on the limit solution in the case $\nu_\tsr{C} \ge 1$ that we consider here.
\item The coefficient $\nu_{\!\matr{\alpha}}^\smallperp$ for the scaling of the Biot tensor~$\hat{\matr{\alpha}}_\rmf$ in normal direction determines whether flow and mechanics equation will be directly coupled inside the fracture in the limit model or not. The coupled case occurs for $2\nu_{\!\matr{\alpha}}^\smallperp = \nu_\tsr{C} -1$, the uncoupled case for  $2\nu_{\!\matr{\alpha}}^\smallperp > \nu_\tsr{C} -1$.
\item The coefficient $\nu_{\!\vct{f}}$ controls the presence of the body force term~$\hat{\vct{f}}_\rmf$ in the limit model. 
The body force term is present for $2 \nu_{\!\vct{f}} = \nu_\tsr{C} -3$ and vanishes for $2 \nu_{\!\vct{f}} > \nu_\tsr{C} -3$.
\item The coefficient $\nu_q$ determines the presence of the fluid source term~$\hat{q}_\rmf$ in the limit model. The source term is present for~$\nu_q = -1$ and vanishes for~$\nu_q > -1$.
\end{itemize}

\begin{remark}
In the literature, rigorous derivations of discrete fracture models for po\-rous media predominantly consider fluid flow without mechanical deformation.
In particular, for single-phase Darcy flow, the regimes $\nu_\matr{K} = -1$~\cite{huy74,morales10}, $\nu_\matr{K} = 1$~\cite{sanchez74}, and the general case~$\nu_\matr{K} \in \mathbb{R}$~\cite{hoerl24} have been analyzed, matching the cases considered here.
For unsaturated flow, the formally derived models in~\cite{kumar20} correspond to $\nu_\omega \ge -1$ and $\nu_\matr{K} \in \mathbb{R}$; rigorous results are available in \cite{list20} for the parameter regimes~$\nu_\omega \ge -1$ with $\nu_\matr{K} \le -1$ and $\nu_\omega > -1$ with $\nu_\matr{K} \in (-1,1)$. 
In addition, a Biot-type model for a thin poroelastic plate without surrounding bulk domains has been analyzed in~\cite{marciniak15} and extended to curved poroelastic shells in~\cite{mikelic19,mikelic16}. 
They consider the case~$\nu_\tsr{C} = -1$, $\nu_\matr{K} = 1$, $\nu_\omega = -1$, and $\nu_{\matr{\alpha}}^\smallpar = \nu_{\matr{\alpha}}^\smallperp = -1$, while accounting for a  different scaling of normal and tangential displacement. 
This regime is not addressed in the present work.
\end{remark}

In order to avoid repetition, we present the different components of the limit models separately. 
First, in \Cref{sec:4.1}, we consider the bulk limit problem. 
\Cref{sec:4.2} addresses the limit problems for the mechanics equation~\eqref{eq:fulldim-weak-trafo-b} within the fracture.
Finally, in \Cref{sec:4.3}, we discuss the different limit problems for the flow equation~\eqref{eq:fulldim-weak-trafo-c} inside the fracture and provide the complete limit models in both strong and weak form, as well as the corresponding convergence theorems.
We remark that the analysis is carried out in a weak setting and the strong formulations are included solely for convenience.

Moreover, we introduce the following averaging operators for the fracture domain~$\Omega_\rmf^1$.%
\begin{definition} \label{def:average}
We define the averaging operators
\begin{subequations}
\begin{alignat}{2}
\mathfrak{A}_{\!\vct{N}} &\colon L^2 (\Omega_\rmf^1 ) \rightarrow L^2 ( \gamma ) , \quad &&\bigl( \mathfrak{A}_{\!\vct{N}} f \bigr) ( \vct{y}  ) := \frac{1}{a (\vct{y})} \int_{-a_- ( \vct{y} ) }^{a_+ ( \vct{y} ) } f ( \vct{y} + s \vct{N} ) \,\rmd s , \\
\mathfrak{A}_\rmf &\colon L^2 (\Omega_\rmf^1 ) \rightarrow \mathbb{R} , \quad &&\mathfrak{A}_\rmf f  := \frac{1}{ \int_\gamma a \,\rmd A } \int_{\Omega_\rmf^1  } f \,\rmd \vct{x} .
\end{alignat}
\end{subequations}
\end{definition}

\subsection{Bulk Limit Problem} 
\label{sec:4.1}
The bulk part of the limit problem has the following form for all cases.

Find $p_\pm \colon \Omega_\pm^0 \times I \rightarrow \mathbb{R}$ and $\vct{u}_\pm \colon \Omega_\pm^0 \times I \rightarrow \mathbb{R}^n$ such that 
\begin{subequations}
\begin{alignat}{2}
-\nabla \cdot\hat{\matr{\sigma}}_\pm \bigl( \vct{u}_\pm , p_\pm \bigr) &= \hat{\vct{f}}_\pm \quad &&\text{in } \Omega_\pm^0 \times I , \\ 
\partial_t \bigl( \hat{\omega}_\pm p_\pm + \nabla \cdot \bigl( \hat{\matr{\alpha}}_\pm \vct{u}_\pm \bigr) \bigr) - \nabla \cdot \bigl( \hat{\matr{K}}_\pm \nabla p_\pm \bigr) &= \hat{q}_\pm \quad &&\text{in } \Omega_\pm^0 \times I , \\
\vct{u}_\pm &= \vct{0} \quad &&\text{on } \rho^0_{\pm } \times I , \\
p_\pm &= 0 \quad &&\text{on } \rho^0_{\pm, \mathrm{D} } \times I , \\
\hat{\matr{K}}_\pm \nabla p_\pm \cdot \vct{n}_\pm &= \vct{0} \quad &&\text{on } \rho^0_{\pm, \mathrm{N} } \times I , \\
p_\pm ( \cdot , 0 ) &= \hat{p}_{0, \pm} \quad &&\text{on } \Omega_\pm^0 . \\
\intertext{Besides, the initial displacement vector~$\vct{u}_{0,\pm} := \vct{u}_\pm(\cdot , 0)$ satisfies}
-\nabla \cdot \hat{\matr{\sigma}}_\pm ( \vct{u}_{0,\pm} , \hat{p}_{0,\pm} ) &= \hat{\vct{f}}_\pm (\cdot , 0) \quad &&\text{in } \Omega_\pm^0 , \\ 
\vct{u}_{0, \pm} &= \vct{0} \quad &&\text{on } \rho^0_{\pm, \mathrm{D}}  .
\end{alignat}%
The total bulk stress tensor is given by
\begin{align}
\hat{\matr{\sigma}}_\pm \bigl( \vct{u}_\pm , p_\pm \bigr) := \hat{\tsr{C}}_\pm \matr{e} ( \vct{u}_\pm ) - \bigl( p_\pm + \hat{G}_\pm \bigr) \hat{\matr{\alpha}}_\pm  .
\end{align}
\label{eq:bulklimit}%
\end{subequations}%
\indent The bulk pressure head~$p_\pm$ and the displacement vector~$\vct{u}_\pm$ in \cref{eq:bulklimit} correspond to the limit functions~$\smash{\hat{p}_\pm^\#}$ and $\smash{\hat{\vct{u}}_\pm^\#}$ from \Cref{cor:conv}.
The bulk limit problem~\eqref{eq:bulklimit} is not uniquely solvable and has to be supplemented with conditions on the interface~$\gamma$ or a coupled fracture problem in~$\Omega_\rmf^1$ (see \Cref{sec:4.2,sec:4.3}).
A weak formulation of the bulk limit problem~\eqref{eq:bulklimit} is presented together with the fracture limit problems in the \Cref{sec:4.2,sec:4.3}.

\subsection{Fracture Limit Problems for the Mechanics Equation \texorpdfstring{\eqref{eq:fulldim-weak-trafo-b}}{}}
\label{sec:4.2}
We consider the limit of the mechanics equation~\eqref{eq:fulldim-weak-trafo-b} as $\epsilon\rightarrow 0$ for $\nu_\tsr{C} = 1$ in \Cref{sec:4.2.1} and for $\nu_\tsr{C} > 1$ in \Cref{sec:4.2.2}.
For the complete limit models (including the limit of the flow equation~\eqref{eq:fulldim-weak-trafo-b}) and the corresponding convergence theorems, we refer to \Cref{sec:4.3}.

Further, we introduce the following effective limit model coefficients and right-hand side.
We define the normal elasticity tensor~$\hat{\matr{C}}_\rmf^{\!\vct{N}} \in \matr{L}^\infty ( \Omega_\rmf^1 ) $, the piecewise constant effective normal Biot vector~$\hat{\vct{\alpha}}_\rmf^\mathrm{eff} \colon  \Omega_\rmf^1  \rightarrow \mathbb{R}^n $, and the effective body force term~$\smash{\hat{\vct{f}}^\mathrm{eff}_\rmf} \in H^1(I ; \vct{L}^2 (\Omega_\rmf^1 ) )$ inside the fracture by
\begin{subequations}
\begin{align}
\label{eq:Ceff} ( \hat{\matr{C}}_\rmf^{\!\vct{N}} )_{ik} &:= \sum_{j=1}^n \sum_{l=1}^n ( \hat{\mathbb{C}}_\rmf )_{ijkl} N_j N_l , \\
\label{eq:alphaeff} \hat{\vct{\alpha}}_\rmf^\mathrm{eff} &:= \begin{cases} 
\hat{\matr{\alpha}}_\rmf \vct{N} &\text{if } \nu_{\!\matr{\alpha}}^\smallperp =  \frac{\nu_\mathbb{C} -1}{2} , \\
\vct{0} &\text{if } \nu_{\!\matr{\alpha}}^\smallperp > \frac{\nu_\mathbb{C} -1}{2} ,  \end{cases} \\
\label{eq:feff} \hat{\vct{f}}_\rmf^\mathrm{eff} &:= \begin{cases} \hat{\vct{f}}_\rmf &\text{if } \nu_{\!\vct{f}} = \frac{\nu_\mathbb{C} -3}{2} , \\ \vct{0} &\text{if } \nu_{\!\vct{f}} > \frac{\nu_\mathbb{C} -3}{2} ,  \end{cases} 
\end{align}%
\label{eq:eff1}%
\end{subequations}%
where $N_j$, $j\in \{1, \dots , n\}$, denote the components of the normal~$\vct{N}$.
As only the normal direction is taken into account, we emphasize that $\smash{\hat{\matr{C}}_\rmf^{\!\vct{N}}}$ is a matrix-valued function and not a tensor-valued function of fourth order, while $\smash{\hat{\alpha}_\rmf^\eff}$ is a vector-valued function and not a matrix-valued function.
The normal elasticity tensor~$\smash{\hat{\matr{C}}_\rmf^{\!\vct{N}}}$ now has the following properties.
\begin{lemma} \label{lem:eff1}
$\hat{\matr{C}}_\rmf^{\!\vct{N}} \in \matr{L}^\infty ( \Omega_\rmf^1 ) $ as defined in \cref{eq:Ceff} is almost everywhere symmetric and uniformly elliptic, i.e., $\hat{\matr{C}}_\rmf^{\!\vct{N}} \vct{\xi} \cdot \vct{\xi} \gtrsim \abs{\vct{\xi}}^2$ for all $\vct{\xi} \in \mathbb{R}^n$. 
\end{lemma}
\begin{proof}
For $\vct{\xi} = (\xi_i) \in \mathbb{R}^n$, we have
\begin{align*}
\hat{\matr{C}}_\rmf^{\!\vct{N}} \vct{\xi} \cdot \vct{\xi} &= ( \hat{\mathbb{C}}_\rmf )_{ijkl} \xi_i N_j \xi_k N_l = \tfrac{1}{4} ( \hat{\mathbb{C}}_\rmf )_{ijkl} [ \xi_iN_j + \xi_j N_i ] [ \xi_kN_l + \xi_lN_k ] \\
&\gtrsim   \tfrac{1}{4} [ \xi_iN_j + \xi_j N_i ] [ \xi_iN_j + \xi_j N_i ] = \tfrac{1}{2} [ \norm{\vct{\xi}}^2 \norm{\vct{N}}^2 + ( \vct{\xi} \cdot \vct{N} )^2 ] \ge \tfrac{1}{2} \norm{\vct{\xi}}^2 ,
\end{align*}
where we have used \ref{asm:C_eps} in \Cref{asm:eps}.
Symmetry is trivial.
\end{proof}

\subsubsection{Limit Models with \texorpdfstring{$\nu_\tsr{C} = 1$}{nuC=1}} \label{sec:4.2.1}
For $\nu_\tsr{C} = 1$, we obtain a limit problem characterized by a jump of the displacement vector but (except for body force terms) continuous normal stress  across the fracture interface~$\gamma$. 
As described in \cref{sec:4.2.1.1}, this first leads to a limit problem within the full-dimensional fracture domain~$\Omega_\rmf^1$, where the displacement vector~$\vct{u}_\rmf$ satisfies an ODE in the normal direction. 
In \cref{sec:4.2.1.2}, we then succeed to reformulate the full-dimensional limit problem from \cref{sec:4.2.1.1} as a discrete fracture model. In this reduced setting, the fracture displacement is eliminated, and the displacement jump across the interface~$\gamma$ scales with an effective normal elasticity.
This reduction is only advantageous if the corresponding limit problem for the flow equation~\eqref{eq:fulldim-weak-trafo-c} (see \cref{sec:4.3}) is independent of the displacement vector~$\vct{u}_\rmf$ within the full-dimensional fracture domain~$\Omega_\rmf^1$ (cf.~\Cref{rem:C1discrete}).

\paragraph{Full-Dimensional Limit Model} 
\label{sec:4.2.1.1}
The strong formulation of the limit problem for $\nu_\tsr{C} = 1$ with full-dimensional fracture reads as follows. 

Find $\vct{u}_\pmf \colon \Omega_\pmf^\odot \times I \rightarrow \mathbb{R}^n$ and $p_\pm \colon \Omega_\pm^0 \times I \rightarrow \mathbb{R}$, as well as $p_\rmf \colon \Omega_\rmf^1 \times I \rightarrow \mathbb{R}$ if $\nu_\matr{K} \le 1$ or $\nu_\omega = -1$, such that
\begin{subequations}
\begin{alignat}{2}
-\partial_{\!\vct{N} } \hat{\vct{\sigma}}_\rmf \bigl( \vct{u}_\rmf , p_\rmf \bigr) &= \hat{\vct{f}}_\rmf^\mathrm{eff} \quad\enspace &&\text{in } \Omega_\rmf^1 \times I , \\
\vct{u}_\pm &= \Pi_\pm \vct{u}_\rmf \quad\enspace &&\text{on } \gamma \times I , \\
\hat{\matr{\sigma}}_\pm \bigl( \vct{u}_\pm , p_\pm \bigr) \vct{n}_\pm &= \Pi_\pm \hat{\vct{\sigma}}_\rmf \bigl( \vct{u}_\rmf , p_\rmf \bigr)  \quad &&\text{on } \gamma \times I ,  \\
\intertext{and the bulk limit problem~\eqref{eq:bulklimit} is satisfied. Besides, the initial displacement vector~$\vct{u}_{0,\pmf} := \vct{u}_\pmf (\cdot , 0)$ satisfies}
-\partial_{\!\vct{N} } \hat{\vct{\sigma}}_\rmf \bigl(\vct{u}_{0, \rmf} , \hat{p}_{0,\rmf} \bigr) &= \hat{\vct{f}}_\rmf^\mathrm{eff} (\cdot, 0) \quad\enspace &&\text{in } \Omega_\rmf^1 \times I , \\
\vct{u}_{0,\pm} &= \Pi_\pm \vct{u}_{0,\rmf} \quad &&\text{on } \gamma \times I , \\
\hat{\matr{\sigma}}_\rmf \bigl( \vct{u}_{0,\pm} , \hat{p}_{0,\pm} \bigr) \vct{n}_\pm &= \Pi_\pm \hat{\vct{\sigma}}_\rmf \bigl(\vct{u}_{0, \rmf} , \hat{p}_{0, \rmf}\bigr) \quad &&\text{on } \gamma \times I . 
\end{alignat}%
Here, the total fracture stress~$\hat{\vct{\sigma}}_\rmf \colon \vct{H}^1_{\!\vct{N}} (\Omega_\rmf^1 ) \times L^2 (\Omega_\rmf^1 ) \rightarrow \vct{L}^2 (\Omega_\rmf^1 )$ is given by
\begin{align}
\label{eq:mechlimit_nuC1_f}\hat{\vct{\sigma}}_\rmf ( \vct{u}_\rmf , p_\rmf ) &:= \begin{cases} \hat{\matr{C}}_\rmf^{\!\vct{N}} \partial_{\!\vct{N}} \vct{u}_\rmf - ( p_\rmf + \hat{G}_\rmf ) \hat{\vct{\alpha}}_\rmf^\mathrm{eff} &\text{if } \nu_\matr{K} \le 1 \text{ or } \nu_\omega = -1 , \\
\hat{\matr{C}}_\rmf^{\!\vct{N}} \partial_{\!\vct{N}} \vct{u}_\rmf &\text{if } \nu_\matr{K} > 1 \text{ and } \nu_\omega > -1 .
\end{cases}
\end{align}%
\label{eq:mechlimit_nuC1}%
\end{subequations}%
Besides, $\Pi_\pm \colon L^2 (\gamma_\pm^1 ) \rightarrow L^2 (\gamma)$ denotes the projection from \cref{eq:proj}.

The pressure heads~$p_\pmf$ and displacement vector~$\vct{u}_\pmf$ in \cref{eq:mechlimit_nuC1} correspond to the limit functions~$\hat{p}_\pmf^\#$ and $\hat{\vct{u}}_\pmf^\#$ from \Cref{cor:conv}. 
The limit problem~\eqref{eq:mechlimit_nuC1} is to be completed by conditions or equations for the pressure head on the interface~$\gamma$ or the fracture domain~$\Omega_\rmf^1$ (see \Cref{sec:4.3}).
A weak formulation of \cref{eq:mechlimit_nuC1} is given by the following problem.
We recall the definition of the space~$\vct{V}^\#$ in \Cref{cor:conv}.

Find $\vct{u}_\pmf \in H^1 ( I ; \vct{V}^\# ) $ and $p_\pm \in L^2 (I ; L^2(\Omega_\pm^0 ))$, as well as $p_\rmf \in L^2 (I; L^2 (\Omega_\rmf^1 ))$ if $\nu_\matr{K} \le 1$ or $\nu_\omega = -1$, such that 
\begin{subequations}
\begin{align}
\label{eq:weak_mechlimit_C1_a}
\begin{split}
&\bigl\langle \hat{\mathbb{C}}_\pm \matr{e} ( \vct{u}_\pm ) , \matr{e} ( \vct{v}_\pm ) \bigr\rangle_{\Omega_\pm^0  } + \bigl\langle \hat{\vct{\sigma}}_\rmf \bigl( \vct{u}_\rmf , p_\rmf \bigr) , \partial_\vct{N}  \vct{v}_\rmf \bigr\rangle_{\Omega_\rmf^1 }  - \bigl\langle p_\pm \hat{\matr{\alpha}}_\pm , \nabla \vct{v}_\pm \bigr\rangle_{\Omega_\pm^0  }  \\
&\hspace{6cm} =  \bigl\langle \hat{\vct{f}}_\pmf^{\fracfac{\! \eff\! }} , \vct{v}_\pmf \bigr\rangle+ \bigl\langle \hat{G}_\pm \hat{\matr{\alpha}}_\pm , \nabla \vct{v}_\pm \bigr\rangle_{\Omega_\pm^0 }  
\end{split}
\intertext{ holds for all $\vct{v}_\pmf \in \vct{V}^\#$ and almost all $t\in I$ with the total stress~$\hat{\matr{\sigma}}_\rmf$ given by \cref{eq:mechlimit_nuC1_f}.
In addition, for all $\vct{v}_\pmf \in \vct{V}^\#$, the initial displacement vector~$\vct{u}_{0,\pmf} := \vct{u}_\pmf (  0)$ satisfies }
\label{eq:weak_mechlimit_C1_b}
\begin{split}
&\bigl\langle \hat{\mathbb{C}}_\pm \matr{e} ( \vct{u}_{0,\pm} ) , \matr{e} ( \vct{v}_\pm ) \bigr\rangle_{\Omega_\pm^0  } + \bigl\langle \hat{\vct{\sigma}}_\rmf \bigl( \vct{u}_{0,\rmf} , \hat{p}_{0,\rmf} \bigr) , \partial_\vct{N}  \vct{v}_\rmf \bigr\rangle_{\Omega_\rmf^1 } -  \bigl\langle\hat{ p}_{0,\pm} \hat{\matr{\alpha}}_\pm , \nabla \vct{v}_\pm \bigr\rangle_{\Omega_\pm^0  }  \\
&\hspace{5.5cm} =  \bigl\langle \hat{\vct{f}}_\pmf^{\fracfac{\! \eff\! }}  (0)  , \vct{v}_\pmf \bigr\rangle + \bigl\langle \hat{G}_\pm \hat{\matr{\alpha}}_\pm , \nabla \vct{v}_\pm \bigr\rangle_{\Omega_\pm^0 }  .
\end{split}%
\end{align}%
\label{eq:weak_mechlimit_C1}%
\end{subequations}%
\indent The weak formulation~\eqref{eq:weak_mechlimit_C1} is to be supplemented with a flow equation as presented in \Cref{sec:4.3}. 
In particular, the limit functions~$\smash{\hat{\vct{u}}_\pmf^\#}$ and $\smash{\hat{p}_\pmf^\#}$ from \Cref{cor:conv} solve \cref{eq:weak_mechlimit_C1}; see \Cref{thm:Klm1,thm:Km1,thm:Km1p1,thm:Kp1,thm:Kgp1}
 in \Cref{sec:4.3}.
\paragraph{Discrete Fracture Limit Model}
\label{sec:4.2.1.2}
Starting from the two-scale limit problem in \cref{eq:weak_mechlimit_C1}, we subsequently present a corresponding discrete fracture limit problem for the bulk displacement only.  
We recall \Cref{def:average} and define the effective interfacial normal elasticity tensor~$\hat{\matr{C}}_\gamma^{\!\vct{N}} \in \matr{L}^\infty (\gamma ) $ and the effective body force term~$\smash{\hat{\vct{F}}_\gamma^\eff} \in \vct{L}^2 (\gamma )$ by
\begin{subequations}
\begin{align}
\label{eq:Cgamma}\hat{\matr{C}}_\gamma^{\!\vct{N}} (\vct{y} ) &:= \Bigl( \mathfrak{A}_{\!\vct{N}} \bigl( (\hat{\matr{C}}_\rmf^{\!\vct{N}})^{\! -1} \bigr) (\vct{y}) \Bigr)^{-1} = \biggl( \frac{1}{a(\vct{y})} \int_{a_-(\vct{y})}^{a_+ (\vct{y})} \bigl( \hat{\matr{C}}_\rmf^{\!\vct{N}} (\vct{y} + s\vct{N} )\bigr)^{-1} \,\rmd s \biggr)^{\!-1}  , \\
\label{eq:Fgamma}\hat{\vct{F}}_\gamma^\mathrm{eff} (\vct{y} ) &:= \frac{\hat{\matr{C}_\gamma^{\!\vct{N}}} (\vct{y})}{a ( \vct{y})}\int_{-a_-(\vct{y})}^{a_+(\vct{y} ) } \biggl[ \int_{-a_-(\vct{y} )}^s \Bigl( \hat{\matr{C}}_\rmf^{\!\vct{N}} (\vct{y} + \bar{s} \vct{N} )  \Bigr)^{\!-1} \rmd \bar{s} \biggr] \, \hat{\vct{f}}_\rmf^\mathrm{eff} (\vct{y} + s \vct{N} ) \,\rmd s .
\end{align}
\end{subequations}
Further, for functions $h_\pm$ on~$\Omega_\pm^0$ with well-defined trace on~$\gamma$, let
\begin{align}
\label{eq:jump}
\jump{h }_\gamma &:= h_+ \bigr\vert_\gamma - h_-\bigr\vert_\gamma 
\end{align}
denote the jump operator across the interface~$\gamma$.
The reduced limit problem now has the following strong formulation with the pressure heads~$p_\pmf$ and displacement vector~$\vct{u}_\pm$ corresponding to the limit functions~$\hat{p}_\pmf^\#$ and $\hat{\vct{u}}_\pm^\#$ from \Cref{cor:conv}.

Find $\vct{u}_\pm \colon \Omega_\pm^0 \times I \rightarrow \mathbb{R}^n$ and $p_\pm \colon \Omega_\pm^0 \times I \rightarrow \mathbb{R}$, as well as $p_\rmf \colon \Omega_\rmf^1 \times I \rightarrow \mathbb{R}$ if $\nu_\matr{K} \le 1$ or $\nu_\omega = -1$, such that 
\begin{subequations}
\begin{alignat}{2}
\jump{\hat{\matr{\sigma}} ( \vct{u}) \vct{N}}_\gamma + a \mathfrak{A}_{\!\vct{N}} \hat{\vct{f}}_\rmf^\mathrm{eff} &= 0 \quad\enspace &&\text{on } \gamma \times I , \\
\hat{\matr{\sigma}}_+ ( \vct{u}_+)  \vct{N} +  \hat{\vct{F}}_{\gamma}^\mathrm{eff}&= \hat{\matr{\sigma}}_\gamma ( \jump{\vct{u}}_\gamma , p_\rmf )  \quad\enspace &&\text{on } \gamma \times I ,
\intertext{and the bulk limit problem~\eqref{eq:bulklimit} is satisfied. Besides, the initial displacement vector~$\vct{u}_{0,\pm} := \vct{u}_\pm (\cdot , 0)$ satisfies}
\jump{\hat{\matr{\sigma}} ( \vct{u}_0) \vct{N}}_\gamma + a \mathfrak{A}_{\!\vct{N}} \hat{\vct{f}}_\rmf^\mathrm{eff} (\cdot ,0 ) &= 0 \quad\enspace &&\text{on } \gamma \times I , \\
\hat{\matr{\sigma}}_+ ( \vct{u}_{0,+})  \vct{N} +  \hat{\vct{F}}_{\gamma}^\mathrm{eff} (\cdot , 0) &= \hat{\matr{\sigma}}_\gamma ( \jump{\vct{u}_0}_\gamma , \hat{p}_{0,\rmf} )  \quad\enspace &&\text{on } \gamma \times I  .
\end{alignat}%
Here, the effective total interface stress~$\hat{\matr{\sigma}}_\gamma$ is given by
\begin{equation}
\begin{multlined}[c][0.875\displaywidth]
\hat{\matr{\sigma}}_\gamma \bigl( \jump{\vct{u}}_\gamma , p_\rmf \bigr) := \\[2pt]
\begin{cases}
\tfrac{1}{a}  \hat{\matr{C}}_\gamma^{\!\vct{N}} \bigl( \jump{\vct{u}}_\gamma - a \mathfrak{A}_{\!\vct{N}} \bigl( \bigl[ p_\rmf - \hat{G}_\rmf \bigr]  \bigl[ \hat{\matr{C}}_\rmf^{\!\vct{N}} \bigr]^{-1} \hat{\vct{\alpha}}_\rmf^\mathrm{eff}  \bigr) \bigr) &\text{if } \nu_\matr{K} \le 1 \text{ or } \nu_\omega = -1 , \\
\tfrac{1}{a} \hat{\matr{C}}_\gamma^{\!\vct{N}} \jump{\vct{u}}_\gamma &\text{if } \nu_\matr{K} > 1 \text{ and } \nu_\omega > -1 . \end{cases}%
\end{multlined}%
\label{eq:mechlimit_nuC1_df_c}
\end{equation}%
\label{eq:mechlimit_nuC1_df}%
\end{subequations}%

Next, with the space~$\vct{V}_1$ defined by
\begin{align}
\vct{V}_1 &:=  \vct{H}^1_{0, \rho_\pm^0 }\!  \bigl( \Omega_\pm^0 \bigr) , \label{eq:V1}
\end{align}
a weak formulation of \cref{eq:mechlimit_nuC1_df} is given by the following problem.

Find $\vct{u}_\pm \in H^1 ( I ; \vct{V}_1 ) $ and $p_\pm \in L^2 ( I ; L^2 ( \Omega_\rmf^1 ) )$, as well as $p_\rmf \in L^2 ( I ;  L^2 ( \Omega_\rmf^1 ) )$ if $\nu_{\matr{K}} \le 1$ or $\nu_\omega = -1$, such that
\begin{subequations}
\begin{equation}
\label{eq:weak_mechlimit_discrete_C1_a}
\begin{multlined}[c][0.875\displaywidth]
\bigl\langle \hat{\mathbb{C}}_\pm \matr{e} ( \vct{u}_\pm ) , \matr{e} ( \vct{v}_\pm ) \bigr\rangle_{\Omega_\pm^0  } + \bigl\langle \hat{\vct{\sigma}}_\gamma \bigl(\jump{ \vct{u}}_\gamma , p_\rmf \bigr) , \jump{\vct{v}}_\gamma \bigr\rangle_{\! \gamma }  - \bigl\langle p_\pm \hat{\matr{\alpha}}_\pm , \nabla \vct{v}_\pm \bigr\rangle_{\Omega_\pm^0  } \\
= \bigl\langle \hat{\vct{f}}_\pm , \vct{v}_\pm \bigr\rangle_{\Omega_\pm^0  } + \bigl\langle \hat{G}_\pm \hat{\matr{\alpha}}_\pm , \nabla \vct{v}_\pm \bigr\rangle_{\Omega_\pm^0 } + \bigl\langle a \mathfrak{A}_{\!\vct{N}} \hat{\vct{f}}_\rmf^\mathrm{eff} , \vct{v}_- \bigr\rangle_{\! \gamma  } + \bigl\langle  \hat{\vct{F}}_\gamma^\mathrm{eff} , \jump{\vct{v}}_\gamma \bigr\rangle_{\!\gamma }   
\end{multlined}
\end{equation}
holds for all $\vct{v}_\pm \in \vct{V}_1$ and almost all $t\in I$ with the effective total stress~$\hat{\matr{\sigma}}_\gamma$ given by \cref{eq:mechlimit_nuC1_df_c}.
In addition, for all $\vct{v}_\pm \in \vct{V}_1$, the initial displacement vector~$\vct{u}_{0,\pm} := \vct{u}_\pm ( 0)$ satisfies
\begin{equation}
\label{eq:weak_mechlimit_discrete_C1_b}
\begin{multlined}[c][0.875\displaywidth]
 \bigl\langle \hat{\mathbb{C}}_\pm \matr{e} ( \vct{u}_{0,\pm} ) , \matr{e} ( \vct{v}_\pm ) \bigr\rangle_{\Omega_\pm^0  } + \bigl\langle \hat{\vct{\sigma}}_\gamma \bigl(\jump{ \vct{u}_0}_\gamma , \hat{p}_{0,\rmf} \bigr) , \jump{\vct{v}}_\gamma \bigr\rangle_{\! \gamma } - \bigl\langle \hat{p}_{0,\pm} \hat{\matr{\alpha}}_\pm , \nabla \vct{v}_\pm \bigr\rangle_{\Omega_\pm^0  }  \\
 = \bigl\langle \hat{\vct{f}}_\pm (0) , \vct{v}_\pm \bigr\rangle_{\Omega_\pm^0  } + \bigl\langle \hat{G}_\pm \hat{\matr{\alpha}}_\pm , \nabla \vct{v}_\pm \bigr\rangle_{\Omega_\pm^0 } +  \bigl\langle a \mathfrak{A}_{\!\vct{N}} \hat{\vct{f}}_\rmf^\mathrm{eff} (0) , \vct{v}_- \bigr\rangle_{\! \gamma  } \\ 
+ \bigl\langle \hat{\vct{F}}_\gamma^\mathrm{eff} (0) , \jump{\vct{v}}_\gamma \bigr\rangle_{\! \gamma }  .
\end{multlined}
\end{equation}%
\label{eq:weak_mechlimit_discrete_C1}%
\end{subequations}%
\begin{remark} \label{rem:C1discrete}
\begin{enumerate}[label=(\roman*)]
\item Solving the reduced limit problem~\eqref{eq:weak_mechlimit_discrete_C1} instead of the two-scale limit problem~\eqref{eq:weak_mechlimit_C1} is only desirable if the corresponding flow limit problem (see \Cref{sec:4.3}) does not require the displacement vector~$\vct{u}_\rmf$ in the full-dimensional fracture domain~$\Omega_\rmf^1$. 
For $\nu_\matr{K} < 1$, this is the case if $\hat{\vct{\alpha}}_\rmf^\eff$ is constant in normal direction (not only piecewise constant), i.e., $\hat{\vct{\alpha}}_\rmf^\eff ( \vct{y} + s\vct{N} ) = ( \mathfrak{A}_{\vct{N}} \hat{\vct{\alpha}}_\rmf^\eff ) (\vct{y} )$ for almost all $\vct{y} + s \vct{N} \in \Omega_\rmf^1$ with $\vct{y} \in \gamma$ and $s \in (- a_-(\vct{y} ) , a_+ (\vct{y} ) )$; see \Cref{rem:df} for details. 
For $\nu_\matr{K} = 1$, this holds only if $2\nu_{\matr{\alpha}}^\smallperp > \nu_\tsr{C} - 1$ (cf.~\Cref{rem:K1}).
For $\nu_\matr{K} > 1$, it holds only if $\nu_\omega > -1$ or $2\nu_{\matr{\alpha}}^\smallperp > \nu_\tsr{C} - 1$ (cf.~\Cref{sec:4.3.5}). 
\item The reduced limit problem \eqref{eq:weak_mechlimit_discrete_C1} may still depend on the full-dimensional fracture pressure head~$p_\rmf$.
In this case, a complete discrete fracture model for both the mechanics and the flow equation is only available if the fracture pressure head~$p_\rmf$ is constant in normal direction, i.e., if $\nu_\matr{K} < 1$ (see \Cref{cor:const}).
\item The reduced limit problem~\eqref{eq:weak_mechlimit_discrete_C1} is equivalent to the two-scale limit problem~\eqref{eq:weak_mechlimit_C1} except for the elimination of the fracture displacement vector.
When solving the reduced limit problem~\eqref{eq:weak_mechlimit_discrete_C1}, we can solve the full-di\-men\-sio\-nal limit problem~\eqref{eq:weak_mechlimit_C1} a posteriori for the fracture displacement~$\vct{u}_\rmf$ with the pressure heads~$p_\pmf$ and bulk displacements~$\vct{u}_\pm$ already known and given as boundary conditions.
\end{enumerate}
\end{remark}

The following result now establishes a relationship between the full-dimensional limit problem~\eqref{eq:weak_mechlimit_C1} and the reduced limit problem~\eqref{eq:weak_mechlimit_discrete_C1}.
\begin{theorem} 
Let $\vct{u}_\pmf \in \vct{V}^\#$ be a solution of the full-dimensional limit problem~\eqref{eq:weak_mechlimit_C1}.
Then, $\vct{u}_\pm \in \vct{V}_1$ solves the discrete fracture limit problem~\eqref{eq:weak_mechlimit_discrete_C1}.
\end{theorem}
\begin{proof}
Let $\vct{u}_\pmf \in \vct{V}^\#$ be a solution of \cref{eq:weak_mechlimit_C1} and let $\vct{v}_\pm \in \vct{V}_1$. We define $\vct{v}_\rmf \in \vct{H}^1_{\!\vct{N}} (\Omega_\rmf^1 )$ by
\begin{align*}
\vct{v}_\rmf (\vct{y} + s\vct{N} ) := \vct{v}_- \vert_\gamma + \frac{1}{a(\vct{y})} \int_{-a_-(\vct{y})}^{s} \bigl[ \hat{\matr{C}}_\rmf^{\!\vct{N}} (\vct{y} + \bar{s} \vct{N} )  \bigr]^{-1} \hat{\matr{C}}_\gamma^{\!\vct{N}} (\vct{y} )\, \jump{\vct{v}}_\gamma ( \vct{y} ) \,\rmd \bar{s}
\end{align*}
where $\vct{y} \in \gamma$ and $s \in (-a_- (\vct{y}) , a_+ (\vct{y}))$.
Clearly, we then have $\vct{v}_\pmf \in \vct{V}^\#$ with 
\begin{align*}
\partial_{\vct{N}}\vct{v}_\rmf (\vct{y} + s\vct{N} ) = \frac{1}{a(\vct{y})}\bigl[ \hat{\matr{C}}_\rmf^{\!\vct{N}} (\vct{y} + s \vct{N} )  \bigr]^{-1} \hat{\matr{C}}_\gamma^{\!\vct{N}} (\vct{y} )\, \jump{\vct{v}}_\gamma ( \vct{y} ) .
\end{align*}
Choosing  $\vct{v}_\pmf \in \vct{V}^\#$ as defined above as the test function in \cref{eq:weak_mechlimit_C1} and using that
\begin{align*}
\bigl\langle \hat{\vct{\sigma}}_\rmf \bigl( \vct{u}_\rmf , p_\rmf \bigr) , \partial_\vct{N}  \vct{v}_\rmf \bigr\rangle_{\Omega_\rmf^1 } 
= \bigl\langle \tfrac{1}{a} \matr{C}_\gamma^{\!\vct{N}} \jump{\vct{u}}_\gamma , \jump{\vct{v}}_\gamma \bigr\rangle_{\! \gamma }  - \bigl\langle \matr{C}_\gamma^{\!\vct{N}}  \mathfrak{A}_{\!\vct{N}} \bigl[ ( p_\rmf - \hat{G}_\rmf )  \bigl( \matr{C}_\rmf^\vct{N} \bigr)^{-1} \hat{\vct{\alpha}}_\rmf^\mathrm{eff} \bigr] , \jump{\vct{v}}_\gamma  \bigr\rangle_{\! \gamma }
\end{align*} 
shows that $\vct{u}_\pm$ satisfies \cref{eq:weak_mechlimit_discrete_C1}, where, w.l.o.g., we have considered the case $\nu_\matr{K} \le 1$ or $\nu_\omega = -1$ in \cref{eq:mechlimit_nuC1_f}.
Moreover, since the map
\begin{align*}
\vct{V}^\# \rightarrow \vct{V}_1 , \quad \vct{w}_\pmf \mapsto \vct{w}_\pm
\end{align*}
clearly defines a continuous embedding, we have $\vct{u}_\pm \in \vct{V}_1$.
\end{proof}

\subsubsection{Limit Models with \texorpdfstring{$\nu_\tsr{C} > 1$}{nuC>1}}
\label{sec:4.2.2}
For $\nu_\tsr{C} > 1$, the fracture contribution to the mechanics equation~\eqref{eq:fulldim-weak-trafo-b} vanishes completely as $k\rightarrow \infty$.
However, in a second step, by rescaling \cref{eq:fulldim-weak-trafo-b} with the factor~$\epsilon^{(1-\nu_\tsr{C})/2}$ inside the fracture, we can derive a limit equation for the (weak) limit~$\smash{\hat{\vct{u}}_\rmf^\#}$  of the scaled fracture displacement~$\epsilon^{(\nu_\tsr{C}- 1)/2}\hat{\vct{u}}_k^{\epsilon_k}$ as $k\rightarrow\infty$ (cf.\ \Cref{cor:conv}). 
Specifically, this is achieved by choosing the test functions $\vct{v}_\pm = \vct{0}$ and $\vct{v}_\rmf = \epsilon^{ (1-\nu_\tsr{C}) /2} \bar{\vct{v}}_\rmf$ with $\bar{\vct{v}}_\rmf \in \vct{H}^1_0 (\Omega_\rmf^1 ) $ in \cref{eq:fulldim-weak-trafo-b}. 
Then, in particular, $\smash{\hat{\vct{u}}_\rmf^\#}$ is not directly coupled to the bulk displacement vector~$\smash{\hat{\vct{u}}_\pm^\#}$ in the limit model, only possibly to the fracture pressure head~$\smash{\hat{p}_\rmf^\#}$ (if $2\nu_{\!\matr{\alpha}}^\smallperp = \nu_\tsr{C} -1$ and if $\nu_\matr{K}\le 1$ or $\nu_\omega = -1$).
The strong formulation of the mechanics limit problem for~$\nu_\tsr{C} > 1$ now reads as follows with $p_\pmf$ and $\vct{u}_\pmf$ corresponding to the limit functions $\hat{p}_\pmf^\#$ and $\hat{\vct{u}}_\pmf^\#$ from \Cref{cor:conv}.

Find $\vct{u}_\pmf \colon \Omega_\pmf^\odot \times I \rightarrow \mathbb{R}^n$ and $p_\pm \colon \Omega_\pm^0 \times I \rightarrow \mathbb{R}$, as well as $p_\rmf \colon \Omega_\rmf^1 \times I \rightarrow \mathbb{R}$ if $\nu_{\matr{K}} \le 1$ or $\nu_\omega = -1$, such that
\begin{subequations}
\begin{alignat}{2}
-\partial_{\!\vct{N} } \hat{\matr{\sigma}}_\rmf ( \vct{u}_\rmf , p_\rmf )  &= \hat{\vct{f}}_\rmf^\mathrm{eff} \quad\enspace &&\text{in } \Omega_\rmf^1 \times I , \\
\vct{u}_\rmf &= \vct{0} \quad\enspace &&\text{on } \gamma_\pm^1 \times I ,\\
\hat{\matr{\sigma}} ( \vct{u}_\pm , p_\pm ) \vct{n}_\pm &= \vct{0} \quad\enspace &&\text{on } \gamma \times I ,
\intertext{
and the bulk limit problem~\cref{eq:bulklimit} is satisfied.
In addition, the initial displacement vector $\vct{u}_{0, \rmf } := \vct{u}_\rmf ( \cdot , 0 ) $ fulfills}
-\partial_{\!\vct{N} } \hat{\matr{\sigma}}_\rmf (\vct{u}_{0,\rmf } , \hat{p}_{0, \rmf} ) &= \hat{\vct{f}}_\rmf^\mathrm{eff} (\cdot , 0) \quad &&\text{in } \Omega_\rmf^1 \times I , \\
\vct{u}_{0,\rmf} &= \vct{0} \quad\enspace &&\text{on } \gamma_\pm^1 \times I , \\
\hat{\matr{\sigma}} ( \vct{u}_{0,\pm } , \hat{p}_{0, \pm} ) \vct{n}_\pm &= \vct{0} \quad\enspace &&\text{on } \gamma \times I .
\end{alignat}
The total fracture stress~$\hat{\matr{\sigma}}_\rmf$ is given by
\begin{align}
\hat{\matr{\sigma}}_\rmf ( \vct{u}_\rmf , p_\rmf ) &:= \begin{cases} \hat{\matr{C}}_\rmf^{\!\vct{N}} \partial_{\!\vct{N}} \vct{u}_\rmf - ( p_\rmf + \hat{G}_\rmf ) \hat{\vct{\alpha}}_\rmf^\mathrm{eff} &\text{if } \nu_\matr{K} \le 1 \text{ or } \nu_\omega = -1 , \\
\hat{\matr{C}}_\rmf^{\!\vct{N}} \partial_{\!\vct{N}} \vct{u}_\rmf &\text{if } \nu_\matr{K} > 1 \text{ and } \nu_\omega > -1 .
\end{cases}
\end{align}%
\label{eq:mechlimit_nuCg1}%
\end{subequations}%
The mechanics limit problem~\eqref{eq:mechlimit_nuCg1} is to be completed by suitable conditions or equations in the fracture domain~$\Omega_\rmf^1$ or on the interface~$\gamma$ (see \Cref{sec:4.3}), in particular to determine the fracture pressure head~$p_\rmf$.

Next, we define the function space~$\vct{V}_{>1}$ by
\begin{align}
\vct{V}_{>1} &:= \bigl\{ (\vct{v}_\pm , \vct{v}_\rmf ) \in  \vct{H}^1_{0, \rho_{\pm}^0 }  \!\bigl( \Omega_\pm^0 \bigr) \times \vct{H}^1_{\!\vct{N}} ( \Omega_\rmf^1 )  \ \big\vert\  \mathfrak{T}_\pm \vct{v}_\rmf = \vct{0} \bigr\} , \label{eq:Vg1}
\end{align}
where $\mathfrak{T}_\pm$ denotes the trace operator in $\vct{H}^1_{\!\vct{N}} (\Omega_\rmf^1 )$; see \Cref{lem:trace_H1N} below.
\begin{lemma} \label{lem:trace_H1N}
There exists a uniquely defined bounded linear operator
\begin{subequations}
\begin{align}
\mathfrak{T}_\pm \colon H^1_{\!\vct{N}} (\Omega_\rmf^1 ) \rightarrow L^2 ( \gamma ) 
\end{align}
such that, for $\vct{y} \in \gamma$ and all $h \in \mathcal{C}^{0,1} ( \bar{\Omega}_\rmf^1 ) $, we have
\begin{align}
\bigl( \mathfrak{T}_\pm h \bigr) ( \vct{y} ) = h \bigl( \vct{y} \pm a_\pm ( \vct{y} ) \vct{N} \bigr) .
\end{align}
\end{subequations}
\end{lemma}
\begin{proof}
See \cite[Lem.\ 3.1]{hoerl24}.
\end{proof}
A weak formulation  of \cref{eq:mechlimit_nuCg1} is now given by the following problem.

Find $\vct{u}_\pmf \colon H^1 ( I ; \vct{V}_{>1})$ and $p_\pm \in L^2 ( I ; L^2 ( \Omega_\pm^0 )) $, as well as $p_\rmf \in L^2 ( I ; L^2 ( \Omega_\rmf^1 ) )$ if $\nu_{\matr{K}} \le 1$ or $\nu_\omega = -1$, such that%
\begin{subequations}
\begin{align}
\begin{split}
\bigl\langle \hat{\mathbb{C}}_\pm \matr{e} ( \vct{u}_\pm ) , \matr{e} ( \vct{v}_\pm ) \bigr\rangle_{\Omega_\pm^0  } - \bigl\langle p_\pm \hat{\matr{\alpha}}_\pm , \nabla \vct{v}_\pm \bigr\rangle_{\Omega_\pm^0 } \hskip-1.5cm  &\\
  &= \bigl\langle \hat{\vct{f}}_\pm , \vct{v}_\pm \bigr\rangle_{\Omega_\pm^0  } + \bigl\langle \hat{G}_\pm \hat{\matr{\alpha}}_\pm , \nabla \vct{v}_\pm \bigr\rangle_{\Omega_\pm^0 } ,
 \end{split} \\
\label{eq:weak_mechlimit_nuCg1_b} \bigl\langle \hat{\matr{\sigma}}_\rmf \bigl( \vct{u}_\rmf , p_\rmf \bigr)  , \partial_{\!\vct{N}} \vct{v}_\rmf \bigr\rangle_{\Omega_\rmf^1 }  &= \bigl\langle \hat{\vct{f}}_\rmf^\mathrm{eff} \! , \vct{v}_\rmf \bigr\rangle_{\Omega_\rmf^1 }  , 
\intertext{holds for all $\vct{v}_\pmf \in \vct{V}_{>1}$ and almost all $t \in I$.
In addition, for all $\vct{v}_\pmf \in \vct{V}_{>1}$, the initial displacement vector~$\vct{u}_{0,\pmf} := \vct{u}_\pmf (  0)$  satisfies}
\begin{split}
\bigl\langle \hat{\mathbb{C}}_\pm \matr{e} ( \vct{u}_{0,\pm} ) , \matr{e} ( \vct{v}_\pm ) \bigr\rangle_{\Omega_\pm^0  } - \bigl\langle \hat{p}_{0,\pm} \hat{\matr{\alpha}}_\pm , \nabla \vct{v}_\pm \bigr\rangle_{\Omega_\pm^0  } \hskip-1.5cm  & \\
 &=\bigl\langle \hat{\vct{f}}_\pm (0) , \vct{v}_\pm \bigr\rangle_{\Omega_\pm^0  } + \bigl\langle \hat{G}_\pm \hat{\matr{\alpha}}_\pm , \nabla \vct{v}_\pm \bigr\rangle_{\Omega_\pm^0 } ,
\end{split} \\
\bigl\langle  \hat{\matr{\sigma}}_\rmf \bigl( \vct{u}_{0 ,\rmf } , \hat{p}_{0,\rmf } ) , \partial_\vct{N} \vct{v}_\rmf \bigr\rangle_{\Omega_\rmf^1 } &= \bigl\langle \hat{\vct{f}}_\rmf^\mathrm{eff} (0) , \vct{v}_\rmf \bigr\rangle_{\Omega_\rmf^1}  .
\end{align}%
\label{eq:weak_mechlimit_nuCg1}%
\end{subequations}%
\indent The mechanics limit problem~\eqref{eq:weak_mechlimit_nuCg1} is to be supplemented by a flow equation to determine the pressure heads~$p_\pmf$ (see \Cref{sec:4.3}).
In particular, the limit functions~$\hat{\vct{u}}_\pmf^\#$ and $\hat{p}_\pmf^\#$ from \Cref{cor:conv} solve \cref{eq:weak_mechlimit_nuCg1}; see \Cref{thm:Klm1,thm:Km1,thm:Km1p1,thm:Kp1,thm:Kgp1}
 in \Cref{sec:4.3}.
A rigorous reduction of the limit problem~\eqref{eq:weak_mechlimit_nuCg1} to a discrete fracture model does not seem to be possible.

\subsection{Fracture Limit Problems for the Flow Equation \texorpdfstring{\eqref{eq:fulldim-weak-trafo-c}}{}}
\label{sec:4.3}
In the following, we consider the limit of the flow equation~\eqref{eq:fulldim-weak-trafo-c} as $\epsilon \rightarrow 0$. 
In analogy to the case of Darcy flow in a non-deformable porous medium~\cite{hoerl24}, we distinguish between five different limit regimes: 
\begin{itemize}[itemsep=0pt]
\item a highly conductive fracture with instantaneously equalized constant pressure head for $\nu_\matr{K} < -1$ presented in \Cref{sec:4.3.1},
\item a conductive fracture with a PDE in tangential direction on the interface~$\gamma$ (see \Cref{sec:4.3.2}),
\item a neutral fracture for $\nu_\matr{K} \in (-1,1)$ discussed in \Cref{sec:4.3.3} where the hydraulic conductivity within the fracture does not influence the flow across~$\gamma$,
\item a permeable barrier with an ODE in normal direction for $\nu_\matr{K} = 1$ (see \Cref{sec:4.3.4}),
\item a solid wall for $\nu_\matr{K} > 1$ presented in \Cref{sec:4.3.5}.
\end{itemize}
Together with a limit model for the mechanics equations~\eqref{eq:fulldim-weak-trafo-b} as presented in \Cref{sec:4.2} we obtain a closed limit model for each of these regimes, for which we verify uniqueness and provide the corresponding convergence theorems for the overall model. 
We start by defining the following effective model coefficients.

We define the normal hydraulic conductivity~$\hat{K}_\rmf^\vct{N} \in L^\infty (\Omega_\rmf^1 )$, the effective tangential hydraulic conductivity~$\smash{\hat{\matr{K}}_\gamma} \in \matr{L}^\infty (\gamma ) $, the effective inverse Biot modulus~$\hat{\omega}_\rmf^\mathrm{eff} \in L^\infty (\Omega_\rmf^1 ) $, and the effective fluid source term~$\hat{q}_\rmf^\mathrm{eff} \in L^2 ( I ; \Omega_\rmf^1 ) $ inside the fracture by
\begin{subequations}
\begin{align}
\label{eq:KfN} \hat{K}_\rmf^{\!\vct{N}} &:= \hat{\matr{K}}_\rmf \vct{N} \cdot \vct{N} , \\
 \label{eq:Kgamma}\hat{\matr{K}}_\gamma &:= \mathfrak{A}_{\!\vct{N}} \bigl( \hat{\matr{K}}_\rmf - \bigl[ \hat{K}_\rmf^{\!\vct{N}} \bigr]^{-1} \hat{\matr{K}}_\rmf \vct{N} \otimes \hat{\matr{K}}_\rmf \vct{N} \bigr) , \\
\label{eq:omegaeff}\hat{\omega}_\rmf^\mathrm{eff} &:= \begin{cases} 
 \hat{\omega}_\rmf  &\text{if } \nu_\omega = -1 , \\
 0 &\text{if } \nu_\omega > -1  , \end{cases} \\
 \label{eq:qeff}\hat{q}_\rmf^\mathrm{eff} &:= \begin{cases}  \hat{q}_\rmf &\text{if } \nu_q = -1, \\ 0 &\text{if } \nu_q > -1  . \end{cases} 
\end{align}%
\label{eq:eff2}%
\end{subequations}%
In particular, $\hat{\matr{K}}_\gamma$ has the following properties.
\begin{lemma} \label{lem:eff2}
\begin{enumerate}[label=(\roman*)]
\item $\hat{\matr{K}}_\gamma \in \matr{L}^\infty (\gamma ) $ is symmetric and positive semidefinite. Besides, we have $\hat{\matr{K}}_\gamma ( \vct{y} ) \vct{\xi} \cdot \vct{\xi} > 0$ for almost all $\vct{y} \in \gamma$ and $\vct{\xi} \in \mathbb{R}^n \setminus \{ \vct{0} \}$ with $\vct{\xi} \cdot \vct{N} = 0$.
\item If $\hat{\matr{K}}_\rmf \in \mathcal{C}^0 ( \bar{\Omega}_\rmf^1 ; \mathbb{R}^{n\times n } ) $, then $\hat{\matr{K}}_\gamma$ is uniformly elliptic on~$\gamma$, i.e., $\hat{\matr{K}}_\gamma ( \vct{y} ) \vct{\xi} \cdot \vct{\xi} \gtrsim \abs{\vct{\xi}}^2$ for  all $\vct{y} \in \gamma$ and $\vct{\xi} \in \mathbb{R}^n$ with $\vct{\xi} \cdot \vct{N} = 0$.
\end{enumerate}
\end{lemma} 
\begin{proof}
See \cite[Lem.\ 4.6]{hoerl24}.
\end{proof}

\subsubsection{Case \texorpdfstring{$\nu_\matr{K} < -1$}{nuK<-1}} 
\label{sec:4.3.1}
For $\nu_\matr{K} < -1$, we obtain a perfectly conductive fracture. 
Fluctuations of the pressure head inside the fracture are instantaneously equalized, i.e., the pressure head inside the fracture is constant in space. 
The set~$W \subset \mathbb{R}$ of admissible constants for the fracture pressure head is given by
\begin{align}
W := \bigl\{ \phi^\ast \in \mathbb{R} \ \big\vert\ \exists\, \phi_\pmf \in \Phi \text{ with } \phi_\rmf \equiv \phi^\ast \bigr\}  \label{eq:W}
\end{align}
and, depending on the boundary conditions, we either have $W = \{ 0 \}$ or $W = \mathbb{R}$.

The strong formulation of the limit problem for $\nu_\matr{K} < -1$ now reads as follows with $p_\pm$, $p_\gamma$, and $\vct{u}_\pmf$ corresponding to the limit functions~$\smash{\hat{p}_\pm^\#}$, $\mathfrak{A}_\rmf \smash{\hat{p}_\rmf^\#}$, and $\smash{\hat{\vct{u}}_\pmf^\#}$ from \Cref{cor:conv}.

Find $p_\pm \colon \Omega_\pm^0 \times I \rightarrow \mathbb{R}$, $p_\gamma \colon I \rightarrow W$, and $\vct{u}_\pmf \colon \Omega_\pmf^\odot \times I \rightarrow \mathbb{R}^n$  such that 
\begin{subequations}
\begin{alignat}{2}
p_\pm &\equiv p_\gamma \quad\enspace &&\text{on } \gamma \times I\\
\intertext{and the bulk limit problem~\eqref{eq:bulklimit}, as well as either the mechanical limit problem \eqref{eq:mechlimit_nuC1} if $\nu_\tsr{C } = 1$ or \eqref{eq:mechlimit_nuCg1} if $\nu_\tsr{C} > 1$ are satisfied.
Moreover, if $W = \mathbb{R}$, we have}
A \bigl( \mathfrak{A}_\rmf \hat{\omega}_\rmf^\mathrm{eff} \bigr) \partial_t p_\gamma -  \int_\gamma \jump{\hat{\matr{K}}\nabla p \cdot \!\vct{N} } \,\rmd S + \bigl\langle  \hat{\vct{\alpha}}_\rmf^\mathrm{eff}\! , \partial_{\!\vct{N}} \partial_t \vct{u}_\rmf \bigr\rangle_{\Omega_\rmf^1} &= A\, \mathfrak{A}_\rmf \hat{q}_\rmf^\mathrm{eff} \quad\enspace &&\text{in }  I , \\
p_\gamma (0) &= \mathfrak{A}_\rmf \hat{p}_{0 , \rmf } \quad\enspace &&\text{on } \gamma ,
\end{alignat}%
\label{eq:Klm1}%
\end{subequations}%
where $A := \int_\gamma a \,\rmd S $ and $\jump{\cdot}$ denotes the jump operator from \cref{eq:jump}.

Further, with the function space~$\Phi_{<-1}$ defined by
\begin{align}
\Phi_{<-1} := \Bigl\{ (\phi_\pm , \phi_\gamma ) \in H^1_{0 , \rho^0_{\pm , \mathrm{D}}} \! (\Omega_\pm^0 ) \times W \ \Big\vert\ \phi_\pm \bigr\vert_\gamma \equiv \phi_\gamma \Bigr\}  , \label{eq:Philm1}
\end{align}
a weak formulation of \cref{eq:Klm1} is given by the following problem.

Find $\vct{u}_\pmf \in H^1 ( I ; \vct{V}^\# ) $ if $\nu_\tsr{C} = 1$ or $\vct{u}_\pmf \in H^1 ( I ; \vct{V}_{>1} )$ if $\nu_\tsr{C} > 1$, as well as
$ ( p_\pm , p_\gamma ) \in L^2 ( I ; \Phi_{<-1} )$ with $p_\pm \in H^1 ( I ; L^2 (\Omega_\pm^0 ) )$ and, if $\nu_\omega = -1$, $p_\gamma \in  H^1 ( I )$ such that $p_\gamma (0) = \mathfrak{A}_\rmf \hat{p}_{0, \rmf }$ and
\begin{equation}
\label{eq:weak_C1_Klm1}
\begin{multlined}[c][0.875\displaywidth]
 \bigl\langle \phi_\pm \hat{\matr{\alpha}}_\pm , \nabla \partial_t \vct{u}_\pm \bigr\rangle_{\Omega_\pm^0  }  +  \bigl\langle \hat{\omega}_\pm \partial_t  p_\pm , \phi_\pm \bigr\rangle_{\Omega_\pm^0 } +  \bigl\langle \hat{\matr{K}}_\pm \nabla p_\pm , \nabla \phi_\pm \bigr\rangle_{ \Omega_\pm^0  } \\
  + A \bigl( \mathfrak{A}_\rmf \hat{\omega}_\rmf^\mathrm{eff} \bigr) \partial_t p_\gamma \phi_\gamma   
 + \bigl\langle \phi_\gamma \hat{\vct{\alpha}}_\rmf^\mathrm{eff} , \partial_{\!\vct{N}} \partial_t \vct{u}_\rmf  \bigr\rangle_{\Omega_\rmf^1  } = \bigl\langle \hat{q}_\pm , \phi_\pm \bigr\rangle_{\Omega_\pm^0  } + A \bigl( \mathfrak{A}_\rmf \hat{q}_\rmf^\mathrm{eff} \bigr) \phi_\gamma 
\end{multlined}
\end{equation}
holds for all $(\phi_\pm , \phi_\gamma ) \in \Phi_{<-1}$ and almost all $t\in I$, and either the mechanics limit problem~\eqref{eq:weak_mechlimit_C1} if $\nu_\tsr{C} = 1$ or \eqref{eq:weak_mechlimit_nuCg1} if $\nu_\tsr{C} > 1$ is satisfied. 

\begin{remark} \label{rem:df}
The limit problem~\eqref{eq:weak_C1_Klm1} is a discrete fracture model with respect to the pressure head but still depends on the displacement vector~$\vct{u}_\rmf$ inside the full-dimensional fracture~$\Omega_\rmf^1$ if $2\nu_{\matr{\alpha}}^\smallperp = \nu_\tsr{C} - 1$ (i.e., if $\vct{\alpha}_\rmf^\eff \not= \vct{0}$).
However, given the additional assumption that $\hat{\vct{\alpha}}_\rmf^\mathrm{eff}$ is constant in normal direction (not only piecewise constant), i.e., $\hat{\vct{\alpha}}_\rmf^\mathrm{eff} ( \vct{y} + s\vct{N}) = (\mathfrak{A}_{\!\vct{N}} \hat{\vct{\alpha}}_\rmf^\mathrm{eff} ) (\vct{y} )$ for almost all~$\vct{y} + s\vct{N} \in \Omega_\rmf^1$ with $\vct{y} \in \gamma$ and $s \in (-a_- (\vct{y} ) , a_+ (\vct{y} ) )$, we can rewrite the corresponding term in \cref{eq:weak_C1_Klm1} as 
\begin{align}
\begin{split}
\bigl\langle \phi_\gamma \hat{\vct{\alpha}}_\rmf^\mathrm{eff} , \partial_{\!\vct{N}} \partial_t \vct{u}_\rmf  \bigr\rangle_{\Omega_\rmf^1  } &= \int_\gamma \phi_\gamma \hat{\vct{\alpha}}_\rmf^\mathrm{eff}  \cdot \int_{-a_-(\vct{y})}^{a_+ (\vct{y} ) } \frac{\rmd }{\rmd s} \partial_t \vct{u}_\rmf (\vct{y} + s \vct{N} ) \,\rmd s \, \rmd \vct{y} \\
&= \begin{cases} \bigl\langle \phi_\gamma \hat{\vct{\alpha}}_\rmf^\mathrm{eff} , \jump{\partial_t \vct{u}} \bigr\rangle_{\gamma } &\text{if } \nu_\tsr{C} = 1 , \\ 
0 &\text{if } \nu_\tsr{C} > 1.
\end{cases}
\end{split}
\end{align}
Thus, in this case, the model is fully reduced to a discrete fracture model and we can solve the reduced mechanics problem~\eqref{eq:weak_mechlimit_discrete_C1} instead of \eqref{eq:weak_mechlimit_C1} if $\nu_\tsr{C} = 1$ and do not have to solve the fracture part~\eqref{eq:weak_mechlimit_nuCg1_b} of the mechanics problem~\eqref{eq:weak_mechlimit_nuCg1} if $\nu_\tsr{C} > 1$.
\end{remark}

We can now prove the following convergence theorem.

\begin{theorem} \label{thm:Klm1} 
Let $\nu_\tsr{C} \ge 1$, $\nu_\matr{K} < -1$, and assume that the \Cref{asm:eps,asm:nu} hold. 
Then,
\begin{align*}
\bigl( \hat{p}_\pm^\# , \mathfrak{A}_\rmf \hat{p}_\rmf^\# \bigr) &\in L^\infty ( I ; \Phi_{<-1} ) \quad\text{with}\enspace\ \begin{cases} 
 \hphantom{\mathfrak{A}_\rmf}\hat{p}_\pm^\# \in H^1 \bigl( I; L^2 (\Omega_\pm^0 )\bigr) , & \\
\mathfrak{A}_\rmf\hat{p}_\rmf^\# \in H^1 (I) \quad\text{if } \nu_\omega = -1,
\end{cases}  \\
\hat{\vct{u}}_\pmf^\# &\in \begin{cases} L^\infty ( I ; \vct{V}^\# ) \cap H^1 ( I ; \vct{V}^\# )  &\text{if } \nu_\tsr{C} = 1, \\
L^\infty ( I ; \vct{V}_{>1} ) \cap H^1 ( I ; \vct{V}_{>1} )   &\text{if } \nu_\tsr{C} > 1
\end{cases}
\end{align*}%
are the unique weak solution of the limit problem in \cref{eq:weak_C1_Klm1}, where $\hat{p}_\pmf^\#$ and $\hat{\vct{u}}_\pmf^\#$ denote the limit functions from \Cref{cor:conv} and $\mathfrak{A}_\rmf$ is the averaging operator from \Cref{def:average}.
In particular, we have $\smash{\hat{p}_\rmf^\# (\vct{x} )} = \mathfrak{A}_\rmf \smash{\hat{p}_\rmf^\#} \in W \subset \mathbb{R}$ for almost all $\vct{x} \in \Omega_\rmf^1 $. 
Moreover, \Cref{cor:conv,thm:strongconv} hold for the entire sequences $\{ \hat{\vct{u}}_\pmf^\epsilon \}_{\epsilon \in ( 0,\epsilon^\ast]}$ and $\{ \hat{p}_\pmf^\epsilon \}_{\epsilon \in ( 0,\epsilon^\ast]}$.
\end{theorem}
\begin{proof} 
See \Cref{sec:6}.
\end{proof}

\subsubsection{Case \texorpdfstring{$\nu_\matr{K} = -1$}{nuK=-1}} 
\label{sec:4.3.2}
In this case, we obtain a conductive fracture with continuous pressure heads across the fracture interface~$\gamma$. Besides, the fracture pressure head satisfies an interfacial PDE on~$\gamma$. 
The strong formulation of the limit problem for $\nu_\matr{K} = -1$ now reads as follows, where $\vct{u}_\pmf$, $p_\pm$, and $p_\gamma$ correspond to the limit functions~$\hat{\vct{u}}_\pmf^\#$, $\hat{p}_\pm^\#$, and $\mathfrak{A}_{\!\vct{N}}\hat{p}_\rmf^\#$ in \Cref{cor:conv}.

Find $p_\pm \colon \Omega_\pm^0 \times I \rightarrow \mathbb{R}$, $p_\gamma \colon \gamma \times I \rightarrow \mathbb{R}$, and $\vct{u}_\pmf \colon \Omega_\pmf^\odot \times I \rightarrow \mathbb{R}^n$ such that 
\begin{subequations}
\begin{align}
 \begin{split}
 a( \mathfrak{A}_{\!\vct{N}} \hat{\omega}_\rmf^\mathrm{eff} ) \partial_t p_\gamma - \nabla \cdot \bigl( a \hat{\matr{K}}_\gamma \nabla p_\gamma \bigr) + a\mathfrak{A}_{\!\vct{N}} \bigl( \hat{\vct{\alpha}}_\rmf^\mathrm{eff}\! \cdot \partial_{\!\vct{N}} \partial_t \vct{u}_\rmf \bigr) \hskip-3cm  & \\
 &= a \mathfrak{A}_{\!\vct{N}} \hat{q}_\rmf^\mathrm{eff} +  \jump{\hat{\matr{K}} \nabla p  \cdot \! \vct{N}}  
 \end{split}  &&\text{on } \gamma \times I , \\
 p_+ = p_- &= p_\gamma \quad &&\text{on } \gamma \times I , \\
p_\gamma (\cdot , 0) &= \mathfrak{A}_{\!\vct{N}}\hat{p}_{0, \rmf} \quad\enspace &&\text{on } \gamma  , \\
p_\gamma &= 0 \quad\enspace &&\text{on } \partial \gamma_\rmD \times I , \\
\matr{K}_\gamma \nabla p_\gamma \cdot \vct{n} &= 0 \quad\enspace &&\text {on } \partial\gamma_\rmN \times I ,
\end{align}%
\label{eq:C1_Km1}%
\end{subequations}%
and the bulk limit problem~\eqref{eq:bulklimit}, as well as either the mechanical limit problem \eqref{eq:mechlimit_nuC1} if $\nu_\tsr{C } = 1$ or \eqref{eq:mechlimit_nuCg1} if $\nu_\tsr{C} > 1$, are satisfied. 

The boundary parts~$\partial\gamma_\rmD$ and~$\partial \gamma_\rmN$ in \cref{eq:C1_Km1} are given by 
\begin{subequations}
\begin{align}
\partial \gamma_\rmD &:= \bigl\{ \vct{y} \in \partial \gamma \ \big\vert\ \exists Z \subset \rho_{\rmf, \rmD}^1 \text{ with } \vct{y} \in Z \text{ and } \abs{Z} > 0  \bigr\}  , \\
\partial \gamma_\rmN &:= \partial \gamma \setminus \partial \gamma_\rmD .
\end{align}%
\label{eq:partialgamma}
\end{subequations}%
Further, we define the function space
\begin{align}
\Phi_{-1} := \Bigl\{ (\phi_\pm , \phi_\gamma ) \in H^1_{0 , \rho^0_{\pm , \mathrm{D}}} \! (\Omega_\pm^0 ) \times H^1(\gamma )  \ \Big\vert\ \phi_\pm \vert_\gamma = \phi_\gamma , \ \phi_\gamma \vert_{\partial \gamma_\mathrm{D}} = 0 \Bigr\}  . \label{eq:Phim1}
\end{align}
Then, the following problem defines a weak formulation of \cref{eq:C1_Km1}.

Find $\vct{u}_\pmf \in H^1 ( I ; \vct{V}^\# ) $ if $\nu_\tsr{C} = 1$ or $\vct{u}_\pmf \in H^1 ( I ; \vct{V}_{>1} )$ if $\nu_\tsr{C} > 1$, as well as $p_\pm \in H^1 ( I ; L^2 (\Omega_\pm^0 ) )$ and $p_\gamma \in  H^1 ( I ; L^2 (\gamma ) )$ with $ ( p_\pm , p_\gamma ) \in L^2 ( I ; \Phi_{-1} )$ and $\hat{p}_\gamma (0) = \mathfrak{A}_{\!\vct{N}}\hat{p}_{0, \rmf}$ such that
\begin{equation}
\label{eq:weak_C1_Km1}
\begin{multlined}[c][0.875\displaywidth]
 \bigl\langle \phi_\pm \hat{\matr{\alpha}}_\pm , \nabla \partial_t \vct{u}_\pm \bigr\rangle_{\Omega_\pm^0  } +
 \bigl\langle \hat{\omega}_\pm \partial_t p_\pm , \phi_\pm \bigr\rangle_{\Omega_\pm^0 }
+  \bigl\langle \hat{\matr{K}}_\pm \nabla p_\pm , \nabla \phi_\pm \bigr\rangle_{ \Omega_\pm^0  }  \\
+ \bigl\langle a ( \mathfrak{A}_{\!\vct{N}}\hat{\omega}_\rmf^\mathrm{eff} ) \partial_t p_\gamma , \phi_\gamma \bigr\rangle_{ \gamma  }  
  + \bigl\langle \phi_\gamma \hat{\vct{\alpha}}_\rmf^\mathrm{eff} \!  , \partial_{\!\vct{N}} \partial_t \vct{u}_\rmf  \bigr\rangle_{\Omega_\rmf^1 }  
 + \bigl\langle a \hat{\matr{K}}_\gamma \nabla p_\gamma , \nabla \phi_\gamma \bigr\rangle_{\gamma } \\
 = \bigl\langle \hat{q}_\pm , \phi_\pm \bigr\rangle_{\Omega_\pm^0  } + \bigl\langle a \mathfrak{A}_{\!\vct{N}}\hat{q}_\rmf^\eff \! ,  \phi_\gamma \bigr\rangle_{\gamma  } 
\end{multlined}
\end{equation}
holds for all $(\phi_\pm , \phi_\gamma ) \in \Phi_{-1}$ and almost all $t\in I$, and either the mechanics limit problem~\eqref{eq:weak_mechlimit_C1} if $\nu_\tsr{C} = 1$ or \eqref{eq:weak_mechlimit_nuCg1} if $\nu_\tsr{C} > 1$ is satisfied. 

The limit problem~\eqref{eq:weak_C1_Km1} is a discrete fracture model with respect to the pressure head, but, in general, not with respect to the displacement vector. A reduction to a full discrete fracture model is possible if $\hat{\vct{\alpha}}_\rmf^\mathrm{eff}$ is constant in normal direction in the same way as for $\nu_\matr{K} < -1$ (see \Cref{rem:df}).

We now obtain the following convergence theorem.
\begin{theorem} \label{thm:Km1} 
Let $\nu_\tsr{C} \ge 1$, $\nu_\matr{K} = -1$, and assume that the \Cref{asm:eps,asm:nu} hold. 
Then,
\begin{align*}
\bigl( \hat{p}_\pm^\# , \mathfrak{A}_{\!\vct{N}} \hat{p}_\rmf^\# \bigr) &\in L^\infty ( I ; \Phi_{-1} ) \quad\text{with}\enspace\ \begin{cases} 
 \hphantom{\mathfrak{A}_{\!\vct{N}}}\hat{p}_\pm^\# \in H^1 \bigl( I; L^2 (\Omega_\pm^0 )\bigr) , & \\
\mathfrak{A}_{\!\vct{N}}\hat{p}_\rmf^\# \in H^1 \bigl( I; L^2 (\gamma ) \bigr) \quad\text{if } \nu_\omega = -1,
\end{cases}  \\
\hat{\vct{u}}_\pmf^\# &\in \begin{cases} L^\infty ( I ; \vct{V}^\# ) \cap H^1 ( I ; \vct{V}^\# )  &\text{if } \nu_\tsr{C} = 1, \\
L^\infty ( I ; \vct{V}_{>1} ) \cap H^1 ( I ; \vct{V}_{>1} )   &\text{if } \nu_\tsr{C} > 1
\end{cases}
\end{align*}%
are a weak solution of the limit problem in~\cref{eq:weak_C1_Km1}, where $\hat{p}_\pmf^\#$ and $\hat{\vct{u}}_\pmf^\#$ denote the limit functions from \Cref{cor:conv} and $\mathfrak{A}_{\!\vct{N}}$ is the averaging operator from \Cref{def:average}. 
In particular, we have $\smash{\hat{p}_\rmf^\#} (\vct{y} +s \vct{N} ) = ( \mathfrak{A}_{\!\vct{N}} \smash{\hat{p}_\rmf^\#} ) (\vct{y})$ for almost all $\vct{y} + s \vct{N} \in \Omega_\rmf^1$ with $\vct{y} \in \gamma$ and $s \in (-a_-(\vct{y}) , a_+(\vct{y}))$.
Moreover, if $\hat{\matr{K}}_\gamma$ is uniformly elliptic on~$\gamma$ (cf.\ \Cref{lem:eff2}), then $\smash{( \hat{p}_\pm^\# , \mathfrak{A}_{\!\vct{N}} \hat{p}_\rmf^\# ) }$ and $\smash{\hat{\vct{u}}_\pmf^\#}$ are the unique weak solution of \cref{eq:weak_C1_Km1} and the entire sequences $\{ \hat{\vct{u}}_\pmf^\epsilon \}_{\epsilon \in ( 0,\epsilon^\ast]}$ and $\{ \hat{p}_\pmf^\epsilon \}_{\epsilon \in ( 0,\epsilon^\ast]}$ converge in \Cref{cor:conv,thm:strongconv} as $\epsilon\rightarrow 0$.
\end{theorem}
\begin{proof}
See \Cref{sec:6}.
\end{proof}

\subsubsection{Case \texorpdfstring{$\nu_\matr{K} \in (-1, 1)$}{-1<nuK<1}} 
\label{sec:4.3.3}  
For $\nu_\matr{K} \in (-1,1 )$, we obtain a fracture with continuous pressure head across the interface~$\gamma$ and neutral hydraulic conductivity, i.e., the fracture conductivity does not have an influence in the limit model. 
The jump of normal velocities across~$\gamma$ is determined only by the fluid source term, storage term, and displacement term (if present) inside the fracture. 
The strong formulation of the limit problem reads as follows with $p_\pm$ and $\vct{u}_\pmf$ corresponding to the limit functions $\hat{p}_\pm^\#$ and $\hat{\vct{u}}_\pmf^\#$ from \Cref{cor:conv}.

Find $p_\pm \colon \Omega_\pm^0 \times I \rightarrow \mathbb{R}$ with $p_\gamma := p_\pm \vert_\gamma$ and $\vct{u}_\pmf \colon \Omega_\pmf^\odot \times I \rightarrow \mathbb{R}^n$ such that 
\begin{subequations}
\begin{align}
 \begin{split}
 a( \mathfrak{A}_{\!\vct{N}} \hat{\omega}_\rmf^\mathrm{eff} ) \partial_t p_\gamma + a\mathfrak{A}_{\!\vct{N}} \bigl( \hat{\vct{\alpha}}_\rmf^\mathrm{eff}\! \cdot \partial_{\!\vct{N}} \partial_t \vct{u}_\rmf \bigr) 
 &= a \mathfrak{A}_{\!\vct{N}} q_\rmf^\mathrm{eff} +  \jump{\hat{\matr{K}} \nabla p \cdot \! \vct{N}}  
 \end{split}  &&\text{on } \gamma \times I , \\
 p_+ &= p_-  \quad &&\text{on } \gamma \times I , 
\end{align}%
\label{eq:Km11}%
\end{subequations}%
and the bulk limit problem~\eqref{eq:bulklimit}, as well as either the mechanical limit problem \eqref{eq:mechlimit_nuC1} if $\nu_\tsr{C } = 1$ or \eqref{eq:mechlimit_nuCg1} if $\nu_\tsr{C} > 1$, are satisfied.

Given the function space
\begin{align}
\Phi_{(-1,1)} := \Bigl\{ \phi_\pm  \in H^1_{0 , \rho^0_{\pm , \mathrm{D}}} \! (\Omega_\pm^0 )  \ \Big\vert\ \phi_+ \vert_\gamma = \phi_- \vert_\gamma  \Bigr\}  , \label{eq:Phim1g1}
\end{align}
the following problem defines a weak formulation of~\cref{eq:Km11}.

Find $\vct{u}_\pmf \in H^1 ( I ; \vct{V}^\# ) $ if $\nu_\tsr{C} = 1$ or $\vct{u}_\pmf \in H^1 ( I ; \vct{V}_{>1} )$ if $\nu_\tsr{C} > 1$, as well as $p_\pm  \in H^1 ( I ; L^2 (\Omega_\pm^0 ) ) \cap L^2 ( I ; \Phi_{(-1,1)} ) $ with $p_\gamma := p_\pm \vert_\gamma$ such that
\begin{equation}
\label{eq:weak_C1_Km11}
\begin{multlined}[c][0.875\displaywidth]
\bigl\langle \phi_\pm \hat{\matr{\alpha}}_\pm , \nabla \partial_t \vct{u}_\pm \bigr\rangle_{\Omega_\pm^0  }
+  \bigl\langle \hat{\omega}_\pm \partial_t p_\pm , \phi_\pm \bigr\rangle_{\Omega_\pm^0 } 
+ \bigl\langle \hat{\matr{K}}_\pm \nabla p_\pm , \nabla \phi_\pm \bigr\rangle_{ \Omega_\pm^0  } \\
 + \bigl\langle a ( \mathfrak{A}_{\!\vct{N}}\hat{\omega}_\rmf^\mathrm{eff} ) \partial_t p_\gamma , \phi_\gamma \bigr\rangle_{\gamma  } 
+ \bigl\langle \phi_\gamma \hat{\vct{\alpha}}_\rmf^\mathrm{eff} \! , \partial_{\!\vct{N}} \partial_t \vct{u}_\rmf \bigr\rangle_{\Omega_\rmf^1  }  \\
 = \bigl\langle \hat{q}_\pm , \phi_\pm \bigr\rangle_{\Omega_\pm^0  } + \bigl\langle a\mathfrak{A}_{\!\vct{N}} \hat{q}_\rmf^\mathrm{eff} \!,  \phi_\gamma \bigr\rangle_{\gamma  } 
\end{multlined}
\end{equation}
holds for all $\phi_\pm  \in \Phi_{(-1,1)}$ with $\phi_\gamma := \phi_\pm \vert_\gamma$ and almost all $t\in I$, and either the mechanics limit problem~\eqref{eq:weak_mechlimit_C1} if $\nu_\tsr{C} = 1$ or \eqref{eq:weak_mechlimit_nuCg1} if $\nu_\tsr{C} > 1$ is satisfied. 

The limit model~\eqref{eq:weak_C1_Km11} still depends on fracture displacement in the  full-di\-men\-sio\-nal fracture domain~$\Omega_\rmf^1$ but can be reduced to a full discrete fracture model if $\vct{\alpha}_\rmf^\mathrm{eff}$ is constant in normal direction as discussed in \Cref{rem:df}.
Further, we have the following convergence theorem.

\begin{theorem} \label{thm:Km1p1}
Let $\nu_\tsr{C} \ge 1$, $\nu_\matr{K} \in (-1, 1)$, and assume that the \Cref{asm:eps,asm:nu} hold.
Then,
\begin{align*}
 \hat{p}_\pm^\#  &\in L^\infty \bigl( I ; \Phi_{(-1, 1)} \bigr) \cap H^1 \bigl( I; L^2 (\Omega_\pm^0 ) \bigr) , \\
\hat{\vct{u}}_\pmf^\# &\in \begin{cases} L^\infty ( I ; \vct{V}^\# ) \cap H^1 ( I ; \vct{V}^\# )  &\text{if } \nu_\tsr{C} = 1, \\
L^\infty ( I ; \vct{V}_{>1} ) \cap H^1 ( I ; \vct{V}_{>1} )   &\text{if } \nu_\tsr{C} > 1
\end{cases}
\end{align*}%
are the unique weak solution of the limit problem in~\eqref{eq:weak_C1_Km11}, where $\hat{p}_\pmf^\#$ and $\hat{\vct{u}}_\pmf^\#$ denote the limit functions from \Cref{cor:conv}.
In particular, we have ${\hat{p}_\pm^\#} = \mathfrak{A}_{\!\vct{N}} \smash{\hat{p}_\rmf^\#}$ on $\gamma$, as well as $\smash{\hat{p}_\rmf^\#} (\vct{y} +s \vct{N} ) = ( \mathfrak{A}_{\!\vct{N}} \smash{\hat{p}_\rmf^\#} ) (\vct{y})$ for almost all $\vct{y} + s \vct{N} \in \Omega_\rmf^1$ with $\vct{y} \in \gamma$ and $s \in (-a_-(\vct{y}) , a_+(\vct{y}))$, where $\mathfrak{A}_{\!\vct{N}}$ denotes the averaging operator from \Cref{def:average}. 
 Moreover, \Cref{cor:conv,thm:strongconv} hold for the entire sequences $\{ \hat{\vct{u}}_\pmf^\epsilon \}_{\epsilon \in ( 0,\epsilon^\ast]}$ and $\{ \hat{p}_\pmf^\epsilon \}_{\epsilon \in ( 0,\epsilon^\ast]}$.
\end{theorem}
\begin{proof}
See \Cref{sec:6}.
\end{proof}

\subsubsection{Case \texorpdfstring{$\nu_\matr{K} = 1$}{nuK=1}} 
\label{sec:4.3.4} 
For $\nu_\matr{K} = 1$, the fracture becomes a permeable barrier with a pressure jump across the interface~$\gamma$. 
In general, as presented in \Cref{sec:4.3.4.1}, the limit model is a two-scale model that includes an ODE for the fracture pressure head inside the full-dimensional fracture domain~$\Omega_\rmf^1$ and depends on the displacement vector inside~$\Omega_\rmf^1$.  
In the special case that $\nu_\omega > -1$ (i.e., the fracture storage term vanishes) and either $2\nu_{\!\matr{\alpha}}^\smallperp > \nu_\tsr{C} -1$ (i.e., flow and mechanics are decoupled inside the fracture in the limit) or the effective Biot vector~$\hat{\vct{\alpha}}_\rmf^\mathrm{eff}$ and the normal hydraulic conductivity~$\smash{\hat{K}_\rmf^{\!\vct{N}}}$ are constant in normal direction, we can reduce the flow equation to an interface condition on~$\gamma$ as discussed in \Cref{sec:4.3.4.2}.
However, in the latter case, the limit model still depends on the displacement vector in~$\Omega_\rmf^1$ so that the model remains a two-scale problem.

\paragraph{Full-Dimensional Limit Model} 
\label{sec:4.3.4.1}
The limit problem for $\nu_\matr{K} = 1$ has the following strong formulation with $p_\pmf$ and $\vct{u}_\pmf$ corresponding to the limit functions~$\hat{p}_\pmf^\#$ and $\smash{\hat{\vct{u}}_\pmf^\#}$ from \Cref{cor:conv}. 

Find $p_\pmf \colon \Omega_\pmf^\odot \times I \rightarrow \mathbb{R}$ and $\vct{u}_\pmf \colon \Omega_\pmf^\odot \times I \rightarrow \mathbb{R}^n$ such that 
\begin{subequations}
\begin{align}
\hat{\omega}_\rmf^\mathrm{eff} \partial_t p_\rmf +\hat{\vct{\alpha}}_\rmf^\mathrm{eff} \!\cdot \partial_{\!\vct{N}} \partial_t \vct{u}_\rmf  - \partial_{\!\vct{N}}  \bigl( \hat{K}_\rmf^{\!\vct{N}} \partial_{\!\vct{N}} p_\rmf \bigr)  &=  q_\rmf^\mathrm{eff}  \ &&\text{on } \Omega_\rmf^1 \times I , \\
p_\rmf (\cdot , 0 ) &= \hat{p}_{0 , \rmf } \quad\enspace &&\text{on } \Omega_\rmf^1 , \\
 p_\pm &= \Pi_\pm p_\rmf \ &&\text{on } \gamma \times I , \\
 \hat{\matr{K}}_\pm \nabla p_\pm \cdot \vct{N} &= \Pi_\pm \bigl( \hat{K}_\rmf^\vct{N} \partial_{\!\vct{N}} p_\rmf \bigr) \ &&\text{on } \gamma \times I ,
\end{align}%
\label{eq:K1}%
and the bulk limit problem~\eqref{eq:bulklimit}, as well as either the mechanical limit problem \eqref{eq:mechlimit_nuC1} if $\nu_\tsr{C } = 1$ or \eqref{eq:mechlimit_nuCg1} if $\nu_\tsr{C} > 1$, are satisfied.
\end{subequations}%

A weak formulation of \cref{eq:K1} is given by the following problem.

Find $\vct{u}_\pmf \in H^1 ( I ; \vct{V}^\# ) $ if $\nu_\tsr{C} = 1$ or $\vct{u}_\pmf \in H^1 ( I ; \vct{V}_{>1} )$ if $\nu_\tsr{C} > 1$, as well as
$p_\pmf  \in L^2 ( I ; \Phi^\# )$ with $p_\pm \in H^1 ( I ; L^2 (\Omega_\pm^0 ) )$ and, if $\nu_\omega = -1$, $p_\rmf \in H^1 ( I; L^2 (\Omega_\rmf^1 ))$ such that $p_\rmf (0) = \hat{p}_{0, \rmf }$ and
\begin{align}
\label{eq:weak_C1_K1}
\begin{split}
&\bigl\langle \phi_\pm \hat{\matr{\alpha}}_\pm , \nabla \partial_t \vct{u}_\pm \bigr\rangle_{\Omega_\pm^0  }
+  \bigl\langle \hat{\omega}_\pmf^{\fracfac{\!\eff\! }} \partial_t p_\pmf , \phi_\pmf \bigr\rangle
+ \bigl\langle \hat{\matr{K}}_\pm \nabla p_\pm , \nabla \phi_\pm \bigr\rangle_{ \Omega_\pm^0  } \\
&\hspace{2cm} + \bigl\langle \phi_\rmf \hat{\vct{\alpha}}_\rmf^\mathrm{eff} , \partial_{\!\vct{N}} \partial_t \vct{u}_\rmf \bigr\rangle_{\Omega_\rmf^1  } 
+ \bigl\langle \hat{K}_\rmf^{\!\vct{N}} \partial_{\!\vct{N}} p_\rmf , \partial_{\!\vct{N}} \phi_\rmf \bigr\rangle_{\Omega_\rmf^1  } =  \bigl\langle \hat{q}_\pmf^{\fracfac{\!\eff\! }} \! ,  \phi_\pmf \bigr\rangle
\end{split}
\end{align}
holds for all $\phi_\pmf  \in \Phi^\#$ and almost all $t\in I$, and either the mechanics limit problem~\eqref{eq:weak_mechlimit_C1} if $\nu_\tsr{C} = 1$ or \eqref{eq:weak_mechlimit_nuCg1} if $\nu_\tsr{C} > 1$ is satisfied. 

We obtain the following convergence theorem.
\begin{theorem} \label{thm:Kp1}
Let $\nu_\tsr{C} \ge 1$, $\nu_\matr{K}  = 1$, and assume that the \Cref{asm:eps,asm:nu} hold. 
Then,
\begin{align*}
\hat{p}_\pmf^\#  &\in L^\infty ( I ; \Phi^\# ) \quad\text{with}\enspace\ \begin{cases} 
 \hat{p}_\pm^\# \in H^1 \bigl( I; L^2 (\Omega_\pm^0 ) \bigr) , & \\
\hat{p}_\rmf^\# \in H^1 \bigl( I; L^2 (\Omega_\rmf^1 ) \bigr) \quad\text{if } \nu_\omega = -1,
\end{cases}  \\
\hat{\vct{u}}_\pmf^\# &\in \begin{cases} L^\infty ( I ; \vct{V}^\# ) \cap H^1 ( I ; \vct{V}^\# )  &\text{if } \nu_\tsr{C} = 1, \\
L^\infty ( I ; \vct{V}_{>1} ) \cap H^1 ( I ; \vct{V}_{>1} )   &\text{if } \nu_\tsr{C} > 1
\end{cases}
\end{align*}%
are the unique weak solution of the limit problem~\eqref{eq:weak_C1_K1}, where $\hat{p}_\pmf^\#$ and $\hat{\vct{u}}_\pmf^\#$ denote the limit functions from \Cref{cor:conv}.
Moreover, \Cref{cor:conv,thm:strongconv} hold for the entire sequences $\{ \hat{\vct{u}}_\pmf^\epsilon \}_{\epsilon \in ( 0,\epsilon^\ast]}$ and $\{ \hat{p}_\pmf^\epsilon \}_{\epsilon \in ( 0,\epsilon^\ast]}$.
\end{theorem}
\begin{proof}
See \Cref{sec:6}.
\end{proof}

\paragraph{Discrete Fracture Limit Model}
\label{sec:4.3.4.2}
Let $\nu_\omega > -1$ (i.e., $\hat{\omega}_\rmf^\mathrm{eff} \equiv 0$) and assume that $\smash{\hat{\vct{\alpha}}_\rmf^\eff}$ and $\smash{\hat{K}_\rmf^{\!\vct{N}}}$ are constant in normal direction, i.e., 
\begin{align}
\hat{\vct{\alpha}}_\rmf^\eff ( \vct{y} + s \vct{N} ) &=  \hat{\vct{\alpha}}_\gamma^\eff (\vct{y} ) , \quad\enspace \hat{K}_\rmf^{\!\vct{N}} ( \vct{y} + s \vct{N}) = \hat{K}_\gamma^{\!\vct{N}} ( \vct{y} ) \label{eq:normalconst}
\end{align}
for almost all $\vct{y} + s\vct{N} \in \Omega_\rmf^1$ with $\vct{y} \in \gamma$ and $s \in ( -a_-(\vct{y} ) , a_+ (\vct{y} )) $, where $\hat{\vct{\alpha}}_\gamma^\eff := \mathfrak{A}_{\!\vct{N}} \hat{\vct{\alpha}}_\rmf^\eff$ and $\smash{\hat{K}_\gamma^{\!\vct{N}}} := \mathfrak{A}_{\!\vct{N}} \smash{\hat{K}_\rmf^{\!\vct{N}}}$.
Then, the flow equation of the full-dimensional limit problem~\eqref{eq:K1} can be reduced to an interface condition on~$\gamma$. 
For this, we define the effective fluid source term $\hat{Q}_\gamma^\mathrm{eff} \in L^2 (\gamma ) $ by 
\begin{align}
\label{eq:Qgamma} \hat{Q}_\gamma^\mathrm{eff} (\vct{y}  ) &:= \frac{1}{a(\vct{y})} \int_{-a_-(\vct{y})}^{a_+(\vct{y})} \bigl( s + a_-(\vct{y} ) \bigr) \, \hat{q}_\rmf^\mathrm{eff} ( \vct{y} + s \vct{N} ) \,\rmd s .
\end{align}
Then, the reduced reformulation of the limit problem~\eqref{eq:K1} reads as follows, where the functions~$p_\pm$ and $\vct{u}_\pmf$ correspond to the limit functions~$\smash{\hat{p}_\pm^\#}$ and $\smash{\hat{\vct{u}}_\pmf^\#}$ from \Cref{cor:conv}.

Find $p_\pm \colon \Omega_\pm^0 \times I \rightarrow \mathbb{R}$ and $\vct{u}_\pmf \colon \Omega_\pmf^\odot \times I \rightarrow \mathbb{R}^n$ such that 
\begin{subequations}
\begin{alignat}{2}
\jump{\matr{K} \nabla p \cdot \! \vct{N} } + a\mathfrak{A}_{\!\vct{N}} \hat{q}_\rmf^\eff &= \begin{cases}  \hat{\vct{\alpha}}_\gamma^\eff \cdot  \jump{\partial_t \vct{u}} &\text{if } \nu_\tsr{C} = 1 , \\ 
0 &\text{if } \nu_\tsr{C} > 1 , \end{cases} \\
\hat{\matr{K}}_+ \nabla p_+  \cdot \vct{N} + \hat{Q}_\gamma^\eff &= \frac{\hat{K}_\gamma^{\!\vct{N}}}{a} \jump{p} +  \begin{cases} \hat{\vct{\alpha}}_\gamma^\eff\! \cdot \bigl(  \partial_t \vct{u}_+ \vert_\gamma  - \mathfrak{A}_{\!\vct{N}} (\partial_t \vct{u}_\rmf ) \bigr) &\text{if } \nu_\tsr{C} = 1 , \\ 
- \hat{\vct{\alpha}}_\gamma^\eff \! \cdot \mathfrak{A}_{\!\vct{N}} (\partial_t \vct{u}_\rmf ) &\text{if } \nu_\tsr{C} > 1  \end{cases}
\end{alignat}%
\label{eq:K1_df}%
\end{subequations}%
holds on~$\gamma \times I$, 
and the bulk limit problem~\eqref{eq:bulklimit}, as well as either the mechanical limit problem \eqref{eq:mechlimit_nuC1} if $\nu_\tsr{C } = 1$ or \eqref{eq:mechlimit_nuCg1} if $\nu_\tsr{C} > 1$, are satisfied.

Next, we define the space
\begin{align}
\Phi_{\ge 1} := H^1_{0 , \rho^0_{\pm , \mathrm{D}}} \! (\Omega_\pm^0 )  . \label{eq:Phi1}
\end{align}
Then, a weak formulation of \cref{eq:K1_df} is given by the following problem.

If $\nu_\tsr{C} = 1$, find $\vct{u}_\pmf \in H^1 ( I ; \vct{V}^\# ) $  and $p_\pm  \in H^1 ( I ; L^2 (\Omega_\pm^0 ) ) \cap L^2 ( I ; \Phi_{\ge 1} )$ such that%
\begin{subequations}
\begin{equation}
\begin{multlined}[c][0.875\displaywidth]
\bigl\langle \phi_\pm \hat{\matr{\alpha}}_\pm , \nabla \partial_t \vct{u}_\pm \bigr\rangle_{\Omega_\pm^0  } 
+ \bigl\langle \hat{\omega}_\pm \partial_t p_\pm , \phi_\pm \bigr\rangle_{\Omega_\pm^0 } 
+ \bigl\langle \hat{\matr{K}}_\pm \nabla p_\pm , \nabla \phi_\pm \bigr\rangle_{ \Omega_\pm^0  }    \\ + \bigl\langle a^{-1}  \hat{K}_\gamma^{\!\vct{N}} \jump{p} , \jump{\phi} \bigr\rangle_\gamma  + \bigl\langle  \hat{\vct{\alpha}}_\gamma^\mathrm{eff} \! , \jump{\phi\, \partial_t \vct{u} } \bigr\rangle_\gamma  - 
\bigl\langle \jump{\phi}  \hat{\vct{\alpha}}_\gamma^\mathrm{eff} \!  , \mathfrak{A}_{\!\vct{N}} ( \partial_t \vct{u}_\rmf ) \bigr\rangle_\gamma 
\\ = \bigl\langle \hat{q}_\pm , \phi_\pm \bigr\rangle_{\Omega_\pm^0  } + \bigl\langle a\mathfrak{A}_{\!\vct{N}} \hat{q}_\rmf^\mathrm{eff} \! , \phi_- \bigr\rangle_\gamma + \bigl\langle \hat{Q}_\gamma^\mathrm{eff} \! , \jump{\phi} \bigr\rangle_\gamma 
\end{multlined} 
\end{equation}
holds for all $\phi_\pm \in \Phi_{\ge 1}$ and almost all $t\in I$, and the mechanics limit problem~\eqref{eq:weak_mechlimit_C1} is satisfied.

Otherwise, if $\nu_\tsr{C} > 1$, find $\vct{u}_\pmf \in H^1 ( I ; \vct{V}_{>1} )$ and  $p_\pm  \in H^1 ( I ; L^2 (\Omega_\pm^0 ) ) \cap L^2 ( I ; \Phi_{\ge 1} )$ with%
\begin{equation}
\begin{multlined}[c][0.875\displaywidth]
\bigl\langle \phi_\pm \hat{\matr{\alpha}}_\pm , \nabla \partial_t \vct{u}_\pm \bigr\rangle_{\Omega_\pm^0  } 
+ \bigl\langle \hat{\omega}_\pm \partial_t p_\pm , \phi_\pm \bigr\rangle_{\Omega_\pm^0 } 
+ \bigl\langle \hat{\matr{K}}_\pm \nabla p_\pm , \nabla \phi_\pm \bigr\rangle_{ \Omega_\pm^0  } \\
+ \bigl\langle a^{-1}  \hat{K}_\gamma^{\!\vct{N}} \jump{p} , \jump{\phi} \bigr\rangle_\gamma
  -\bigl\langle \jump{\phi} \hat{\vct{\alpha}}_\gamma^\mathrm{eff} \!  , \mathfrak{A}_{\!\vct{N}}( \partial_t \vct{u}_\rmf ) \bigr\rangle_\gamma  \\
 = \bigl\langle \hat{q}_\pm , \phi_\pm \bigr\rangle_{\Omega_\pm^0  } + \bigl\langle a\mathfrak{A}_{\!\vct{N}} \hat{q}_\rmf^\mathrm{eff} \! , \phi_- \bigr\rangle_\gamma + \bigl\langle \hat{Q}_\gamma^\mathrm{eff}\! , \jump{\phi} \bigr\rangle_\gamma 
\end{multlined} 
\end{equation}%
\label{eq:weak_C1_K1_df}%
\end{subequations}%
holds for all $\phi_\pm \in \Phi_{\ge 1}$ and almost all $t\in I$, and the mechanics limit problem~\eqref{eq:weak_mechlimit_nuCg1} is satisfied.

The reduced limit problem~\eqref{eq:weak_C1_K1_df} emerges from the full-dimensional limit problem in \cref{eq:weak_C1_K1} as detailed in the following theorem.
\begin{theorem}
Let $\nu_\omega > -1$ and assume that $\hat{\vct{\alpha}}_\rmf^\eff$ and $\hat{K}_\rmf^{\!\vct{N}}$ are constant in normal direction in the sense of \cref{eq:normalconst}.
Besides, let $p_\pmf \in L^2 ( I ; \Phi^\# ) $ with $p_\pm \in H^1 ( I ; L^2 (\Omega_\pm^0 ) )$ and, if $\nu_\omega = -1$, $p_\rmf \in H^1 ( I ; L^2 ( \Omega_\rmf^1 ) )$ be a solution of the full-di\-men\-sio\-nal limit problem~\eqref{eq:weak_C1_K1}.
Then, $p_\pm  \in H^1 ( I ; L^2 (\Omega_\pm^0 ) ) \cap L^2 ( I ; \Phi_{\ge 1} )$ solves the reduced limit problem~\eqref{eq:weak_C1_K1_df}.
\end{theorem}
\begin{proof}
Let $\phi_\pm \in \Phi_{\ge 1}$. 
We define $\phi_\rmf \in H^1_{\!\vct{N}}(\Omega_\rmf^1 ) $ by 
\begin{align*}
\phi_\rmf ( \vct{y} +s \vct{N} ) := \phi_- \bigr\vert_\gamma ( \vct{y} ) + \frac{s + a_-(\vct{y})}{a ( \vct{y}) } \jump{\phi} ( \vct{y} ) 
\end{align*}
for $\vct{y} \in \gamma$ and $s \in (a_-(\vct{y} ) , a_+ (\vct{y}))$
such that $\phi_\pmf \in \Phi^\#$ with 
\begin{align*}
\partial_{\!\vct{N}} \phi_\rmf (\vct{y} + s \vct{N} )  =  a^{-1} (\vct{y} ) \jump{\phi} ( \vct{y}) .
\end{align*}
By inserting $\phi_\pmf \in \Phi^\#$ as test function in \cref{eq:weak_C1_K1} and using the relations
\begin{align*}
\bigl\langle \phi_\rmf \hat{\vct{\alpha}}_\rmf^\mathrm{eff} , \partial_{\!\vct{N}} \partial_t \vct{u}_\rmf \bigr\rangle_{\Omega_\rmf^1  } &=  \begin{cases}
\bigl\langle  \hat{\vct{\alpha}}_\gamma^\mathrm{eff} , \jump{\phi\, \partial_t \vct{u} } \bigr\rangle_\gamma - 
\bigl\langle \jump{\phi}  \hat{\vct{\alpha}}_\gamma^\mathrm{eff}  , \mathfrak{A}_{\!\vct{N}} ( \partial_t \vct{u}_\rmf ) \bigr\rangle_\gamma  &\text{if } \nu_\tsr{C} = 1 , \\
-\bigl\langle \jump{\phi} \hat{\vct{\alpha}}_\gamma^\mathrm{eff}  , \mathfrak{A}_{\!\vct{N}} ( \partial_t \vct{u}_\rmf ) \bigr\rangle_\gamma  &\text{if } \nu_\tsr{C} > 1, 
\end{cases}  \\
\bigl\langle \hat{K}_\rmf^\vct{N} \partial_{\!\vct{N}} p_\rmf , \partial_{\!\vct{N}} \phi_\rmf \bigr\rangle_{\Omega_\rmf^1  } &= \bigl\langle a^{-1}  \hat{K}_\gamma^{\!\vct{N}} \jump{p} , \jump{\phi} \bigr\rangle_\gamma  ,
\end{align*}
we find that $p_\pm$ satisfies \cref{eq:weak_C1_K1_df}.
In addition, we have $p_\pm \in L^2 (I ; \Phi_{\ge 1})$ since the map
\begin{align*}
\Phi^\# \rightarrow \Phi_{\ge 1} ,\enspace \phi_\pmf \mapsto \phi_\pm
\end{align*}
clearly defines a continuous embedding.
\end{proof}

\begin{remark} \label{rem:K1}
\begin{enumerate}[label=(\roman*)]
\item Except for the elimination of the fracture pressure head, the reduced limit problem~\eqref{eq:weak_C1_K1_df} is equivalent to the two-scale limit problem~\eqref{eq:weak_C1_K1}. 
When solving the reduced problem~\eqref{eq:weak_C1_K1_df}, the fracture pressure head can be obtained a~posteriori by solving \cref{eq:weak_C1_K1} with the displacement vectors~$\vct{u}_\pmf$ and bulk pressure heads~$p_\pm$  already given.
\item In general, if $2\nu_{\!\matr{\alpha}}^\smallperp = \nu_\tsr{C} - 1$ and hence $\hat{\vct{\alpha}}_\rmf^\eff \not= \vct{0}$, the reduced limit problem~\eqref{eq:weak_C1_K1_df} depends on the normal average of the fracture displacement vector~$\vct{u}_\rmf$ and requires the solution of a two-scale problem with full-dimensional fracture~$\Omega_\rmf^1$ for the mechanics equation.  
\item If $2\nu_{\!\matr{\alpha}}^\smallperp > \nu_\tsr{C} - 1$ (i.e., $\hat{\vct{\alpha}}_\rmf^\eff \equiv \vct{0}$), we can drop the assumption that $\smash{\hat{K}_\rmf^{\!\vct{N}}}$ is constant in normal direction for the reduced reformulation of \cref{eq:weak_C1_K1}. 
In this case, the effective normal conductivity~$\hat{K}_\gamma^{\!\vct{N}} \in L^\infty ( \gamma ) $ and the effective fluid source~$\smash{\hat{Q}_\rmf^\eff} \in L^2 (\gamma ) $ are given by%
\end{enumerate}
\vspace*{-9pt}
\begin{subequations}
\begin{align}
\hat{K}_\gamma^{\!\vct{N}} (\vct{y} )  &:= \biggl( \frac{1}{a(\vct{y})}\int_{-a_-(\vct{y} ) }^{a_+(\vct{y} ) } \bigl[ \hat{K}_\rmf^{\!\vct{N}} (\vct{y} +s \vct{N}) \bigr]^{-1} \,\rmd s \biggr)^{-1} , \\
\hat{Q}_\gamma^\eff (\vct{y}) &:= \frac{\hat{K}_\gamma^{\!\vct{N}} ( \vct{y})}{a(\vct{y})} \int_{-a_-(\vct{y} ) }^{a_+ ( \vct{y} ) } \hat{q}_\rmf^\eff (\vct{y} + s \vct{N} ) \int_{-a_-(\vct{y} ) }^{s } \bigl[ \hat{K}_\rmf^{\!\vct{N}} (\vct{y} + \bar{s} \vct{N}) \bigr]^{-1} \rmd \bar{s} \,\rmd s   .
\end{align}
\end{subequations}
\begin{enumerate}[resume]
For details, we refer to \cite[\S 4.4]{hoerl24}.
\end{enumerate}
\end{remark}

\subsubsection{Case \texorpdfstring{$\nu_\matr{K} > 1$}{nuK>1}} 
\label{sec:4.3.5} 
For $\nu_\matr{K} > 1$, the fracture becomes an impermeable wall with no flow from either bulk domains across the fracture interface~$\gamma$. 
For $\nu_\omega = -1$, we still get a flow equation inside the fracture that is not directly coupled to the bulk flow, but depends on the displacement vector inside the fracture. In contrast, for $\nu_\omega > -1$, we do not obtain convergence for the pressure head inside the fracture.
The strong formulation of the limit problem now reads as follows, with $p_{\pm (\rmf )}$ and $\vct{u}_\pmf$ corresponding to the limit functions~$\smash{\hat{p}_{\pm (\rmf )}^\#}$ and~$\smash{\hat{\vct{u}}_{\pmf}^\#}$ in \Cref{cor:conv}.

Find $p_\pm \colon \Omega_\pm^0 \times I \rightarrow \mathbb{R}$ and $\vct{u}_\pmf \colon \Omega_\pmf^\odot \times I \rightarrow \mathbb{R}^n$ such that 
\begin{subequations}
\begin{align}
 \hat{\matr{K}}_\pm \nabla p_\pm \cdot \vct{N} &= 0 \quad\enspace \text{on } \gamma \times I ,
\end{align}
and the bulk limit problem~\eqref{eq:bulklimit}, as well as either the mechanical limit problem \eqref{eq:mechlimit_nuC1} if $\nu_\tsr{C } = 1$ or \eqref{eq:mechlimit_nuCg1} if $\nu_\tsr{C} > 1$, are satisfied. Moreover, if $\nu_\omega = -1$, find $p_\rmf \colon \Omega_\rmf^1 \times I \rightarrow \mathbb{R}$ such that
\begin{alignat}{2}
\label{eq:strong_C1Kg1_c} \hat{\omega}_\rmf \partial_t p_\rmf + \hat{\vct{\alpha}}_\rmf^\mathrm{eff} \!\cdot \partial_{\!\vct{N}} \partial_t \vct{u}_\rmf  &= \hat{q}_\rmf^\mathrm{eff} \quad\enspace &&\text{in } \Omega_\rmf^1 \times I , \\
\label{eq:strong_C1Kg1_d} p_\rmf ( \cdot , 0 ) &= \hat{p}_{0, \rmf } \quad\enspace &&\text{on } \Omega_\rmf^1 .
\end{alignat}%
\label{eq:strong_C1Kg1}
\end{subequations}%
\indent A weak formulation of \cref{eq:strong_C1Kg1} is given by the following problem with the space $\Phi_{\ge 1}$ given by \cref{eq:Phi1}.

Find $\vct{u}_\pmf \in H^1 ( I ; \vct{V}^\# ) $ if $\nu_\tsr{C} = 1$ or $\vct{u}_\pmf \in H^1 ( I ; \vct{V}_{>1} )$ if $\nu_\tsr{C} > 1$, as well as $p_\pm \in H^1 ( I ; L^2 (\Omega_\pm^0 ) ) \cap L^2 ( I ; \Phi_{\ge 1})$ such that%
\begin{equation}
\label{eq:weak_C1_Kg1}
\begin{multlined}[c][0.875\displaywidth]
\bigl\langle \phi_\pm \hat{\matr{\alpha}}_\pm , \nabla \partial_t \vct{u}_\pm \bigr\rangle_{\Omega_\pm^0  } 
+ \bigl\langle \hat{\omega}_\pm \partial_t p_\pm , \phi_\pm \bigr\rangle_{\Omega_\pm^0 } 
+ \bigl\langle \hat{\matr{K}}_\pm \nabla p_\pm , \nabla \phi_\pm \bigr\rangle_{ \Omega_\pm^0  } 
 = \bigl\langle \hat{q}_\pm , \phi_\pm \bigr\rangle_{\Omega_\pm^0  } 
\end{multlined}
\end{equation}
holds for all $\phi_\pm  \in \Phi_{\ge 1}$ and almost all $t\in I$, and either the mechanics limit problem~\eqref{eq:weak_mechlimit_C1} if $\nu_\tsr{C} = 1$ or \eqref{eq:weak_mechlimit_nuCg1} if $\nu_\tsr{C} > 1$ is satisfied. 
Moreover, if $\nu_\omega = -1$, find $p_\rmf \in H^1 (I ; L^2 (\Omega_\rmf^1 ) )$ with $p_\rmf ( 0 ) = \hat{p}_{0,\rmf}$ such that \cref{eq:strong_C1Kg1_c} is satisfied almost everywhere.

We now obtain the following convergence theorem.
\begin{theorem}  \label{thm:Kgp1}
Let $\nu_\tsr{C} \ge 1$, $\nu_\matr{K} > 1$, and assume that the \Cref{asm:eps,asm:nu} hold.  
Besides, if $\nu_\omega > - 1$, let additionally $2 \nu_{\!\matr{\alpha}}^\smallpar > \nu_\tsr{C } + \min\{ \nu_\omega - 2, \nu_\matr{K} - 4 \}$ and $2 \nu_{\!\matr{\alpha}}^\smallperp \ge \nu_\tsr{C} +  \min\{ \nu_\omega , \nu_\matr{K} - 2 \}$.
Then,
\begin{align*}
\hat{p}_\pm^\#  &\in L^\infty ( I ; \Phi_{\ge 1} ) \cap  H^1 \bigl( I ; L^2 (\Omega_\pm^0 ) \bigr) , \\
\hat{p}^\#_\rmf &\in L^\infty \bigl( I ; L^2 ( \Omega_\rmf^1 ) \bigr) \cap H^1 \bigl( I ; L^2 (\Omega_\rmf^1 ) \bigr) \quad\text{if } \nu_\omega = -1 , \\
\hat{\vct{u}}_\pmf^\# &\in \begin{cases} L^\infty ( I ; \vct{V}^\# ) \cap H^1 ( I ; \vct{V}^\# )  &\text{if } \nu_\tsr{C} = 1, \\
L^\infty ( I ; \vct{V}_{>1} ) \cap H^1 ( I ; \vct{V}_{>1} )   &\text{if } \nu_\tsr{C} > 1
\end{cases}
\end{align*}%
are the unique weak solution of the limit problem in \cref{eq:weak_C1_Kg1}, where $\hat{p}_\pmf^\#$ and $\hat{\vct{u}}_\pmf^\#$ denote the limit functions from \Cref{cor:conv}. 
Moreover, \Cref{cor:conv,thm:strongconv} hold for the entire sequences $\{ \hat{\vct{u}}_\pmf^\epsilon \}_{\epsilon \in ( 0,\epsilon^\ast]}$ and $\{ \hat{p}_\pmf^\epsilon \}_{\epsilon \in ( 0,\epsilon^\ast]}$.
\end{theorem}
\begin{proof}
See \Cref{sec:6}.
\end{proof}

\section{Proof of \texorpdfstring{\Cref{thm:apriori}}{Theorem 3.4}} \label{sec:5}
Given the \Cref{asm:eps,asm:nu}, we subsequently prove the a~priori estimates in \Cref{thm:apriori} for the solutions of the transformed Biot problem~\eqref{eq:fulldim-weak-trafo}.
We begin by stating estimates for transformed gradients and strain tensors in \cref{eq:gradeps,eq:Eeps}.
\begin{lemma} \label{lem:Meps}
Let $\epsilon > 0$ and $\phi_\pm \in H^1 (\Omega_\pm^0 )$, $\vct{v}_\pm \in \vct{H}^1 (\Omega_\pm^0 )$.
Then, for $i \in \{ + , - \}$, we have 
\begin{subequations}
\begin{align}
\label{eq:ineqMeps1} \norm[\big ]{\matr{M}_i^\epsilon \nabla \phi_i }_{\vct{L}^2 (\Omega_i^0 ) } &\le \bigl( 1 + \epsilon \norm{\nablapar a_i}_{\vct{L}^\infty ( \gamma ) }\bigr) \norm[\big ]{ \nabla \phi_i }_{\vct{L}^2 (\Omega_i^0 ) } , \\
\label{eq:ineqMeps2}  \norm[\big ]{ \nabla \phi_i }_{\vct{L}^2 (\Omega_i^0 ) }  &\le \bigl( 1 + \epsilon \norm{\nablapar a_i}_{\vct{L}^\infty ( \gamma ) }\bigr) \norm[\big ]{\matr{M}_i^\epsilon \nabla \phi_i }_{\vct{L}^2 (\Omega_i^0 ) }, \\
\label{eq:ineqEeps1} \norm[\big ]{\matr{E}_i^\epsilon ( \vct{v}_i ) }_{\matr{L}^2 (\Omega_i^0 )} &\le  \norm{\matr{e} ( \vct{v}_i ) }_{\matr{L}^2 (\Omega_i^0 ) } + \epsilon \norm{\nablapar a_i}_{\vct{L}^\infty ( \gamma ) } \norm{\nabla \vct{v}_i}_{\matr{L}^2 (\Omega_i^0 ) } , \\
\label{eq:ineqEeps2} \norm{\matr{e} ( \vct{v}_i ) }_{\matr{L}^2 (\Omega_i^0 ) } & \le \norm[\big ]{\matr{E}_i^\epsilon ( \vct{v}_i ) }_{\matr{L}^2 (\Omega_i^0 )} + \epsilon \norm{\nablapar a_i}_{\vct{L}^\infty ( \gamma ) } \norm{\nabla \vct{v}_i [ \matr{M}_i^\epsilon ]^\rmt }_{\matr{L}^2 (\Omega_i^0 ) } .
\end{align}
\end{subequations}
\end{lemma}
\begin{proof}
We prove the results for $i=+$; the case $i=-$ follows analogously. \cref{eq:ineqMeps1} follows directly from the definition of~$\matr{M}_+^\epsilon$ in \cref{eq:Meps}. \cref{eq:ineqMeps2} follows from
\begin{align*}
\nabla \phi_+ (\vct{x} )  = \bigl[ \matr{I}_n + \epsilon \nablapar a_+ \bigl(\vct{x} - (\vct{x} \cdot \vct{N} ) \vct{N} \bigr) \vct{N}^\rmt \bigr] \matr{M}_+^\epsilon (\vct{x} ) \nabla \phi_+ (\vct{x} ) .
\end{align*}
Moreover, \cref{eq:ineqEeps1} is a consequence of 
\begin{align*}
\matr{E}_+^\epsilon ( \vct{v}_+) &= \matr{e} ( \vct{v}_+ )  - \frac{ \epsilon}{2} \Bigl( [\nabla \vct{v}_+ ] \vct{N} [ \nablapar a_+ ]^\rmt + [\nablapar a_+ ] \vct{N}^\rmt [ \nabla \vct{v}_+ ]^\rmt \Bigr) .
\end{align*}
Finally, \cref{eq:ineqEeps2} follows from 
\begin{align*}
\matr{e} ( \vct{v}_+ )  = \matr{E}_+^\epsilon (\vct{v}_+ )  + \frac{\epsilon}{2} \Bigl( \nabla \vct{v}_+ [ \matr{M}_+^\epsilon ]^\rmt \vct{N} [ \nablapar a_+ ]^\rmt + [\nablapar a_+ ] \vct{N}^\rmt \matr{M}_+^\epsilon [ \nabla v_+]^\rmt  \Bigr)  .
\end{align*}
\end{proof}
\begin{remark} \label{rem:Meps}
Let $0 < \epsilon \le \epsilon^\ast := \min \bigl\{ 1 , \tfrac{1}{2} C^\mathrm{K} [ C_a (1 +C_a ) ]^{-1} \bigr\}$, where 
\begin{align*}
C_a := \max\nolimits_{i\in \{ + , - \}} \bigl\{ \norm{\nablapar a_i }_{\vct{L}^\infty (\gamma ) } \bigr\} \quad\ \text{and} \quad\ C^\mathrm{K} := \min\nolimits_{i \in \{ + , - \} } \bigl\{ C^\mathrm{K}_i \bigr\} ,
\end{align*}
where $C^\mathrm{K}_i > 0$ denotes the constant in Korn's inequality 
\begin{align}
\norm{\matr{e} (\vct{v}_i ) }_{\matr{L}^2 (\Omega_i^0 ) } \ge C_i^\mathrm{K} \norm{\nabla \vct{v}_i}_{\matr{L}^2 (\Omega_i^0 ) } \label{eq:CK}
\end{align}
that holds for all $\vct{v}_i \in \vct{H}^1_{0, \rho_i^0 } (\Omega_i^0 ) $, $i\in \{ + ,-\}$. 
Then, for all $\phi_i \in H^1 (\Omega_\pm^0 )$ and $\vct{v}_i \in \vct{H}^1_{0,\rho_i^0 } (\Omega_i^0 )$,  $i \in \{ + , - \}$, we have 
\begin{subequations}
\begin{alignat}{2}
\norm[\big ]{ \nabla \phi_i }_{\vct{L}^2 (\Omega_i^0 ) } &\lesssim\ &\norm[\big ]{\matr{M}_i^\epsilon \nabla \phi_i }_{\vct{L}^2 (\Omega_i^0 ) } \: &\lesssim  \norm[\big ]{ \nabla \phi_i }_{\vct{L}^2 (\Omega_i^0 ) } , \\
\norm[\big ]{ \nabla \vct{v}_i }_{\matr{L}^2 (\Omega_i^0 ) } &\lesssim\ &\norm[\big ]{\nabla \vct{v}_i [\matr{M}_i^\epsilon  ]^\rmt }_{\matr{L}^2 (\Omega_i^0 ) } &\lesssim  \norm[\big ]{ \nabla \vct{v}_i }_{\matr{L}^2 (\Omega_i^0 ) } , \\
\label{eq:remE} \norm{\matr{e} ( \vct{v}_i ) }_{\matr{L}^2 (\Omega_i^0 ) } &\lesssim &\norm[\big ]{\matr{E}_i^\epsilon ( \vct{v}_i ) }_{\matr{L}^2 (\Omega_i^0 )} \enspace &\lesssim \norm{\matr{e} ( \vct{v}_i ) }_{\matr{L}^2 (\Omega_i^0 ) } 
\end{alignat}
\end{subequations}
as a consequence of \Cref{lem:Meps}.
We remark that the condition $2\epsilon \le C^\mathrm{K} [ C_a ( 1+ C_a)]^{-1}$ is only required for the first inequality in \cref{eq:remE}.
\end{remark}

Moreover, we require \enquote{domain-decomposed} versions of Poincaré's and Korn's inequality which fit the geometric setting after the coordinate transformation~\eqref{eq:trafo} with scaling factors that are only present inside the fracture domain~$\Omega_\rmf^1$ and not inside the bulk domains~$\Omega_\pm^0$. 
First, we have the following Poincaré inequalities.
\begin{lemma} 
\label{lem:poincare_decomp}
Let $\alpha \ge 0$ and $\epsilon > 0$. Then, for all $\phi_\pmf \in \Phi$, we have
\begin{subequations}
\begin{align}
\norm[\big ]{\fracfac{\epsilon^{\alpha} }\phi_\pmf}_\Lambda^2 \lesssim \norm[\big ]{\fracfac{\epsilon^{\frac{1}{2}\! } } \nabla^{\fracfac{\!\epsilon\! }} \phi_\pmf}_{\vct{\Lambda}}^2 .
\end{align}
Moreover, if both bulk domains~$\Omega_\pm^0$ have a boundary section with Dirichlet conditions, i.e., $\abs{\rho^0_{+, \rmD}} > 0$ and $\abs{\rho^0_{-, \rmD}} > 0$, we have
\begin{align}
\norm[\big ]{\fracfac{\epsilon^{\alpha} }\phi_\pmf}_\Lambda^2 \lesssim \sum\nolimits_{i = \pm } \norm{\nabla \phi_i}_{\vct{L}^2 (\Omega_i^0)}^2 + \norm{  \epsilon^\alpha \nabla_{\!\vct{N}} \phi_\rmf}_{\vct{L}^2 (\Omega_\rmf^1 ) }^2  \le \norm[\big ]{\fracfac{\epsilon^{\alpha + 1} } \nabla^{\fracfac{\!\epsilon\! }} \phi_\pmf}_{\vct{\Lambda}}^2 .
\end{align}
In addition, for all $\vct{v}_\pmf \in \vct{V}$, it is
\begin{align}
\norm[\big ]{\fracfac{\epsilon^{\alpha} }\vct{v}_\pmf}_{\vct{\Lambda}}^2 \lesssim \sum\nolimits_{i = \pm } \norm{\nabla \vct{v}_i}_{\matr{L}^2 (\Omega_i^0)}^2 + \norm{  \epsilon^\alpha \nabla_{\!\vct{N}} \vct{v}_\rmf}_{\matr{L}^2 (\Omega_\rmf^1 ) }^2  \le \norm[\big ]{\fracfac{\epsilon^{\alpha + 1} } \nabla^{\fracfac{\!\epsilon\! }} \vct{v}_\pmf}_{\matr{\Lambda}}^2 .
\end{align}
\end{subequations}
\end{lemma}
\begin{proof}
See \cite[Lem.\ 3.6]{hoerl24}. 
\end{proof}
Further, the following Korn inequality holds true.  
\begin{lemma}
\label{lem:korn_decomp}
Let $\alpha \ge \frac{1}{2}$ and $\epsilon \in ( 0, \epsilon^\ast ] $. 
Then, for all $\vct{v}_\pmf \in \vct{V}$, we have 
\begin{align}
\norm[\big ]{ \fracfac{\epsilon^\alpha} \nablaeps \vct{v}_\pmf}^2_{\matr{\Lambda}} \lesssim \norm[\big ]{\fracfac{\epsilon^\alpha} \matr{e}^{\fraceps\! } ( \vct{v}_\pmf ) }^2_{\matr{\Lambda}} . 
\end{align}
\end{lemma}
\begin{proof}
Let $\epsilon > 0$ and $\vct{v}_\pmf^\epsilon \in \vct{V}^\epsilon$. 
Further, define~$\vct{v}^\epsilon \in \vct{H}^1_0 (\Omega^\epsilon )$ so that $\vct{v}^\epsilon \vert_{\Omega_i^\epsilon } := \vct{v}_i^\epsilon $ for $i \in \{ + , - , \rmf\} $. 
Then, we can apply Korn's inequality in~$\Omega_\pm^\epsilon$ and~$\Omega^\epsilon$, i.e.,
\begin{align*}
\norm[\big ]{\nabla \vct{v}_\pm^\epsilon}_{\matr{L}^2 (\Omega_\pm^\epsilon )} \lesssim \norm[\big ]{\matr{e}( \vct{v}_\pm^\epsilon ) }_{\matr{L}^2 (\Omega_\pm^\epsilon )}  \quad\ \text{and}\quad\ \norm{\nabla \vct{v}^\epsilon }_{\matr{L}^2 (\Omega^\epsilon )} \lesssim \norm{\matr{e} (\vct{v}^\epsilon )}_{\matr{L}^2 (\Omega^\epsilon )} ,
\end{align*}
where the implicit constants are independent of~$\epsilon$. 
For the second inequality, this follows directly from the fact that $\vct{v}^\epsilon$ satisfies $\vct{v}^\epsilon \vert_{\partial \Omega } = \vct{0}$ on the whole boundary~$\partial\Omega^\epsilon$.
For the first inequality, we invoke \Cref{lem:Meps}. 
Indeed, in $\Omega_+^\epsilon$ (and, analogously, in $\Omega_-^\epsilon$), we have 
\begin{multline*}
\frac{C_+^\mathrm{K}}{1 + \epsilon C_+^a} \norm[\big]{\nabla \hat{\vct{v}}_+^\epsilon [ \matr{M}_+^\epsilon ]^\rmt }_{\matr{L}^2 (\Omega_+^0 ) } \le C_+^\mathrm{K} \norm{\hat{\vct{v}}_+^\epsilon }_{\matr{L}^2 (\Omega_+^0 ) } \\
\le \norm{\matr{e}( \hat{\vct{v}}_+^\epsilon ) }_{\matr{L}^2 (\Omega_+^0 ) }
\le \norm{\matr{E}_+^\epsilon ( \hat{\vct{v}}_+^\epsilon )}_{\matr{L}^2 (\Omega_+^0 ) }  + \epsilon C_+^a \norm{\hat{\vct{v}}_+^\epsilon [ \matr{M}_+^\epsilon ]^\rmt }_{\matr{L}^2 ( \Omega_+^\epsilon ) }
\end{multline*}
with $C_+^\mathrm{K} > 0$ as defined in \cref{eq:CK} and $\hat{\vct{v}}_+^\epsilon := \vct{v}_+^\epsilon \circ (\vct{T}_+^\epsilon)^{-1}$, as well as $C_+^a := \norm{\nablapar a_+}_{\vct{L}^\infty (\gamma ) }$. 
Consequently, it is 
\begin{align*}
\frac{1}{2} \frac{C_+^\mathrm{K}}{1 + C_+^a}\norm{\nabla \vct{v}_+^\epsilon }_{\matr{L}^2 ( \Omega_+^\epsilon ) } \le \biggl( \frac{C_+^\mathrm{K}}{1 + \epsilon C_+^a} - \epsilon C_+^a \biggr) \norm{\nabla \vct{v}_+^\epsilon }_{\matr{L}^2 ( \Omega_+^\epsilon ) } \le \norm{\matr{e}  ( \vct{v}_+^\epsilon ) }_{\matr{L}^2 ( \Omega_+^\epsilon ) } ,
\end{align*}
where we have used that $2 \epsilon \le C_+^\mathrm{K} [ C_+^a ( 1+ C_+^a) ]^{-1} $ (see \Cref{rem:Meps}).
Thus, we find
\begin{align*}
\norm[\big ]{ \fracfac{\epsilon^{ \alpha - \frac{1}{2}}} \nabla \vct{v}_\pmf^\epsilon }^2_{\matr{\Lambda}^{\!\epsilon} } &\le \norm[\big ]{\nabla \vct{v}_\pm^\epsilon }^2_{\matr{L}^2 (\Omega_\pm^\epsilon ) } + \epsilon^{2\alpha - 1} \norm{\nabla \vct{v}^\epsilon }_{\matr{L}^2 (\Omega^\epsilon ) }^2 \\
&\lesssim \norm[\big ]{\matr{e} ( \vct{v}_\pm^\epsilon ) }^2_{\matr{L}^2 (\Omega_\pm^\epsilon ) } + \epsilon^{2\alpha - 1} \norm{\matr{e} (\vct{v}^\epsilon ) }^2_{\matr{L}^2 ( \Omega^\epsilon ) } \lesssim \norm[\big ]{\fracfac{\epsilon^{ \alpha - \frac{1}{2}}} \matr{e} ( \vct{v}_\pmf^\epsilon ) }^2_{\matr{\Lambda}^{\!\epsilon }} .
\end{align*}
The assertion now follows by applying the coordinate transformation~\eqref{eq:trafo} and using \Cref{rem:Meps}. 
\end{proof}

Next, we derive the following estimate for the initial displacement~$\hat{\vct{u}}_{0 , \pmf}^\epsilon \in \vct{V}$.
\begin{lemma}
\label{lem:apriori_u0} 
Let $\iota ( \nu ) = \frac{1}{2} ( \nu + 1 )$. 
Given \Cref{asm:eps}, as well as \ref{asm:nuf} and~\ref{asm:nualpha} in \Cref{asm:nu}, there exists $\epsilon^\ast > 0$ such that, for all $\epsilon \in ( 0, \epsilon^\ast]$, the initial displacement vector $\hat{\vct{u}}_{0,\pmf}^\epsilon \in \vct{V}$ satisfies  
\begin{align}
\norm[\big ]{\fracfac[\big ]{\epsilon^{\iota ( \nu_{\tsr{C}} ) } } \matr{e}^\fraceps ( \hat{\vct{u}}_{0 , \pmf}^\epsilon ) }_{\matr{\Lambda}} \lesssim 1.
\end{align}
\end{lemma}
\begin{proof}
We choose $\vct{v}_\pmf = \hat{\vct{u}}_{0 , \pmf}^\epsilon$ in \cref{eq:fulldim-weak-trafo-d} and obtain
\begin{multline*}
\bigl\langle \fracfac{\epsilon^{\nu_\tsr{C} + 1}} \hat{\tsr{C}}_\pmf^\epsilon \matr{E}_\pmf^\epsilon ( \hat{\vct{u}}_{0,\pmf}^\epsilon ) , \matr{E}_\pmf^\epsilon (  \hat{\vct{u}}_{0,\pmf}^\epsilon  )  \bigr\rangle  = \bigl\langle \fracfac{\epsilon^{\!\matr{\nu}_{\!\matr{\alpha}} + \matr{I}_n\!  } } \hat{p}_{0,\pmf}^\epsilon \hat{\matr{\alpha}}_\pmf^\epsilon  , \nablaeps  \hat{\vct{u}}_{0,\pmf}^\epsilon   \bigl[ \matr{M}_\pmf^\epsilon \bigr]^\rmt \bigr\rangle  \\
+  \bigl\langle \fracfac{\epsilon^{\nu_{\!\vct{f}} + 1} }\hat{\vct{f}}_\pmf^\epsilon ( 0) ,  \hat{\vct{u}}_{0,\pmf}^\epsilon  \bigr\rangle + \bigl\langle  \fracfac{\epsilon^{\!\matr{\nu}_{\!\matr{\alpha}} + \matr{I}_n\!  } } \hat{G}_\pmf^\epsilon \hat{\matr{\alpha}}_\pmf^\epsilon , \nablaeps  \hat{\vct{u}}_{0,\pmf}^\epsilon  \bigl[ \matr{M}_\pmf^\epsilon \bigr]^\rmt \bigr\rangle . 
\end{multline*}
Using \ref{asm:alpha_eps}, \ref{asm:f_eps}, \ref{asm:p0_eps}, and~\ref{asm:C_eps} in \Cref{asm:eps} and \Cref{rem:Meps}, we find
\begin{align*}
\norm[\big ]{\fracfac[\big]{\epsilon^{\iota ( \nu_{\tsr{C}} )} } \matr{e}^\fraceps ( \hat{\vct{u}}_{0 , \pmf}^\epsilon ) }_{\matr{\Lambda}}^2 \lesssim  \norm[\big ]{\fracfac[\big]{\epsilon^{\nu_{\!\vct{f}} +1 }} \hat{\vct{u}}_{0,\pmf}^\epsilon }_{\vct{\Lambda}} + \norm[\big ]{ \nabla^\fraceps \hat{\vct{u}}_{0 , \pmf}^\epsilon  \fracfac[\big ]{\epsilon^{\!\matr{\nu}_{\! {\matr{\alpha}}} + \matr{I}_n\! }}  }_{\matr{\Lambda}} .
\end{align*}
Applying the \Cref{lem:poincare_decomp,lem:korn_decomp} and \ref{asm:nuf} in \Cref{asm:nu} yields
\begin{align*}
\norm[\big ]{\fracfac[\big]{\epsilon^{\nu_{\!\vct{f}} +1 }} \hat{\vct{u}}_{0,\pmf}^\epsilon }_{\vct{\Lambda}} \lesssim \norm[\big ]{\fracfac[\big]{\epsilon^{\nu_{\!\vct{f}} +2 }} \nabla^\fraceps \hat{\vct{u}}_{0,\pmf}^\epsilon }_{\vct{\Lambda}} \lesssim \norm[\big ]{\fracfac[\big]{\epsilon^{\iota( \nu_{\tsr{C}} ) }} \matr{e}^\fraceps ( \hat{\vct{u}}_{0,\pmf}^\epsilon ) }_{\vct{\Lambda}} .
\end{align*}
Besides, by using the identity 
\begin{align*}
\norm[\big ]{ \nabla^\epsilon \hat{\vct{u}}_{0 , \rmf}^\epsilon \,\epsilon^{\!\matr{\nu}_{\! {\matr{\alpha}}} + \matr{I}_n} }_{\matr{L}^2 (\Omega_\rmf^1 ) }^2  &= \norm[\big ]{\epsilon^{\nu_{\! {\matr{\alpha}}}^\smallpar + 1} \nabla_{\!\smallpar} \hat{\vct{u}}_{0 , \rmf}^\epsilon }_{\matr{L}^2 (\Omega_\rmf^1 ) }^2 + \norm[\big ]{\epsilon^{\nu_{\! {\matr{\alpha}}}^\smallperp } \nabla_{\!\vct{N}} \hat{\vct{u}}_{0 , \rmf}^\epsilon }_{\matr{L}^2 (\Omega_\rmf^1 ) }^2 
\end{align*}
and applying \Cref{lem:korn_decomp} and \ref{asm:nualpha} in \Cref{asm:nu}, we find
\begin{align*}
\norm[\big ]{ \nabla^\fraceps \hat{\vct{u}}_{0 , \pmf}^\epsilon \fracfac[\big ]{\epsilon^{\!\matr{\nu}_{\! {\matr{\alpha}}} + \matr{I}_n\! }}  }_{\matr{\Lambda}} \lesssim \norm[\big ]{\fracfac[\big ]{\epsilon^{ \min\{  \nu_{\!\matr{\alpha}}^\smallpar + 1, \nu_{\!\matr{\alpha}}^\smallperp  +1  \} \! }} \nabla^\fraceps \hat{\vct{u}}_{0 , \pmf}^\epsilon }_{\matr{\Lambda}} \lesssim \norm[\big ]{\fracfac[\big]{\epsilon^{\iota (\nu_{\tsr{C}} )  }} \matr{e}^\fraceps ( \hat{\vct{u}}_{0,\pmf}^\epsilon ) }_{\vct{\Lambda}} \! . \tag*{\qedhere}
\end{align*}
\end{proof}

Using the \Cref{lem:poincare_decomp,lem:korn_decomp,lem:Meps,lem:apriori_u0} and the \Cref{asm:eps,asm:nu}, we can now prove \Cref{thm:apriori}.
\begin{proof}[Proof of \Cref{thm:apriori}]
The proof is divided into two steps. 
In the first step, we establish uniform-in-time bounds for the displacement vector and its gradient, as well as the pressure, together with an $L^2$-in-time bound for the pressure gradient. 
Building on these estimates, the second step yields higher regularity, in particular $L^2$-in-time bounds for the time derivatives and an improved uniform-in-time bound for the pressure gradient.

\emph{Step 1.} 
We choose $\vct{v}_\pmf = \partial_t \hat{\vct{u}}_\pmf^\epsilon $ and $\phi_\pmf = \hat{p}_\pmf^\epsilon$ in \cref{eq:fulldim-weak-trafo} and find%
\begin{align}
\begin{split}
&\frac{\rmd }{\rmd t } \Bigl( \bigl\langle \fracfac{\epsilon^{\nu_\tsr{C} + 1}} \hat{\tsr{C}}_\pmf^\epsilon \matr{E}_\pmf^\epsilon ( \hat{\vct{u}}_\pmf^\epsilon ) , \matr{E}_\pmf^\epsilon ( \hat{\vct{u}}_\pmf^\epsilon )  \bigr\rangle   +\bigl\langle \fracfac{\epsilon^{\nu_\omega +1 }}  \hat{\omega}_\pmf^\epsilon  \hat{p}_\pmf^\epsilon  , \hat{p}_\pmf^\epsilon \bigr\rangle  \Bigr) \\
&\hspace{1cm} + \norm[\big]{\fracfac[\big]{\epsilon^{ \iota (\nu_\matr{K} ) } } \nabla^{\fracfac{\! \epsilon\! }} \hat{p}^\epsilon_\pmf }_{\vct{\Lambda}}^2
\lesssim \bigl\langle \fracfac{\epsilon^{\nu_{\!\vct{f}} + 1}} \hat{\vct{f}}_\pmf^\epsilon , \partial_t \hat{\vct{u}}_\pmf^\epsilon \bigr\rangle  \\[3pt]
&\hspace{3cm} + \bigl\langle  \fracfac{\epsilon^{\!\matr{\nu}_{\!\matr{\alpha}} + \matr{I}_n\!  } } \hat{G}_\pmf^\epsilon \hat{\matr{\alpha}}_\pmf^\epsilon , \nabla^\fraceps \partial_t \hat{\vct{u}}^\epsilon_\pmf [\matr{M}_\pmf^\epsilon ]^\rmt \bigr\rangle
 + \norm[\big]{\fracfac[\big]{\epsilon^{\nu_q + 1} } \hat{p}_\pmf^\epsilon }_{\Lambda} ,
\end{split}%
\label{eq:apriori_1}%
\end{align}%
where we have used \Cref{rem:Meps} and \ref{asm:q_eps} and~\ref{asm:K_eps} in  \Cref{asm:eps}.
Next, we integrate from~$0$ to~$t \in \bar{I}$ in \cref{eq:apriori_1}. 
Integration by parts yields%
\begin{multline*} 
\int_0^t \Bigl( \bigl\langle \fracfac{\epsilon^{\nu_{\!\vct{f}} + 1}} \hat{\vct{f}}_\pmf^\epsilon , \partial_t \hat{\vct{u}}_\pmf^\epsilon \bigr\rangle  + \bigl\langle  \fracfac{\epsilon^{\!\matr{\nu}_{\!\matr{\alpha}} + \matr{I}_n\!  } } \hat{G}_\pmf^\epsilon \hat{\matr{\alpha}}_\pmf^\epsilon , \nabla^\fraceps \partial_t \hat{\vct{u}}^\epsilon_\pmf [\matr{M}_\pmf^\epsilon ]^\rmt\bigr\rangle \Bigr) \, \rmd \bar{t}\\
\phantom{i}= \bigl\langle \fracfac{\epsilon^{\nu_{\!\vct{f}} + 1}} \hat{\vct{f}}_\pmf^\epsilon , \hat{\vct{u}}_\pmf^\epsilon \bigr\rangle \Bigr\vert_0^t 
+  \bigl\langle  \fracfac{\epsilon^{\!\matr{\nu}_{\!\matr{\alpha}} + \matr{I}_n\!  } } \hat{G}_\pmf^\epsilon \hat{\matr{\alpha}}_\pmf^\epsilon , \nabla^\fraceps \hat{\vct{u}}^\epsilon_\pmf [\matr{M}_\pmf^\epsilon ]^\rmt \bigr\rangle\Bigr\vert_0^t \\
- \int_0^t \bigl\langle \fracfac{\epsilon^{\nu_{\!\vct{f}} + 1}} \partial_t \hat{\vct{f}}_\pmf^\epsilon ,  \hat{\vct{u}}_\pmf^\epsilon \bigr\rangle \, \rmd \bar{t} .
\end{multline*}
Besides, we use \Cref{rem:Meps} and \ref{asm:omega_eps} and~\ref{asm:C_eps} in \Cref{asm:eps} to obtain 
\begin{multline*}
\bigl\langle \fracfac{\epsilon^{\nu_\tsr{C} + 1}} \hat{\tsr{C}}_\pmf^\epsilon \matr{E}_\pmf^\epsilon ( \hat{\vct{u}}_\pmf^\epsilon ) , \matr{E}_\pmf^\epsilon ( \hat{\vct{u}}_\pmf^\epsilon )  \bigr\rangle  +\bigl\langle \fracfac{\epsilon^{\nu_\omega +1 }}  \hat{\omega}_\pmf^\epsilon  \hat{p}_\pmf^\epsilon  , \hat{p}_\pmf^\epsilon \bigr\rangle  \\
\gtrsim \norm[\big ]{\fracfac[\big ]{\epsilon^{\iota ( \nu_{\tsr{C}})} } \matr{e}^\fraceps ( \hat{\vct{u}}_\pmf^\epsilon )  }_{\matr{\Lambda}}^2  
+ \norm[\big ]{\fracfac[\big ]{\epsilon^{ \iota (\nu_\omega  )}} \hat{p}_\pmf^\epsilon }^2_{\Lambda } .
\end{multline*}
Thus, integrating in~\cref{eq:apriori_1}, applying Young's inequality on the right-hand side, and using \Cref{rem:Meps} and \ref{asm:alpha_eps} and~\ref{asm:f_eps} in \Cref{asm:eps} yields the inequality 
\begin{align*}
\begin{split}
&\norm[\big ]{\fracfac[\big ]{\epsilon^{\iota ( \nu_{\tsr{C}})} } \matr{e}^\fraceps ( \hat{\vct{u}}_\pmf^\epsilon (t) ) }_{\matr{\Lambda}}^2  
+ \norm[\big ]{\fracfac[\big ]{\epsilon^{ \iota (\nu_\omega  )}} \hat{p}_\pmf^\epsilon  (t) }^2_{\Lambda } 
+ \norm[\big]{\fracfac[\big]{\epsilon^{ \iota (\nu_\matr{K} ) } } \nabla^{\fracfac{\! \epsilon\! }} \hat{p}^\epsilon_\pmf }_{L^2 (0, t;  \vct{\Lambda})}^2 \\
&\hspace{0.75cm}\lesssim 
 \delta^{-1} + \norm[\big ]{\fracfac[\big ]{\epsilon^{ \iota (\nu_{\tsr{C}}  )}} \matr{e}^\fraceps ( \hat{\vct{u}}_{0, \pmf}^\epsilon ) }^2_{\matr{\Lambda} } 
+ \norm[\big ]{\fracfac[\big ]{\epsilon^{ \iota (\nu_\omega  )}} \hat{p}_{0, \pmf}^\epsilon  }^2_{\Lambda } 
+  \norm[\big]{\fracfac[\big]{\epsilon^{\nu_{\!\vct{f}} + 1} }  \hat{\vct{u}}_{0, \pmf}^\epsilon }_{ \vct{\Lambda}} 
\\
&\hspace{1.5cm}+   \norm[\big]{ \nabla^\fraceps  \hat{\vct{u}}^\epsilon_{0, \pmf } \fracfac[\big]{\epsilon^{\! \matr{\nu}_{\!\matr{\alpha}} + \matr{I}_n\! } }  }_{\vct{\Lambda}} +  \delta \Bigl(  \norm[\big]{\fracfac[\big]{\epsilon^{\nu_{\!\vct{f}} +1} }  \hat{\vct{u}}_\pmf^\epsilon (t) }_{ \vct{\Lambda}}^2  
+ \norm[\big]{\fracfac[\big]{\epsilon^{\nu_{\!\vct{f}} +1 } }  \hat{\vct{u}}_\pmf^\epsilon }_{L^2 ( I ; \vct{\Lambda})}^2  \\
&\hspace{5cm} + \norm[\big]{\fracfac[\big]{\epsilon^{\nu_q +1 } } \hat{p}_\pmf^\epsilon }_{L^2 ( 0,t ; \Lambda ) }^2 + \norm[\big]{  \nabla^\fraceps  \hat{\vct{u}}^\epsilon_\pmf (t) \fracfac[\big]{\epsilon^{\! \matr{\nu}_{\!\matr{\alpha}} + \matr{I}_n\! } }  }_{\vct{\Lambda}}^2  \Bigr)
\end{split} 
\end{align*}
for any $\delta > 0$.
Further, with \Cref{lem:poincare_decomp} and \ref{asm:nuq} and~\ref{asm:dirichlet} in \Cref{asm:nu}, we find
\begin{align*}
\norm[\big]{\fracfac[\big]{\epsilon^{\nu_q +1 } } \hat{p}_\pmf^\epsilon }_{L^2 ( 0,t ; \Lambda ) }^2 &\lesssim 
\begin{cases}
\norm[\big]{\fracfac[\big]{\epsilon^{\iota ( \nu_\omega )  } } \hat{p}_\pmf^\epsilon }_{L^2 ( I ; \Lambda ) }^2 &\text{if } 2\nu_q \ge \nu_\omega - 1 , \\[4pt]
\norm[\big]{\fracfac[\big]{\epsilon^{ \iota (\nu_\matr{K} ) } } \nabla^{\fracfac{\! \epsilon\! }} \hat{p}^\epsilon_\pmf }_{L^2 (0, t;  \vct{\Lambda})}^2 &\text{if } 2 \nu_q \ge \max \{ - 2 ,  \nu_\matr{K} - 3 \} .
\end{cases}
\end{align*}
Then, by using the \Cref{lem:poincare_decomp,lem:korn_decomp,lem:apriori_u0}, \ref{asm:p0_eps} in \Cref{asm:eps}, and \ref{asm:nuomega}, \ref{asm:nuf}, and~\ref{asm:nualpha} in \Cref{asm:nu}, we obtain
\begin{align}
\begin{split}
 &\norm[\big ]{\fracfac[\big ]{\epsilon^{\iota ( \nu_{\tsr{C}}) } } \matr{e}^\fraceps ( \hat{\vct{u}}_\pmf^\epsilon (t) )}_{\matr{\Lambda}}^2  
+ \norm[\big ]{\fracfac[\big ]{\epsilon^{ \iota (\nu_\omega  ) }} \hat{p}_\pmf^\epsilon  (t) }^2_{\Lambda }  
+ \norm[\big]{\fracfac[\big]{\epsilon^{ \iota (\nu_\matr{K} ) } } \nabla^{\fracfac{\! \epsilon\! }} \hat{p}^\epsilon_\pmf }_{L^2 (0, t;  \vct{\Lambda})}^2 \\
&\hspace{1.7cm}\lesssim 1 + \delta^{-1} + \delta \Bigl( \norm[\big]{\fracfac[\big]{\epsilon^{\iota ( \nu_\tsr{C} ) } } \matr{e}^\fraceps (  \hat{\vct{u}}_\pmf^\epsilon ) }_{L^2 ( I ; \vct{\Lambda})}^2 + \norm[\big]{\fracfac[\big]{\epsilon^{\iota ( \nu_\omega )  } } \hat{p}_\pmf^\epsilon }_{L^2 ( I ; \Lambda ) }^2 \Bigr)
\end{split}%
\label{eq:apriori_2}%
\end{align}%
for $\delta > 0 $ sufficiently small. 
Integrating \cref{eq:apriori_2} from~$0$ to~$T$ now yields 
\begin{align}
\label{eq:apriori_3}
\norm[\big ]{\fracfac[\big ]{\epsilon^{\iota ( \nu_{\tsr{C}})} } \matr{e}^\fraceps ( \hat{\vct{u}}_\pmf^\epsilon )  }_{L^2 ( I ; \matr{\Lambda})}^2  
+ \norm[\big ]{\fracfac[\big ]{\epsilon^{ \iota (\nu_\omega  )}} \hat{p}_\pmf^\epsilon  }^2_{ L^2 ( I ; \Lambda ) } \lesssim 1 .
\end{align}
Then, by inserting \cref{eq:apriori_3} back into \cref{eq:apriori_2}, we obtain 
\begin{equation}
\label{eq:apriori_3b}
\begin{multlined}[c][0.875\displaywidth]
\norm[\big ]{\fracfac[\big ]{\epsilon^{\iota ( \nu_{\tsr{C}})} } \matr{e}^\fraceps ( \hat{\vct{u}}_\pmf^\epsilon ) }_{L^\infty ( I ; \matr{\Lambda})}
+ \norm[\big ]{\fracfac[\big ]{\epsilon^{ \iota (\nu_\omega  )}} \hat{p}_\pmf^\epsilon  }_{L^\infty ( I ; \Lambda ) } \\
+ \norm[\big]{\fracfac[\big]{\epsilon^{ \iota (\nu_\matr{K} ) } } \nabla^{\fracfac{\! \epsilon\! }} \hat{p}^\epsilon_\pmf }_{L^2 (I;  \vct{\Lambda})}^2 \lesssim 1 .
\end{multlined}
\end{equation}
Further, applying the \Cref{lem:poincare_decomp,lem:korn_decomp} yields 
\begin{align*}
\norm[\big ]{\fracfac[\big ]{\epsilon^{ \theta ( \nu_{\tsr{C}} )  } } \hat{\vct{u}}_\pmf^\epsilon }_{L^\infty ( I ; \vct{\Lambda})} \lesssim \norm[\big ]{\fracfac[\big ]{\epsilon^{\max \{ \frac{1}{2} \! , \, \iota ( \nu_{\tsr{C}}) \}\! }  } \nabla^\fraceps  \hat{\vct{u}}_\pmf^\epsilon  }_{L^\infty ( I ; \matr{\Lambda})} \lesssim 1.
\end{align*}

\emph{Step 2.}
Next, for any $\bar{\vct{v}}_\pmf \in H^2 ( I ; \vct{V} ) $, we choose $\vct{v}_\pmf = \partial_{tt} \bar{\vct{v}}_\pmf $ in \cref{eq:fulldim-weak-trafo-b}. 
Then, by integrating \cref{eq:fulldim-weak-trafo-b} from~$0$ to $t\in\bar{I}$, integrating by parts, and using~\eqref{eq:fulldim-weak-trafo-d}, we find that 
\begin{align}
\begin{split}
&\int_0^t \Bigl( \bigl\langle \fracfac{\epsilon^{\nu_\tsr{C} + 1}} \hat{\tsr{C}}_\pmf^\epsilon \matr{E}_\pmf^\epsilon ( \partial_t \hat{\vct{u}}_\pmf^\epsilon ) , \matr{E}_\pmf^\epsilon ( \partial_t \bar{\vct{v}}_\pmf )  \bigr\rangle  \\
&\enspace -  \bigl\langle \fracfac{\epsilon^{\!\matr{\nu}_{\!\matr{\alpha}} + \matr{I}_n\!  } } \partial_t \hat{p}_\pmf^\epsilon \hat{\matr{\alpha}}_\pmf^\epsilon  , \nablaeps \partial_t \bar{\vct{v}}_\pmf [\matr{M}_\pmf^\epsilon ]^\rmt \bigr\rangle ) \Bigr) \,\rmd \bar{t} 
 = \int_0^t \bigl\langle \fracfac[\big ]{\epsilon^{\nu_{\!\vct{f}} + 1}} \partial_t \hat{\vct{f}}_\pmf^\epsilon , \partial_t \bar{\vct{v}}_\pmf \bigr\rangle \,\rmd \bar{t}  .
\end{split}%
\label{eq:apriori_dtu}%
\end{align}
A density argument shows that \cref{eq:apriori_dtu} also holds for~$\bar{\vct{v}}_\pmf \in H^1 ( I ; \vct{V} ) $. 
Now, by choosing $\bar{\vct{v}}_\pmf = \hat{\vct{u}}_\pmf^\epsilon$ in \cref{eq:apriori_dtu} and $\phi_\pmf = \partial_t \hat{p}^\epsilon_\pmf$ in \cref{eq:fulldim-weak-trafo-c}, we obtain  
\begin{align}
\begin{split}
& \int_0^t \Bigl( \bigl\langle \fracfac{\epsilon^{\nu_\tsr{C} + 1}} \hat{\tsr{C}}_\pmf^\epsilon \matr{E}_\pmf^\epsilon ( \partial_t \hat{\vct{u}}_\pmf^\epsilon ) , \matr{E}_\pmf^\epsilon ( \partial_t \hat{\vct{u}}_\pmf^\epsilon )  \bigr\rangle
+ \bigl\langle \fracfac{\epsilon^{\nu_\omega +1 }}  \hat{\omega}_\pmf^\epsilon \partial_t \hat{p}_\pmf^\epsilon  , \partial_t \hat{p}_\pmf^\epsilon \bigr\rangle  \Bigr) \,\rmd \bar{t}  \\
&\hspace{1.25cm} + \tfrac{1}{2} \bigl\langle \fracfac{\epsilon^{\nu_\matr{K} + 1 }} \hat{\matr{K}}_\pmf^\epsilon \matr{M}_\pmf^\epsilon \nabla^{\fracfac{\!\epsilon\! }} \hat{p}_\pmf^\epsilon  , \matr{M}_\pmf^\epsilon \nabla^\fraceps\! \hat{p}_\pmf^\epsilon \bigr\rangle \Bigr\vert_0^t \\
&\hspace{3.1cm} = \int_0^t \Bigl( \bigl\langle \fracfac{\epsilon^{\nu_{\!\vct{f}}  + 1}} \partial_t \hat{\vct{f}}_\pmf^\epsilon , \partial_t \hat{\vct{u}}_\pmf^\epsilon \bigr\rangle
 + \bigl\langle \fracfac{\epsilon^{\nu_q  + 1}} \hat{q}_\pmf^\epsilon , \partial_t \hat{p}_\pmf^\epsilon  \bigr\rangle  \Bigr) \,\rmd \bar{t} . 
\end{split}%
\label{eq:apriori_4}%
\end{align}%
For the last term in \cref{eq:apriori_4}, we can now distinguish between two cases in \ref{asm:nuq} in \Cref{asm:nu}.
The case $2\nu_q \ge \nu_\omega - 1$ is straightforward and, using \ref{asm:q_eps} in  \Cref{asm:eps} and Young's inequality, we directly obtain
\begin{align*}
\int_0^t  \bigl\langle \fracfac{\epsilon^{\nu_q  + 1}} \hat{q}_\pmf^\epsilon , \partial_t \hat{p}_\pmf^\epsilon  \bigr\rangle   \,\rmd \bar{t} \le \delta^{-1} + \delta \norm[\big]{ \fracfac{\epsilon^{\iota (\nu_\omega )}} \partial_t \hat{p}_\pmf^\epsilon } ^2_{L^2 (0,t ;\Lambda ) } 
\end{align*}
for any $\delta > 0$. Otherwise, if $2\nu_q \ge \max \{ -2 , \nu_\matr{K} - 3\}$, we can integrate by parts in time and find
\begin{align*}
&\int_0^t  \bigl\langle \fracfac{\epsilon^{\nu_q  + 1}} \hat{q}_\pmf^\epsilon , \partial_t \hat{p}_\pmf^\epsilon  \bigr\rangle   \,\rmd \bar{t}  =  \bigl\langle \fracfac{\epsilon^{\nu_q  + 1}} \hat{q}_\pmf^\epsilon ,  \hat{p}_\pmf^\epsilon  \bigr\rangle \Bigr\vert_0^t - \int_0^t \bigl\langle \fracfac{\epsilon^{\nu_q  + 1}} \partial_t \hat{q}_\pmf^\epsilon ,  \hat{p}_\pmf^\epsilon  \bigr\rangle  \,\rmd \bar{t} \\
&\hspace{1cm} \lesssim  \norm[\big]{ \fracfac{\epsilon^{\theta ( \nu_\matr{K} ) } } \hat{p}_\pmf^\epsilon (t) }_\Lambda +  \norm[\big]{ \fracfac{\epsilon^{\theta ( \nu_\matr{K} )} } \hat{p}_{0, \pmf}^\epsilon  }_\Lambda   + \norm[\big]{ \fracfac{\epsilon^{\theta ( \nu_\matr{K} ) }} \hat{p}_\pmf^\epsilon }_{L^2 (I ; \vct{\Lambda } ) }\\
&\hspace{1cm} \lesssim \norm[\big]{ \fracfac{\epsilon^{\iota ( \nu_\matr{K} ) } } \nabla^\fraceps \hat{p}_\pmf^\epsilon (t) }_\Lambda +  \norm[\big]{ \fracfac{\epsilon^{\iota ( \nu_\matr{K} )} } \nabla^\fraceps \hat{p}_{0, \pmf}^\epsilon  }_\Lambda   + \norm[\big]{ \fracfac{\epsilon^{\iota ( \nu_\matr{K} ) }} \nabla^\fraceps \hat{p}_\pmf^\epsilon }_{L^2 (I ; \vct{\Lambda } ) } \\
&\hspace{1cm} \lesssim 1 +  \delta^{-1} + \delta \norm[\big ]{ \fracfac[\big ]{\epsilon^{\iota ( \nu_\matr{K})}} \nabla^\fraceps \hat{p}^\epsilon_\pmf (t) }^2_{\vct{\Lambda}} 
\end{align*}
for any $\delta > 0$, where we have used \ref{asm:q_eps} in  \Cref{asm:eps} for the first inequality, \ref{asm:dirichlet} in \Cref{asm:nu} and  \Cref{lem:poincare_decomp} for the second inequality, and Young's inequality, \ref{asm:p0_eps} in \Cref{asm:eps}, and \cref{eq:apriori_3b} for the last inequality.  
Thus, applying the \Cref{lem:poincare_decomp,lem:korn_decomp}, \Cref{rem:Meps}, \ref{asm:f_eps}, \ref{asm:p0_eps}, \ref{asm:omega_eps}, \ref{asm:C_eps}, and \ref{asm:K_eps} in \Cref{asm:eps}, and \ref{asm:nuf} in  \Cref{asm:nu}, we find that \cref{eq:apriori_4} implies the inequality 
\begin{align*}
\norm[\big ]{\fracfac[\big ]{\epsilon^{\iota ( \nu_\tsr{C} ) }} \matr{e}^\fraceps ( \partial_t  \hat{\vct{u}}_\pmf^\epsilon ) }^2_{L^2 ( 0,t ; \vct{\Lambda} ) } + \norm[\big ]{\fracfac[\big ]{\epsilon^{\iota ( \nu_\omega ) }} \partial_t \hat{p}_\pmf^\epsilon }_{L^2 ( 0, t ; \Lambda ) }^2 + \norm[\big ]{\fracfac[\big ]{\epsilon^{\iota ( \nu_{\matr{K}} )}} \nabla^\fraceps \hat{p}^\epsilon_\pmf (t) }_{\vct{\Lambda}}^2 \lesssim 1 .
\end{align*}
Further, the \Cref{lem:poincare_decomp,lem:korn_decomp} now imply 
\begin{align*}
\norm[\big ]{\fracfac[\big]{\epsilon^{ \theta (\nu_\tsr{C} )  }} \partial_t \hat{\vct{u}}_\pmf^\epsilon }_{L^2 ( I ; \vct{\Lambda})} \lesssim \norm[\big ]{\fracfac[\big ]{\epsilon^{\max \{ \frac{1}{2} \! , \, \iota ( \nu_\tsr{C} )  \} \! }} \partial_t \nabla^\fraceps \hat{\vct{u}}_\pmf^\epsilon  }_{L^2 (I ; \vct{\Lambda} ) }  \lesssim 1 .
\end{align*}
Besides, with \Cref{lem:poincare_decomp} and \ref{asm:dirichlet} in \Cref{asm:nu}, we obtain
\begin{align*}
\norm[\big ]{\fracfac[\big ]{\epsilon^{  \theta ( \nu_\matr{K} )   } } \hat{p}_\pmf^\epsilon }_{L^\infty ( I ; \Lambda ) } \lesssim 1 .  \tag*{\qedhere}
\end{align*}
\end{proof}

\section{Proof of the \texorpdfstring{\Cref{thm:Klm1,thm:Km1,thm:Km1p1,thm:Kp1,thm:Kgp1}}{Theorems~4.9 to~4.12 and~4.15}} \label{sec:6}
In the following, we verify that the limit functions~$\hat{p}_\pmf^\#$ and $\vct{u}_\pmf^\#$ from \Cref{cor:conv} (uniquely) solve the limit models presented in \Cref{sec:4}, i.e., we prove the \Cref{thm:Klm1,thm:Km1,thm:Km1p1,thm:Kp1,thm:Kgp1}.
We begin by deriving a few auxiliary results to ensure that the limit solutions belong to the stated solution spaces. 
Then, we analyze the converge of the individual terms in the original full-dimensional model~\eqref{eq:fulldim-weak-trafo}; see \Cref{lem:bulkconv,lem:fracconv1,lem:fracconv2}.
Finally, we assemble these ingredients to conclude the proof. 

First, as a consequence of the a~priori estimates in \Cref{thm:apriori}, we can conclude that the limit pressure head~$\hat{p}_\rmf^\#$ inside the fracture is constant or constant in normal direction for certain ranges of the scaling parameter~$\nu_\matr{K}$.
\begin{corollary} \label{cor:const}
Let the \Cref{asm:eps,asm:nu} hold true. Then, for almost all~$\vct{y} + s \vct{N} \in \Omega_\rmf^1$ and $t\in I$ with $\vct{y} \in \gamma$ and $s \in (-a_- ( \vct{y} ) , a_+ ( \vct{y} ) )$, we have
\begin{subequations}
\begin{alignat}{4}
\nabla \hat{p}_\rmf^\# (\vct{y} + s \vct{N} ) &= \vct{0} \quad\ &&\text{if } \nu_\matr{K} < -1 \quad &&\Rightarrow \enspace  &\hat{p}_\rmf^\# (\vct{y} + s\vct{N} )  &= \mathfrak{A}_\rmf \hat{p}_\rmf^\# = \mathrm{const}. , \\
\nabla_{\!\vct{N}} \hat{p}_\rmf^\# (\vct{y} + s \vct{N} ) &= \vct{0} \quad\ &&\text{if } \nu_\matr{K} < 1 \quad &&\Rightarrow \enspace  &\hat{p}_\rmf^\# (\vct{y} + s\vct{N} )  &= \bigl( \mathfrak{A}_{\!\vct{N}} \hat{p}_\rmf^\# \bigr) (\vct{y})   .
\end{alignat}
\end{subequations}
\end{corollary}
\begin{proof}
The result follows directly from \Cref{thm:apriori} and \Cref{cor:conv}.
\end{proof}
Further, we find that the bulk limit pressure heads~$\hat{p}_\pm^\#$ are continuous across the fracture interface~$\gamma$ for~$\nu_\matr{K} < 1$.
\begin{proposition} \label{prop:ptrace}
Given the \Cref{asm:eps,asm:nu}, the limit pressure head $\hat{p}_\pmf^\# \in \Phi^\#$ satisfies
\begin{alignat}{2}
\hat{p}_\pm^\# &= \mathfrak{A}_{\!\vct{N}} \hat{p}_\rmf^\# \quad &&\text{a.e.\ on}~\gamma \enspace\ \text{if } \nu_\matr{K} < 1 .
\end{alignat}
\end{proposition}
\begin{proof}
The result follows analogously to \cite[Prop.\ 3.10]{hoerl24}. 
\end{proof}
In addition, using \Cref{lem:trace_H1N}, we can show that the trace of~$\smash{\hat{\vct{u}}_\rmf^\#}$ vanishes on~$\gamma_\pm^1$ for~$\nu_\tsr{C} > 1$, i.e., the limit~$\smash{\hat{\vct{u}}_\rmf^\#}$ of the scaled fracture displacement is not directly coupled to the bulk displacement~$\smash{\hat{\vct{u}}_\pm^\#}$.
\begin{proposition} \label{prop:uftrace}
For $\nu_\tsr{C} > 1$, we have $\mathfrak{T}_\pm \hat{\vct{u}}_\rmf^\# = \vct{0}$, where $\mathfrak{T}_\pm$ denotes the trace operator from \Cref{lem:trace_H1N}.
\end{proposition}
\begin{proof} Let $\nu_\tsr{C} > 1$ and $\theta ( \nu_\tsr{C} ) := \tfrac{1}{2} ( \nu_\tsr{C} - 1)$. 
With \Cref{lem:trace_H1N} and \cref{eq:conv_u_b} in \Cref{cor:conv}, we have 
\begin{align*}
\epsilon_k^{\theta ( \nu_\tsr{C} ) } \mathfrak{T}_\pm \hat{\vct{u}}_\rmf^{\epsilon_k} \rightharpoonup \mathfrak{T}_\pm \hat{\vct{u}}_\rmf^\# \quad\enspace \text{weakly  in } \vct{L}^2( \gamma ) 
\end{align*}
 as $k\rightarrow \infty$.
Besides, \cref{eq:apriori_a,eq:apriori_b} in \Cref{thm:apriori}  and the trace theorem in $H^1(\Omega_\pm^0)$ imply $\smash{\norm{\epsilon_k^{\theta ( \nu_\tsr{C})} \hat{\vct{u}}_\pm^{\epsilon_k} }_{L^2 ( \gamma ) }} \rightarrow 0$ as $k\rightarrow \infty$.
Thus, for all $\vct{\zeta} \in \vct{L}^2 ( \gamma ) $, we have 
\begin{align*}
\bigl\langle \mathfrak{T}_\pm \hat{\vct{u}}_\rmf^\# , \vct{\zeta} \bigr\rangle_{\!\gamma } = \bigl\langle \mathfrak{T}_\pm \bigl( \hat{\vct{u}}_\rmf^\# - \epsilon_k^{\theta (\nu_\tsr{C})}  \hat{\vct{u}}_\rmf^{\epsilon_k} \bigr) , \vct{\zeta} \bigr\rangle_{\! \gamma } + \bigl\langle \epsilon_k^{\theta ( \nu_\tsr{C} ) } \hat{\vct{u}}_\pm^{\epsilon_k} \bigr\vert_\gamma , \vct{\zeta} \bigr\rangle_{\! \gamma } \rightarrow 0.
\end{align*}
\end{proof}

Next, in \Cref{lem:bulkconv}, we address the convergence of the individual terms related to the bulk domains~$\Omega_\pm^0$ in the weak formulation~\eqref{eq:fulldim-weak-trafo}.
\begin{lemma} \label{lem:bulkconv}
Let the \Cref{asm:eps,asm:nu} hold. Then, for all $\vct{v}_\pm \in \vct{H}^1 (\Omega_\pm^0 ) $ and $\phi_\pm \in H^1 (\Omega_\pm^0 )$,  we have 
\begin{subequations}
\begin{align}
\label{eq:bulkconv_a} \bigl\langle \hat{\tsr{C}}_\pm^{\epsilon_k} \matr{E}_\pm^{\epsilon_k} ( \hat{\vct{u}}_\pm^{\epsilon_k} ) ,  \matr{E}_\pm^{\epsilon_k} ( \vct{v}_\pm )  \bigr\rangle_{\Omega_\pm^0    } &\rightarrow \bigl\langle \hat{\tsr{C}}_\pm \matr{e} ( \hat{\vct{u}}_\pm^\# ) , \matr{e} ( \vct{v}_\pm )  \bigr\rangle_{\Omega_\pm^0 } , \\
\label{eq:bulkconv_b} \bigl\langle   \hat{p}_\pm^{\epsilon_k} \hat{\matr{\alpha}}_\pm^{\epsilon_k}  , \nabla \vct{v}_\pm [\matr{M}_\pm^{\epsilon_k} ]^\rmt \bigr\rangle_{\Omega_\pm^0 } &\rightarrow \bigl\langle   \hat{p}_\pm^\# \hat{\matr{\alpha}}_\pm , \nabla \vct{v}_\pm \bigr\rangle_{ \Omega_\pm^0  } , \\
\label{eq:bulkconv_c}\bigl\langle   \hat{\omega}_\pm^{\epsilon_k}  \partial_t \hat{p}_\pm^{\epsilon_k}  , \phi_\pm \bigr\rangle_{\Omega_\pm^0  } &\rightarrow \bigl\langle   \hat{\omega}_\pm \partial_t \hat{p}_\pm^\#  , \phi_\pm \bigr\rangle_{\Omega_\pm^0 } , \\
\label{eq:bulkconv_d} \bigl\langle \hat{\matr{K}}_\pm^{\epsilon_k} \matr{M}_\pm^{\epsilon_k}\nabla \hat{p}_\pm^{\epsilon_k}  ,  \matr{M}_\pm^{\epsilon_k} \nabla \phi_\pm \bigr\rangle_{\Omega_\pm^0 } &\rightarrow  \bigl\langle \hat{\matr{K}}_\pm \nabla \hat{p}_\pm^\#  , \nabla \phi_\pm \bigr\rangle_{\Omega_\pm^0  }, \\
\label{eq:bulkconv_e} \bigl\langle \phi_\pm \hat{\matr{\alpha}}_\pm^{\epsilon_k} \! , \nabla \partial_t \hat{\vct{u}}_\pm^{\epsilon_k} [\matr{M}_\pm^{\epsilon_k} ]^\rmt \bigr\rangle_{\Omega_\pm^0 } &\rightarrow  \bigl\langle \phi_\pm \hat{\matr{\alpha}}_\pm , \nabla \partial_t \hat{\vct{u}}_\pm^\# \bigr\rangle_{\Omega_\pm^0 } 
\end{align}
as $k\rightarrow \infty$. Moreover, as $\epsilon \rightarrow 0$, it is
\begin{align}
\label{eq:bulkconv_f} \bigl\langle \hat{\vct{f}}_\pm^\epsilon , \vct{v}_\pm \bigr\rangle_{\Omega_\pm^0} & \rightarrow  \bigl\langle \hat{\vct{f}}_\pm , \vct{v}_\pm \bigr\rangle_{\Omega_\pm^0 } , \\
\label{eq:bulkconv_g} \bigl\langle  \hat{G}_\pm^\epsilon \hat{\matr{\alpha}}_\pm^\epsilon , \nabla \vct{v}_\pm  [\matr{M}_\pm^{\epsilon_k} ]^\rmt  \bigr\rangle_{\Omega_\pm^0 } &\rightarrow \bigl\langle  \hat{G}_\pm \hat{\matr{\alpha}}_\pm , \nabla \vct{v}_\pm \bigr\rangle_{\Omega_\pm^0 } , \\
\label{eq:bulkconv_h} \bigl\langle \hat{q}_\pm^\epsilon , \phi_\pm \bigr\rangle_{\Omega_\pm^0 } &\rightarrow \bigl\langle \hat{q}_\pm , \phi_\pm \bigr\rangle_{\Omega_\pm^0 } .
\end{align}
\end{subequations}
\end{lemma}
\begin{proof}
The result follows directly from \Cref{asm:eps}, \Cref{lem:G}, and \Cref{cor:conv}.
\end{proof}
Further, using \Cref{thm:apriori} and \Cref{cor:conv}, we examine the convergence of the individual terms in \cref{eq:fulldim-weak-trafo-b} that are associated with the fracture domain~$\Omega_\rmf^1$. 
\begin{lemma} \label{lem:fracconv1}
Let the \Cref{asm:eps,asm:nu} hold and let $\nu_\tsr{C} \ge 1$.
\begin{subequations}
\begin{enumerate}[label=(\roman*),wide]
\item As $k\rightarrow \infty$, we have for all $\vct{v}_\rmf \in \vct{H}^1 (\Omega_\rmf^1 ) $ that
\begin{align}
\label{eq:fracconv1_a} \bigl\langle \epsilon_k^{\nu_\tsr{C} + 1 } \hat{\tsr{C}}_\rmf^{\epsilon_k} \matr{e}^{\epsilon_k} ( \hat{\vct{u}}_\rmf^{\epsilon_k} ) , \matr{e}^{\epsilon_k} ( \vct{v}_\rmf ) \bigr\rangle_{ \Omega_\rmf^1  } &\rightarrow \begin{cases} \bigl\langle \hat{\tsr{C}}_\rmf^{\!\vct{N}} \partial_{\!\vct{N}} \hat{\vct{u}}_\rmf^\#\! , \partial_{\!\vct{N}} \vct{v}_\rmf \bigr\rangle_{\Omega_\rmf^1  } &\text{if } \nu_\tsr{C} = 1 , \\ 
0 &\text{if } \nu_\tsr{C} > 1 , \end{cases} \\
\label{eq:fracconv1_b} \bigl\langle \epsilon_k^{\frac{\nu_\tsr{C} + 3}{2} } \hat{\tsr{C}}_\rmf^{\epsilon_k} \matr{e}^{\epsilon_k} ( \hat{\vct{u}}_\rmf^{\epsilon_k} ) , \matr{e}^{\epsilon_k} ( \vct{v}_\rmf ) \bigr\rangle_{ \Omega_\rmf^1  } &\rightarrow \bigl\langle \hat{\tsr{C}}_\rmf^{\!\vct{N}} \partial_{\!\vct{N}} \hat{\vct{u}}_\rmf^\#\! , \partial_{\!\vct{N}} \vct{v}_\rmf \bigr\rangle_{\Omega_\rmf^1  }  .
\end{align}
\item Let $\nu_\matr{K} \le 1$ or $\nu_\omega = -1$. Then, as $k\rightarrow\infty$, we have 
\begin{align}
\label{eq:fracconv1_c} \bigl\langle  \epsilon_k^{\matr{\nu}_{\!\matr{\alpha}} + \matr{I}_n}  \hat{p}_\rmf^{\epsilon_k} \hat{\matr{\alpha}}_\rmf^{\epsilon_k} \! , \nabla^{\epsilon_k} \vct{v}_\rmf  \bigr\rangle_{\Omega_\rmf^1  }  &\rightarrow \begin{cases}
\bigl\langle \hat{p}_\rmf^\# \hat{\vct{\alpha}}_\rmf^\mathrm{eff} \! , \partial_{\!\vct{N}} \vct{v}_\rmf \bigr\rangle_{\Omega_\rmf^1  } &\text{if } \nu_\tsr{C} = 1, \\
0 &\text{if } \nu_\tsr{C} > 1 \end{cases} \\
\label{eq:fracconv1_d}  \bigl\langle \epsilon_k^\frac{1-\nu_\tsr{C}}{2} \epsilon_k^{\matr{\nu}_{\!\matr{\alpha}} + \matr{I}_n}  \hat{p}_\rmf^{\epsilon_k} \hat{\matr{\alpha}}_\rmf^{\epsilon_k} \! , \nabla^{\epsilon_k} \vct{v}_\rmf  \bigr\rangle_{\Omega_\rmf^1  }  &\rightarrow 
\bigl\langle \hat{p}_\rmf^\# \hat{\vct{\alpha}}_\rmf^\mathrm{eff} \! , \partial_{\!\vct{N}} \vct{v}_\rmf \bigr\rangle_{\Omega_\rmf^1  } ,
\end{align}
for all $\vct{v}_\rmf \in \vct{H}^1 (\Omega_\rmf^1 ) $.
Further, if $\nu_\matr{K} > 1$ and $\nu_\omega > -1$, as well as $2 \nu_{\!\matr{\alpha}}^\smallpar > \nu_\tsr{C} +  \min \{ \nu_\omega - 2 , \nu_\matr{K} - 4 \} $ and $2 \nu_{\!\matr{\alpha}}^\smallperp >  \nu_\tsr{C} +  \min \{ \nu_\omega  , \nu_\matr{K} - 2\} $, it is 
\begin{align}
\label{eq:fracconv1_e}  \bigl\langle  \epsilon_k^{\matr{\nu}_{\!\matr{\alpha}} + \matr{I}_n }  \hat{p}_\rmf^{\epsilon_k} \hat{\matr{\alpha}}_\rmf^{\epsilon_k} \!  , \nabla^{\epsilon_k} \vct{v}_\rmf  \bigr\rangle_{\Omega_\rmf^1  } &\rightarrow 0 , \\
\label{eq:fracconv1_f}  \bigl\langle   \epsilon_k^\frac{1-\nu_\tsr{C}}{2} \epsilon_k^{\matr{\nu}_{\!\matr{\alpha}} + \matr{I}_n }  \hat{p}_\rmf^{\epsilon_k} \hat{\matr{\alpha}}_\rmf^{\epsilon_k} \! , \nabla^{\epsilon_k} \vct{v}_\rmf  \bigr\rangle_{\Omega_\rmf^1  } &\rightarrow 0 .
\end{align}
\item As $\epsilon \rightarrow 0$, for all $\vct{v}_\rmf \in \vct{H}^1 (\Omega_\rmf^1 )$, it is 
\end{enumerate}
\begin{align}
\label{eq:fracconv1_g}\bigl\langle \epsilon^{\nu_{\!\vct{f}} + 1 } \hat{\vct{f}}^{\epsilon}_\rmf , \vct{v}_\rmf \bigr\rangle_{\Omega_\rmf^1  } &\rightarrow \begin{cases} 
\bigl\langle \hat{\vct{f}}_\rmf^\mathrm{eff} \! , \vct{v}_\rmf \bigr\rangle_{\Omega_\rmf^1  } &\text{if } \nu_\tsr{C} = 1 , \\
0 &\text{if } \nu_\tsr{C} >  1 , \end{cases} \\
\label{eq:fracconv1_h}\bigl\langle \epsilon^\frac{2\nu_{\!\vct{f}} - \nu_\mathbb{C} +3}{2} \hat{\vct{f}}_\rmf^\epsilon , \vct{v}_\rmf \bigr\rangle_{\!\vct{\Lambda} }  &\rightarrow \bigl\langle \hat{\vct{f}}_\rmf^\mathrm{eff} \! , \vct{v}_\rmf \bigr\rangle_{\Omega_\rmf^1  } , \\
\label{eq:fracconv1_i} \bigl\langle  \epsilon^{\matr{\nu}_{\!\matr{\alpha }} + \matr{I}_n}   \hat{G}_\rmf^{\epsilon} \hat{\matr{\alpha}}_\rmf^{\epsilon} , \nabla^{\epsilon} \vct{v}_\rmf \bigr\rangle_{\Omega_\rmf^1  } &\rightarrow \begin{cases} \bigl\langle \hat{G}_\rmf \hat{\vct{\alpha}}_\rmf^\mathrm{eff}\! , \partial_{\!\vct{N}} \vct{v}_\rmf \bigr\rangle_{\Omega_\rmf^1}  &\text{if } \nu_\tsr{C} = 1  , \\ 
0   &\text{if } \nu_\tsr{C} > 1 , \end{cases}  \\
\label{eq:fracconv1_j}\bigl\langle   \epsilon_k^\frac{1-\nu_\tsr{C}}{2} \epsilon^{\matr{\nu}_{\!\matr{\alpha }} + \matr{I}_n}  \hat{G}_\rmf^{\epsilon} \hat{\matr{\alpha}}_\rmf^{\epsilon}   , \nabla^{\epsilon} \vct{v}_\rmf \bigr\rangle_{\Omega_\rmf^1  }&\rightarrow 
\bigl\langle \hat{G}_\rmf \hat{\vct{\alpha}}_\rmf^\mathrm{eff}\! , \partial_{\!\vct{N}} \vct{v}_\rmf \bigr\rangle_{\Omega_\rmf^1} 
\end{align}
\end{subequations}
\end{lemma}
\begin{proof} 
\begin{enumerate}[label=(\roman*),wide]
\item 
In the following, we prove \cref{eq:fracconv1_b}; \cref{eq:fracconv1_a} then follows as an immediate consequence.
With \cref{eq:apriori_a,eq:apriori_b} in \Cref{thm:apriori}, we have 
\begin{align*}
\epsilon^\frac{\nu_\tsr{C} - 1}{2} \norm[\big]{\hat{\vct{u}}_\rmf^\epsilon}_{H^1 ( I ; \vct{L}^2 (\Omega_\rmf^1 ) ) } + \epsilon^\frac{\nu_\tsr{C} + 1}{2} \norm[\big]{\nabla_{\!\smallpar} \hat{\vct{u}}_\rmf^\epsilon}_{H^1 ( I ; \matr{L}^2 (\Omega_\rmf^1 ) )} \lesssim 1 
\end{align*}
for $\nu_\tsr{C} \ge 1$ and hence
\begin{align} \label{eq:fracconv1_proof_1}
\epsilon_k^\frac{\nu_\tsr{C} + 1}{2} \nabla_{\!\smallpar} \hat{\vct{u}}_\rmf^{\epsilon_k} \rightharpoonup \matr{0} \quad\enspace\text{weakly in } H^1 ( I ; \matr{L}^2 (\Omega_\rmf^1 ) ). 
\end{align}
Now, let $2\matr{e}_\smallpar (\vct{v} ) := \nabla_{\!\smallpar} \vct{v} + ( \nabla_{\!\smallpar} \vct{v})^\rmt$ and $2\matr{e}_\vct{N} (\vct{v} ) := \nabla_{\!\vct{N}} \vct{v} + ( \nabla_{\!\vct{N}} \vct{v})^\rmt$. 
We rewrite the term $\bigl\langle \epsilon_k^{\frac{\nu_\tsr{C} + 3}{2} } \hat{\tsr{C}}_\rmf^{\epsilon_k} \matr{e}^{\epsilon_k} ( \hat{\vct{u}}_\rmf^{\epsilon_k} ) , \matr{e}^{\epsilon_k} ( \vct{v}_\rmf ) \bigr\rangle_{ \Omega_\rmf^1  }$ as
\begin{equation}
\begin{multlined}[c][0.875\displaywidth]
 \epsilon_k^{\frac{\nu_{\tsr{C} } +3}{2} } \bigl\langle \hat{\mathbb{C}}_\rmf^{\epsilon_k} \matr{e}_\smallpar ( \hat{\vct{u}}_\rmf^{\epsilon_k} ) , \matr{e}_\smallpar ( \vct{v}_\rmf ) \bigr\rangle_{\Omega_\rmf^1 } +  \epsilon_k^{\frac{\nu_{\tsr{C} } + 1}{2}  } \bigl\langle \hat{\mathbb{C}}_\rmf^{\epsilon_k} \matr{e}_\smallpar ( \hat{\vct{u}}_\rmf^{\epsilon_k} ) , \matr{e}_{\vct{N}} ( \vct{v}_\rmf ) \bigr\rangle_{\Omega_\rmf^1  } \\
  +  \epsilon_k^{\frac{\nu_{\tsr{C} } + 1}{2}  } \bigl\langle \hat{\mathbb{C}}_\rmf^{\epsilon_k} \matr{e}_{\vct{N}} ( \hat{\vct{u}}_\rmf^{\epsilon_k} ) , \matr{e}_\smallpar ( \vct{v}_\rmf ) \bigr\rangle_{\Omega_\rmf^1  } + \epsilon_k^{\frac{\nu_{\tsr{C} } - 1}{2} } \bigl\langle \hat{\mathbb{C}}_\rmf^{\epsilon_k} \matr{e}_{\vct{N}} ( \hat{\vct{u}}_\rmf^{\epsilon_k} ) , \matr{e}_{\vct{N}} ( \vct{v}_\rmf ) \bigr\rangle_{\Omega_\rmf^1  }.
\end{multlined}%
 \label{eq:fracconv1_proof_2}%
\end{equation}%
With \ref{asm:C_eps} in \Cref{asm:eps} and \cref{eq:conv_u_b} in \Cref{cor:conv}, we obtain
\begin{align*}
\epsilon_k^{\frac{\nu_{\tsr{C} } - 1}{2} } \bigl\langle \hat{\mathbb{C}}_\rmf^{\epsilon_k} \matr{e}_{\vct{N}} ( \hat{\vct{u}}_\rmf^{\epsilon_k} ) , \matr{e}_{\vct{N}} ( \vct{v}_\rmf ) \bigr\rangle_{\Omega_\rmf^1  } &\rightarrow \bigl\langle
\hat{\tsr{C}}_\rmf \matr{e}_{\!\vct{N}} (\hat{\vct{u}}_\rmf^\# ) , \matr{e}_{\!\vct{N}} (\vct{v}_\rmf ) \bigr\rangle_{\Omega_\rmf^1  }
\end{align*}
as $k\rightarrow \infty$.
Besides, using \cref{eq:apriori_b} in \Cref{thm:apriori} and \cref{eq:fracconv1_proof_1}, we find that the other terms in \cref{eq:fracconv1_proof_2} vanish as $k\rightarrow \infty$.
Thus, the result follows with \ref{asm:C_eps} in \Cref{asm:eps}.
\item We write
\begin{align*}
\bigl\langle \epsilon_k^{\matr{\nu}_{\!\matr{\alpha}} + \matr{I}_n}  \hat{p}_\rmf^{\epsilon_k} \hat{\matr{\alpha}}_\rmf^{\epsilon_k} \! , \nabla^{\epsilon_k} \vct{v}_\rmf  \bigr\rangle_{\Omega_\rmf^1  }  = \bigl\langle \epsilon_k^{\nu_{\!\matr{\alpha}}^\smallpar + 1} \hat{p}_\rmf^{\epsilon_k} \hat{\matr{\alpha}}_\rmf^{\epsilon_k} \! , \nablapar \vct{v}_\rmf  \bigr\rangle_{\Omega_\rmf^1  }  + \bigl\langle \epsilon_k^{\nu_{\!\matr{\alpha}}^\smallperp }\hat{p}_\rmf^{\epsilon_k} \hat{\matr{\alpha}}_\rmf^{\epsilon_k} \! , \nablaperp \vct{v}_\rmf  \bigr\rangle_{\Omega_\rmf^1  } 
\end{align*}
for \cref{eq:fracconv1_c,eq:fracconv1_e} and analogously for \cref{eq:fracconv1_d,eq:fracconv1_f}.
Then, the assertion follows with \cref{eq:apriori_d} in \Cref{thm:apriori}, \cref{eq:conv_p_d,eq:conv_p_e} in \Cref{cor:conv}, and \ref{asm:alpha_eps} in \Cref{asm:eps}.
\item The result follows directly from \ref{asm:alpha_eps} and \ref{asm:f_eps} in \Cref{asm:eps}, \Cref{lem:G}, and \ref{asm:nuf} and~\ref{asm:nualpha} in \Cref{asm:nu}. 
\end{enumerate}
\end{proof}
Moreover, the following lemma is concerned with the convergence of the individual terms associated with the fracture domain~$\Omega_\rmf^1$ in the flow equation~\eqref{eq:fulldim-weak-trafo-c}.
\begin{lemma} \label{lem:fracconv2}
Let the \Cref{asm:eps,asm:nu} hold.
\begin{subequations}
\begin{enumerate}[label=(\roman*),wide]
\item Let $\nu_\tsr{C} \ge 1$. Then, as $k\rightarrow \infty$, for all $\phi_\rmf \in H^1 (\Omega_\rmf^1 ) $, we have
\begin{align}
\label{eq:fracconv2_a} &\bigl\langle    \epsilon_k^{\matr{\nu}_{\!\matr{\alpha }} + \matr{I}_n}  \phi_\rmf \hat{\matr{\alpha}}_\rmf^{\epsilon_k}  \!  , \nabla^{\epsilon_k} \partial_t \hat{\vct{u}}_\rmf^{\epsilon_k} \bigr\rangle_{\Omega_\rmf^1  } \rightarrow \bigl\langle \phi_\rmf \hat{\vct{\alpha}}_\rmf^\mathrm{eff}\! , \partial_{\!\vct{N}} \partial_t \hat{\vct{u}}_\rmf^\#  \bigr\rangle_{ \Omega_\rmf^1  } .
\end{align}
\item For all $\phi_\rmf \in H^1 (\Omega_\rmf^1 )$ with $\nabla_{\!\vct{N}} \phi_\rmf \equiv \vct{0}$, we have
\begin{align}
\label{eq:fracconv2_b}\bigl\langle \epsilon_k^{\nu_\matr{K} +1 } \hat{\matr{K}}_\rmf^{\epsilon_k} \nabla^{\epsilon_k} \hat{p}_\rmf^{\epsilon_k} \! , \nabla^{\epsilon_k} \phi_\rmf \bigr\rangle_{\Omega_\rmf^1  } &\rightarrow \begin{cases}  \bigl\langle a \hat{\matr{K}}_\gamma \nabla (\mathfrak{A}_{\!\vct{N}} \hat{p}_\rmf^\# ) , \mathfrak{A}_{\!\vct{N}} \phi_\rmf \bigr\rangle_{\gamma  } &\text{if } \nu_\matr{K} = -1 , \\
0 &\text{if } \nu_\matr{K} > -1 \end{cases} 
\end{align}
as $k\rightarrow \infty$. In addition, for all $\phi_\rmf \in H^1 (\Omega_\rmf^1 )$, it is 
\begin{align}
\label{eq:fracconv2_c}\bigl\langle \epsilon_k^{\nu_\matr{K} +1 } \hat{\matr{K}}_\rmf^{\epsilon_k} \nabla^{\epsilon_k} \hat{p}_\rmf^{\epsilon_k} \! , \nabla^{\epsilon_k} \phi_\rmf \bigr\rangle_{\Omega_\rmf^1  } &\rightarrow \begin{cases} \bigl\langle \hat{K}_\rmf^{\!\vct{N}} \partial_{\!\vct{N}} \hat{p}_\rmf^\# \! , \partial_{\!\vct{N}} \phi_\rmf \bigr\rangle_{ \Omega_\rmf^1  } &\text{if } \nu_\matr{K} = 1 , \\
0 &\text{if } \nu_\matr{K} > 1 .
\end{cases} 
\end{align}
\item For all $\phi_\rmf \in H^1 (\Omega_\rmf^1 )$, as $k\rightarrow \infty$, we have
\begin{align}
\label{eq:fracconv2_d} \bigl\langle \epsilon_k^{\nu_\omega + 1 } \hat{\omega}_\rmf^{\epsilon_k} \partial_t \hat{p}_\rmf^{\epsilon_k}\! , \phi_\rmf \bigr\rangle_{ \Omega_\rmf^1  } &\rightarrow \bigl\langle \hat{\omega}_\rmf^\mathrm{eff} \partial_t \hat{p}_\rmf^\# \! , \phi_\rmf \bigr\rangle_{ \Omega_\rmf^1  }  .
\end{align}
\item As $\epsilon \rightarrow 0$, for all $\phi_\rmf \in H^1 (\Omega_\rmf^1 ) $, it is 
\begin{align}
\label{eq:fracconv2_e}\bigl\langle \epsilon^{\nu_q + 1 } \hat{q}_\rmf^{\epsilon} , \phi_\rmf \bigr\rangle_{\Omega_\rmf^1  } &\rightarrow \bigl\langle \hat{q}_\rmf^\mathrm{eff} \! , \phi_\rmf \bigr\rangle_{\Omega_\rmf^1  }  . 
\end{align}
\end{enumerate}
\end{subequations}
\end{lemma}
\begin{proof}
\begin{enumerate}[label=(\roman*),wide]
\item 
We rewrite the term~$\bigl\langle   \epsilon_k^{\matr{\nu}_{\!\matr{\alpha }} + \matr{I}_n}   \phi_\rmf \hat{\matr{\alpha}}_\rmf^{\epsilon_k} \! , \nabla^{\epsilon_k} \partial_t \hat{\vct{u}}_\rmf^{\epsilon_k} \bigr\rangle_{\Omega_\rmf^1  }$ as
\begin{align*}
\bigl\langle  \epsilon_k^{\nu_{\!\matr{\alpha}}^\smallpar + 1}  \phi_\rmf \hat{\matr{\alpha}}_\rmf^{\epsilon_k} \!  , \nablapar \partial_t \hat{\vct{u}}_\rmf^{\epsilon_k} \bigr\rangle_{\Omega_\rmf^1  } 
+ \bigl\langle \epsilon_k^{\nu_{\!\matr{\alpha}}^\smallperp}   \phi_\rmf \hat{\matr{\alpha}}_\rmf^{\epsilon_k} \! , \nablaperp \partial_t \hat{\vct{u}}_\rmf^{\epsilon_k} \bigr\rangle_{\Omega_\rmf^1  } .
\end{align*}
Then, \cref{eq:fracconv2_a} follows from \cref{eq:apriori_a,eq:apriori_b} in \Cref{thm:apriori}, \cref{eq:conv_u_b} in \Cref{cor:conv}, \ref{asm:alpha_eps} in \Cref{asm:eps}, and \ref{asm:nualpha} in \Cref{asm:nu}. In particular, we have 
\begin{align*}
\epsilon_k^{\iota( \nu_\tsr{C}) } \hat{\vct{u}}_\rmf^{\epsilon_k} \rightarrow \vct{0} \quad\text{strongly in } H^1 (I ; \vct{H}^1_{\!\vct{N}} (\Omega_\rmf^1 ) ) \quad\ \text{and} \quad \ \norm[\big ]{\epsilon_k^{\iota ( \nu_\tsr{C} ) }\nablapar \hat{\vct{u}}_\rmf^{\epsilon_k} }_{H^1 ( I ; \vct{L}^2 ( \Omega_\rmf^1 ) ) } \lesssim 1
\end{align*}
for $\nu_\tsr{C} \ge 1$ and hence $\epsilon_k^{\iota ( \nu_\tsr{C} ) } \hat{\vct{u}}_\rmf^{\epsilon_k} \rightharpoonup \vct{0}$ weakly in $H^1 (I ; \vct{H}^1 (\Omega_\rmf^1 ) ) $, where $\iota ( \nu ) := \tfrac{1}{2} ( \nu + 1)$.
\item 
We first prove \cref{eq:fracconv2_b}. For $\phi_\rmf \in H^1 (\Omega_\rmf^1)$ with $\nablaperp \phi_\rmf \equiv \vct{0}$, we have 
\begin{multline*}
\bigl\langle \epsilon_k^{\nu_\matr{K} +1 } \hat{\matr{K}}_\rmf^{\epsilon_k} \nabla^{\epsilon_k} \hat{p}_\rmf^{\epsilon_k}\! , \nabla^{\epsilon_k} \phi_\rmf \bigr\rangle_{\Omega_\rmf^1  } \\
 =  \bigl\langle \epsilon_k^{\nu_\matr{K} +1 } \hat{\matr{K}}_\rmf^{\epsilon_k} \nablapar \hat{p}_\rmf^{\epsilon_k}\! , \nablapar \phi_\rmf \bigr\rangle_{\Omega_\rmf^1  } 
+ \bigl\langle \epsilon_k^{\nu_\matr{K}  } \hat{\matr{K}}_\rmf^{\epsilon_k} \nablaperp \hat{p}_\rmf^{\epsilon_k} \! , \nablapar \phi_\rmf \bigr\rangle_{\Omega_\rmf^1  } .
\end{multline*}
Here, with the estimate~\eqref{eq:apriori_e} from \Cref{thm:apriori}, the weak convergence \eqref{eq:conv_p_c} from \Cref{cor:conv}, and \ref{asm:K_eps} in  \Cref{asm:eps}, we find
\begin{align}
 \bigl\langle \epsilon_k^{\nu_\matr{K} +1 } \hat{\matr{K}}_\rmf^{\epsilon_k} \nablapar \hat{p}_\rmf^{\epsilon_k}\! , \nablapar \phi_\rmf \bigr\rangle_{\Omega_\rmf^1  }  &\rightarrow \begin{cases}
 \bigl\langle \hat{\matr{K}}_\rmf \nablapar \hat{p}_\rmf^\# \! , \nablapar \phi_\rmf \bigr\rangle_{\Omega_\rmf^1  } &\text{if } \nu_\matr{K} = - 1, \\
 0 &\text{if } \nu_\matr{K} > -1.
 \end{cases} 
\label{eq:fracconv_ii_proof1}%
\end{align}%
Besides, for $\nu_\matr{K} = -1$, we have
\begin{align*}
\norm[\big]{\epsilon^{-1} \nabla_{\!\vct{N}} \hat{p}_\rmf^\epsilon }_{L^\infty( I ; \vct{L}^2 (\Omega_\rmf^1 )) } \lesssim 1
\end{align*}
using \cref{eq:apriori_e} in \Cref{thm:apriori}.
Therefore, there exists $\hat{\xi}^\# \in L^\infty(I ; L^2 (\Omega_\rmf^1 ))$ such that 
\begin{align*}
\epsilon_k^{-1} \nabla_{\!\vct{N}} \hat{p}_\rmf^{\epsilon_k} \rightharpoonup \hat{\xi}^\# \vct{N} \quad\ \text{weakly in } L^2 ( I ; \vct{L}^2 (\Omega_\rmf^1 )).
\end{align*}
Now, for $\nu_\matr{K} = -1$, we multiply \cref{eq:fulldim-weak-trafo-c} with~$\epsilon_k$ for $\phi_\pmf \in \Phi$ and take the limit $k\rightarrow\infty$ to obtain
\begin{align}
\bigl\langle \hat{\matr{K}}_\rmf \nablapar \hat{p}_\rmf^\# , \nablaperp \phi_\rmf \bigr\rangle_{\Omega_\rmf^1  } + \bigl\langle \hat{\xi}^\# \hat{\matr{K}}_\rmf \vct{N} , \nabla_{\!\vct{N}} \phi_\rmf \bigr\rangle_{\Omega_\rmf^1 } = 0 . \label{eq:xi_eq}
\end{align}
Clearly, one solution of \cref{eq:xi_eq} is 
\begin{align}
\hat{\xi}^\# = -\bigl( \hat{K}_\rmf^{\!\vct{N}} \bigr)^{\! -1} \hat{\matr{K}}_\rmf \nablapar \hat{p}_\rmf^\# \cdot \vct{N} . \label{eq:xi}
\end{align}
Given another solution~$\bar{\xi}^\# \in L^\infty ( I ; L^2 ( \Omega_\rmf^1 ) ) $ and choosing 
\begin{align*}
\phi_\rmf (\vct{x} ) = \phi_\rmf ( \vct{y} + x_\vct{N}\vct{N} ) := \int^{x_\vct{N}}_{-a_-(\vct{y})} \bigl( \hat{\xi}^\# - \bar{\xi}^\# \bigr) ( \vct{y} + s \vct{N} ) \,\rmd s
\end{align*} 
in \cref{eq:xi_eq}, where $\vct{y} := \vct{x} - x_\vct{N} \vct{N}$ and $x_\vct{N} := \vct{x} \cdot\vct{N}$, now yields
\begin{align*}
\bigl\langle  \hat{K}_\rmf^{\!\vct{N}}  \bigl( \hat{\xi}^\# - \bar{\xi}^\# \bigr) , \hat{\xi}^\# - \bar{\xi}^\# \bigr\rangle_{\Omega_\rmf^1  } = 0
\end{align*}
and hence $\hat{\xi}^\# = \bar{\xi}^\#$. 
As a result, we obtain
\begin{equation}
\begin{multlined}[c][0.875\displaywidth]
 \bigl\langle \epsilon_k^{\nu_\matr{K}  } \hat{\matr{K}}_\rmf^{\epsilon_k} \nablaperp \hat{p}_\rmf^{\epsilon_k} \! , \nablapar \phi_\rmf \bigr\rangle_{\Omega_\rmf^1  } \\ \rightarrow \begin{cases}
- \bigl\langle \bigl( \hat{K}_\rmf^{\!\vct{N}} \bigr)^{\! -1} \bigl( \hat{\matr{K}}_\rmf \vct{N} \otimes \hat{\matr{K}}_\rmf \vct{N} \bigr)  \nablapar \hat{p}_\rmf^\# ,  \nablapar \phi_\rmf \bigr\rangle_{\Omega_\rmf^1} &\text{if } \nu_\matr{K} = -1 , \\ 0 &\text{if } \nu_\matr{K} > -1 , \end{cases}
\end{multlined}%
\label{eq:fracconv_ii_proof2}
\end{equation}%
which, using \Cref{cor:const}, completes the proof of \cref{eq:fracconv2_b}.

Next, we derive \cref{eq:fracconv2_c}. 
For $\phi_\rmf \in H^1 (\Omega_\rmf^1)$ and $\nu_\matr{K} \ge 1$, we have 
\begin{multline*}
\bigl\langle \epsilon_k^{\nu_\matr{K} +1 } \hat{\matr{K}}_\rmf^{\epsilon_k} \nabla^{\epsilon_k} \hat{p}_\rmf^{\epsilon_k}\! , \nabla^{\epsilon_k} \phi_\rmf \bigr\rangle_{\Omega_\rmf^1  } \\
= \bigl\langle \epsilon_k^{\nu_\matr{K}  } \hat{\matr{K}}_\rmf^{\epsilon_k} \nablapar \hat{p}_\rmf^{\epsilon_k} \! , \nablaperp \phi_\rmf \bigr\rangle_{\Omega_\rmf^1  } 
+ \bigl\langle \epsilon_k^{\nu_\matr{K} -1 } \hat{\matr{K}}_\rmf^{\epsilon_k} \nablaperp \hat{p}_\rmf^{\epsilon_k} \!  , \nablaperp \phi_\rmf \bigr\rangle_{\Omega_\rmf^1  } + \smallo ( \epsilon_k ) ,
\end{multline*}
where we have used \cref{eq:fracconv_ii_proof1,eq:fracconv_ii_proof2}.
With the estimate~\eqref{eq:apriori_e} from \Cref{thm:apriori}, the weak convergence~\eqref{eq:conv_p_d} from \Cref{cor:conv}, and \ref{asm:K_eps} in \Cref{asm:eps}, we find 
\begin{align*}
 \bigl\langle \epsilon_k^{\nu_\matr{K} -1 } \hat{\matr{K}}_\rmf^{\epsilon_k} \nablaperp \hat{p}_\rmf^{\epsilon_k} \!  , \nablaperp \phi_\rmf \bigr\rangle_{\Omega_\rmf^1  } &\rightarrow \begin{cases} 
\bigl\langle  \hat{\matr{K}}_\rmf \nablaperp \hat{p}_\rmf^\# \!  , \nablaperp \phi_\rmf \bigr\rangle_{\Omega_\rmf^1  } &\text{if } \nu_\matr{K} = 1 , \\
 0 &\text{if } \nu_\matr{K} > 1 .
 \end{cases}
\end{align*}
In addition, for $\nu_\matr{K} = 1$, we have
\begin{align*}
\norm{\hat{p}_\rmf^\epsilon }_{L^\infty( I ; H^1_{\!\vct{N}} (\Omega_\rmf^1 )) } + \epsilon \norm{\nabla_{\!\smallpar}  \hat{p}_\rmf^\epsilon}_{L^\infty ( I ; \vct{L}^2 (\Omega_\rmf^1 )) } \lesssim 1
\end{align*}
with \cref{eq:apriori_d} and~\eqref{eq:apriori_e} in \Cref{thm:apriori}.
Thus, there exists $\vct{\eta}^\#_\rmf \in L^2 ( I ; H^1 (\Omega_\rmf^1 ) ) $ such that $\epsilon_k \hat{p}_\rmf^{\epsilon_k} \rightharpoonup\eta^\#_\rmf$ weakly in $L^2 ( I ; H^1 (\Omega_\rmf^1 ) )$.
Besides, we have $\epsilon_k \hat{p}_\rmf^{\epsilon_k} \rightarrow 0$ strongly in $L^2 ( I ; \vct{H}^1_{\!\vct{N}} (\Omega_\rmf^1 ) )$ and hence $\smash{\eta^\#_\rmf} = 0$. 
Consequently, it is
\begin{align*}
 \bigl\langle \epsilon_k^{\nu_\matr{K}  } \hat{\matr{K}}_\rmf^{\epsilon_k} \nablapar \hat{p}_\rmf^{\epsilon_k} \! , \nablaperp \phi_\rmf \bigr\rangle_{\Omega_\rmf^1  } &\rightarrow 0 \quad\ \text{if } \nu_\matr{K} \ge 1
\end{align*}
so that \cref{eq:fracconv2_c} follows.
\item \cref{eq:fracconv2_d} follows directly from \cref{eq:apriori_d} in \Cref{thm:apriori}, \cref{eq:conv_p_e} in \Cref{cor:conv}, and \ref{asm:omega_eps} in \Cref{asm:eps}.
\item The result follows directly from \ref{asm:q_eps} in \Cref{asm:eps}.
\end{enumerate}
\end{proof}

We can now piece together the proof of the \Cref{thm:Klm1,thm:Km1,thm:Km1p1,thm:Kp1,thm:Kgp1}.

\begin{proof}[Proof of the \Cref{thm:Klm1,thm:Km1,thm:Km1p1,thm:Kp1,thm:Kgp1}]
First, as a consequence of the \Cref{cor:conv,cor:const,prop:ptrace,prop:uftrace}, we find that the limit functions belong to the stated solution spaces.
Next, we verify that limit solutions solve the respective limit problems as stated by applying the \Cref{lem:bulkconv,lem:fracconv1,lem:fracconv2} to pass to the limit~$k\rightarrow \infty$ in \cref{eq:fulldim-weak-trafo}. 
Here, the choice of test functions is important and depends on the choice of the scaling coefficients~$\nu_\matr{K}$ and~$\nu_\tsr{C}$:
\begin{itemize}
\item Case $\nu_\matr{K} <-1$: We choose $\phi_\pmf \in \Phi$ with $\phi_\rmf \equiv \phi_\gamma \in W$ in \cref{eq:fulldim-weak-trafo}.
Then, we use \cref{eq:bulkconv_c,eq:bulkconv_d,eq:bulkconv_e,eq:bulkconv_h} in \Cref{lem:bulkconv} and \cref{eq:fracconv2_a,eq:fracconv2_d,eq:fracconv2_e} in \Cref{lem:fracconv2} to take the limit~$k\rightarrow\infty$ in the flow equation~\eqref{eq:fulldim-weak-trafo-c}.
\item Case $\nu_\matr{K} \in [ -1 , 1)$: We choose $\phi_\pmf\in \Phi$ with $\nablaperp \phi_\rmf = \vct{0}$ in \cref{eq:fulldim-weak-trafo}.
Then, we use \cref{eq:bulkconv_c,eq:bulkconv_d,eq:bulkconv_e,eq:bulkconv_h} in \Cref{lem:bulkconv} and \cref{eq:fracconv2_a,eq:fracconv2_b,eq:fracconv2_d,eq:fracconv2_e} in \Cref{lem:fracconv2} to take the limit~$k\rightarrow\infty$ in the flow equation~\eqref{eq:fulldim-weak-trafo-c}.
\item Case $\nu_\matr{K} \ge 1$: We choose $\phi_\pmf\in \Phi$ without restrictions in \cref{eq:fulldim-weak-trafo}. 
Then, we use \cref{eq:bulkconv_c,eq:bulkconv_d,eq:bulkconv_e,eq:bulkconv_h} in \Cref{lem:bulkconv} and \cref{eq:fracconv2_a,eq:fracconv2_c,eq:fracconv2_d,eq:fracconv2_e} in \Cref{lem:fracconv2} to take the limit~$k\rightarrow\infty$ in the flow equation~\eqref{eq:fulldim-weak-trafo-c}.
\item Case $\nu_\tsr{C} = 1$: We choose $\vct{v}_\pmf\in \vct{V}$ without restrictions in \cref{eq:fulldim-weak-trafo}. Then, we use \cref{eq:bulkconv_a,eq:bulkconv_b,eq:bulkconv_f,eq:bulkconv_g} in \Cref{lem:bulkconv} and \cref{eq:fracconv1_a}, \cref{eq:fracconv1_c} or \cref{eq:fracconv1_e}, \cref{eq:fracconv1_g}, and \cref{eq:fracconv1_i} in \Cref{lem:fracconv1} to take the limit~$k\rightarrow \infty$ in the mechanics equation~\cref{eq:fulldim-weak-trafo-b}.
\item Case $\nu_\tsr{C} > 1$: First, as for $\nu_\tsr{C} = 1$, we choose $\vct{v}_\pmf\in \vct{V}$ without restrictions and use \cref{eq:bulkconv_a,eq:bulkconv_b,eq:bulkconv_f,eq:bulkconv_g} in \Cref{lem:bulkconv} and \cref{eq:fracconv1_a}, \cref{eq:fracconv1_c} or \cref{eq:fracconv1_e}, \cref{eq:fracconv1_g}, and \cref{eq:fracconv1_i} in \Cref{lem:fracconv1} to take the limit~$k\rightarrow \infty$ in the mechanics equation~\cref{eq:fulldim-weak-trafo-b}.
Then, we proceed in a second step and choose $\vct{v}_\pm = \vct{0}$ and $\vct{v}_\rmf = \epsilon^{ (1-\nu_\tsr{C}) /2} \bar{\vct{v}}_\rmf$ with $\bar{\vct{v}}_\rmf \in \vct{H}^1_0 (\Omega_\rmf^1 ) $ in \cref{eq:fulldim-weak-trafo-b}.
We employ \cref{eq:fracconv1_b}, \cref{eq:fracconv1_d} or \cref{eq:fracconv1_f}, \cref{eq:fracconv1_h}, and \cref{eq:fracconv1_j} in \Cref{lem:fracconv1} to pass to the limit~$k\rightarrow\infty$.
\end{itemize}
\noindent Further, using \Cref{lem:eff1}, the uniqueness of the solution follows in analogy to the proof of~\Cref{thm:wellposed_biot_sd} in Appendix~\ref{sec:A}.
This directly implies the weak convergence of the entire sequences $\{ \hat{\vct{u}}_\pmf^\epsilon \}_{\epsilon \in ( 0,\epsilon^\ast]}$ and $\{ \hat{p}_\pmf^\epsilon \}_{\epsilon \in ( 0,\epsilon^\ast]}$, as any weakly convergent subsequence must convergence to the same unique limit.
\end{proof}

\section{Proof of \texorpdfstring{\Cref{thm:strongconv}}{Theorem 3.6}} \label{sec:7}
In the following, exploiting the knowledge of the limit models in \Cref{sec:4}, we prove the strong convergence results in \Cref{thm:strongconv}.
As a first step, we derive the following strong convergence result for the initial displacement vector. 
\begin{lemma} 
\label{lem:u0conv}
Let $\nu_\tsr{C} \ge 1$ and assume that the \Cref{asm:eps,asm:nu} hold. 
In addition, if $\nu_\matr{K} > 1$ and $\nu_\omega > -1$, let $2 \nu_{\!\matr{\alpha}}^\smallpar >  \nu_\tsr{C} + \min\{ \nu_\omega - 2, \nu_\matr{K} - 4 \}$ and $2 \nu_{\!\matr{\alpha}}^\smallperp > \nu_\tsr{C} + \min\{ \nu_\omega , \nu_\matr{K} -2  \}$.
Let $\hat{\vct{u}}_{0,\pmf}^\epsilon := \hat{\vct{u}}_\pmf^\epsilon ( 0)$ and $\smash{\hat{\vct{u}}_{0,\pmf}^\#} := \smash{\hat{\vct{u}}_\pmf^\#} (  0)$.
Then, we have
\begin{subequations}
\begin{alignat}{2}
\hat{\vct{u}}_{0,\pm}^{\epsilon_k} &\rightarrow \hat{\vct{u}}_{0,\pm}^\# \quad\enspace &&\text{strongly in } \vct{H}^1 (\Omega_\pm^0 ) , \\
\epsilon_k^{\iota (\nu_\tsr{C} ) } \matr{e}^{\epsilon_k}\bigl( \hat{\vct{u}}_{0,\rmf}^{\epsilon_k} \bigr) &\rightarrow \matr{e}_{\!\vct{N}} \bigl( \hat{\vct{u}}_{0,\rmf}^\#  \bigr) \quad\enspace &&\text{strongly in } \matr{L}^2 (\Omega_\rmf^1 ) ,
\end{alignat}%
\label{eq:strongu0}%
\end{subequations}%
where $2\matr{e}_{\!\vct{N}} (\vct{v} ) := \nablaperp \vct{v} + ( \nablaperp \vct{v})^\rmt$ and $\iota ( \nu ) := \tfrac{1}{2} ( \nu + 1 )$.
\end{lemma}%
\begin{proof}
We define the space
\begin{align}
\label{eq:L2sym}
 \matr{L}^2_{\mathrm{sym}} (\Omega_\rmf^1 ) := \{ \matr{M} \in \matr{L}^2 (\Omega_\rmf^1 ) \ \vert\ \matr{M} \text{ is symmetric a.e.}\} 
\end{align}
as a subspace of~$\matr{L}^2 (\Omega_\rmf^1)$. Then, as a consequence of \ref{asm:C_eps} in \Cref{asm:eps},
\begin{align}
\label{eq:normiii_new}
\normiii[\big]{ ( \vct{v}_\pm , \matr{M} ) }^2 := \bigl\langle \hat{\mathbb{C}}_\pm \matr{e} ( \vct{v}_\pm ) , \matr{e} ( \vct{v}_\pm ) \bigr\rangle_{\Omega_\pm^0  } + \bigl\langle \hat{\tsr{C}}_\rmf  \matr{M}  ,  \matr{M} \bigr\rangle_{\Omega_\rmf^1 }
\end{align}
defines an equivalent norm on~$\smash{\vct{H}^1_{0, \rho_\pm^0}\! (\Omega_\pm^0 )} \times \matr{L}^2_{\mathrm{sym}} (\Omega_\rmf^1)$.
As  a consequence of the \Cref{lem:poincare_decomp,lem:korn_decomp,lem:apriori_u0}, we have 
\begin{align*}
\norm[\big]{\hat{\vct{u}}_{0,\pm}^\epsilon }_{\vct{H}^1 (\Omega_\pm^0 ) } + \norm[\big]{\epsilon^{\frac{\nu_\tsr{C} - 1}{2}} \hat{\vct{u}}_{0,\rmf}^\epsilon }_{\vct{H}^1_{\!\vct{N}} (\Omega_\rmf^1 )} + \norm[\big]{\epsilon^{\iota ( \nu_\tsr{C} ) } \nablapar \hat{\vct{u}}_{0,\rmf}^\epsilon }_{\vct{L}^2 (\Omega_\rmf^1 )} \lesssim 1.
\end{align*}
Thus, the strong convergences in \cref{eq:strongu0} already hold weakly and, as $k \rightarrow \infty$, we have 
\begin{align*}
\epsilon_k^{\iota ( \nu_\tsr{C})} \nablapar \hat{\vct{u}}_{0, \rmf}^{\epsilon_k} \rightharpoonup \matr{0}, \quad \epsilon_k^{\frac{\nu_\tsr{C} - 1}{2}} \nablaperp \hat{\vct{u}}_{0, \rmf}^{\epsilon_k} \rightharpoonup \nablaperp \hat{\vct{u}}_{0,\rmf}^\#   \quad\enspace \text{weakly in } \vct{L}^2 (\Omega_\rmf^1 ) .
\end{align*}
In particular, using \Cref{lem:G} and \ref{asm:alpha_eps}, \ref{asm:f_eps}, and~\ref{asm:p0_eps} in \Cref{asm:eps}, we obtain
\begin{align*}
\bigl\langle\bigl( \hat{p}_{0,\pm}^{\epsilon_k} + \hat{G}_\pm^{\epsilon_k}  \bigr) \hat{\matr{\alpha}}_\pm^{\epsilon_k} , \nabla \hat{\vct{u}}_{0,\pm}^{\epsilon_k} \bigr\rangle_{\Omega_\pm^0} & \rightarrow \bigl\langle \bigl( \hat{p}_{0, \pm} + \hat{G}_\pm \bigr) \hat{\matr{\alpha}}_\pm , \nabla \hat{\vct{u}}_{0,\pm}^\# \bigr\rangle_{\Omega_\pm^0} , \\
\bigl\langle \epsilon_k^{\matr{\nu}_{\!\matr{\alpha}} + \matr{I}_n} \bigl( \hat{p}_{0,\rmf}^{\epsilon_k} + \hat{G}_\rmf^{\epsilon_k}  \bigr) \hat{\matr{\alpha}}_\rmf^{\epsilon_k} , \nabla^\epsilon \hat{\vct{u}}_{0,\rmf}^{\epsilon_k} \bigr\rangle_{\Omega_\rmf^1} & \rightarrow \bigl\langle \bigl( \hat{p}_{0, \rmf} + \hat{G}_\rmf \bigr) \hat{\vct{\alpha}}_\rmf^\eff \! , \partial_{\!\vct{N}} \hat{\vct{u}}_{0,\rmf}^\# \bigr\rangle_{\Omega_\rmf^1} , \\
\bigl\langle \fracfac[\big ]{\epsilon_k^{\nu_\vct{f} +1 }} \hat{\vct{f}}_\pmf^{\epsilon_k} (0) , \hat{\vct{u}}_{0,\pmf}^{\epsilon_k} \bigr\rangle & \rightarrow \bigl\langle \hat{\vct{f}}_\pmf^{\fracfac{\! \eff\! } } (0) , \hat{\vct{u}}_{0,\pmf}^\# \bigr\rangle .
\end{align*}
Thus, with \cref{eq:fulldim-weak-trafo-d}, \ref{asm:C_eps} in \Cref{asm:eps}, and \Cref{lem:apriori_u0}, we find
\begin{align*}
&\normiii[\big]{ \bigl(  \hat{\vct{u}}_{0,\pm}^{\epsilon_k} ,\, \epsilon_k^{\iota ( \nu_\tsr{C} )} \matr{e}^{\epsilon_k} ( \hat{\vct{u}}_{0, \rmf }^{\epsilon_k} ) \bigr) }^2 +\smallo (\epsilon_k ) \\
&\hspace{1cm}= \bigl\langle \fracfac{\epsilon_k^{\nu_\tsr{C} + 1}} \hat{\tsr{C}}_\pmf^{\epsilon_k} \matr{E}_\pmf^\epsilon ( \hat{\vct{u}}_{0, \pmf}^{\epsilon_k} ) , \matr{E}_\pmf^\epsilon ( \hat{\vct{u}}_{0,\pmf}^{\epsilon_k} )  \bigr\rangle   \\
&\hspace{1cm}= \bigl\langle \fracfac[\big ]{\epsilon_k^{\matr{\nu}_{\!\matr{\alpha}} + \matr{I}_n}} \bigl( \hat{p}_{0,\pmf}^{\epsilon_k} + \hat{G}_\pmf^{\epsilon_k}  \bigr) \hat{\matr{\alpha}}_\pmf^{\epsilon_k} , \nablaeps \hat{\vct{u}}_{0,\pmf}^{\epsilon_k} \bigr\rangle + \bigl\langle \fracfac[\big ]{\epsilon_k^{\nu_\vct{f} +1 }} \hat{\vct{f}}_\pmf^{\epsilon_k} (0) , \hat{\vct{u}}_{0,\pmf}^{\epsilon_k} \bigr\rangle \\
&\hspace{1cm}\rightarrow \bigl\langle \hat{\mathbb{C}}_\pm \matr{e} ( \hat{\vct{u}}_{0,\pm}^\# ) , \matr{e} ( \hat{\vct{u}}_{0,\pm}^\# ) \bigr\rangle_{\Omega_\pm^0  } + \bigl\langle \hat{\matr{C}}^{\!\vct{N}}_\rmf \partial_{\!\vct{N}} \hat{\vct{u}}_{0,\rmf}^\#, \partial_{\!\vct{N}}  \hat{\vct{u}}_{0,\rmf}^\# \bigr\rangle_{\Omega_\rmf^1} = \normiii[\big]{ \bigl(   \hat{\vct{u}}_{0,\pm}^\# ,\, \matr{e}_{\!\vct{N}} ( \hat{\vct{u}}_{0,\rmf }^\# ) \bigr) }^2
\end{align*}
as $k \rightarrow \infty$, where we have used either \cref{eq:weak_mechlimit_C1} if $\nu_\tsr{C} = 1$ or \cref{eq:weak_mechlimit_nuCg1} if $\nu_\tsr{C} > 1$.
\end{proof}

We can now prove \Cref{thm:strongconv}. We first consider the case $\nu_\tsr{C} = 1$ and $\nu_\matr{K} = -1$. 
After completing the proof, we briefly indicate how the remaining cases can be handled analogously.
\begin{proof}[Proof of \Cref{thm:strongconv} (case $\nu_\tsr{C} = 1$, $\nu_\matr{K} = -1$)]
\begin{enumerate}[label=(\roman*),wide]
\item We define the norms 
\begin{subequations}
\begin{align}
\label{eq:norm_i}
\begin{split}
&\Phi \times L^2 (\Omega_\rmf^1 ) \rightarrow \mathbb{R}_0^+ , \enspace\ \norm[\big ]{\bigl( \phi_\pmf , \xi \bigr) }_\mathrm{i}^2 := \bigl\langle \hat{\matr{K}}_\pm \nabla \phi_\pm , \nabla \phi_\pm \bigr\rangle_{\Omega_\pm^0 }
 \\ &\hspace{2.25cm} + \, \bigl\langle \hat{\matr{K}}_\rmf \nablaperp \phi_\rmf , \nablaperp \phi_\rmf \bigr\rangle_{\Omega_\rmf^1}   +   \bigl\langle \hat{\matr{K}}_\rmf \bigl[ \nablapar \phi_\rmf + \xi \vct{N} \bigr] , \nablapar \phi_\rmf + \xi \vct{N} \bigr\rangle_{\Omega_\rmf^1}  , 
\end{split} \\ 
\label{eq:norm_ii} &\Lambda \rightarrow \mathbb{R}_0^+ , \enspace\  \norm{\phi_\pmf }_\mathrm{ii}^2 := \bigl\langle \hat{\omega}_\pmf^{\fracfac{\!\mathrm{eff}\!}} \phi_\pmf ,  \phi_\pmf \bigr\rangle .
\end{align}%
\label{eq:norm_i_ii}%
\end{subequations}%
Note that $\norm{\cdot }_\mathrm{ii}$ is only a norm on~$\Lambda$ if $\nu_\omega = -1$. Otherwise, if $\nu_\omega > -1$, $\norm{\cdot }_\mathrm{ii}$ defines a norm on~$L^2 (\Omega_\pm^0 )$. 
In addition, we use the norm $\normiii{\cdot} $ on $\smash{\vct{H}^1_{0 ,\rho_\pm^0}\! (\Omega_\pm^0 )}  \times \matr{L}^2_{\mathrm{sym}} (\Omega_\rmf^1 )$ as defined in \cref{eq:normiii_new} with the space~$\matr{L}^2_{\mathrm{sym}}(\Omega_\rmf^1 )$ given by~\cref{eq:L2sym}.

Next, we choose $\vct{v}_\pmf = \partial_t \hat{\vct{u}}_\pmf^\epsilon $ and $\phi_\pmf = \hat{p}_\pmf^\epsilon$ in \cref{eq:fulldim-weak-trafo} and integrate from $0$ to~$t\in\bar{I}$ to find
\begin{align}
\label{eq:thm_C1Klm1_1}
\begin{split}
&\tfrac{1}{2} \bigl\langle \fracfac{\epsilon^{2}} \hat{\tsr{C}}_\pmf^\epsilon \matr{E}_\pmf^\epsilon ( \hat{\vct{u}}_\pmf^\epsilon ) , \matr{E}_\pmf^\epsilon ( \hat{\vct{u}}_\pmf^\epsilon )  \bigr\rangle \bigr\vert_0^t + \tfrac{1}{2} \bigl\langle \fracfac{\epsilon^{\nu_\omega +1 }}  \hat{\omega}_\pmf^\epsilon  \hat{p}_\pmf^\epsilon  , \hat{p}_\pmf^\epsilon \bigr\rangle \bigr\vert_0^t \\
&\ + \int_0^t \bigl\langle  \hat{\matr{K}}_\pmf^\epsilon \matr{M}_\pmf^\epsilon \nablaeps \hat{p}_\pmf^\epsilon (\bar{t})  , \matr{M}_\pmf^\epsilon \nabla^\fraceps \hat{p}_\pmf^\epsilon (\bar{t}) \bigr\rangle \,\rmd\bar{t} = \int_0^t \Bigl[   \bigl\langle \fracfac{\epsilon^{\nu_q + 1}} \hat{q}_\pmf^\epsilon (\bar{t}) , \hat{p}_\pmf^\epsilon (\bar{t}) \bigr\rangle  \\
&\quad\enspace + \bigl\langle \fracfac{\epsilon^{\nu_{\!\vct{f}} + 1} }\hat{\vct{f}}_\pmf^\epsilon (\bar{t}) , \partial_t \hat{\vct{u}}_\pmf^\epsilon (\bar{t}) \bigr\rangle   + \bigl\langle  \fracfac{\epsilon^{\!\matr{\nu}_{\!\matr{\alpha}} + \matr{I}_n \!  } } \hat{G}_\pmf^\epsilon \hat{\matr{\alpha}}_\pmf^\epsilon , \nabla^\fraceps \partial_t \hat{\vct{u}}_\pmf^\epsilon (\bar{t}) [\matr{M}_\pmf^\epsilon ]^\rmt \bigr\rangle  \Bigr] \,\rmd\bar{t} .
\end{split}
\end{align}
In the same way, by choosing $\phi_\pm = \hat{p}_\pm^\#$, $\phi_\gamma = \mathfrak{A}_\vct{N} \hat{p}_\rmf^\#$ in \cref{eq:weak_C1_Km1}, as well as $\vct{v}_\pmf = \partial_t \hat{\vct{u}}_\pmf^\#$ in \cref{eq:weak_mechlimit_C1}, we obtain  
\begin{align}
\label{eq:thm_C1Klm1_2}
\begin{split}
&\tfrac{1}{2} \bigl\langle \hat{\mathbb{C}}_\pm \matr{e} ( \hat{\vct{u}}_\pm^\# ) , \matr{e} ( \hat{\vct{u}}_\pm^\# ) \bigr\rangle_{\Omega_\pm^0  }\bigr\vert_0^t + \tfrac{1}{2} \bigl\langle \hat{\matr{C}}^{\!\vct{N}}_\rmf \partial_{\!\vct{N}} \hat{\vct{u}}_\rmf^\# \! , \partial_{\!\vct{N}} \hat{\vct{u}}_\rmf^\# \bigr\rangle_{\Omega_\rmf^1 } \bigr\vert_0^t  \\
&\hspace{0.75cm}  + \tfrac{1}{2} \bigl\langle  \hat{\omega}_\pmf^{\fracfac{\! \eff\!}}  \hat{p}_\pmf^\# , \hat{p}_\pmf^\# \bigr\rangle \bigr\vert_0^t 
+ \int_0^t  \Bigl[ \bigl\langle \hat{\matr{K}}_\pm \nabla \hat{p}_\pm^\# (\bar{t}) , \nabla \hat{p}_\pm^\# (\bar{t}) \bigr\rangle_{ \Omega_\pm^0  } \\
&\hspace{1.5cm}  +\, \bigl\langle \hat{\matr{K}}_\rmf \bigl( \nablapar \hat{p}_\rmf^\# (\bar{t}) + \hat{\xi}^\# (\bar{t}) \vct{N} \bigr) , \bigl( \nablapar \hat{p}_\rmf^\# (\bar{t}) + \hat{\xi}^\# (\bar{t}) \vct{N} \bigr) \bigr\rangle_{\Omega_\rmf^1} \Bigr]   \,\rmd \bar{t} \\
&\hspace{2.75cm} = \int_0^t \Bigl[ \bigl\langle \hat{\vct{f}}_\pmf^{\fracfac{\!\eff\! }} (\bar{t}) ,  \partial_t \hat{\vct{u}}_\pmf^\# (\bar{t}) \bigr\rangle  
 +  \bigl\langle \hat{G}_\pm \hat{\matr{\alpha}}_\pm , \nabla  \partial_t \hat{\vct{u}}_\pm^\# (\bar{t})  \bigr\rangle_{\Omega_\pm^0 }  \\
&\hspace{4cm} + \bigl\langle \hat{G}_\rmf \hat{\vct{\alpha}}_\rmf^\eff \! , \partial_{\!\vct{N}} \partial_t \hat{\vct{u}}_\rmf^\# ( \bar{t} ) \bigr\rangle_{\Omega_\rmf^1} 
 + \bigl\langle \hat{q}_\pmf^{\fracfac{\!\mathrm{eff}\!}} (\bar{t}) , \hat{p}_\pmf^\# (\bar{t}) \bigr\rangle  \Bigr] \,\rmd\bar{t}
\end{split}
\end{align}
after integrating from~$0$ to~$t\in\bar{I}$, where $\smash{\hat{\xi}^\#}$ is given by \cref{eq:xi}.
Consequently, for $t \in \bar{I}$, we have 
\begin{align*}
&\int_0^t \norm[\big ]{ \bigl( \hat{p}_\pmf^{\epsilon_k} (\bar{t}) , \epsilon_k^{-1} \partial_{\!\vct{N}} \hat{p}_\rmf^{\epsilon_k} (\bar{t} )  \bigr) }_\mathrm{i}^2 \,\rmd \bar{t} + \tfrac{1}{2}\norm[\big ]{\hat{p}_\pmf^{\epsilon_{k}} (t) }_\mathrm{ii}^2 \\
&\hspace{4cm} + \tfrac{1}{2}\normiii[\big ]{\bigl( \hat{\vct{u}}_\pm^{\epsilon_k } (t)  , \, \epsilon_k^{\iota ( \nu_\tsr{C} ) } \matr{e}^{\epsilon_k} (\hat{\vct{u}}_\rmf^{\epsilon_k} (t) ) \bigr)  }^2  + \smallo ( \epsilon_{k} ) \\
&\quad \le \int_0^t  \bigl\langle \hat{\matr{K}}_\pmf^{\epsilon_k} \matr{M}_\pmf^{\epsilon_k} \nabla^{\fracfac{\! \epsilon_k \! }} \hat{p}_\pmf^{\epsilon_k} (\bar{t})  , \matr{M}_\pmf^{\epsilon_k} \nabla^{\fracfac{\! \epsilon_k \! }} \hat{p}_\pmf^{\epsilon_k} (\bar{t}) \bigr\rangle  \,\rmd \bar{t} + \tfrac{1}{2} \bigl\langle \fracfac{\epsilon_k^{\nu_\omega +1 }}  \hat{\omega}_\pmf^{\epsilon_k}  \hat{p}_\pmf^{\epsilon_k} (t)  , \hat{p}_\pmf^{\epsilon_k} (t) \bigr\rangle \\
&\hspace{6.625cm} + \tfrac{1}{2}  \bigl\langle \fracfac{\epsilon_k^{2}} \hat{\tsr{C}}_\pmf^{\epsilon_k} \matr{E}_\pmf^{\epsilon_k} ( \hat{\vct{u}}_\pmf^{\epsilon_k} (t) ) , \matr{E}_\pmf^{\epsilon_k} ( \hat{\vct{u}}_\pmf^{\epsilon_k} (t) )  \bigr\rangle   \\[3pt]
&\quad = \tfrac{1}{2} \bigl\langle \fracfac{\epsilon_k^{2}} \hat{\tsr{C}}_\pmf^{\epsilon_k} \matr{E}_\pmf^{\epsilon_k} ( \hat{\vct{u}}_{0,\pmf}^{\epsilon_k} ) , \matr{E}_\pmf^{\epsilon_k} ( \hat{\vct{u}}_{0,\pmf}^{\epsilon_k} )  \bigr\rangle   +  \tfrac{1}{2} \bigl\langle \fracfac{\epsilon_k^{\nu_\omega +1 }}  \hat{\omega}_\pmf^{\epsilon_k}  \hat{p}_{0,\pmf}^{\epsilon_k}  , \hat{p}_{0,\pmf}^{\epsilon_k} \bigr\rangle  \\
&\hspace{2cm} + \int_0^t \Bigl[ \bigl\langle \fracfac{\epsilon_k^{\nu_{\!\vct{f}} + 1} }\hat{\vct{f}}_\pmf^{\epsilon_k} , \partial_t \hat{\vct{u}}_\pmf^{\epsilon_k} \bigr\rangle + \bigl\langle  \fracfac{\epsilon_k^{\!\matr{\nu}_{\!\matr{\alpha}} + \matr{I}_n\!  } } \hat{G}_\pmf^{\epsilon_k} \hat{\matr{\alpha}}_\pmf^{\epsilon_k} , \nabla^{\fracfac{\! \epsilon_k \! }} \partial_t \hat{\vct{u}}_\pmf^{\epsilon_k} [\matr{M}_\pmf^{\epsilon_k}]^\rmt \bigr\rangle  \\
&\hspace{9.05cm}  + \bigl\langle \fracfac{\epsilon_{k}^{\nu_q + 1}} \hat{q}_\pmf^{\epsilon_{k}} , \hat{p}_\pmf^{\epsilon_{k}}  \bigr\rangle \Bigr] \,\rmd\bar{t}  \\
&\quad \rightarrow \tfrac{1}{2} \Bigl[ \bigl\langle \hat{\mathbb{C}}_\pm \matr{e} ( \hat{\vct{u}}_{0,\pm}^\# ) , \matr{e} ( \hat{\vct{u}}_{0,\pm}^\# ) \bigr\rangle_{\Omega_\pm^0  }  + \bigl\langle \hat{\matr{C}}^{\!\vct{N}}_\rmf \partial_{\!\vct{N}} \hat{\vct{u}}_{0,\rmf}^{\#} , \partial_{\!\vct{N}} \hat{\vct{u}}_{0,\rmf}^{\#}  \bigr\rangle_{\Omega_\rmf^1} +   \bigl\langle \hat{\omega}_\pmf^{\fracfac{\! \eff \!}}  \hat{p}_{0, \pmf} , \hat{p}_{0, \pmf} \bigr\rangle  \Bigr] \\ 
&\hspace{4cm} +  \int_0^t \Bigl[ \bigl\langle \hat{\vct{f}}_\pmf^{\fracfac{\!\eff\! }} \! , \partial_t \hat{\vct{u}}_\pmf^\# \bigr\rangle + \bigl\langle \hat{G}_\pm \hat{\matr{\alpha}}_\pm , \nabla \partial_t \hat{\vct{u}}_\pm^\# \bigr\rangle_{\Omega_\pm^0 } \\
&\hspace{6.5cm} + \bigl\langle \hat{G}_\rmf \hat{\vct{\alpha}}_\rmf^\eff \! , \partial_{\!\vct{N}} \partial_t \hat{\vct{u}}_\rmf^\# \bigr\rangle_{\Omega_\rmf^1} + \bigl\langle \hat{q}_\pmf^{\fracfac{\!\mathrm{eff}\!}}  \! , \hat{p}_\pmf^\#  \bigr\rangle  \Bigr] \,\rmd\bar{t} \\
&\quad \le  \int_0^t \norm[\big ]{ \bigl( \hat{p}_\pmf^\# (\bar{t}) , \hat{\xi}^\# (\bar{t} ) \bigr) }_\mathrm{i}^2  \,\rmd \bar{t} + \tfrac{1}{2}\norm[\big ]{\hat{p}_\pmf^\# (t) }_\mathrm{ii}^2 + \tfrac{1}{2}\normiii[\big ]{\bigl( \hat{\vct{u}}_\pm^\# (t) , \matr{e}_\vct{N} (\hat{\vct{u}}_\rmf^\# \! (t) ) \bigr)  }^2 
\end{align*}
as $k \rightarrow \infty$, where we have applied \Cref{thm:apriori} and \ref{asm:omega_eps}, \ref{asm:C_eps}, and~\ref{asm:K_eps} in \Cref{asm:eps} in the first step and \cref{eq:thm_C1Klm1_1} in the second step.
Besides, the third step follows from \Cref{asm:eps}, \Cref{lem:G}, \Cref{cor:conv}, and \Cref{lem:u0conv}, as well as the fact that $\epsilon^{\iota(\nu_\tsr{C} )} \nablapar \partial_t \hat{\vct{u}}_\rmf^\epsilon \rightharpoonup \vct{0}$ weakly in $L^2 ( I ; \vct{L}^2 (\Omega_\rmf^1 ) ) $ as a consequence of \cref{eq:apriori_b} in \Cref{thm:apriori}.
The last step follows with \cref{eq:thm_C1Klm1_2}. 
Thus, with the weak lower semicontinuity of the norms, we obtain 
\begin{equation}
\label{eq:thm_C1Klm1_3}
\begin{multlined}[c][0.875\displaywidth]
\lim_{k\rightarrow\infty } \biggl[ \int_0^t \norm[\big ]{ \bigl( \hat{p}_\pmf^{\epsilon_k} (\bar{t}) , \epsilon_k^{-1} \partial_{\!\vct{N}} \hat{p}_\rmf^{\epsilon_k} (\bar{t}) \bigr) }_\mathrm{i}^2  \,\rmd \bar{t} + \tfrac{1}{2}\norm[\big ]{\hat{p}_\pmf^{\epsilon_{k}} (t) }_\mathrm{ii}^2  \\
+   \tfrac{1}{2}\normiii[\big ]{\bigl( \hat{\vct{u}}_\pm^{\epsilon_k} (t)  , \, \epsilon_k^{\iota ( \nu_\tsr{C} ) } \matr{e}^{\epsilon_k} (\hat{\vct{u}}_\rmf^{\epsilon_k} (t) ) \bigr)  }^2   \bigg] \\
 =    \int_0^t \norm[\big ]{\bigl( \hat{p}_\pmf^\# (\bar{t}) , \hat{\xi}^\# (\bar{t}) \bigr) }_\mathrm{i}^2  \,\rmd \bar{t} +  \tfrac{1}{2}\norm[\big ]{\hat{p}_\pmf^\# (t) }_\mathrm{ii}^2  +  \tfrac{1}{2}\normiii[\big ]{\bigl( \hat{\vct{u}}_\pm^\# (t) , \matr{e}_\vct{N} ( \hat{\vct{u}}_\rmf^\# \!(t) \bigr)  }^2  .
\end{multlined}
\end{equation}
\item  
 Considering \cref{eq:weak_mechlimit_C1} and \cref{eq:weak_C1_Km1} yields
\begin{equation}
\label{eq:thm_C1Klm1_4}
\begin{multlined}[c][0.875\displaywidth]
\int_0^t \Bigl[ \bigl\langle \hat{\mathbb{C}}_\pm \matr{e} \bigl(  \partial_t \hat{\vct{u}}_\pm^\# (\bar{t}) \bigr) , \matr{e} \bigl(  \partial_t \hat{\vct{u}}_\pm^\# (\bar{t}) \bigr) \bigr\rangle_{\Omega_\pm^0  } + \bigl\langle \hat{\matr{C}}_\rmf^{\!\vct{N}} \partial_{\!\vct{N}} \partial_t \hat{\vct{u}}_\rmf^\# \! (\bar{t}) , \partial_{\!\vct{N}} \partial_t \hat{\vct{u}}_\rmf^\# \! (\bar{t})\bigr\rangle_{\Omega_\rmf^1 } \\ 
+ \bigl\langle \hat{\omega}_\rmf^{\fracfac{\!\eff\!}} \partial_t \hat{p}_\pmf^\# (\bar{t}), \partial_t \hat{p}_\pmf^\# (\bar{t})\bigr\rangle  \Bigr] \,\rmd\bar{t}  
+  \tfrac{1}{2} \bigl\langle \hat{\matr{K}}_\pm \nabla \hat{p}_\pm^\# , \nabla \hat{p}_\pm^\# \bigr\rangle_{ \Omega_\pm^0  }  \bigr\vert_0^t \\
 + \tfrac{1}{2} \bigl\langle \hat{\matr{K}}_\rmf \bigl( \nablapar \hat{p}_\rmf^\#  + \hat{\xi}^\# \vct{N} \bigr) , \bigl( \nablapar \hat{p}_\rmf^\# + \hat{\xi}^\# \vct{N} \bigr) \bigr\rangle_{\Omega_\rmf^1} \bigr\vert_0^t    \\
 =  \int_0^t \Bigl( \bigl\langle \partial_t \hat{\vct{f}}_\pmf^{\fracfac{\!\mathrm{eff}\!}} \! (\bar{t}), \partial_t \hat{\vct{u}}_\pmf^\# (\bar{t}) \bigr\rangle
 + \bigl\langle \hat{q}_\pmf^{\fracfac{\!\mathrm{eff}\!}} \! (\bar{t}) , \partial_t \hat{p}_\pmf^\#  (\bar{t}) \bigr\rangle  \Bigr) \,\rmd \bar{t} 
\end{multlined}
\end{equation}
by applying analogous arguments as in the second step of the proof of \Cref{thm:apriori} for the derivation of \cref{eq:apriori_4}.
Then, as $k\rightarrow\infty $, we have 
\begin{align*}
&\int_0^t \Bigl[ \norm[\big ]{\partial_t \hat{p}_\pmf^{\epsilon_{k}} (\bar{t}) }_\mathrm{ii}^2 + \normiii[\big ]{\bigl( \partial_t \hat{\vct{u}}_\pm^{\epsilon_k} (\bar{t} ) , \, \epsilon_k^{\iota ( \nu_\tsr{C} ) } \matr{e}^{\epsilon_k} \bigl( \partial_t \hat{\vct{u}}_\rmf^{\epsilon_k} (\bar{t}) \bigr) \bigr) }^2 \Bigr] \,\rmd\bar{t}  \\
&\hspace{5.5cm} + \tfrac{1}{2} \norm[\big ]{\bigl( \hat{p}_\pmf^{\epsilon_k} (t) , \epsilon_k^{-1} \partial_{\!\vct{N}} \hat{p}_\rmf^{\epsilon_k} (t) \bigr) }_\mathrm{i}^2 + \smallo (\epsilon_{k} ) \\
&\enspace \le \tfrac{1}{2} \bigl\langle \hat{\matr{K}}_\pmf^{\epsilon_k} \matr{M}_\pmf^{\epsilon_k} \nabla^{\fracfac{\! \epsilon_k \! }} \hat{p}_\pmf^{\epsilon_k} (t)  , \matr{M}_\pmf^{\epsilon_k} \nabla^{\fracfac{\! \epsilon_k \! }} \hat{p}_\pmf^{\epsilon_k} (t) \bigr\rangle  + \int_0^t \Bigl[ \bigl\langle \fracfac{\epsilon_k^{\nu_\omega +1 }}  \hat{\omega}_\pmf^{\epsilon_k}  \partial_t \hat{p}_\pmf^{\epsilon_k} (\bar{t})  , \partial_t \hat{p}_\pmf^{\epsilon_k} (\bar{t}) \bigr\rangle   \\
&\hspace{5.25cm} + \bigl\langle \fracfac{\epsilon_k^{2}} \hat{\tsr{C}}_\pmf^{\epsilon_k} \matr{E}_\pmf^{\epsilon_k} \bigl( \partial_t \hat{\vct{u}}_\pmf^{\epsilon_k} (\bar{t}) \bigr) , \matr{E}_\pmf^{\epsilon_k} \bigl ( \partial_t \hat{\vct{u}}_\pmf^{\epsilon_k} (\bar{t}) \bigr)  \bigr\rangle   \Bigr]\,\rmd \bar{t} \\
&\enspace = \tfrac{1}{2}\bigl\langle\hat{\matr{K}}_\pmf^{\epsilon_k} \matr{M}_\pmf^{\epsilon_k} \nabla^{\fracfac{\!\epsilon_k \! }} \hat{p}_{0,\pmf}^{\epsilon_k}  , \matr{M}_\pmf^{\epsilon_k}\nabla^{\fracfac{\!\epsilon_k \! }} \hat{p}_{0,\pmf}^{\epsilon_k} \bigr\rangle  \\
&\hspace{4.1cm} + \int_0^t \Bigl[ \bigl\langle \fracfac{\epsilon_{k}^{\nu_{\!\vct{f}}  + 1}} \partial_t \hat{\vct{f}}_\pmf^{\epsilon_{k}} , \partial_t \hat{\vct{u}}_\pmf^{\epsilon_{k}} \bigr\rangle
 + \bigl\langle \fracfac{\epsilon^{\nu_q  + 1}_{k}} \hat{q}_\pmf^{\epsilon_{k}} , \partial_t \hat{p}_\pmf^{\epsilon_{k}}  \bigr\rangle  \Bigr] \,\rmd \bar{t}  \\
&\enspace \rightarrow \tfrac{1}{2} \bigl\langle \hat{\matr{K}}_\pm \nabla \hat{p}_{0,\pm} , \nabla \hat{p}_{0,\pm} \bigr\rangle_{ \Omega_\pm^0  }  + \tfrac{1}{2} \bigl\langle \hat{\matr{K}}_\rmf \bigl( \nablapar \hat{p}_{0,\rmf} + \hat{\xi}_0^\# \vct{N} \bigr), \bigl(  \nablapar \hat{p}_{0,\rmf} + \hat{\xi}_0^\# \vct{N}  \bigr) \bigr\rangle_{\Omega_\rmf^1}  \\
&\hspace{6cm} +  \int_0^t \Bigl[ \bigl\langle \partial_t \hat{\vct{f}}_\pmf^{\fracfac{\!\mathrm{eff}\!}}  \! , \partial_t \hat{\vct{u}}_\pmf^\# \bigr\rangle
 + \bigl\langle \hat{q}_\pmf^{\fracfac{\!\mathrm{eff}\!}} \! , \partial_t \hat{p}_\pmf^\#  \bigr\rangle  \Bigr] \,\rmd \bar{t} \\
&\enspace \le \tfrac{1}{2} \norm[\big ]{\bigl( \hat{p}_\pmf^\# (t) , \hat{\xi}^\# (t) \bigr) }_\mathrm{i}^2 +  \int_0^t \Bigl[ \norm[\big ]{\partial_t \hat{p}_\pmf^\# (\bar{t}) }_\mathrm{ii}^2 + \normiii[\big ]{\bigl( \partial_t \hat{\vct{u}}_\pm^\# (\bar{t} ) ,\matr{e}_\vct{N} \bigl( \partial_t \hat{\vct{u}}_\rmf^\# \! (\bar{t} ) \bigr) \bigr) }^2 \Bigr] \,\rmd\bar{t}  ,
\end{align*}
where we have used \Cref{thm:apriori} and \ref{asm:omega_eps}, \ref{asm:C_eps}, and \ref{asm:K_eps} in \Cref{asm:eps} in the first step and \cref{eq:apriori_4} in the second step.
Additionally, the third the step follows from \Cref{cor:conv} and \ref{asm:f_eps}, \ref{asm:q_eps}, and \ref{asm:K_eps} in \Cref{asm:eps}, as well as $\epsilon^{-1} \nablaperp \hat{p}_{0,\rmf} \rightarrow \smash{\hat{\xi}_0^\#} \vct{N}$ strongly in $L^2 ( I ; \vct{L}^2 ( \Omega_\rmf^1 ) )$ and the fact that $\epsilon^{\nu_q +1  } \partial_t \hat{p}_\rmf^\epsilon \rightharpoonup 0$ weakly in $L^2(I ; L^2(\Omega_\rmf^\epsilon))$ for $\nu_q > -1$ as a consequence of \cref{eq:apriori_d} in \Cref{thm:apriori} and the assumption $2\nu_q > \nu_\omega - 1$. 
The last step follows from \cref{eq:thm_C1Klm1_4}.
Thus, using the weak lower semicontinuity of the norms, we find 
\begin{equation}
\label{eq:thm_C1Klm1_5}
\begin{aligned}
&\lim_{k\rightarrow \infty  } \biggl[ \tfrac{1}{2} \norm[\big ]{\bigl( \hat{p}_\pmf^{\epsilon_k} (t), \epsilon_k^{-1} \partial_{\!\vct{N}} \hat{p}_\rmf^{\epsilon_k} (t) \bigr) }_\mathrm{i}^2 \\
&\hspace{2cm} +   \int_0^t \Bigl[ \norm[\big ]{\partial_t \hat{p}_\pmf^{\epsilon_{k}} (\bar{t}) }_\mathrm{ii}^2 + \normiii[\big ]{\bigl( \partial_t \hat{\vct{u}}_\pm^{\epsilon_k} (\bar{t} ) , \, \epsilon_k^{\iota ( \nu_\tsr{C} ) } \matr{e}^{\epsilon_k} \bigl( \partial_t \hat{\vct{u}}_\rmf^{\epsilon_k} ( \bar{t} ) \bigr) \bigr) }^2 \Bigr] \,\rmd\bar{t}  \biggr] \\
&\ =  \tfrac{1}{2} \norm[\big ]{\bigl( \hat{p}_\pmf^\# (t) , \hat{\xi}^\# (t) \bigr) }_\mathrm{i}^2 + \int_0^t \Bigl[ \norm[\big ]{\partial_t \hat{p}_\pmf^\# (\bar{t}) }_\mathrm{ii}^2 + \normiii[\big ]{\bigl( \partial_t \hat{\vct{u}}_\pm^\# (\bar{t} ) ,\matr{e}_\vct{N} \bigl( \partial_t \hat{\vct{u}}_\rmf^\# (\bar{t}) \bigr) \bigr) }^2 \Bigr] \,\rmd\bar{t} 
\end{aligned}
\end{equation}
so that \cref{eq:strongconv_b} follows together with \cref{eq:thm_C1Klm1_3}. \qedhere
\end{enumerate}
\end{proof}
The proof can be adapted as follows to cover the remaining cases.
\begin{proof}[Proof of \Cref{thm:strongconv} (all cases)]
For $\nu_\tsr{C} > 1$, we can proceed analogously as above by considering the mechanics limit equation \cref{eq:weak_mechlimit_nuCg1} instead of~\cref{eq:weak_mechlimit_C1}.
Moreover, for a different value of~$\nu_\matr{K} \in \mathbb{R}$, the proof works in the same way if we redefine the norm~$\norm{\cdot}_\mathrm{i}$ and use the flow limit equations~\cref{eq:weak_C1_Klm1}, \cref{eq:weak_C1_Km11}, \cref{eq:weak_C1_K1}, or \cref{eq:weak_C1_Kg1} instead of \cref{eq:weak_C1_Km1}.
Specifically, we define 
\begin{subequations}
\begin{align}
\Phi \rightarrow \mathbb{R}_0^+ , \quad \norm{\phi_\pmf }_\mathrm{i}^2 := 
\bigl\langle \hat{\matr{K}}_\pmf \nabla \phi_\pmf , \nabla \phi_\pmf \bigr\rangle 
\end{align}
for $\nu_\matr{K} < - 1$ and  
\begin{align}
\Phi^\#  \rightarrow \mathbb{R}_0^+ , \quad  \norm{\phi_\pmf }_\mathrm{i}^{2} := 
\bigl\langle \hat{\matr{K}}_\pm \nabla \phi_\pm , \nabla \phi_\pm \bigr\rangle_{\Omega_\pm^0 }   + \bigl\langle \hat{\matr{K}}_\rmf \nablaperp \phi_\rmf , \nablaperp \phi_\rmf \bigr\rangle_{\Omega_\rmf^1}  
\end{align}
for $ \nu_\matr{K} \in (-1,1)$. 
Besides, for $\nu_\matr{K} = 1$ and $\vct{L}_\smallpar^2 (\Omega_\rmf^1) := \{ \vct{v} \in \vct{L}^2 (\Omega_\rmf^1 ) \ \vert\ \vct{v} \cdot \vct{N} = 0 \text{ a.e.} \}$, we set 
\begin{equation}
\begin{multlined}[c][0.875\displaywidth]
\Phi^\# \times \vct{L}_\smallpar^2 (\Omega_\rmf^1 ) \rightarrow \mathbb{R}_0^+ , \\ \norm[\big ]{\bigl( \phi_\pmf , \vct{\xi} \bigr) }_\mathrm{i}^{2} := \bigl\langle \hat{\matr{K}}_\pm \nabla \phi_\pm , \nabla \phi_\pm \bigr\rangle_{\Omega_\pm^0 }   +    \bigl\langle \hat{\matr{K}}_\rmf \bigl[ \nablaperp \phi_\rmf + \vct{\xi}  \bigr] , \nablaperp \phi_\rmf + \vct{\xi}  \bigr\rangle_{\Omega_\rmf^1} 
\end{multlined}%
\end{equation}%
and consider $\vct{\xi} = \epsilon_{k} \nablapar \hat{p}_\rmf^{\epsilon_{k}}$ in \cref{eq:thm_C1Klm1_3,eq:thm_C1Klm1_5}. 
Finally, for $\nu_\matr{K} > 1$, we have 
\begin{align}
\Phi_{\ge 1}  \rightarrow \mathbb{R}_0^+ , \quad \norm{ \phi_\pm  }_\mathrm{i}^{2} := \bigl\langle \hat{\matr{K}}_\pm \nabla \phi_\pm , \nabla \phi_\pm \bigr\rangle_{\Omega_\pm^0 }   . 
\end{align}%
\label{eq:normi_2}%
\end{subequations}%
\end{proof}

\section*{Acknowledgements}
This work was funded by the Deutsche Forschungsgemeinschaft (DFG, German Research Foundation) – Project Number 327154368 – SFB 1313. 
K.K.\ would like to acknowledge the financial support from the CSSR supported by the Norwegian Research Council and Vista CSD supported by the Norwegian Academy of Sciences and Equinor. 
In addition, K.K.
acknowledges support as a Mercator Fellow within SFB 1313, as well as support from a Humboldt Research Fellowship from the Alexander von Humboldt Foundation. 

\begin{appendices}
\section{Existence and Uniqueness for Linear Poroelasticity}
\label{sec:A}
Let $D \subset \mathbb{R}^n$ be a bounded domain with $\mathcal{C}^{0,1}$-boundary representing an anisotropic poroelastic medium and consider the time interval $I := ( 0, T)$ with $T>0$.
Besides, let $\partial D = \partial D_\rmD \cup \partial D_\rmN$ with $\partial D_\rmD \cap \partial D_\rmN = \emptyset$ and $\abs{\partial D_\rmD} > 0$.
We assume that the deformation of~$D$ and the flow of a single-phase fluid in~$D$ are governed by the quasi-static Biot equations~\cite{biot41,mikelic12}. 
Specifically, we consider the following Biot system in dimensionless form (see \Cref{sec:2.4} for a non-dimensionalization).

Find the pressure $p \colon D \times I \rightarrow \mathbb{R}$ and the displacement vector $\vct{u}\colon D\times I \rightarrow \mathbb{R}^n$ such that 
\begin{subequations}
\begin{alignat}{2}
-\nabla \cdot \bigl( \tsr{C} \matr{e} ( \vct{u} ) - p \matr{\alpha} \bigr) &= \vct{f} - \nabla \cdot ( G \matr{\alpha} ) \quad &&\text{in } D \times I , \\
\partial_t \bigl( \omega p + \nabla \cdot ( \matr{\alpha} \vct{u} ) \bigr) - \nabla \cdot ( \matr{K} \nabla p  ) &= q \quad\enspace &&\text{in } D \times I, \\
u &= \vct{0} \quad\enspace &&\text{on } \partial D \times I , \\
p &= 0 \quad\enspace &&\text{on } \partial D_\rmD \times I , \\
\matr{K} \nabla p \cdot \vct{n} &= 0 \quad\enspace &&\text{on } \partial D_\rmN \times I , \\
\label{eq:initial_p_sd} p ( \cdot , 0 ) &= p_0  \quad\enspace &&\text{on } D.
\intertext{Besides, given the initial pressure head~$p_0 \colon D \rightarrow \mathbb{R}$, the initial displacement vector $\vct{u}_0 := \vct{u} ( \cdot, 0 )  \colon D \rightarrow \mathbb{R}^n$ is defined such that}
-\nabla \cdot \bigl( \tsr{C} \matr{e} ( \vct{u}_0 ) - p_0 \matr{\alpha} \bigr) &= \vct{f} (\cdot , 0) - \nabla \cdot ( G \matr{\alpha } ) \quad &&\text{in } D , \\
\vct{u}_0 &= \vct{0} \quad\enspace &&\text{on } \partial D .
\end{alignat}%
\label{eq:biot_sd_strong}%
\end{subequations}%
In \cref{eq:biot_sd_strong}, $\tsr{C} \colon D \rightarrow \mathbb{R}^{n\times n \times n \times n}$ denotes the elasticity tensor and $\matr{e} ( \vct{u})$ is the symmetrized gradient of~$\vct{u}$ given by
\begin{align}
2 \matr{e} ( \vct{u} ) := \nabla \vct{u} + ( \nabla \vct{u} )^\rmt .
\end{align}
Further, $\omega \colon D \rightarrow \mathbb{R}$ denotes the inverse Biot modulus and $\matr{\alpha} \colon D \rightarrow \mathbb{R}^{n\times n }$ is Biot's coupling tensor, which is regarded here as a matrix, following \cite{mikelic12}. 
The hydraulic conductivity matrix is denoted by $\matr{K} \colon D \rightarrow \mathbb{R}^{n\times n }$ and $G\colon D \rightarrow \mathbb{R}$ is the function $G ( \vct{x} ) := \vct{x} \cdot \vct{g}$ with the gravitational acceleration~$\matr{g} \in \mathbb{R}^n$. 
In addition, $\vct{f} \colon D \times I \rightarrow \mathbb{R}^n$ is a body force term, $q \colon D \times I \rightarrow \mathbb{R}$ a fluid source term, and $\vct{n} \colon \partial D \rightarrow \mathbb{R}^n$ denotes the outer unit normal on~$\partial D$.
We make the following assumptions on the model coefficients, right-hand sides, and initial data.
\begin{assumption} 
\label{asm:biot_sd}
We assume that
\begin{enumerate}[label=(\roman*),labelsep=0em,topsep=0pt,leftmargin=*,align=left,widest={$(\mathbb{A}_{\vct{f}\!,q})$}]
\item[$(\mathbb{A}_{\matr{\alpha}})$] \customlabel{asm:alpha_sd}{$(\mathbb{A}_{\matr{\alpha}})$} $\matr{\alpha} \colon D \rightarrow \mathbb{R}^{n\times n }$ is symmetric and piecewise constant,
\item[$(\mathbb{A}_{\vct{f}\!,q})$] \customlabel{asm:fq_sd}{$(\mathbb{A}_{\vct{f}\!,q})$} $\vct{f} \in H^2 (I ; \vct{L}^2 (D))$ and $q \in H^1 (I ; L^2 (D) )$,
\item[$(\mathbb{A}_{p_0})$] \customlabel{asm:p0_sd}{$(\mathbb{A}_{p_0})$} $p_0 \in H^1_{0,\partial D_\rmD}\! (D)$ and $\matr{K} \nabla p_0 \in \vct{H}_{\mathrm{div}} (D)$,
\item[$(\mathbb{A}_{\omega})$] \customlabel{asm:omega_sd}{$(\mathbb{A}_{\omega})$} $\omega \in L^\infty (D)$ is almost everywhere uniformly positive, i.e., $\omega \gtrsim 1 $,
\item[$(\mathbb{A}_{\tsr{C}})$] \customlabel{asm:C_sd}{$(\mathbb{A}_{\tsr{C}})$} $\tsr{C}\in L^\infty ( D ; \mathbb{R}^{n \times n \times n \times n} )$  is almost everywhere uniformly elliptic, i.e., we have $\tsr{C} \matr{M} : \matr{M} \gtrsim \matr{M} : \matr{M} $ for all symmetric $\matr{M} \in \mathbb{R}^{n\times n }$,  and satisfies the symmetry properties
\begin{align*}
 \tsr{C}_{ijkl} =  \tsr{C}_{klij} , \quad \tsr{C}_{ijkl} =  \tsr{C}_{jikl} , \quad  \tsr{C}_{ijkl} =  \tsr{C}_{ijlk},
\end{align*}
\item[$(\mathbb{A}_{\matr{K}})$] \customlabel{asm:K_sd}{$(\mathbb{A}_{\matr{K}})$} $\matr{K} \in \matr{L}^\infty (D)$ is almost everywhere symmetric and uniformly elliptic, i.e., $\matr{K} \vct{w} \cdot \vct{w} \gtrsim \vert \vct{w} \vert^2$ for all $\vct{w} \in \mathbb{R}^n $.
\end{enumerate}
\end{assumption}
A weak formulation of the Biot system~\eqref{eq:biot_sd_strong} is now given by the following problem.

Find the pressure head~$p \in H^1 (I ; L^2 (D)) \,\cap\, L^2 (I ; H^1_{0, \partial D_\rmD} \! (D))$ and the displacement vector $\vct{u} \in H^1 (I ; \vct{H}^1_0 (D))$ such that $p(0) = p_0$ and the equations%
\begin{subequations}
\begin{align}
\bigl\langle \tsr{C} \matr{e} ( \vct{u} ) , \matr{e} ( \vct{v} ) \bigr\rangle_{\! D } - \bigl\langle p \matr{\alpha} , \nabla \vct{v} \bigr\rangle_{\! D } &= \bigl\langle \vct{f} , \vct{v} \bigr\rangle_{\! D} + \bigl\langle G \matr{\alpha} , \nabla \vct{v} \bigr\rangle_{\! D } , \\
\bigl\langle \omega \partial_t p + \partial_t \nabla \cdot ( \matr{\alpha} \vct{u} ) , \phi \bigr\rangle_{\! D } + \bigl\langle \matr{K}  \nabla p , \nabla \phi \bigr\rangle_{\! D} &= \bigl\langle q , \phi \bigr\rangle_{\! D} ,
\intertext{hold for all $\vct{v} \in \vct{H}^1_0 (D) $ and~$\phi \in  H^1_{0,\partial D_\rmD } \! (D)$ and almost all~$t \in I$. 
Besides, the initial displacement vector~$\vct{u}_0  := \vct{u} (  0 ) \in \vct{H}^1_0 ( D )$ satisfies } 
\bigl\langle \tsr{C} \matr{e} ( \vct{u_0} ) , \matr{e} ( \vct{v} ) \bigr\rangle_{\! D } - \bigl\langle p_0 \matr{\alpha} , \nabla \vct{v} \bigr\rangle_{\! D } &= \bigl\langle \vct{f} (0) , \vct{v} \bigr\rangle_{\! D}  + \bigl\langle G \matr{\alpha} , \nabla \vct{v} \bigr\rangle_{ \! D } 
\end{align}%
\label{eq:biot_sd_weak}%
\end{subequations}%
for all~$\vct{v} \in \vct{H}^1_0 (D)$.

Subsequently, we prove the following existence and uniqueness theorem for the Biot problem~\eqref{eq:biot_sd_weak}.
\begin{theorem} 
\label{thm:wellposed_biot_sd}
Given the \Cref{asm:biot_sd}, the Biot problem in \cref{eq:biot_sd_weak} admits unique solutions 
\begin{align*}
p \in H^1 ( I ; H^1_{0 ,\partial D_\rmD} \! (D)) \,\cap\, L^\infty ( I ; H^1 (D)) \,\cap\, W^{1,\infty } (I ; L^2 (D) ) \quad\text{and}\quad \vct{u} \in W^{1,\infty } (I ; \vct{H}^1_0 (D)).
\end{align*}
\end{theorem}
\begin{remark}
While the wellposedness of the Biot equations are well studied in the literature~\cite{auriault77,marciniak15,showalter00,zenisek84}, we continue to provide a proof for \Cref{thm:wellposed_biot_sd} for completeness, as no proof is available in this specific setting.
The result of \Cref{thm:wellposed_biot_sd} is used in the main part of the paper in \Cref{prop:wellposed-fulldim} to establish the wellposedness of a Biot problem in a fractured domain involving multiple subdomains.
In particular, our analysis covers the assumptions stated in \Cref{asm:biot_sd} with matrix-valued piecewise constant Biot coupling and otherwise function-valued model coefficients, as well as weak regularity assumptions on the initial data.
\end{remark}
In order to prove \Cref{thm:wellposed_biot_sd}, we consider a sequence of Galerkin approximations to the Biot problem~\eqref{eq:biot_sd_weak}. 
Let $\{ \vct{w}_j \}_{j=1}^\infty$ be an arbitrary basis of $\vct{H}^1_0 (D )$.
Further, let $\{ \varphi_j \}_{j=1}^\infty \subset H^1_{0, \partial D_\rmD }\!  (D)$ be the $H^1 (D)$-orthogonal and $L^2 (D)$-orthonormal basis of eigenvectors of the elliptic operator~$L  \phi  := -\nabla \cdot ( \matr{K} \nabla \phi ) $ (cf.\ \cite[Satz~8.39]{dobrowolski10}).
Then, for $k\in \mathbb{N}$, there is $\lambda_k \in \mathbb{R}$ such that
\begin{align}
\label{eq:eigenfunction}
\bigl\langle \matr{K} \nabla \varphi_k , \nabla \psi \bigr\rangle_{\! D } &= \lambda_k \bigl\langle \varphi_k , \psi \bigr\rangle_{\! D} \quad\enspace\text{for all } \psi \in H^1_{0,\partial D_\rmD } \! (D)
\end{align}
and hence~$\matr{K} \nabla \varphi_k \in \vct{H}_\mathrm{div} (D)$ with $-\nabla \cdot ( \matr{K} \nabla \varphi_k ) = \lambda_k \varphi_k \in H^1_{0, \partial D_\rmD }\! ( D)$, as well as $\matr{K}\nabla \varphi_k \cdot\vct{n} = 0$ on~$\partial D_\rmN$.
Further, for $N \in \mathbb{N}$, we define the finite-dimensional spaces 
\begin{align}
\vct{V}^N := \operatorname{span} \{ \vct{w}_j \}_{j=1}^N , \quad \Phi^N := \operatorname{span} \{ \varphi_j \}_{j=1}^N .
\end{align}
In addition, let $\Pi_{\Phi^{\! N}} \colon L^2 (D) \rightarrow L^2 (D)$ denote the orthogonal projection onto~$\Phi^N$.
We now consider the following finite-dimensional approximation for~\eqref{eq:biot_sd_weak}.

Find $p^N := \sum_{j=1}^N P_j^N\! (t) \varphi_j \in H^1 (I ; \Phi^N )$ and $\vct{u}^N := \sum_{j=1}^N U_j^N\! (t) \vct{w}_j \in H^1 ( I ;  \vct{V}^N )$ such that%
\begin{subequations}
\begin{align}
\label{eq:galerkin-a} p^N\! ( 0) &= \Pi_{\Phi^{\! N}} p_0 =: p_0^N
\end{align}
and the equations
\begin{align}
\bigl\langle \tsr{C} \matr{e} ( \vct{u}^N ) , \matr{e} ( \vct{v} ) \bigr\rangle_{ \! D } - \bigl\langle p^N \! \matr{\alpha} , \nabla \vct{v} \bigr\rangle_{ \! D } &= \bigl\langle \vct{f} , \vct{v} \bigr\rangle_{\! D} + \bigl\langle G\matr{\alpha} , \nabla \vct{v} \bigr\rangle_{\! D } , \\
\label{eq:galerkin-c} \bigl\langle \omega \partial_t p^N + \partial_t \nabla \cdot ( \matr{\alpha} \vct{u}^N ) , \phi \bigr\rangle_{\! D } + \bigl\langle \matr{K}  \nabla p^N\! , \nabla \phi \bigr\rangle_{\! D } &= \bigl\langle q , \phi \bigr\rangle_{\! D } \\
\intertext{hold for all~$\vct{v} \in \vct{V}^N$ and $\phi \in \Phi^N$ and almost all~$t \in I$. Moreover, the initial displacement vector~$ \vct{u}_0^N := \vct{u}^N\! (0) = \sum_{j=1}^N U_{0,j}^N \vct{w}_j \in \vct{V}^N$ satisfies } 
\label{eq:galerkin-d} \bigl\langle \tsr{C} \matr{e} ( \vct{u}^N_0 ) , \matr{e} ( \vct{v} ) \bigr\rangle_{ \! D } - \bigl\langle p^N_0\! \matr{\alpha} , \nabla \vct{v} \bigr\rangle_{\! D } &= \bigl\langle \vct{f} (0) , \vct{v} \bigr\rangle_{\! D }  + \bigl\langle G\matr{\alpha} , \nabla \vct{v} \bigr\rangle_{\! D }
\end{align}%
\label{eq:galerkin}%
\end{subequations}%
for all~$\vct{v} \in \vct{V}^N$. 

The Galerkin system~\eqref{eq:galerkin} can be equivalently rewritten as a linear system of ordinary differential equations. 
For this, we can define the matrices~$\matr{A}^{\! N} = (A_{ij}^N )$, $\matr{E}^N = (E_{ij}^N )$, $\matr{K}^N = (K_{ij}^N )$, $\matr{O}^N = ( O_{ij}^N ) \in \mathbb{R}^{N\times N }$ with
\begin{subequations}
\begin{equation}
\begin{alignedat}{2}
A_{ij}^N &:= \bigl\langle \varphi_j \matr{\alpha} , \nabla \vct{w}_i \bigl\rangle_{\! D} , \quad &E_{ij}^N &:= \bigl\langle \tsr{C} \matr{e} ( \vct{w}_j ) , \matr{e} ( \vct{w}_i) \bigr\rangle_{\! D } , \\
K_{ij}^N &:= \bigl\langle \matr{K}  \nabla \varphi_j , \nabla \varphi_i \bigr\rangle_{\! D }, \quad &O_{ij}^N &:= \bigl\langle \omega \varphi_j , \varphi_i \bigl\rangle_{\! D }
\end{alignedat}
\end{equation}
for $i,j \in \{ 1, \dots , N \}$, as well as the vectors $\vct{f}^N = (f^N_i)$,  $\vct{g}^N = (g^N_i)$,  $\vct{P}^N_0 = (P^N_{0,i})$, $\vct{q}^N = (q^N_i) \in \mathbb{R}^n $ with
\begin{equation}
\begin{alignedat}{2}
f^N_i \! (t) &:= \bigl\langle \vct{f} (t) , \vct{w}_i \bigr\rangle_{\! D} , \quad 
&g^N_i &:= \bigl\langle G \matr{\alpha} , \nabla \vct{w}_i \bigr\rangle_{\! D} ,  \\
P^N_{0,i} &:= \bigl\langle p_0 , \varphi_i \bigr\rangle_{\! D } , \quad 
&q_i^N \! (t) &:= \bigl\langle q (t) , \varphi_i \bigr\rangle_{\! D} 
\end{alignedat}
\end{equation}
\end{subequations}
for $i\in \{ 1, \dots , N \}$.
Then, using \ref{asm:alpha_sd} in \Cref{asm:biot_sd}, we can equivalently formulate the Galerkin system~\eqref{eq:galerkin} as the following system of ordinary differential equations.

Find $\vct{P}^N = (P^N_i) \in H^1 ( I ; \mathbb{R}^N)$ and $\vct{U}^N = (U^N_i) \in H^1 ( I ;  \mathbb{R}^N)$ such that
\begin{subequations}
\begin{align}
\label{eq:galerkin_PN_UN-a} \matr{E}^N \vct{U}^N\! (t) - \matr{A}^{\! N} \vct{P}^N\! (t) &= \vct{f}^N\! (t) + \vct{g}^N\! , \\
\label{eq:galerkin_PN_UN-b} \matr{O}^N \tfrac{\rmd }{\rmd t } \vct{P}^N\! (t) + \matr{K}^N \vct{P}^N\! (t) + (\matr{A}^{\! N} )^\rmt \tfrac{\rmd }{\rmd t }\vct{U}^N \! (t) &= \vct{q}^N\! (t)  , \\
 \vct{P}^N\! (0) &= \vct{P}^N_0\! .
 \end{align}
 Besides, $\vct{U}^N_0 := \vct{U}^N\! ( 0)$ is given by
\begin{align}
 \vct{U}^N_0  &= \bigl(\matr{E}^N\bigr)^{\! -1} \bigl( \vct{f}^N\! (0) + \vct{g}^N  + \matr{A}^{\! N} \vct{P}^N_0\bigr).
\end{align}%
\label{eq:galerkin_PN_UN}%
\end{subequations}%
Here, the matrix~$\matr{E}^N$ is positive definite and hence invertible as a consequence of \ref{asm:C_sd} in \Cref{asm:biot_sd} and Korn's inequality.
By inserting \cref{eq:galerkin_PN_UN-a} into \cref{eq:galerkin_PN_UN-b}, we obtain a linear system of ordinary differential equations for~$\vct{P}^N$ only, viz.,%
\begin{subequations}
\begin{align}
\label{eq:galerkin_PN-a} \matr{D}^N \tfrac{\rmd }{\rmd t } \vct{P}^N\! (t) + \matr{K}^N \vct{P}^N\! (t)  &= \vct{q}^N\! (t)  - (\matr{A}^{\! N} )^\rmt ( \matr{E}^N )^{-1} \tfrac{\rmd }{\rmd t}\vct{f}^N\! (t) , \\
\vct{P}^N\! (0) &= \vct{P}^N_0\! ,
\end{align}%
\label{eq:galerkin_PN}%
\end{subequations}%
where $\matr{D}^N := \matr{O}^N + (\matr{A}^{\! N} )^\rmt (\matr{E}^N )^{-1} \matr{A}^{\! N} \in \mathbb{R}^{N \times N }$.
Now, from the wellposedness of the system~\eqref{eq:galerkin_PN}, we can infer the wellposedness of the Galerkin system~\eqref{eq:galerkin}.
\begin{lemma} \label{lem:sdH2}
Given \Cref{asm:biot_sd}, the Galerkin system~\eqref{eq:galerkin} admits unique solutions $p^N \in H^2 \bigl(I ; \Phi^N \bigr)$ and $\vct{u}^N \in H^2 \bigl( I ; \vct{V}^N \bigr) $.
\end{lemma}
\begin{proof}
The matrix~$\matr{D}^N$ is positive definite and hence invertible as a consequence of \ref{asm:omega_sd} in \Cref{asm:biot_sd} and the positive definiteness of~$\matr{E}^N$. 
Thus, we can define $\vct{h} \colon I \times \mathbb{R}^N \rightarrow \mathbb{R}^N$ by 
\begin{align*}
 \vct{h} \bigl( t, \vct{P}^N\! (t) \bigr) := (\matr{D}^N)^{-1}  \Bigl( - \matr{K}^N \vct{P}^N \! (t) +   \vct{q}^N\! (t)  - (\matr{A}^{\!N})^\rmt ( \matr{E}^N )^{-1} \tfrac{\rmd }{\rmd t } \vct{f}^N\! (t) \Bigr) 
\end{align*}
and rewrite \cref{eq:galerkin_PN} as
\begin{align*}
\tfrac{\rmd }{\rmd t} \vct{P}^N\! (t) &= \vct{h} \bigl( t , \vct{P}^N\! (t) \bigr) , \quad \vct{P}^N\! (0) = \vct{P}^N_0\! . 
\end{align*}
Here, the function $\vct{h}$ is continuous on~$I \times \mathbb{R}^N$ due to \ref{asm:fq_sd} in \Cref{asm:biot_sd} and the continuity of the embedding~$H^1 (I ) \hookrightarrow \mathcal{C}^0 (\bar{I} )$. 
Besides, $h$ is clearly globally Lipschitz continuous with respect to the second argument.
Thus, the Picard-Lindelöf theorem yields the existence of a unique solution~$\vct{P}^N \! \in \mathcal{C}^1 ( \bar{I}; \mathbb{R}^N )$ to the system in~\eqref{eq:galerkin_PN}.
Moreover, by weakly differentiating \cref{eq:galerkin_PN-a}, we find 
\begin{align}
\tfrac{\rmd^2}{\rmd t^2} \vct{P}^N\! (t) &= (\matr{D}^N)^{-1}  \Bigl( - \matr{K}^N \tfrac{\rmd }{\rmd t} \vct{P}^N\! (t) +  \tfrac{\rmd }{\rmd t} \vct{q}^N\! (t)  - (\matr{A}^{\!N})^\rmt ( \matr{E}^N )^{-1} \tfrac{\rmd^2 }{\rmd t^2 } \vct{f}^N\! (t) \Bigr) 
\label{eq:dt2PN}%
\end{align}%
and hence $\vct{P}^N \!\in H^2 (I ; \mathbb{R}^N ) $.
Then, with \cref{eq:galerkin_PN_UN-a}, we also have $\vct{U}^N \in H^2 (I ; \mathbb{R}^N )$. $\qed$
\end{proof}
Next, we derive a~priori estimates for the unique solution of the Galerkin system~\eqref{eq:galerkin}. 
\begin{lemma}
\label{lem:apriori_sd} 
Given \Cref{asm:biot_sd}, the solutions $p^N \in H^2 ( I ; \Phi^N ) $ and $\vct{u}^N \in H^2 (I ; \vct{V}^N)$ of the Galerkin system~\eqref{eq:galerkin} satisfy the estimates
\begin{subequations}
\begin{align}
\norm[\big]{ \vct{u}^N }_{W^{1, \infty } (I ; \vct{H}^1 (D ))}  &\lesssim 1 , \\
\norm[\big]{p^N  }_{ W^{1, \infty } (I ; L^2 (D))}  + \norm[\big]{ \nabla  p^N}_{ L^\infty (I ; L^2 (D) )} + \norm[\big ]{\partial_t \nabla p^N }_{L^2 ( I ; \vct{L}^2 (D) ) } &\lesssim 1 ,
\end{align}%
\end{subequations}%
where the implicit constants are independent from~$N$.
\end{lemma}
\begin{proof} 
The proof is divided into three steps. 
The first step establishes uniform-in-time bounds for the displacement vector~$\vct{u}^N$ and pressure~$p^N$ and an $L^2$-in-time bound for the pressure gradient~$\nabla p^N$. 
The second step yields $L^2$-in-time bounds for the time derivatives and an improved uniform-in-time bound for the pressure gradient. 
Finally, the third step establishes uniform-in-time bounds for the time derivatives and an $L^2$-in-time bound for the time derivative of the pressure gradient.

\emph{Step~1.} We choose $\vct{v} = \partial_t \vct{u}^N$ and $\phi = p^N$ in \cref{eq:galerkin}. 
Then, with \ref{asm:alpha_sd} in \Cref{asm:biot_sd}, we obtain
\begin{equation}
\begin{multlined}[c][0.875\displaywidth]
\frac{1}{2} \frac{\rmd }{\rmd t} \Bigl( \bigl\langle \tsr{C} \matr{e} (\vct{u}^N ) , \matr{e} ( \vct{u}^N ) \bigr\rangle_{\! D } + \bigl\langle \omega p^N \! , p^N \bigr\rangle_{\! D} \Bigr) + \bigl\langle \matr{K} \nabla p^N\!  , \nabla p^N \bigr\rangle_{\! D } \\
  = \bigl\langle \vct{f} , \partial_t \vct{u}^N \bigr\rangle_{\! D} + \bigl\langle q , p^N \bigr\rangle_{\! D} + \bigl\langle G \matr{\alpha } , \partial_t \nabla \vct{u}^N \bigr\rangle_{\! D} .
\label{eq:apriori_sd_1}%
\end{multlined}%
\end{equation}%
Integrating from $0$ to $t \in \bar{I}$ in \cref{eq:apriori_sd_1} yields
\begin{equation}
\begin{multlined}[c][0.875\displaywidth]
\frac{1}{2} \Bigl( \bigl\langle \tsr{C} \matr{e} (\vct{u}^N  ) , \matr{e} ( \vct{u}^N  ) \bigr\rangle_{\! D }  + \bigl\langle \omega p^N \! , p^N  \bigr\rangle_{\! D} \Bigr) \Bigr\vert_t + \int_0^t \bigl\langle \matr{K} \nabla p^N\! , \nabla p^N \bigr\rangle_{\! D }\rmd \bar{t} \\
= \frac{1}{2} \Bigl( \bigl\langle \tsr{C} \matr{e} (\vct{u}^N_0  ) , \matr{e} ( \vct{u}^N_0  ) \bigr\rangle_{\! D  }  + \bigl\langle \omega p^N_0 \! , p^N_0  \bigr\rangle_{\! D } \Bigr) \\
+ \int_0^t \Bigl(  \bigl\langle \vct{f} , \partial_t \vct{u}^N \bigr\rangle_{\! D } + \bigl\langle q , p^N \bigr\rangle_{\! D } + \bigl\langle G \matr{\alpha } , \partial_t \nabla \vct{u}^N \bigr\rangle_{\! D} \Bigr)\, \rmd \bar{t} .
\label{eq:apriori_sd_2}%
\end{multlined}%
\end{equation}%
Now, with \ref{asm:omega_sd} and~\ref{asm:C_sd} in \Cref{asm:biot_sd} and by applying Korn's and Poincaré's inequality, we have 
\begin{align*}
\Bigl( \bigl\langle \tsr{C} \matr{e} (\vct{u}^N  ) , \matr{e} ( \vct{u}^N  ) \bigr\rangle_{\! D }  + \bigl\langle \omega p^N \! , p^N  \bigr\rangle_{\! D } \Bigr) \Bigr\vert_t \gtrsim \norm[\big]{ \vct{u}^N\! (t) }_{\vct{H}^1 (D )}^2 + \norm[\big]{p^N\! (t) }^2_{L^2 (D)} .
\end{align*}
Moreover, on the right-hand side of~\eqref{eq:apriori_sd_2}, integration by parts and Young's inequality yield the estimate 
\begin{align*}
&\int_0^t  \Bigl(  \bigl\langle \vct{f} , \partial_t \vct{u}^N \bigr\rangle_{\! D} + \bigl\langle q , p^N \bigr\rangle_{\! D} + \bigl\langle G \matr{\alpha } , \partial_t \nabla \vct{u}^N \bigr\rangle_{\! D}  \Bigr) \,\rmd \bar{t} \\
 &\quad = \Bigl( \bigl\langle \vct{f } , \vct{u}^N \bigr\rangle_{\! D} + \bigl\langle G\matr{\alpha }, \nabla \vct{u}^N \bigr\rangle_{\! D } \Bigr) \Bigr\vert_0^t + \int_0^t \Bigl( \bigl\langle q , p^N \bigr\rangle_{\! D} - \bigl\langle \partial_t \vct{f} , \vct{u}^N \bigr\rangle_{\! D } \Bigr) \,\rmd \bar{t} \\
 &\quad \lesssim\norm[\big]{\vct{u}^N_0 }_{\vct{H}^1 (D)} +  \delta^{-1} + \delta \Bigl(  \norm[\big]{p^N }_{ L^2 (I ; L^2 (D))}^2 +  \norm[\big]{\vct{u}^N\! (t)}_{\vct{H}^1 (D)}^2 + \norm[\big]{\vct{u}^N}_{L^2 (I ; \vct{L}^2 (D))}^2 \Bigr) 
\end{align*}
for any $\delta > 0$.
In addition, choosing~$\vct{v} = \vct{u}_0^N$ in \cref{eq:galerkin-d} and applying Korn's and Poincaré's inequality yields
\begin{align*}
\norm[\big ]{\vct{u}^N_0}_{\vct{H}^1 (D) }  \lesssim\norm{ \vct{f} (0 )}_{\vct{L}^2 (D) } + \norm{G \matr{\alpha }}_{\vct{L}^2 (D) } + \norm{p_0}_{L^2 (D) } .
\end{align*}
Thus, by using \ref{asm:K_sd} in \Cref{asm:biot_sd}, \cref{eq:apriori_sd_2} now yields the inequality
\begin{align}
\begin{split}
&(1 - \delta ) \norm[\big]{ \vct{u}^N\! (t) }_{\vct{H}^1 (D )}^2 + \norm[\big]{p^N \! (t) }^2_{L^2 (D)} +  \norm[\big]{\nabla p^N}_{L^2 ( 0, t; \, \vct{L}^2 (D)) } \\
&\hspace{3.5cm} \lesssim 1 +   \delta^{-1} + \delta \Bigl( \norm[\big]{\vct{u}^N}_{L^2 (I; \vct{L}^2 (D))}^2 + \norm[\big]{p^N }_{L^2 (I; L^2 (D))}^2 \Bigr) 
\end{split}%
\label{eq:apriori_sd_3}
\end{align}%
for any~$\delta > 0$. 
By integrating \cref{eq:apriori_sd_3} from~$0$ to~$T$ and choosing~$\delta > 0$ sufficiently small, we find
\begin{align}
\norm[\big]{ \vct{u}^N }_{L^2 (I ; \vct{H}^1 (D ))}^2 + \norm[\big]{p^N  }^2_{L^2 (I ; L^2 (D))} \lesssim 1 . \label{eq:apriori_sd_4}
\end{align}
Inserting \eqref{eq:apriori_sd_4} back into \cref{eq:apriori_sd_3} yields
\begin{align*}
\norm[\big]{ \vct{u}^N }_{L^\infty (I ; \vct{H}^1 (D ))}^2 + \norm[\big]{p^N  }^2_{L^\infty (I ; L^2 (D))} + \norm[\big]{\nabla p^N}_{L^2 ( I; \vct{L}^2 (D)) } \lesssim 1 .
\end{align*}

\emph{Step~2}. Next, we differentiate~$\vct{U}^N$ in~\cref{eq:galerkin_PN_UN-a} and find that 
\begin{align}
\bigl\langle \tsr{C} \matr{e} (\partial_t \vct{u}^N  ), \matr{e} ( \vct{v} ) \bigr\rangle_{\! D } - \bigl\langle \partial_t p^N\! \matr{\alpha} , \nabla \vct{v} \bigr\rangle_{ \! D } &= \bigl\langle \partial_t \vct{f} , \vct{v} \bigr\rangle_{\! D } \label{eq:apriori-dtun}
\end{align}
holds for all $\vct{v} \in \vct{V}^N$. 
Choosing $\vct{v} = \partial_t \vct{u}^N$ in \cref{eq:apriori-dtun} and $\phi = \partial_t p^N$ in \cref{eq:galerkin-c} yields
\begin{equation}
\begin{multlined}[c][0.875\displaywidth]
\bigl\langle \tsr{C} \matr{e} ( \partial_t \vct{u}^N ) , \matr{e} ( \partial_t \vct{u}^N ) \bigr\rangle_{\! D} + \bigl\langle \omega \partial_t p^N \! , \partial_t p^N \bigr\rangle_{\! D}  + \frac{1}{2} \frac{\rmd }{\rmd t} \bigl\langle \matr{K} \nabla p^N\!  , \nabla p^N \bigr\rangle_{\! D } \\
 = \bigl\langle \partial_t \vct{f} , \partial_t \vct{u}^N \bigr\rangle_{\! D } + \bigl\langle q , \partial_t p^N \bigr\rangle_{\! D }  .
 \label{eq:apriori_sd_5}
\end{multlined}
\end{equation}
Further, we integrate from~$0$ to~$t \in \bar{I}$ in \cref{eq:apriori_sd_5}. 
Then, by using \ref{asm:omega_sd}, \ref{asm:C_sd}, and~\ref{asm:K_sd} in \Cref{asm:biot_sd} and by applying Young's inequality on the right-hand side and Korn's and Poincaré's inequality for the first term on the left-hand side of~\cref{eq:apriori_sd_5}, we obtain 
\begin{align*}
\norm[\big]{\partial_t \vct{u}^N }^2_{L^2 ( 0, t ; \, \vct{H}^1 (D)) } + \norm[\big]{\partial_t p^N }^2_{L^2 (0, t ; \, L^2 (D)) } + \norm[\big]{\nabla p^N \! (t) }^2_{\vct{L}^2 (D) } \lesssim 1 + \bigl\langle \matr{K} \nabla p^N_0 \! , \nabla p^N_0 \bigr\rangle_{\! D} .
\end{align*}
Further, by definition of the space~$\Phi^N$, we have~$\nabla \cdot ( \matr{K} \nabla p^N ) \in \Phi^N $.  
Consequently, by using \cref{eq:galerkin-a}, \ref{asm:p0_sd} in \Cref{asm:biot_sd}, and integration by parts, we find
\begin{align*}
\bigl\langle \matr{K} \nabla p^N_0 , \nabla p^N_0 \bigr\rangle_{\! D } = - \bigl\langle \nabla \cdot \bigl( \matr{K } \nabla p^N_0 \bigr) , p_0 \bigr\rangle_{\! D } = \bigl\langle \matr{K} \nabla p^N_0 , \nabla p_0 \bigr\rangle_{\! D } 
\end{align*}
and hence $\norm{\nabla p^N_0 }_{\vct{L}^2 (D) } \le \norm{ \nabla  p_0 }_{\vct{L}^2 (D) }$.
Thus, we have 
\begin{align}
\norm[\big]{\partial_t \vct{u}^N }^2_{L^2 ( I ; \, \vct{H}^1 (D)) } + \norm[\big]{\partial_t p^N }^2_{L^2 (I ; \, L^2 (D)) } + \norm[\big]{\nabla p^N }^2_{L^\infty ( I ; \vct{L}^2 (D)) } \lesssim 1 . \label{eq:apriori_sd_6}
\end{align}

\emph{Step~3}. 
We weakly differentiate \cref{eq:galerkin_PN_UN-b} in time and obtain
\begin{align}
 \bigl\langle \omega \partial_{tt} p^N + \partial_{tt} \nabla \cdot ( \matr{\alpha} \vct{u}^N ) , \phi \bigr\rangle_{\! D  } +  \bigl\langle \matr{K}  \partial_t \nabla p^N \! , \nabla \phi \bigr\rangle_{\! D }  =  \bigl\langle \partial_t q , \phi \bigr\rangle_{\! D } 
 \label{eq:apriori_dttpN}
\end{align}
for all $\phi \in \Phi^N$ and almost all~$t\in I$.
By integrating from $0$ to~$t\in\bar{I}$ and choosing $\vct{v} = \partial_{tt} \vct{u}^N$ in \cref{eq:apriori-dtun} and $\phi = \partial_t p^N$ in \cref{eq:apriori_dttpN}, we find 
\begin{align}
\begin{split}
&(1-\delta ) \norm[\big ]{\partial_t  \vct{u}^N \! (t) }^2_{\vct{H}^1 (D)} + \norm[\big ]{\partial_t p^N \! (t) }^2_{L^2 (D)} + \norm[\big]{\partial_t \nabla p^N }_{L^2 (0,t; \vct{L}^2 (D))}^2 \\
&\hspace{2cm}\lesssim 1 + \delta^{-1} + \norm[\big ]{\partial_t \vct{u}^N \! (0) }^2_{\vct{H}^1 (D)} + \norm[\big ]{\partial_t p^N \! (0) }^2_{L^2 (D)} \\
&\hspace{4cm}+ \delta\Bigl( \norm[\big ]{\partial_t \vct{u}^N }^2_{L^2 (I ; \vct{L}^2 (D))} + \norm[\big ]{\partial_t p^N }^2_{L^2 (I ; L^2 (D))}  \Bigr)
\end{split}%
\label{eq:apriori_sd_7}%
\end{align}%
for any~$\delta > 0$, where we have used \ref{asm:fq_sd}, \ref{asm:omega_sd}, \ref{asm:C_sd}, and~\ref{asm:K_sd} in \Cref{asm:biot_sd}, applied Korn's, Poincaré's, and Young's inequality, and used the relation
\begin{align*}
\int_0^t \bigl\langle \partial_t \vct{f} , \partial_{tt} \vct{u}^N \bigr\rangle_{\! D} \,\rmd \bar{t} = \bigl\langle \partial_t \vct{f} , \partial_t \vct{u}^N \bigr\rangle_{\! D} \Bigr\vert_0^t -\int_0^t \bigl\langle \partial_{tt} \vct{f} , \partial_t \vct{u}^N \bigr\rangle_{\! D } \,\rmd\bar{t}  .
\end{align*}
Besides, using the continuous embedding $H^1 (I ; X ) \hookrightarrow \mathcal{C}^0 (\bar{I} ; X )$ for a Banach space~$X$, as well as \Cref{lem:sdH2} and \ref{asm:fq_sd} in \Cref{asm:biot_sd}, we can evaluate   \cref{eq:apriori_sd_5} at $t=0$. This yields 
\begin{multline*}
\norm[\big ]{\partial_t \vct{u}^N \! (0) }_{\vct{H}^1 (D)}^2 + \norm[\big ]{\partial_t p^N \! ( 0 ) }_{L^2 (D)}^2 \\ 
\lesssim  \norm[\big ]{\partial_t \vct{f} (0)}_{\vct{L}^2 (D) }^2 + \norm{q (0) }_{L^2(D)}^2 + \norm[\big]{\nabla \cdot \bigl( \matr{K} \nabla p^N_0 \bigr) }_{L^2 (D)}^2 ,
\end{multline*}
where we have used \ref{asm:C_sd} and \ref{asm:omega_sd} in  \Cref{asm:biot_sd}, applied Korn's, Poincaré's, and Young's inequality, and integrated by parts to obtain
\begin{align*}
\bigl\langle \matr{K} \nabla p^N_0 , \partial_t \nabla p^N\! (0) \bigr\rangle_{\! D } &= - \bigl\langle \nabla \cdot \bigl( \matr{K} \nabla p^N_0 \bigr) , \partial_t p^N \! (0) \bigr\rangle_{\! D }  .
\end{align*}
Using \cref{eq:eigenfunction}, we find $\nabla \cdot ( \matr{K} \nabla p_0^N ) = \Pi_{\Phi^N} ( \nabla \cdot (\matr{K} \nabla p_0 ))$
and, consequently, obtain the inequality $\norm{\nabla \cdot ( \matr{K} \nabla p^N_0 ) }_{L^2 (D) } \le \norm{\nabla \cdot ( \matr{K} \nabla p_0 ) }_{L^2 (D) }$.
Now, with \cref{eq:apriori_sd_6}, \cref{eq:apriori_sd_7} yields 
\begin{align*}
\norm[\big]{\partial_t \vct{u}^N}_{L^\infty (I ; \vct{H}^1 (D))} + \norm[\big ]{\partial_t p^N }_{L^\infty (I ; L^2 (D))} + \norm[\big ]{\partial_t \nabla p^N }_{L^2 ( I ; \vct{L}^2 (D))} \lesssim 1 .
\end{align*} 
\end{proof}
With the estimates in \Cref{lem:apriori_sd} we can now take the limit~$N\rightarrow \infty$ in the Galerkin system~\eqref{eq:galerkin} and prove~\Cref{thm:wellposed_biot_sd}.
\begin{proof}[Proof of \Cref{thm:wellposed_biot_sd}]
As a consequence of \Cref{lem:apriori_sd}, there exist functions 
\begin{align*}
p \in H^1 ( I ; H^1_{0;\partial D_\rmD} \! (D)) \cap L^\infty ( I ; H^1 (D)) \cap W^{1,\infty } (I ; L^2 (D) ) \quad\text{and}\quad \vct{u} \in W^{1,\infty } (I ; \vct{H}^1_0 (D))
\end{align*}
and a strictly increasing subsequence~$\{ N_k \}_{k \in \mathbb{N}} \subset \mathbb{N}$ such that 
\begin{subequations}
\begin{alignat}{2}
\label{eq:pN_conv_sd} p^{N_k} &\rightharpoonup p \quad\enspace &&\text{weakly in } H^1 \bigl( I ; H^1 (D) \bigr)  , \\
\label{eq:uN_conv_sd} \vct{u}^{N_k} &\rightharpoonup \vct{u} \quad\enspace &&\text{weakly in } H^1 \bigl( I ; \vct{H}^1 (D) \bigr) 
\end{alignat}
\end{subequations}
as $k \rightarrow \infty$. 
By taking the limit $N_k \rightarrow \infty$ in the Galerkin system~\eqref{eq:galerkin}, we find that the limit functions~$p$ and~$\vct{u}$ solve the Biot system~\eqref{eq:biot_sd_weak}.

It remains to prove uniqueness. 
For this, we consider a solution~$(p, \vct{u} )$ to the homogeneous system corresponding to \cref{eq:biot_sd_weak} (i.e., $\vct{f} \equiv \vct{0}$, $q \equiv 0$, $p_0 \equiv 0$, $\vct{u}_0 \equiv \vct{0}$ and $G \equiv 0$).
Choosing $\vct{v} = \partial_t \vct{u}$ and $\phi = p$ in \cref{eq:biot_sd_weak} and integrating from $0$ to~$t\in\bar{I}$ yields
\begin{multline*}
\norm{\vct{u} (t) }^2_{\vct{H}^1 (D) } + \norm{p (t) }^2_{L^2 (D) } \lesssim \tfrac{1}{2} \bigl\langle \tsr{C} \matr{e} ( \vct{u} (t) ) , \matr{e} ( \vct{u} (t) ) \bigr\rangle_{\! D } \\
 +  \tfrac{1}{2} \bigl\langle \omega p (t), p(t) \bigr\rangle_{\! D} + \int_0^t \bigl\langle \matr{K} \nabla p , \nabla p \bigr\rangle_{\! D } \,\rmd \bar{t} = 0
\end{multline*}
and hence $p\equiv 0 $ and $\vct{u} \equiv \vct{0}$ so that uniqueness follows.
\end{proof}

\pagebreak[4]
\section{List of Symbols} 
\label{sec:B} 
{
\small
\subsection{Physical Quantities} 
\subsubsection{Dimensional Quantities}
\begin{longtable}[l]{L{0.08\textwidth}L{0.565\textwidth}L{0.255\textwidth}}
Symbol & Meaning & Introduction \\
\hline \\[-0.95em]
\endfirsthead

Symbol & Meaning & Introduction \\
\hline \\[-0.95em]
\endhead
$\dimq{a}$ & aperture & \Cref{sec:2.2} \\
$\dimq{a}_\pm $ & left/right aperture & \Cref{sec:2.2} \\
$\dimq{\tsr{C}}_\pmf$ & elasticity tensor in~$\dimq{\Omega}_\pmf$ &  \cref{eq:fulldim-dimensional} \\
$\dimq{\vct{f}}_\pmf$ & body force term in~$\dimq{\Omega}_\pmf$ &  \cref{eq:fulldim-dimensional} \\
$\dimq{\vct{g}}$ & gravitational acceleration & \cref{eq:Gdim} \\
$\dimq{G}_\pmf$ & gravitational term in $\dimq{\Omega}_\pmf$ &  \cref{eq:Gdim} \\
$\dimq{\matr{K}}_\pmf$ & hydraulic conductivity matrix in $\dimq{\Omega}_\pmf$ & \cref{eq:fulldim-dimensional} \\
$\dimq{p}_\pmf$ & pressure head in $\dimq{\Omega}_\pmf$ & \cref{eq:fulldim-dimensional} \\
$\dimq{p}_{0, \pmf}$ & initial pressure head in $\dimq{\Omega}_\pmf$ & \cref{eq:fulldim-dimensional} \\
$\dimq{q}_\pmf$ & fluid source term in $\dimq{\Omega}_\pmf$ & \cref{eq:fulldim-dimensional} \\
$\dimq{t}$ & time & \Cref{sec:2.3} \\
$\dimq{T}$ & final time & \Cref{sec:2.3} \\
$\dimq{\vct{u}}_\pmf$ & displacement vector in~$\dimq{\Omega}_\pmf$ & \cref{eq:fulldim-dimensional} \\
$\dimq{\vct{u}}_{0,\pmf}$ & initial displacement vector in $\dimq{\Omega}_\pmf$ & \cref{eq:fulldim-dimensional} \\
$\dimq{\vct{x}}$ & position vector & \Cref{sec:2.3} \\
$\matr{\alpha}_\pmf$ & Biot tensor in $\dimq{\Omega}_\pmf$ & \cref{eq:fulldim-dimensional} \\
$\dimq{\pi}_\pmf$ & pressure in $\dimq{\Omega}_\pmf$ & \cref{eq:pi} \\
$\dimq{\matr{\sigma}}_\pmf$ & stress tensor in $\dimq{\Omega}_\pmf$ & \cref{eq:fulldim-dimensional} \\
$\dimq{\omega}_\pmf$ & inverse Biot modulus in $\dimq{\Omega}_\pmf$ & \cref{eq:fulldim-dimensional}
\end{longtable}

\subsubsection{Non-Dimensional Quantities} 
\begin{longtable}[l]{L{0.08\textwidth}L{0.565\textwidth}L{0.255\textwidth}}
Symbol & Meaning & Introduction \\
\hline \\[-0.95em]
\endfirsthead

Symbol & Meaning & Introduction \\
\hline \\[-0.95em]
\endhead
$a$ & aperture  & \cref{eq:nondim2} \\
$a_\pm$ & left/right aperture &  \cref{eq:nondim2} \\
$A$ & integrated aperture & \cref{eq:Klm1} \\
$\tsr{C}_\pmf^\epsilon$ & elasticity tensor in~$\Omega_\pmf^\epsilon$ & \cref{eq:nondim2} \\
$\vct{f}_\pmf^\epsilon$ & body force term in~$\Omega_\pmf^\epsilon$ &  \cref{eq:nondim2} \\
$\vct{g}$ & gravitational acceleration &  \cref{eq:nondim2} \\
$G_\pmf^\epsilon$ & gravitational term in~$\Omega_\pmf^\epsilon$ &  \cref{eq:nondim2} \\
$\matr{K}_\pmf^\epsilon$ & hydraulic conductivity matrix in~$\Omega_\pmf^\epsilon$ &  \cref{eq:nondim2} \\
$p_\pmf^\epsilon$ & pressure head in~$\Omega_\pmf^\epsilon$ &  \cref{eq:fulldim-nondim}, \cref{eq:fulldim-weak} \\
$p_{0, \pmf}^\epsilon$ & initial pressure head in~$\Omega_\pmf^\epsilon$ &  \cref{eq:nondim2} \\
$q_\pmf^\epsilon$ & fluid source term in~$\Omega_\pmf^\epsilon$ &  \cref{eq:nondim2} \\
$t$ & time & \cref{eq:xt} \\
$T$ & final time &  \cref{eq:nondim2} \\
$\vct{u}_\pmf^\epsilon$ & displacement vector in~$\Omega_\pmf^\epsilon$ & \cref{eq:fulldim-nondim}, \cref{eq:fulldim-weak} \\
$\vct{u}_{0,\pmf}^\epsilon$ & initial displacement vector in~$\Omega_\pmf^\epsilon$ & \cref{eq:fulldim-nondim}, \cref{eq:fulldim-weak} \\
$\vct{x}$ & position vector & \cref{eq:xt} \\
$\matr{\alpha}_\pmf^\epsilon$ & Biot tensor in~$\Omega_\pmf^\epsilon$ & \cref{eq:eps_scaling2} \\
$\matr{\sigma}_\pmf^\epsilon$ & stress tensor in~$\Omega_\pmf^\epsilon$ & \cref{eq:fulldim-nondim} \\
$\omega_\pmf^\epsilon$ & inverse Biot modulus in~$\Omega_\pmf^\epsilon$ &  \cref{eq:nondim2} \\
\end{longtable}

\subsubsection{Non-Dimensional Transformed Quantities}
\begin{longtable}[l]{L{0.08\textwidth}L{0.565\textwidth}L{0.255\textwidth}}
Symbol & Meaning & Introduction \\
\hline \\[-0.95em]
\endfirsthead

Symbol & Meaning & Introduction \\
\hline \\[-0.95em]
\endhead
$\hat{\tsr{C}}_\pmf^\epsilon$ & elasticity tensor in~$\Omega_\pmf^\odot$ & \cref{eq:paramtrafo} \\
$\hat{\matr{C}}_\rmf^{\!\vct{N}}$ & normal elasticity tensor in~$\Omega_\rmf^1$ & \cref{eq:Ceff} \\ 
$\hat{\matr{C}}_\gamma^{\!\vct{N}}$ & effective normal elasticity tensor on~$\gamma$ & \cref{eq:Cgamma} \\
$\hat{\vct{f}}_\pmf^\epsilon$ & body force term in $\Omega_\pmf^\odot$ & \cref{eq:paramtrafo} \\
$\hat{\vct{f}}_\rmf^\eff$ & effective normal body force term in~$\Omega_\rmf^1$ & \cref{eq:feff} \\
$\hat{\vct{F}}_\gamma^\eff$ & effective body force term on~$\gamma$ & \cref{eq:Fgamma} \\
$\hat{G}_\pmf^\epsilon $ & gravitational term in~$\Omega_\pmf^\odot$ & \cref{eq:paramtrafo}, \Cref{lem:G}  \\
$\hat{\matr{K}}_\pmf^\epsilon$ & hydraulic conductivity matrix in~$\Omega_\pmf^\odot$ & \cref{eq:paramtrafo} \\
$\hat{K}_\rmf^{\!\vct{N}}$ & normal hydraulic conductivity in~$\Omega_\rmf^1$ & \cref{eq:KfN} \\
$\hat{\matr{K}}_\gamma$ & effective tangential hydraulic conductivity on~$\gamma$ & \cref{eq:Kgamma} \\
$\hat{p}_\pmf^\epsilon$ & pressure head in~$\Omega_\pmf^\odot$ & \cref{eq:fulldim-weak-trafo} \\
$\hat{p}_\pmf^\#$ & limit pressure head in~$\Omega_\pmf^\odot$ & \Cref{cor:conv} \\
$\hat{p}_{0,\pmf}^\epsilon$ & initial pressure head in~$\Omega_\pmf^\odot$ & \cref{eq:paramtrafo} \\
$\hat{q}_\pmf^\epsilon$ & fluid source term in~$\Omega_\pmf^\odot$ & \cref{eq:paramtrafo} \\
$\hat{q}_\rmf^\eff$ & effective fluid source term in~$\Omega_\rmf^1$ & \cref{eq:qeff} \\
$\hat{Q}_\gamma^\eff$ & effective fluid source term on~$\gamma$ & \cref{eq:Qgamma} \\
$\hat{\vct{u}}_\pmf^\epsilon$ & displacement vector in~$\Omega_\pmf^\odot$ & \cref{eq:fulldim-weak-trafo} \\
$\hat{\vct{u}}_\pmf^\#$ & limit displacement vector in~$\Omega_\pmf^\odot$ & \Cref{cor:conv}\\
$\hat{\vct{u}}_{0,\pmf}^\epsilon$ & initial displacement vector in~$\Omega_\pmf^\odot$& \cref{eq:fulldim-weak-trafo} \\
$\hat{\vct{u}}_{0,\pmf}^\#$ & initial limit displacement vector in~$\Omega_\pmf^\odot$ & \Cref{lem:u0conv} \\
$\hat{\matr{\alpha}}_\pmf^\epsilon$ & Biot tensor in~$\Omega_\pmf^\odot$ & \cref{eq:paramtrafo} \\
$\hat{\vct{\alpha}}_\rmf^\eff$ & effective normal Biot vector in~$\Omega_\rmf^1$ & \cref{eq:alphaeff} \\
$\hat{\xi}^\#$ & limit of $\epsilon^{-1} \partial_{\!\vct{N}} \hat{p}_\rmf^\epsilon$ for $\nu_\matr{K} = -1$ & \cref{eq:xi} \\
$\hat{\matr{\sigma}}_\pm$ & stress tensor in~$\Omega_\pm^0$ & \cref{eq:bulklimit} \\
$\hat{\matr{\sigma}}_\rmf$ & stress tensor in~$\Omega_\rmf^1$ & \cref{eq:mechlimit_nuC1}, \cref{eq:mechlimit_nuCg1} \\
$\hat{\matr{\sigma}}_\gamma$ & effective stress tensor on~$\gamma$ & \cref{eq:mechlimit_nuC1_df} \\
$\hat{\omega}_\pmf^\epsilon$ & inverse Biot modulus in~$\Omega_\pmf^\odot$ & \cref{eq:paramtrafo}  \\
$\hat{\omega}_\rmf^\eff$ & effective inverse Biot modulus in~$\Omega_\rmf^1$ & \cref{eq:omegaeff}
\end{longtable}

\subsubsection{Reference Quantities} 
\begin{longtable}[l]{L{0.08\textwidth}L{0.565\textwidth}L{0.255\textwidth}}
Symbol & Meaning & Introduction \\
\hline \\[-0.95em]
\endfirsthead

Symbol & Meaning & Introduction \\
\hline \\[-0.95em]
\endhead
$\refq{a}$ & characteristic fracture aperture & \Cref{sec:2.4} \\
$\refq{C}_\rmb$  & characteristic bulk elasticity & \Cref{sec:2.4} \\
$\refq{C}_\rmf$ & characteristic fracture elasticity & \cref{eq:eps_scaling1} \\
$\refq{f}_\rmb$ & characteristic bulk body force & \Cref{sec:2.4} \\
$\refq{f}_\rmf$ & characteristic fracture body force & \cref{eq:eps_scaling1} \\
$\refq{g}$ & gravitational acceleration $\refq{g} = \abs{\dimq{\vct{g}}}$ & \Cref{sec:2.3} \\
$\refq{K}_\rmb$ & characteristic bulk hydraulic conductivity & \Cref{sec:2.4} \\
$\refq{K}_\rmf$ & characteristic fracture hydraulic conductivity & \cref{eq:eps_scaling1} \\
$\refq{L}$ & characteristic longitudinal fracture extension & \Cref{sec:2.4} \\
$\refq{q}_\rmb$ & characteristic bulk fluid source & \Cref{sec:2.4} \\
$\refq{q}_\rmf$ & characteristic fracture fluid source & \cref{eq:eps_scaling1} \\
$\refq{t}$ & characteristic time & \Cref{sec:2.4} \\
$\refq{\rho}$ & fluid density & \Cref{sec:2.3} \\
$\refq{\omega}_\rmb$ & characteristic bulk inverse Biot modulus & \Cref{sec:2.4} \\
$\refq{\omega}_\rmf$ & characteristic fracture inverse Biot modulus & \cref{eq:eps_scaling1}
\end{longtable}

\subsubsection{Scaling Quantities} 
\begin{longtable}[l]{L{0.08\textwidth}L{0.565\textwidth}L{0.255\textwidth}}
Symbol & Meaning & Introduction \\
\hline \\[-0.95em]
\endfirsthead

Symbol & Meaning & Introduction \\
\hline \\[-0.95em]
\endhead
$\epsilon$ & characteristic fracture width-to-length ratio & \cref{eq:epsilon} \\
$\epsilon^\ast$ & maximum characteristic width-to-length ratio & \Cref{rem:Meps} \\
$\nu_\tsr{C}$ & scaling coefficient of elasticity tensor & \cref{eq:eps_scaling1} \\
$\nu_\vct{f}$ & scaling coefficient of body force term & \cref{eq:eps_scaling1} \\
$\nu_\matr{K}$ & scaling coefficient of hydraulic conductivity matrix & \cref{eq:eps_scaling1} \\
$\nu_q$ & scaling coefficient of fluid source term & \cref{eq:eps_scaling1} \\
$\matr{\nu}_{\!\matr{\alpha}}$ & scaling matrix of Biot tensor & \cref{eq:nualpha} \\
$\nu_{\matr{\alpha}}^\smallpar $ & tangential scaling coefficient of Biot tensor & \cref{eq:nualpha} \\
$\nu_{\matr{\alpha}}^\smallperp$ & normal scaling coefficient of Biot tensor & \cref{eq:nualpha} \\
$\nu_\omega$ & scaling coefficient of inverse Biot modulus & \cref{eq:eps_scaling1} \\
\end{longtable} 

\subsection{Operators and Functions} 
\begin{longtable}[l]{L{0.08\textwidth}L{0.565\textwidth}L{0.255\textwidth}}
Symbol & Meaning & Introduction \\
\hline \\[-0.95em]
\endfirsthead

Symbol & Meaning & Introduction \\
\hline \\[-0.95em]
\endhead
$\mathfrak{A}_{\!\vct{N}}$ & normal averaging operator in~$\Omega_\rmf^1$ & \Cref{def:average} \\
$\mathfrak{A}_\rmf$ & total averaging operator in~$\Omega_\rmf^1$ & \Cref{def:average} \\
$\dimq{\matr{e}}  $ & symmetrized Jacobian (dimensional) & \cref{eq:edim} \\
$\matr{e}$ & symmetrized Jacobian (non-dimensional) & \cref{eq:fulldim-nondim} \\
$\matr{e}^\epsilon  $ & $\epsilon$-scaled symmetrized Jacobian & \cref{eq:straineps} \\
$\matr{e}_{\smallpar } $ & tangential symmetrized Jacobian & \cref{eq:fracconv1_proof_2} \\
$\matr{e}_{\!\vct{N}}  $ & normal symmetrized Jacobian & \Cref{thm:strongconv} \\
$\matr{E}_\pmf^\epsilon$ & symmetrized Jacobian with transformation terms & \cref{eq:Eeps} \\
$\matr{M}_\pmf^\epsilon$ & transformation matrix associated with $\vct{T}_\pmf^\epsilon$ & \cref{eq:Meps} \\
$\mathfrak{T}_\pm$ & trace operator in $H^1_{\!\vct{N}} (\Omega_\rmf^1)$ & \Cref{lem:trace_H1N} \\
$\vct{T}_\pmf^\epsilon$ & coordinate transformation & \cref{eq:trafo} \\
$\theta  $ & $\theta (\nu ) := \max \{ 0 , \tfrac{1}{2} ( \nu - 1) \}$ & \Cref{thm:apriori} \\
$\iota $ & $\iota (\nu) := \tfrac{1}{2} (\nu + 1)$ & \Cref{thm:apriori} \\
$\Pi_\pm$ & $L^2$-projection from $\gamma_\pm^1$ onto $\gamma$ & \cref{eq:proj} \\
$\dimq{\nabla} $ & gradient/Jacobian (dimensional) & \Cref{sec:2.3} \\
$\nabla $ & gradient/Jacobian (non-dimensional) & \cref{eq:nabla} \\
$\nabla^\epsilon$ & $\epsilon$-scaled gradient/Jacobian & \cref{eq:nablaeps} \\
$\nablapar$ & tangential gradient/Jacobian & \cref{eq:nablaeps} \\
$\nablaperp$ & normal gradient/Jacobian & \cref{eq:nablaeps} \\
$\partial_{\!\vct{N}}$ & normal derivative & \cref{eq:nablaeps} \\
$\norm{\cdot}_\mathrm{i}$ & alternative norm on~$\Phi$ & \cref{eq:norm_i_ii}, \cref{eq:normi_2} \\
$\norm{\cdot}_\mathrm{ii}$ & alternative norm on~$\Lambda$ & \cref{eq:norm_i_ii} \\
$\normiii{\cdot }$ & alternative norm on $\smash{\vct{H}^1_{0, \rho_\pm^0}\! (\Omega_\pm^0 )} \times  \matr{L}^2_{\mathrm{sym}} (\Omega_\rmf^1 )$ & \cref{eq:normiii_new} \\
$\fracfac{\cdot }$ & fracture factor operator & Notations~\ref{not:fracfac},~\ref{not:odot} \\
$\jump{\cdot}_\gamma$ & jump operator across~$\gamma$ & \cref{eq:jump} \\
$\langle \cdot , \cdot \rangle_\epsilon$ & scalar product in~$\Lambda^\epsilon$\!, $\vct{\Lambda}^\epsilon$\!, $\matr{\Lambda}^\epsilon$ & \cref{eq:inner_Lambda} \\
$\langle \cdot , \cdot \rangle $ & scalar product in~$\Lambda$, $\vct{\Lambda}$, $\matr{\Lambda}$ & \Cref{sec:2.6} 
\end{longtable}

\pagebreak[4]
\subsection{Sets} 
\subsubsection{Dimensional Sets} 
\begin{longtable}[l]{L{0.08\textwidth}L{0.565\textwidth}L{0.255\textwidth}}
Symbol & Meaning & Introduction \\
\hline \\[-0.95em]
\endfirsthead

Symbol & Meaning & Introduction \\
\hline \\[-0.95em]
\endhead
$\dimq{I}$ & time interval & \Cref{sec:2.3} \\
$\dimq{\gamma}$ & fracture interface & \Cref{sec:2.2} \\
$\dimq{\gamma}_\pm$ & boundary between $\dimq{\Omega}_\pm$ and $\dimq{\Omega}_\rmf$ & \cref{eq:gammapm} \\
$\dimq{\rho}_{\pmf}$ & outer boundary of~$\partial\dimq{\Omega}_{\pmf}$ & \cref{eq:rho} \\
$\dimq{\rho}_{\pmf, \rmD}$ & Dirichlet boundary of~$\partial\dimq{\Omega}_\pmf$ & \cref{eq:rhoDN} \\
$\dimq{\rho}_{\pmf , \rmN}$ & Neumann boundary of~$\partial\dimq{\Omega}_\pmf$ & \cref{eq:rhoDN} \\
$\dimq{\Omega}$ & overall domain & \cref{eq:omega} \\
$\dimq{\Omega}_\pm$ & bulk subdomains & \cref{eq:Omega+_dim}, \cref{eq:Omega-_dim} \\
$\dimq{\Omega}_\rmf$ & fracture subdomain & \cref{eq:Omegaf_dim} \\
\end{longtable}

\subsubsection{Non-Dimensional Sets} 
\begin{longtable}[l]{L{0.08\textwidth}L{0.565\textwidth}L{0.255\textwidth}}
Symbol & Meaning & Introduction \\
\hline \\[-0.95em]
\endfirsthead

Symbol & Meaning & Introduction \\
\hline \\[-0.95em]
\endhead
$I$ & time interval & \cref{eq:nondim1} \\
$\gamma$ & fracture interface & \cref{eq:nondim1} \\
$\gamma_\pm^\epsilon $ & boundary between $\Omega_\pm^\epsilon$ and $\Omega_\rmf^\epsilon$ & \cref{eq:nondim1} \\
$\partial\gamma_\rmD$ & Dirichlet boundary of~$\gamma$ & \cref{eq:partialgamma} \\
$\partial\gamma_\rmN$ & Neumann boundary of~$\gamma$ & \cref{eq:partialgamma} \\
$\rho_\pmf^\epsilon $ & outer boundary of~$\partial \Omega_\pmf^\epsilon$ & \cref{eq:nondim1} \\
$\rho_{\pmf,\rmD}^\epsilon$ & Dirichlet boundary of~$\partial \Omega_\pmf^\epsilon$ & \cref{eq:nondim1} \\
$\rho_{\pmf,\rmN}^\epsilon $ & Neumann boundary of~$\partial\Omega_\pmf^\epsilon$ & \cref{eq:nondim1} \\
$\rho_\rmD^\epsilon$ & overall Dirichlet boundary & \cref{eq:spaces_eps} \\
$\Omega^\epsilon$ & overall domain & \cref{eq:nondim1}, \cref{eq:Omegaeps} \\
$\Omega_\pm^\epsilon$ & bulk subdomains & \cref{eq:nondim1}, \cref{eq:Omegapmf_eps} \\ 
$\Omega_\rmf^\epsilon$ & fracture subdomain & \cref{eq:nondim1}, \cref{eq:Omegapmf_eps} \\ 
\end{longtable}

\subsection{Function Spaces} 
\begin{longtable}[l]{L{0.24\textwidth}L{0.405\textwidth}L{0.255\textwidth}}
Symbol & Meaning & Introduction \\
\hline \\[-0.95em]
\endfirsthead

Symbol & Meaning & Introduction \\
\hline \\[-0.95em]
\endhead
$\mathcal{C}^{0,1} (D)$ & Lipschitz functions  & \Cref{sec:2.2} \\
$H^1 (D)$, $\vct{H}^1 (D) $  & Sobolev spaces & \Cref{sec:2.1} \\
$H^1_0 (D)$, $\vct{H}^1_0 (D) $  & Sobolev spaces, zero trace on $\partial D$ & \Cref{sec:2.1} \\
$H^1_{0,\Gamma} (D)$, $\vct{H}^1_{0,\Gamma} (D) $  & Sobolev spaces, zero trace on $\Gamma $ & \Cref{sec:2.1} \\
$\vct{H}_\mathrm{div} (D)$ & Sobolev space for divergences & \cref{eq:spaces_eps} \\
$H^1_{\!\vct{N}} (D)$, $\vct{H}^1_{\!\vct{N}} (D)$ & Sobolev spaces for normal derivatives & \cref{eq:H1N} \\ 
$H^1 (J; X ) $ & Sobolev-Bochner space & \Cref{sec:2.1} \\
$L^r (D)$, $\vct{L}^r (D)$, $\matr{L}^r (D)$ & Lebesgue spaces & \Cref{sec:2.1} \\
$\matr{L}^2_{\mathrm{sym}} (D)$ & symmetric $L^2$-space & \cref{eq:L2sym} \\
$L^r ( J ; X) $ & Lebesgue-Bochner space & \Cref{sec:2.1} \\
$\vct{V}^\epsilon$ & displacement space in~\eqref{eq:fulldim-weak} & \cref{eq:spaces_eps_a} \\
$\vct{V}$ & displacement space in~\eqref{eq:fulldim-weak-trafo} & \cref{eq:V} \\
$\vct{V}^\#$ & closure of $\vct{V}$ in $\vct{H}^1 (\Omega_\pm^0) \times \vct{H}^1_{\!\vct{N}} (\Omega_\rmf^1 ) $& \Cref{cor:conv} \\
$\vct{V}_1$ & displacement space for $\nu_\tsr{C} = 1$ & \cref{eq:V1} \\
$\vct{V}_{>1}$ & displacement space for $\nu_\tsr{C} > 1$ & \cref{eq:Vg1} \\
$W$ & admissible values for $p_\gamma$ in~\eqref{eq:weak_C1_Klm1} & \cref{eq:W} \\
$W^{1,r} (D)$ & Sobolev space & \Cref{sec:2.1}  \\
$W^{1,r} (J;X ) $ & Sobolev-Bochner space & \Cref{sec:2.1} \\ 
$\Lambda^\epsilon$, $\vct{\Lambda}^\epsilon$, $\matr{\Lambda}^\epsilon$ & product $L^2$-space on~$\Omega_\pmf^\epsilon$ & \cref{eq:spaces_eps_d} \\
$\Lambda$, $\vct{\Lambda}$, $\matr{\Lambda}$ & product $L^2$-space on~$\Omega_\pmf^\odot$ & \cref{eq:Lambda}  \\
$\vct{\Xi}^\epsilon$ & velocity space in~\eqref{eq:fulldim-weak} & \cref{eq:spaces_eps_c} \\
$\Phi^\epsilon$ & pressure head space in~\eqref{eq:fulldim-weak} & \cref{eq:spaces_eps_b} \\
$\Phi$ & pressure head space in~\eqref{eq:fulldim-weak-trafo} & \cref{eq:Phi} \\
$\Phi^\#$ & closure of~$\Phi$ in $H^1(\Omega_\pm^0) \times H^1_{\!\vct{N}} (\Omega_\rmf^1 )$ & \Cref{cor:conv} \\
$\Phi_{<-1}$ & pressure head space for $\nu_\matr{K} < -1$ & \cref{eq:Philm1} \\
$\Phi_{-1}$ & pressure head space for $\nu_\matr{K} = -1$ & \cref{eq:Phim1} \\
$\Phi_{(-1,1) }$ & pressure head space for $\nu_\matr{K} \in (-1,1 )$ & \cref{eq:Phim1g1} \\
$\Phi_{\ge 1}$ & pressure head space for $\nu_\matr{K} \ge 1$ & \cref{eq:Phi1} 
\end{longtable}

\subsection{Other Symbols} 
\begin{longtable}[l]{L{0.08\textwidth}L{0.565\textwidth}L{0.255\textwidth}}
Symbol & Meaning & Introduction \\
\hline \\[-0.95em]
\endfirsthead

Symbol & Meaning & Introduction \\
\hline \\[-0.95em]
\endhead
$\matr{I}_n$ & $n$-dimensional identity matrix & \cref{eq:nualpha} \\
$n$      & dimension & \Cref{sec:2.1} \\
$\vct{n}$ & unit normal vector on~$\partial \Omega$ (resp.~$\partial\dimq{\Omega}$) & \cref{eq:fulldim-dimensional} \\
$\vct{n}_\pm$ & unit normal vector on~$\gamma_\pm^\epsilon$ (resp.~$\dimq{\gamma}_\pm$) & \cref{eq:fulldim-dimensional} \\
$\vct{N}$ & unit normal vector on~$\gamma$ (resp.~$\dimq{\gamma}$) & \cref{eq:Omegaf_dim} \\
$\lesssim$, $\gtrsim$ & inequality up to multiplicative constant & \Cref{sec:2.1} \\
$\pm$, $\pmf$ & shorthand for the pair~$(+, -)$ and triple $(+ , -, \rmf )$ & Notations~\ref{not:pmf},~\ref{not:fracfac}, \ref{not:odot} \\
$\odot$ & bulk-fracture index: $0$ in~$\Omega_\pm^0$ and $1$ in~$\Omega_\rmf^1$ & \Cref{not:odot} \\
\end{longtable} 
}

\normalsize
\end{appendices}

\printbibliography

\end{document}